\documentclass[a4paper,11pt,twoside]{article}
\usepackage[utf8]{inputenc}
\usepackage[T1]{fontenc}
\usepackage[english]{babel}
\usepackage{charter}
\usepackage{indentfirst}
\usepackage{xcolor}
\usepackage{verbatim}
\usepackage{tikz}
\usepackage{hyperref}
\usepackage{pifont}
\usepackage{bm}

\usepackage[top=3.cm, bottom=4.0cm, left=2.2cm, right=2.2cm]{geometry}
\usepackage{amsmath}
\usepackage{amsthm}	%package pour les styles des ?nonc?s
\usepackage{amsfonts}	%pourles polices mathematiques vraisemblablement
\usepackage{amssymb, bbm, units}
\usepackage{enumerate}
\usepackage[nobysame, alphabetic, initials]{amsrefs}
\usepackage{ulem} 

\usepackage{hyperref}
\usepackage{url}
\usepackage[caption=false]{subfig}
\numberwithin{equation}{section}
\numberwithin{figure}{section}
 \usepackage[nodayofweek]{datetime}
 
\renewcommand*{\thefootnote}{\fnsymbol{footnote}}

\theoremstyle{plain}
\newtheorem{theorem}{Theorem}[section]
\newtheorem{lemma}[theorem]{Lemma}
\newtheorem{proposition}[theorem]{Proposition}
\newtheorem{corollary}[theorem]{Corollary}

\newtheorem{definition}[theorem]{Definition}
\newtheorem{remark}[theorem]{Remark}

\newtheorem{assumption}[theorem]{Assumption}

\newcommand{\cB}{\mathcal{B}}

\newcommand{\cF}{\mathcal{F}}

\newcommand{\cK}{\mathcal{K}}
\newcommand{\cL}{\mathcal{L}}
\newcommand{\cM}{\mathcal{M}}

\newcommand{\cP}{\mathcal{P}}

\newcommand{\cS}{\mathcal{S}}

\newcommand{\cU}{\mathcal{U}}

\newcommand{\AAA}{\mathbb{A}}
\newcommand{\CC}{\mathbb{C}}
\newcommand{\DD}{\mathbb{D}}
\newcommand{\EE}{\mathbb{E}}
\newcommand{\FF}{\mathbb{F}}
\newcommand{\II}{\mathbb{I}}
\newcommand{\NN}{\mathbb{N}}
\newcommand{\PP}{\mathbb{P}}
\newcommand{\RR}{\mathbb{R}}
\newcommand{\ZZ}{\mathbb{Z}}

\DeclareMathOperator*{\argmax}{arg\,max}

\makeatletter
\newsavebox{\@brx}
\newcommand{\llangle}[1][]{\savebox{\@brx}{\(\m@th{#1\langle}\)}%
  \mathopen{\copy\@brx\mkern2mu\kern-0.9\wd\@brx\usebox{\@brx}}}
\newcommand{\rrangle}[1][]{\savebox{\@brx}{\(\m@th{#1\rangle}\)}%
  \mathclose{\copy\@brx\mkern2mu\kern-0.9\wd\@brx\usebox{\@brx}}}
\makeatother

\newcommand{\toP}{\stackrel{p}{\rightarrow}}

\definecolor{darkgreen}{rgb}{0,0.35,0}

\title{Many-Server Queueing Systems with Heterogeneous Strategic Servers in Heavy Traffic}
\author{Burak B\"{u}ke\textit{$^{a}$}, Gon\c calo dos Reis\textit{$^{a,b}$} and Vadim Platonov\textit{$^{a}$}}
 \date{%
    \footnotesize 
    $^{a}$~School of Mathematics, University of Edinburgh, The King's Buildings, Edinburgh, UK
    \\
    $^{b}$~Centro de Matem\'atica e Aplica\c c$\tilde{\text{o}}$es (CMA), FCT, UNL, Portugal
    \\[2ex]
    \longdate \today \ (\currenttime)
}

\begin{document}

\selectlanguage{english}

\maketitle
%%%%%%%%%%%%%%%%%%%%%%%%%%%%%%%%%%%%%%%%%%%%%%%%%%%%%%%%%%%%%%%%%
%%%%%%%%%%%%%%%%%%%%%%%%%%%%%%%%%%%%%%%%%%%%%%%%%%%%%%%%%%%%%%%%%
%%%%%%%%%%%%%%%%%%%%%%%%%%%%%%%%%%%%%%%%%%%%%%%%%%%%%%%%%%%%%%%%%
%%%%%%%%%%%%%%%%%%%%%%%%%%%%%%%%%%%%%%%%%%%%%%%%%%%%%%%%%%%%%%%%%
%%%%%%%%%%%%%%%%%%%%%%%%%%%%%%%%%%%%%%%%%%%%%%%%%%%%%%%%%%%%%%%%%
\renewcommand*{\thefootnote}{\arabic{footnote}}

%%%%%%%%%%%%%%%%%%%%%%%%%%%%%%%%%
\vspace{-1cm}
\begin{abstract} 
In most service systems, the servers are humans who desire to experience a certain level of idleness. In call centers, this manifests itself as the call avoidance behavior, where servers strategically adjust their service rate to strike a balance between the idleness they receive and effort to work harder. Moreover, being humans, each server values this trade-off differently and has different capabilities. Drawing ideas from mean-field game theory, we develop a novel framework, relying on measure-valued processes,  to simultaneously address strategic server behavior and inherent server heterogeneity in service systems. 
This framework enables us to extend the literature on strategic servers in four new directions by: (i) incorporating individual choices of servers, (ii) incorporating individual abilities of servers, (iii) modeling the discomfort experienced by servers due to low levels of idleness, and (iv) considering more general routing policies. Using our framework, we are able to asymptotically characterize asymmetric Nash equilibria for many-server systems with strategic servers. 

In simpler cases, it has been shown that the purely quality-driven regime is asymptotically optimal when the servers are strategic. However, we show that if the discomfort increases fast enough as the idleness approaches zero, the quality-and-efficiency-driven regime and other quality driven regimes can be optimal. This is the first time this conclusion appears in the literature. 
\end{abstract}
%%%%%%%%%%%%%%%%%%%%%%%%%%%%%%%%%

\vspace{0.3cm}


\footnotesize
\setcounter{tocdepth}{2}
\tableofcontents
\normalsize

\section{Introduction}

% The design and optimization of many-server service systems has been the subject of growing interest in the last two decades, due to the economic importance of these systems and the challenging problems they pose. With their worldwide market size of \$339.4 billion in 2020 and an expected market size around \$500 billion in 2027, call centers constitute a perfect example of these systems~\citep{statista2020}. Around 76\% of these call centers serve institutions which generate revenue more than \$250 million per year and more than 13\% serve institutions with an annual revenue of \$25 billion~\citep{deloitte2017}. Hence, it is no surprise that even an improvement of a second in service time can yield saving millions of dollars annually~\citep{Gans2010}. 
The design and optimization of many-server service systems has been the subject of vigorous research in the last two decades due to the economic importance of these systems and the challenging problems underpinning them. With a worldwide market size of \$339.4 billion in 2020 and an expected market size of \$500 billion by 2027, call centers constitute a perfect example of these systems~\cite{statista2020}.  Around 76\% of these call centers serve institutions that generate revenues above \$250 million/year and more than 13\% serve institutions with an annual revenue of \$25 billion~\cite{deloitte2017}. Hence, it is no surprise that a small improvement in the service time can lead to millions of dollars in savings~\cite{Gans2010}. 

% Despite the rapid development of technology, 82\% of the US and 74\% of the non-US customers want more human interaction \citep[see][]{pwc2018} and the vast majority of the interactions in call centers are carried out by human agents. The presence of this human factor  poses interesting modeling challenges for researchers. For example, the agents should be able to experience a certain level of idleness in between calls in order to avoid burnout and keep their motivation higher. Failure to provide these idle times will yield agents strategically slowing down the service in order to create additional idle time, which is generally referred as \emph{call avoidance behavior}~\citep{go4cust21}. For example, agents might choose  to delay pressing the terminate call button when a call ends or is transferred, and in systems where the call is routed to the longest available agent, agents can become unavailable for a short period of time to put themselves at the end of the list~\citep[see e.g.][]{CallCentreHelper17}. As a more extreme example of call avoidance behavior,~\cite{Gans2003} observe that a significant number of calls last nearly zero seconds, which indicates that some agents are strategically hanging up on the customers to create extra idle time~\citep[see also][]{Gopalakrishnan_etal2016}. 
Despite the rapid technological developments, 82\% of the US and 74\% of the non-US customers want more human interaction \cite{pwc2018} and the vast majority of the interactions in call centers are carried out by human agents (servers). The presence of this human factor poses interesting modeling challenges. One such challenge involves ensuring that agents experience adequate idle time between calls to prevent burnout and maintain high motivation levels. Failure to provide these idle times yields in servers strategically slowing down their service in order to create additional idle time -- this is generally referred as \emph{call avoidance behavior}~\cite{go4cust21}. 
For example, servers might choose to delay pressing the terminate call button when a call ends or is transferred, and in systems where the call is routed to the server that has been available for longest, servers can become unavailable for a short period of time to put themselves at the end of the list see e.g.~\cite{CallCentreHelper17}. As a more extreme example of this behavior,~\cite{Gans2003} observe that a significant number of calls last nearly zero seconds, which indicates that some servers are strategically hanging up on the customers to create extra idle time.

% Starting with~\cite{Gopalakrishnan_etal2016}, there has been an increasing research interest in the design of many-server systems with strategic servers to create incentives for agents to work faster. The main stream of work concentrates on the analysis of Markovian many-server queues and aims to analyze equilibria using the stationary probabilities for such systems. Hence, the characterization of the stationary probabilities constitutes a key step in this analysis. To the best of our knowledge, the only existing result in this direction is in~\cite{Gumbel1960}, where the stationary probabilities for many-server queues with heterogeneous servers is provided under a random routing policy where the customers are routed to one of the idle agents uniformly at random. Unfortunately, even for this special case, the equations for the stationary probabilities are extremely complicated, rendering an exact analysis of a general case intractable. Hence, the recent work on strategic servers mainly concentrates on identifying symmetric equilibria by setting all service rates except for the rate for a tagged server to be equal.

Starting with~\cite{Gopalakrishnan_etal2016}, there has been an increasing research interest in modeling many-server systems with strategic servers to determine incentives for servers to work faster. The main stream of work concentrates on the analysis of Markovian queues and aims to analyze equilibria using the stationary probabilities. Hence, the characterization of the stationary probabilities constitutes a key step in this analysis. To the best of our knowledge, the only existing result in this direction is in~\cite{Gumbel1960}, where the stationary probabilities for systems with heterogeneous servers is provided under a random routing policy where the customers are routed to an idle server chosen uniformly at random. Unfortunately, even for this special case, the equations for the stationary probabilities are extremely involved, rendering an exact analysis of the general case intractable. As a result, the recent work on strategic servers mainly concentrates on identifying symmetric equilibria by setting all service rates equal except for the rate of a tagged server.

% In addition to the strategic behavior, another consequence of servers being humans is agents having inherent individual differences. Each agent has a different capability to serve customers, values idle time differently and  one expects each agent to serve with a different rate in equilibrium if it exists. Hence, it is crucial to analyze asymmetric equilibria in which agents serve with heterogeneous rates to capture the human aspect of the problem better. 
Furthermore, beyond strategic behavior, another consequence of servers being humans is the presence of inherent individual differences and preferences among servers. Each server possesses unique capabilities in serving customers, values idle time differently and, at equilibria, one expects each server to serve with a different rate if such equilibrium exists. It is then crucial for the next generation of models to capture these behaviors and analyze the so-called \textit{asymmetric equilibria} in which servers serve with heterogeneous rates. %to capture the human aspect of the problem better. 

% The key step in the analysis of asymmetric equilibria is identifying how idleness is distributed among agents with heterogeneous rates, which heavily depends on the routing policy. In a recent paper,~\cite{bukeqin2019} propose the use of a measure-valued process which they refer as the \emph{fairness process} to capture the individual differences in service rates. The fairness process indicates the distribution of idleness under a given routing policy. \cite{bukeqin2019} prove that in a modified version of Halfin-Whitt scaling~\citep{halfinwhitt1981}, the sequence of fairness processes are tight and  converge in probability to a constant deterministic process for some common routing policies. Then, they use the fairness process to prove diffusion limits for many-server systems with heterogeneous servers. In this work, we use the fairness process as our main technical tool to analyze many-server systems with strategic heterogeneous servers. 
The key step in the analysis of asymmetric equilibria is identifying how idleness is distributed among servers featuring heterogeneous service rates, which heavily depends on the routing policy used by the call center. In a recent paper,~\cite{bukeqin2019} propose the use of a measure-valued process, referred to as the \emph{fairness process}, to capture the individual differences in service rates as a means to obtain access to the distribution of server idleness under a given routing policy. \cite{bukeqin2019} prove that in a modified version of Halfin-Whitt scaling~\cite{halfinwhitt1981}, the sequence of fairness processes is tight and converges in probability to a constant deterministic process for some common routing policies. Then, they use the fairness process to prove diffusion limits for many-server systems with heterogeneous servers. In this work, we use the fairness process as our main technical tool to analyze many-server systems with strategic heterogeneous servers. 

The literature from human resource management indicates that the perception of fairness has a significant effect on employee performance~\cites{colquittetal2001,cohencharash2001}. The distribution of idleness among servers is generally perceived as a measure of fairness. An interesting question regarding the distribution of idleness is whether one can design routing policies in order to distribute the idleness among servers according to servers' service rates in a pre-specified manner. For example, it is well-known that if the fastest-server-first policy is employed, i.e., the customers are routed to the idle server with the highest service rate, only the slowest servers experience idleness~\cites{Armony2005, Atar2008, bukeqin2019}. As another example,~\cite{bukeqin2019} show that if customers are routed using a class of policies called totally blind policies, where the server to which a customer is routed is asymptotically independent of the service rate of that server, then the cumulative idleness servers receive is proportional to their service rate. Totally blind policies include longest-idle-server-first, where customers are routed to the server that has been available the longest, and random routing as studied in~\cite{Gumbel1960}. Several routing policies have been designed in the literature for service systems with pools of servers to ensure that each pool asymptotically receives a certain proportion of idleness~\cites{Armony2005, atarschwartzshaki11, armward10, gurwhitt09, reedshaki2015, wararm13}. However, to the best of our knowledge, the question of how to design policies to distribute idleness on an individual server level remains an open problem.

%{We need to talk about the quality-driven regime vs qed regime}

\subsection{Our Contributions}\label{sec:contributions}

Our aim is to push the frontier for the analysis of service systems with strategic servers. Our contributions can be described under two major headings.

\subsubsection*{Technical Contributions}
\begin{enumerate}
\item \emph{A General Framework to Study Heterogeneous Strategic Servers.} In this paper, we introduce a novel framework to study systems with strategic servers and extend the results in four major directions. First, we suggest a utility function to model the individual preferences of servers regarding the trade-off between utility of idleness and cost of working faster. Second, we address the inherent differences between the abilities of servers using random individual maximum and minimum attainable service rates. Third, we model the contribution of percentage of idleness as a concave increasing function of idleness, which we refer as the \emph{utility of idleness function} to better reflect the marginal benefit of idleness and allows us to model the discomfort experienced by the servers. Finally, we characterize the Nash equilibria for service rate distribution under a more general class of routing policies.

\item \emph{Service Systems with Heterogeneity in Quality-Driven Regimes.} In Section~\ref{sec:convergence_system}, we provide fluid limits for many-server queues with heterogeneous service rates in quality-driven regimes. A key question regarding fluid and diffusion approximations is whether these approximations can be used to study the stationary behavior of the systems, which require an interchange of limits argument. We prove that this is indeed the case for both quality- and quality-and-efficiency driven regimes. 

\item \emph{Analysis of Generalized Random Routing Policies.} The random routing policy, where an incoming job is assigned to one of the idle servers uniformly at random, is well-understood. However, to the best of our knowledge, there are no results in the literature where incoming customers are randomly assigned to servers with probabilities that depend on the service rates.  
In Section~\ref{sec:grr_policy}, we present a detailed asymptotical analysis of generalized random routing policies where the routing probabilities can be a function of the service rate and characterize the distribution of idleness among servers with heterogeneous rates under these policies in a quality-driven system. %As we show in Section~\ref{sec:convergence_system}, this result can be used to analyze the stationary behavior of these many-server systems. 
Moreover, using another interchange of limits argument, we also show that this result can also be used to characterize the long run percentage a server with a given service rate is idle. 
\end{enumerate}
\noindent\subsubsection*{Practical Contributions and Managerial Insights}
\begin{enumerate}
    \item \emph{Fair Distribution of Idleness Among Individual Servers in a Quality-Driven Regime.} 
    In the quality-and-efficiency-driven regime, it is possible to design routing policies to attain any pre-specified distribution of idleness among server pools~\cites{atarschwartzshaki11, armward10, reedshaki2015, wararm13}. 
    We show that in the quality-driven regime, where the safety staffing scales linearly with the offered load, this result no longer holds and it may not be possible to design routing policies to attain certain idleness distributions. We characterize necessary conditions for an idleness distribution to be attainable and also prove that the `strict inequality' versions of these conditions are also sufficient by designing general random routing policies to attain the desired distribution of idleness. Moreover, we show that one can use these policies to distribute idleness at an individual server level.

    \item \emph{The Relationship Between the Optimal Staffing Regime and Servers' Sensitivity to Idleness.} When the servers are homogeneous and idleness contributes to the utility as itself, i.e., the utility of the idleness function is identity, \cite{Gopalakrishnan_etal2016} prove that it is optimal to scale the safety staffing linearly with the offered load. In Theorem~\ref{thm:best_response_scaling}, we prove that the main factor in establishing this result is the sensitivity of servers to low levels of idleness, which can be expressed in terms of the derivative of the utility of idleness function near zero and extend the optimality result to a more general class of functions. Moreover, we show that if the servers are sufficiently sensitive to low levels of idleness, then it is asymptotically optimal to adopt a quality-and-efficiency-driven regime.

    \item \emph{Characterization of the Equilibrium Service Rate Distribution.} When servers are heterogeneous and strategic, a central practical question concerns how the equilibrium service rate distribution is influenced by primitive model assumptions. 
    In Section~\ref{sec:strategic_grr}, we use our framework to derive an equation that characterizes the equilibrium service rate distribution under any generalized random routing policy. 
    Through numerical experiments, we demonstrate that our general framework is very effective in estimating the equilibrium service rate for large many-server systems. 
    Moreover, we present an extensive numerical study to understand how various model parameters affect the shape and scale of the equilibrium service rate distribution.

\end{enumerate}

\subsection{Notation}\label{sec:notation}

In this work, we assume that all the stochastic processes and the random variables lie in the probability space $(\Omega, \cF, \PP)$. We use the shorthand notation `w.p.~1' instead of `with probability 1' with respect to $\PP$. We denote the set of real numbers, set of positive real numbers and the positive integers as $\RR, \RR_+$ and $\NN$, respectively and use $\cB(\RR_+)$ to denote the Borel $\sigma$-algebra on $\RR_+$. We use $\to$, $\toP$ and $\Rightarrow$ to denote convergence in $\RR$, convergence  in probability and convergence in measure, respectively. To emphasize that a certain system parameter $x$ is random, we use the '$\sim$' notation as $\tilde{x}$. Also, we define $(x)^{+}:=\max\{x,0\}$ and $(x)^-:=\max\{0,-x\}$. We define $\CC_{[\alpha,\beta]}^b[0,\infty)$ to be the set of continuous bounded functions defined on the interval $[\alpha,\beta]$, $\iota(\cdot)$ to be the identity function and $\II$ to be the indicator function where $\II(A)$ is 1 if $A$ holds and $0$ otherwise. The spaces of right-continuous functions and left-continuous functions are endowed with the Skorokhod-$J_1$ topology. For any measure $\eta$, $\langle f,\eta\rangle:=\int f(x)d\eta(x)$, and hence, if $\eta$ is a probability measure $\langle \iota,\eta\rangle$ denotes the first moment of $\eta$. With a slight abuse of notation, we use $F$ to denote both the law and the cumulative distribution of random variable, i.e., $F(A)$ denotes the probability of set $A$ and $F(a)=F((-\infty, a])$ for any real argument $a$. We also denote the Dirac measure which assigns a unit measure to point $x$ as $\delta_x(\cdot)$. Finally, for any stochastic process $X(t)$, we use $X(\infty)$ to denote a random variable distributed with its stationary distribution. For measure-valued processes, we find that it is more appropriate to use $\eta_t$ and $\eta_\infty$ notation instead of $\eta(t)$ and $\eta(\infty)$.

\section{Literature on Strategic Behavior in Queues}

The vast majority of the literature on queues with strategic players concentrates on the strategic behavior of customers while joining a queue. We refer the reader to the excellent monographs \cite{hassin_haviv_2003} and \cite{hassin_2020} for an extensive review of the literature on strategic customers. Motivated by the competition between two firms, two-server queues with strategic servers has also received attention in the literature see, e.g. \cites{Kalai1992, Gilbert1998, Cachon2002, Cachon2007}. More recently, \cite{Allon2010} extend this line of research to cover multiple firms each of which are operating as $M/M/N$ queues. 

The study of  servers who are behaving strategically within the same organization is relatively recent and to the best of our knowledge started with the seminal work of \cite{Gopalakrishnan_etal2016}. They assume that the servers (agents) are identically sensitive to the idleness they receive and there is a non-cooperative game between servers to set their service rates in order to maximize their utility which is defined as a combination of the server's idle time and effort to speed up the service. The system operator routes the customers to the servers with the aim of incentivizing servers using the idleness experienced -- they refer to this as \emph{incentive-aware routing}. In an $M/M/N$ queue setting, they show the existence of symmetric equilibrium under policies which route customers to servers according to their idle time. To solve the problem, \cite{Gopalakrishnan_etal2016} explicitly assume that the servers need to choose service rates in a way that ensures system stability. \cite{armony2021capacity} show that this assumption is not necessary using Tarski's intersection theorem. \cite{gopalakrishnan2021} studies the above results comparing queueing systems with pooled and dedicated queues under \emph{r-routing policies}, a class of randomized policies which routes the customers to servers randomly with probabilities proportional to the $r$th power of the service rate. He provides analytical results for systems with dedicated queues, but due to their complexity he only studies the pooled queues numerically. One of the major contributions of our work is to provide analytical tools for the asymptotic study of pooled queues. In addition to incentive-aware routing, one can also provide monetary incentives based on the service speed. \cite{zhan2019staffing} characterize a joint routing and payment policy to optimize performance under the assumption that there is a trade-off between service speed and service quality. \cite{zhong2021} analyze equilibrium in loss systems with finite buffer capacity when both incentive-aware routing and payment incentives are used. 
An interesting recent work by \cite{Bayraktar2019} uses mean-field games ideas to study strategic servers with dedicated queues and designs a dynamic arrival rate control policy. 
% \hl{should we mention their use of mean-field games? loosely relating to what we do?}. 
To the best of our knowledge, the approach we offer is novel, in particular, for the application of mean-field games to queueing systems. 

\section{A Many-Server Queueing Model with Strategic and Heterogeneous Servers}\label{sec:setting}

We consider a sequence of queueing systems where arrivals at the $n$th system follow a Poisson process with rate $\lambda^n$. We have the following standard assumption on arrival rates.
\begin{assumption}\label{asm:lambdascaling}
There exists a $0<\bar{\lambda}<\infty$ such that $n^{-1/2}(\lambda^n-n \bar{\lambda})\to 0$ as $n\to \infty$. 
\end{assumption}
If there are idle servers upon the arrival of a customer, the customer is routed to one of these servers according to a prespecified routing policy $\pi$. All servers in the system have the skills to serve any arriving customer, albeit with different rates. If all servers are busy, the customer waits in a queue. Customers are impatient with exponential($\gamma$) patience times, independent of other customers. If the patience time of a customer expires before the service commences, the customer abandons the system. Once the service commences, the customer only departs when the service finishes. We assume that customers are served on a first-come-first-served basis. 

For the system outlined above, the system operator decides on the routing policy $\pi$ and the staffing level. The routing policy determines how idleness is distributed among servers with different rates. At this point, we only assume that the routing policy is non-idling, i.e., there can only be customers in the queue if all the servers in the system are busy. 

Due to the non-cooperative game between servers, there is a circular relationship between staffing level and individual service rates. 
The staffing level is determined by the expected service rate, and the distribution of service rates is determined by the idleness experienced by each server with different rate, the latter being directly influenced by the staffing level and the routing policy. 
Our main goal in this work is to identify whether equilibria exist and characterize any equilibrium given the system operator's staffing and routing policies. As the servers differ individually in terms of their capabilities and preferences, the resulting equilibria will be asymmetric, in general. 

The system operator aims to set the staffing level to be the offered load (the minimum number of staff required to stabilize the system if there were no abandonments), and a  safety staffing proportional to a power $\alpha$ of the offered load, where $1/2\leq \alpha\leq 1$, to achieve a certain quality of service. Suppose that the server $k$ in the $n$th system serves with rate $\tilde{\mu}_k^n$. If the system operator knows the distribution of these service rates to be $F$ with an expected value $\bar{\mu}_F$, she sets the staffing level for the $n$th system to be
\begin{equation}
    N_\alpha^n=\frac{\lambda^n}{\bar{\mu}_F}+\beta\left(\frac{\lambda^n}{\bar{\mu}_F}\right)^{\alpha},
    \label{eq:staffing_level}
\end{equation}
where $\beta>0$ see, e.g., \cite{Gans2003}. Following the classification in \cite{garnettetal02}, we say that the system operates in a \textit{quality-driven} (QD) regime when $1/2<\alpha\leq 1$ and in a quality-and-efficiency-driven (QED) regime when $\alpha=1/2$. In the literature, it is quite common to restrict the quality-driven regime to the case with $\alpha=1$ and this case also plays a major role in our study. Hence, to differentiate, we say that the system operates in a \emph{purely quality-driven regime} if $\alpha=1$. The system operator sets the staffing parameters, $\alpha,\beta$, and the routing policy $\pi$ in order to minimize an operating cost, $C_O(\alpha, \beta, \pi)$, which can be expressed as 
\begin{equation}\label{eq:operator_cost}
     C_O^n(\alpha, \beta, \pi) = c_SN_\alpha^n + c_W\lambda^n\EE[W^n(\infty)] + c_A\EE[\bar{R}^n(\infty)]
\end{equation}
where $c_S, c_W$  and $c_A$ are  the unit staffing cost, unit queueing time cost per customer and unit abandonment cost, respectively, $\EE[W^n(\infty)]$ is the long run average queueing time per customer and $\EE[\bar{R}^n(\infty)]$ is the long run abandonment rate from the system. 

On the other hand, each server in the system aims to set her service rate to maximize her own utility, where utility is defined as a trade-off between the expected long run proportion of idleness she experiences and an \emph{effort cost} she incurs for working faster. This setting is similar to \cite{Gopalakrishnan_etal2016} with two fundamental differences. First, we model the contribution of idleness to the utility as a concave increasing function of the long run percentage idleness experienced rather than the long run percentage idleness itself. In addition to modeling the decreasing marginal return of idleness, this approach also helps us model the situations where working without any breaks is unacceptable to the servers. Second, we model the trade-off between `benefits of idleness' and `effort cost' to be server specific by introducing a multiplicative random coefficient to the effort cost function. The idleness each server experiences depends ultimately on the service rates of all servers in the system. 
Letting $I_k^n(t)$ to be the idleness process of server $k$, where $I_k^n(t)$ takes the value $1$ if the server $k$ in system $n$ is idle  at time $t$ and is $0$ otherwise, we define the utility function of server $k$ in the $n$th system, serving with rate $\mu$, to be 
\begin{equation}
    U_k^n(\mu, F) = u_I\big(\EE[I_k^n(\infty)|\tilde{\mu}_k^n=\mu]\big)-\tilde{a}_k^n c(\mu),\label{eq:utility}
\end{equation}
% We refer to the function $f(\cdot)$ as the utility of idleness function and assume it to be concave and increasing. The effort cost function, $c(\cdot)$, is assumed to be convex and increasing. We also assume that both the utility of idleness function and the effort cost function are twice continuously differentiable. The coefficient $\tilde{a}_k^n$ is a positive random variable with distribution $F_a(\cdot)$ and density function $f_a(\cdot)$, and determines the personal preference of server $k$ in regards to the trade-off between utility of idleness and the effort cost. 
 where $F$ is the distribution of individual service rates and determines the expected stationary idleness of the given server. We refer to $u_I(\cdot)$ as the \textit{utility of idleness} function and assume it to be concave and increasing. The \textit{effort cost} function, $c(\cdot)$, is assumed to be convex and increasing. We also assume that both $u_I(\cdot)$ and $c(\cdot)$ are twice continuously differentiable. The coefficient $\tilde{a}_k^n$ is a positive random variable with distribution $F_a(\cdot)$ and determines the personal preference of server $k$ in regards to the trade-off between utility of idleness and the effort cost and use $f_a(\cdot)$ to denote the density function of $\tilde{a}_k^n$ when it is a continuous random variable.  We assume that each server $k$ deterministically knows her trade-off parameter $\tilde{a}_k^n$.

Each server $k$ in the $n$th system has an inherent minimum service rate, $\tilde{\mu}_{{\min},k}^n$, which the  server can achieve with minimal effort and an inherent maximum achievable service rate, $\tilde{\mu}_{{\max},k}^n$, which the server cannot improve upon by working faster regardless of how much effort she puts into her work. The system operator does not have prior information about these individual minimum and maximum rates and hence, they are modeled as random variables following a common joint distribution $F_{\min,\max}$ with  marginals $F_{\min}$ and $F_{\max}$. We assume that these individual minimum and maximum service rates are uniformly bounded, i.e., there exists $\mu_{\min}$ and $\mu_{\max}$ such that 
 \begin{equation}
     0<\mu_{\min}\leq \tilde{\mu}_{\min,k}^n\leq \tilde{\mu}_k^n\leq \tilde{\mu}_{\max,k}^n\leq \mu_{\max}<\infty, \quad \mbox{ for all }k,n\in \NN \quad \mbox{ w.p. }1.\label{eq:boundedrates}
 \end{equation}
We also assume independence among vectors $(\tilde{a}_k^n, \tilde{\mu}_{\min,k}^n, \tilde{\mu}_{\max,k}^n)$ for each $k$. 
% Without loss of generality, we assume that $\mu_{\min}=1$ by changing the time units as necessary.

\label{par:fixed_point_discussion}Given a distribution $F^{(0)}$ for the service rates and the system operator's decisions on $\alpha, \beta$ and $\pi$, each server $k$ in the $n$th system serves with the service rate $\tilde{\mu}_k^n$ that maximizes her utility, i.e.,  
\begin{equation}
\tilde{\mu}_k^n\in \argmax_{\mu\in[\tilde{\mu}_{{\min},k}^n, \tilde{\mu}_{{\max},k}^n]}U_k^n(\mu, F^{(0)}).\label{eq:mu_update_equation}
\end{equation}
As $\tilde{a}_k^n, \tilde{\mu}_{{\min},k}^n$ and $\tilde{\mu}_{{\max},k}^n$ are random and only known by server $k$, the resulting service rate is random and the distributions of these parameters determine the distribution $F^{(1)}$ of the  optimal service rates. Our goal in this work is to study the equilibrium service rate distribution $F$ by characterizing the fixed point where $F^{(0)}=F^{(1)}=F$, based on the given values of $\alpha, \beta$ and $\pi$. Our analysis also provides important insights on the optimal regime parameter $\alpha$ for the system operator. 

\subsection{Dynamics of the System Processes}\label{sec:systemdynamics}

In this section, we define the stochastic processes in our model. We denote the number of customers in the $n$th system at time $t$ as $X^n(t)$. For the $n$th system, the arrival process is a Poisson process with rate $\lambda^n$ and $A^n(t)$ denotes the number of arrivals by time $t$. The service time of a customer depends on the server and takes an exponential time with rate $\tilde{\mu}_k^n$ if the customer is served by server $k$. In this section, we assume that the service rates $\tilde{\mu}_k^n, 1\leq k\leq N_\alpha^n$, are i.i.d. random variables with a common known distribution $F$ having mean $\bar{\mu}_F$ and variance $\sigma_F^2$. %For the $n$th system, we denote the vector of service rates as $\boldsymbol{\mu}^n=(\tilde{\mu}_1^n,\ldots, \tilde{\mu_{N_\alpha^n}^n})$. 
In Section~\ref{sec:strategic}, we analyze in detail how this distribution is determined as a result of the strategic decisions of the servers. The departure process for server $k$  and the number of abandonments by time $t$ is denoted $D_k^n(t)$ and  $R_0^n(t)$, respectively. We have the following balance equation for the system length process $X^n(t)$, 
\begin{equation}
X^n(t) = X^n(0) + A^n(t) - \sum_{k=1}^{N_\alpha^n}D_k^n(t)-R_0^n(t), \quad \mbox{ for all }t\geq 0,\ n \in \NN.
\label{eq:basic_system_length}
\end{equation}
Taking $S_k^n(t)$ and $R^n(t)$ as independent unit rate Poisson processes for all $k$ and $n$, we can equivalently define the departure and the abandonment processes as 
\[
D_k^n(t)
:=S_k^n\left(\tilde{\mu}_k^n\int_0^t\big(1-I_k^n(s)\big)ds\right)
\quad \mbox{and}\quad 
R_0^n(t):=R^n\left(\gamma\int_0^t\big(X^n(s)-N^n_\alpha\big)^+ds\right),
\]
 and can write \eqref{eq:basic_system_length} as 
 \begin{equation}
 \label{eq:auxiliareqSystemLengthProcess}
 X^n(t) = X^n(0) + A^n(t) - \sum_{k=1}^{N_\alpha^n}S_k^n\left(\tilde{\mu}_k^n\int_0^t\big(1-I_k^n(s)\big)ds\right)-R^n\left(\gamma\int_0^t\big(X^n(s)-N^n_\alpha\big)^+ds\right),  
 \end{equation}
 for all $t\geq 0, n \in \NN$. The non-idling property of the routing policy implies 
 \[
 \left(X^n(t)-N^n_\alpha\right)^-=\sum_{k=1}^{N_\alpha^n}I^n_k(t), \mbox{ for all }t\geq 0,\  n\in \NN.
%   \qquad\boxed{\textrm{should it be }I^n_k?}
 \]
 For all $1/2\leq \alpha\leq 1$, we define the scaled processes $
  \hat{X}_\alpha^n(t):=n^{-\alpha}(X^n(t)-N_\alpha^n)$ and $\hat{I}_{k,\alpha}^n(t):=n^{-\alpha}I_k^n(t)$.
We also use the shorthand notation $\hat{I}_{\alpha}^n(t):=\sum_{k=1}^{N_\alpha^n}\hat{I}_{k,\alpha}^n(t)$ for the total scaled idleness in the system. 
Regarding the initial system length, we make the following assumption.

\begin{assumption}\label{asm:initialcondition}
The initial system lengths $\hat{X}_\alpha^n(0)$ are uniformly integrable and $\hat{X}_\alpha^n(0)\Rightarrow \xi_0$ as $n\to\infty$, where $\xi_0$ is an integrable random variable. 
\end{assumption}

We also have the following assumption on the routing policy.
\begin{assumption}\label{asm:non-discrimating}
Under the routing policy $\pi$, for any $0\leq k,j\leq N_\alpha^n$ and $\mu\in[\mu_{\min},\mu_{\max}]$,
\[
\EE[I_k^n(\infty)|\tilde{\mu}_k^n=\mu]=\EE[I_j^n(\infty)|\tilde{\mu}_j^n=\mu].
\]
\end{assumption}

Assumption~\ref{asm:non-discrimating} ensures that the routing policy does not discriminate against any server and all servers working with the same rate are expected to receive the same proportional idleness in the long run. Using symmetry arguments, it can be easily seen that any routing policy which does not explicitly use the server index $k$ satisfies this assumption. 

\subsection{The Fairness Process}\label{sec:fairnessprocess}

The utility of a server is defined based on the long run proportion of time the server stays idle. This quantity is difficult to calculate and manipulate, especially when servers are heterogeneous. To be able to analyze this quantity in the heterogeneous setting, we use the  \emph{fairness process} introduced by \cite{bukeqin2019}. In this section, we provide the definition of the fairness process and how it can be used to analyze strategic server behavior.

Defining the empirical distribution of service rates for the $n$th system as
$F^n[\mu,\mu+\Delta) :=(N_\alpha^n)^{-1}\sum_{j=1}^{N_\alpha^n}\delta_{\mu_j^n}[\mu,\mu+\Delta) \mbox{ for all }\mu\in [\mu_{\min},\mu_{\max}]\mbox{ and }\Delta>0,$
we can write
\begin{align}
\nonumber
&\EE[I_k^n(\infty)|\tilde{\mu}_k^n=\mu]\\& \nonumber\quad= \lim_{t\to\infty}\frac{1}{t}\EE\left[\int_0^tI_k^n(s)ds|\tilde{\mu}_k^n=\mu\right]\\
    \nonumber&\quad=\lim_{\Delta\to 0}\lim_{t\to\infty}\frac{1}{t}\EE\left[\frac{\sum_{j=1}^{N_\alpha^n}\delta_{\tilde{\mu}_j^n}[\mu,\mu+\Delta)\int_0^tI_j^n(s)ds}{N_\alpha^nF^n[\mu,\mu+\Delta)} \Big| \tilde{\mu}_k^n=\mu\right]
    \\
    &\quad=\lim_{\Delta\to 0}\lim_{t\to\infty}\frac{1}{t}\EE\left[\frac{n^\alpha}{N_\alpha^nF^n[\mu,\mu+\Delta)}    \times\frac{\sum_{j=1}^{N_\alpha^n}\delta_{\tilde{\mu}_j^n}[\mu,\mu+\Delta) 
    \int_0^t\hat{I}_{j,\alpha}^n(s)ds}{\int_0^t \hat{I}_\alpha^{n}(s)ds}
    \times
    \int_0^t \hat{I}_\alpha^{n}(s)ds \Big| \tilde{\mu}_k^n=\mu\right].
    \label{eq:idleness_decomposition}
\end{align}

The second term inside the expectation represents the proportion of total cumulative idleness experienced by servers whose service rates lie between $\mu$ and $\mu + \Delta$. \cite{bukeqin2019} represent the distribution of cumulative idleness among servers with different service rates using a measure-valued process and they coin this process as the \emph{fairness process}. 
The name is motivated by the idea that the distribution of idleness is generally viewed as a measure of fairness towards servers see, e.g., \cites{armward10, wararm13}. By selecting a specific routing policy, the system operator determines the fairness process, i.e.,  how idleness is distributed among servers as a function of service effort. For example, a routing policy might aim to distribute idleness equally, reflecting an egalitarian notion of fairness, or reward faster servers with more idleness, representing a performance-based perspective on fairness. In response, servers decide their service rates to maximize their utility as defined in \eqref{eq:utility}. Thus, the fairness process corresponding to a routing policy plays a key role in our analysis by shaping the structure of the utility function.

Setting 
 $\tau_0^n:=\inf\{t>0:\int_0^t\hat{I}_\alpha^n(s)ds>0\}$  as the first time idleness is experienced by the system, the fairness process for the $n$th system is defined as
\begin{equation}
    \eta_{\alpha,t}^{\pi,n}(\AAA):=\left\{\begin{array}{ll}
    \displaystyle\frac{\sum_{k=1}^{N_\alpha^n}\delta_{\mu_k^n}(\AAA)\int_0^t\hat{I}_{k,\alpha}^n(s)ds}{\int_0^t\hat{I}_{\alpha}^n(s)ds} &, \mbox{if }t>\tau_0^n,
    \\
    \delta_0(\AAA) &, \mbox{if }t\leq \tau_0^n,
    \end{array}\right.
\end{equation}
for all $\AAA \in \cB(\RR_+)$, i.e., $\eta_{\alpha,t}^{\pi, n}(\AAA)$ represents the proportion of total cumulative idleness received up to time $t$  by the servers whose service rates lie in set $\AAA$. To make this quantity well-defined and avoid division by 0, it is set to $\delta_0(\AAA)$ as a placeholder until the total cumulative idleness in the system is positive. Hence, \eqref{eq:idleness_decomposition} can be written more compactly using the fairness process as 
\begin{align}
    \EE[I_k^n(\infty)|\tilde{\mu}_k^n=\mu]& =\lim_{\Delta\to 0}\lim_{t\to\infty}\frac{1}{t}\EE\left[\frac{n^\alpha}{N_\alpha^n}    \times\frac{\eta_{\alpha,t}^{\pi,n}[\mu,\mu+\Delta)}{F^n[\mu,\mu+\Delta)}
    \times
    \int_0^t \hat{I}_\alpha^{n}(s)ds|\tilde{\mu}_k^n=\mu\right]. \label{eq:idleness_decomposition_fairness}
\end{align}
Hence, the long run proportion of time a server is idle is determined by the total idleness the system experiences and how this idleness is distributed among servers as a consequence of the routing policy. In the QED regime ($\alpha = 1/2$), \cite{bukeqin2019} show that if the limiting fairness process can be characterized for a routing policy, then the limiting behavior of the total cumulative idleness can also be determined, which we extend to QD regimes ($1/2<\alpha< 1$) in Theorem~\ref{thm:system_convergence}.  Unfortunately, the fairness processes do not converge in measure in any of the four Skorokhod topologies as $n\to \infty$ and the limiting process should be defined based on shifted versions of the fairness processes. The formal definition of this limit is presented in the Appendix~\ref{app:fairness_definition}. 
%As we show in Section~\ref{sec:convergence_system}, the total idleness is relatively easier to analyze once the distribution of idleness is characterized. However, the distribution of idleness is more difficult to analyze and  constitutes the major challenge in identifying the utility of servers. %Hence, the recent literature \citep[e.g.,][]{Gopalakrishnan_etal2016, armony2021capacity} has concentrated on servers with identical preferences and has aimed to identify a symmetric equilibrium by considering a system where all the servers, except the one under consideration, have the same rate.  

\cite{bukeqin2019} observe that under most stationary policies the fairness processes converge on a faster scale than the scaled system length processes. Hence the limiting fairness process does not depend on $t$ after $\tau_0:=\lim_{n\to\infty}\tau_0^n$ and is deterministic, i.e., $\eta_{t,\alpha}^\pi=\eta_\alpha^\pi$ for all $t>\tau_0$. For these cases, we refer to $\eta_\alpha^\pi$ as \textit{the limiting fairness measure}, slightly abusing the terminology. Proposition~\ref{thm:specialpolicies} summarizes several findings of~\cite{bukeqin2019}, extending from $\alpha=1/2$ to $1/2\leq \alpha <1$. 
\begin{proposition}\label{thm:specialpolicies}
Suppose $1/2\leq \alpha <1$ and let   $\eta_\alpha^{SSF}, \eta_\alpha^{FSF}, \eta_\alpha^{LISF}$ and $\eta_\alpha^{RR}$ be the limiting fairness processes corresponding to slowest-server-first, fastest-server-first, longest-idle-server-first and uniformly random routing, respectively. Then, for all $t>\tau_0$ and $\AAA\in \cB(\RR_+)$
\begin{align*}
     &\eta_{\alpha,t}^{SSF} = \delta_{\mu_{\max}}(\AAA), \;
     \eta_{\alpha,t}^{FSF}= \delta_{\mu_{\min}}(\AAA) \; \textrm{ and }\eta_{\alpha,t}^{LISF}( \AAA)=\eta_{\alpha,t}^{RR}( \AAA)=
     \frac{\int_\AAA\mu F(d\mu)}{\int_{\RR_+}\mu F(d\mu)}
     .
\end{align*}
\end{proposition}

Theorem 1 indicates that under the slowest-server-first and fastest-server-first policies, only the fastest and slowest servers, respectively, experience non-negligible idleness and in general cannot be considered fair. However, under both the longest-idle-server-first and uniformly random routing policies, the idleness a server experiences is proportional to her service rate and can be considered fair for some definitions of fairness. Unfortunately, the case with $\alpha=1$ is far more involved and similar results do not hold as we show in Section~\ref{sec:grr_policy}.

\section{Convergence of the Scaled System Length Processes}\label{sec:convergence_system}

In this section, we derive fluid and diffusion approximations as the weak limit of the scaled system length processes $\{\hat{X}_\alpha^n\}_{n\in\NN}$ when $1/2\leq \alpha\leq 1$ using the fairness process defined in Section~\ref{sec:fairnessprocess}. Recall that  $\hat{X}_\alpha^n(0)\Rightarrow \xi_0$ and $\langle \iota, \eta_{\alpha,t}^\pi\rangle$ corresponds to the first moment of the measure $\eta_{\alpha,t}^\pi$. 

\begin{theorem}\label{thm:system_convergence}
    Suppose that the limiting fairness process under the adopted routing policy $\pi$ is $\eta_{\alpha,t}^\pi$. Then, $\hat{X}_\alpha^n\Rightarrow \xi_\alpha$ as $n\to\infty$, where
    \begin{enumerate}
        \item if $\alpha=1/2$, $\xi_{1/2}$ is the strong solution to the stochastic differential equation
    \begin{equation}\xi_{1/2}(t)=\xi_0+(2\bar{\lambda})^{1/2} W(t)-\left(\beta_F(\bar{\lambda}\bar{\mu}_F)^{1/2}+\zeta\right) t +\langle \iota, \eta_{1/2,t}^\pi \rangle\int_0^t(\xi_{1/2}(s))^-ds-\gamma \int_0^t(\xi_{1/2}(s))^+ds, 
    \label{eq:diffusionlimit}\end{equation}
    for all $t\geq 0$ where $\zeta\sim$ $\textrm{Normal}(0,\sigma_F^2\bar{\lambda}^{\alpha}\bar{\mu}_F^{-\alpha})$ and $W$ is a standard Brownian motion.
    
    \item if $1/2<\alpha\leq 1$, the process $\xi_\alpha$ is the solution to the ordinary differential equation
    \begin{equation}
    \xi_\alpha(t)=\xi_0-\beta\bar{\lambda}^\alpha\bar{\mu}_F^{1-\alpha}t +\langle \iota, \eta_{\alpha,t}^\pi \rangle\int_0^t(\xi_\alpha(s))^-ds - \gamma\int_0^t(\xi_\alpha(s))^+ds, \text{ for all } t\geq 0.
    \label{eq:fluid_limit}
    \end{equation}
    \end{enumerate} 
\end{theorem}

 A key question regarding the stochastic process limits above is whether it is possible to approximate the stationary behavior of the many-server queues using the diffusion and fluid limits given above, i.e., whether the many-server limit and the limit as $t\to \infty$ is interchangeable.  As the next step, we prove the interchangeability of limits in Theorem~\ref{thm:interchangibility_stationary_n}. 
\begin{theorem}
\label{thm:interchangibility_stationary_n}
For many-server systems with random and heterogeneous service rates, for any $1/2\leq \alpha\leq 1$ the following convergence results hold as $n\to \infty$, (i) $\hat{X}_\alpha^n(\infty)\Rightarrow \xi_\alpha(\infty)$, (ii) $\EE[\hat{X}_\alpha^n(\infty)]\to \EE[\xi_\alpha(\infty)]$,  and (iii)  $\eta_{\alpha,\infty}^{\pi,n}\Rightarrow \eta_{\alpha, \infty}^\pi$ 
where $\xi_{\alpha}(t)\Rightarrow \xi_{\alpha}(\infty)$ and $\eta_{\alpha, t}^\pi\Rightarrow \eta_{\alpha, \infty}^\pi$ as $t\to \infty$.
\end{theorem}

A simple fixed point analysis using part (iii) of Theorem~\ref{thm:interchangibility_stationary_n} implies that for any $1/2<\alpha\leq 1$, 
%if $\eta_{\alpha,t}=\eta_\alpha$ for all $t\geq\tau_0$, then the solution to \eqref{eq:fluid_limit} can be written as 
% \[
% \xi(t)=-\frac{\beta\bar{\lambda}^\alpha\bar{\mu}_F^{1-\alpha}}{\langle \iota, \eta_\alpha \rangle}+\left(\xi_0+\frac{\beta\bar{\lambda}^\alpha\bar{\mu}_F^{1-\alpha}}{\langle \iota, \eta_\alpha \rangle}\right)e^{-\langle \iota, \eta_\alpha \rangle t},
% \]
% and as $t\to \infty$ we have $ 
% \xi(t)\to\beta\bar{\lambda}^\alpha\bar{\mu}_F^{1-\alpha}\langle \iota, \eta_\alpha \rangle^{-1}$. 
% Using a similar reasoning, 
we have $$\xi_{\alpha}(t)\to \xi_\alpha(\infty):=-\beta\bar{\lambda}^\alpha\bar{\mu}_F^{1-\alpha}\langle \iota, \eta_{\alpha,\infty}^\pi \rangle^{-1},\qquad \textrm{ as $t\to \infty$.}$$ Part (ii) of Theorem~\ref{thm:system_convergence} implies $ \EE[\hat{X}_\alpha^n(\infty)]\approx \xi_\alpha(\infty)$ and $\EE[(\hat{X}_\alpha^n(\infty))^+]\approx 0$ for sufficiently large $n$. Finally, using the relation that $\hat{I}^n_\alpha(\infty)=(\hat{X}_\alpha^n(\infty))^-$, we can approximate the expected long run number of idle servers as 
\begin{equation}
    \EE[\hat{I}_\alpha^n(\infty)]\approx\beta\bar{\lambda}^\alpha\bar{\mu}_F^{1-\alpha}\langle \iota, \eta_{\alpha,\infty}^\pi\rangle^{-1}.\label{eq:stationary_idle_length}
\end{equation}
This argument underpins much of the success mean-field games theory has collected in recent years, see e.g. \cite{Bayraktar2019} and references therein.

\begin{remark}\label{rem:specialpolicies}
Equations \eqref{eq:diffusionlimit} and \eqref{eq:fluid_limit} imply that $\PP(\tau_0<\infty)=1$ for all $1/2\leq \alpha< 1$. Proposition~\ref{thm:specialpolicies} and Theorem \ref{thm:interchangibility_stationary_n} imply
$\langle\iota,\eta_{\alpha,\infty}^{SSF}\rangle = \mu_{\max}$, $
     \langle\iota,\eta_{\alpha,\infty}^{S
     FSF}\rangle= \mu_{\min}$ and $\langle\iota,\eta_{\alpha,\infty}^{LISF}\rangle=\langle\iota,\eta_{\alpha,\infty}^{RR}\rangle=
     \EE[(\tilde{\mu}_k^n)^2]/\EE[\tilde{\mu}_k^n]$. These generalize the transient convergence results in \cite{Atar2008} and \cite{bukeqin2019}.
\end{remark}

Theorems~\ref{thm:system_convergence} and \ref{thm:interchangibility_stationary_n} demonstrate the importance of characterizing the limiting fairness process under a suggested routing policy. In Proposition~\ref{thm:specialpolicies}, we presented some results in this direction when $1/2\leq \alpha<1$. In the next section, we concentrate on the purely quality-driven regime ($\alpha = 1$) and characterize the limiting fairness processes for a general class of randomized policies.  
 
\section{A Generalized Random Routing Policy in the Purely Quality-Driven Regime}
\label{sec:grr_policy}
In this section, we address the following two questions: (i) Given a service rate distribution $F$, which limiting fairness processes are attainable in a purely quality-driven regime? (ii) How should one design a routing policy to attain a specified fairness process? In Proposition~\ref{thm:specialpolicies}, we have seen that for $1/2\leq \alpha <1$, the limiting fairness process can be quite general and even measures with a mass concentrating at a certain service rate are attainable. However, when $\alpha=1$, the dynamics differ substantially. To understand this difference, consider any $\AAA\in\cB(\RR_+)$ with $F(\AAA)>0$. The number of agents with service rates in set $\AAA$ scales with $n$ and Theorem~\ref{thm:system_convergence} implies that the total idleness observed in the system scales with $n^\alpha$. Hence, when $\alpha<1$, it is possible for all the idleness to concentrate on the agents of set $\AAA$ for $n$ large enough and have point masses in the limiting fairness measures as seen in Proposition~\ref{thm:specialpolicies}. On the other hand, when $\alpha=1$, both the total idleness and the number of agents in $\AAA$ scale with $n$, and the number of agents in $\AAA$ can potentially be less than the total idleness experienced in the system for any $n$. As we always need to have the total proportional idleness experienced by the agents in set $\AAA$ less than the number of agents in that set, for any $n\in \NN$ and $T>0$, 
\begin{align}
    % \nonumber
    \frac{\sum_{k=1}^{N_\alpha^n}\delta_{\tilde{\mu}_k^n}(\AAA)\int_0^T I_k^n(t)dt}{T}&\leq \sum_{k=1}^{N_\alpha^n}\delta_{\tilde{\mu}_k^n}(\AAA),
     \mbox{ which implies } 
    \frac{\eta_{\alpha,\infty}^{\pi,n}(\AAA)}{F^n(\AAA)}\leq \frac{N_\alpha^n T}{\sum_{k=1}^{N_\alpha^n}\int_0^TI_k^n(t)dt}.\label{eq:boundnumberservers}
\end{align}
When $1/2\leq \alpha<1$, Theorem~\ref{thm:system_convergence} implies that the right-hand side diverges to infinity as $n\to\infty$, and hence, \eqref{eq:boundnumberservers} does not pose any restrictions on the limiting fairness measure.  However, when $\alpha=1$, taking the limit and replacing from \eqref{eq:stationary_idle_length}, we get the following necessary condition for any probability measure to be the stationary limiting fairness measure under a routing policy.
\begin{proposition}\label{thm:necessary_idleness_distribution}
In the purely quality-driven regime, a necessary condition for  a probability measure $\eta_{1,\infty}^\pi$ to be the stationary limiting fairness measure under a routing policy $\pi$ is  
\begin{equation}
    0\leq \frac{\eta_{1,\infty}^\pi(\AAA)}{F(\AAA)}\leq \frac{(1+\beta)\langle \iota, \eta_{1,\infty}^\pi\rangle}{\beta \bar{\mu}_F} \mbox{ for all }\AAA\in \cB(\RR_+).\label{eq:boundqdregime}
\end{equation}
\end{proposition}

Equation~\eqref{eq:boundqdregime} implies that, in the purely quality-driven regime, the stationary limiting fairness measure is absolutely continuous, and hence, has a density $g(\mu)$ with respect to $F$. Considering sets $[\mu,\mu+\Delta)$ and taking the limit as $\Delta\to 0$, \eqref{eq:boundqdregime} then takes the form
\begin{equation}
        0\leq g(\mu)\leq \frac{(1+\beta)\langle \iota, \eta_{1,\infty}^\pi\rangle}{\beta \bar{\mu}_F}.\label{eq:boundqddensity}
\end{equation}

We naturally ask the reverse implication, concretely, is the necessary condition provided in Proposition~\ref{thm:necessary_idleness_distribution} also sufficient for the existence of a routing policy that attains the provided stationary limiting fairness measure. In this section, we provide a class of generalized random routing policies that can attain any stationary limiting fairness measure with a density that satisfies~\eqref{eq:boundqddensity} strictly in the purely quality-driven regime. 
We first formally define our routing policy. 
\begin{definition}
 Given a real-valued function $h(\mu)$ such that $h(\mu)>0$ for all $\mu_{\min}\leq \mu\leq \mu_{\max}$,  a routing policy is called $h$-random if the probability of an incoming job arriving at $t$ to be routed upon arrival to the idle server $k$ having service rate $\tilde{\mu}_k^n$ is proportional to $h(\tilde{\mu}_k^n)$ and is given by 
 \[
 \frac{h(\tilde{\mu}_k^n)I_k^n(t-)}{\sum_{l=1}^{N_1^n}h(\tilde{\mu}_l^n)I_l^n(t-)},
 \]
 where 0/0 is interpreted as 0.
 \end{definition}

To be able to analyze $h$-random policies, we define the finite measure-valued instantaneous allocation of idleness processes as $\psi_t^n(\AAA):=\sum_{k=1}^{N_1^n}\delta_{\tilde{\mu}_k^n}(\AAA)I_k^n(t)$, and this process denotes the number of servers with service rate in set $\AAA$ who are idle at time $t$ at the $n$th system. 
We also define the scaled instantaneous allocation as $\bar{\psi}_t^n(\AAA):=n^{-1}\psi_t^n(\AAA)$.
\begin{lemma}\label{lem:integral_equation}
The set of scaled instantaneous allocations, $\{\bar{\psi}_t^n\}_{n\in\NN}$, is tight. Moreover, if the subsequence $\bar{\psi}_t^{n_k}\Rightarrow \bar{\psi}_t$ as $n_k\to\infty$, then, $\bar{\psi}_t$ satisfies 
\begin{align}
        \nonumber\langle f,\bar{\psi}_{t}\rangle &=\langle f,\bar{\psi}_{0}\rangle + \frac{\bar{\lambda}}{\bar{\mu}}(1+\beta)\langle f\times \iota, F\rangle \int_0^t\II(\xi_{\alpha}(s)\leq 0)ds - \int_0^t\langle f\times \iota, \bar{\psi}_{s-}\rangle \II(\xi_{\alpha}(s)\leq 0)ds
        \\
        &\quad 
        - \bar{\lambda}\int_0^t \frac{\langle f\times h, \bar{\psi}_{s-}\rangle }{\langle h, \bar{\psi}_{s-}\rangle}\II(\xi_{\alpha}(s)\leq 0)ds \mbox{ for all }t\geq 0.
      \label{eq:limitrandom1}
    \end{align}
\end{lemma}
% \label{lem:tightness_allocations}
% \end{lemma}

% The above tightness result implies that any subsequence of $\bar{\psi}_t^n$ has a further subsequence that converges. In Lemma~\ref{lem:integral_equation}, we provide an equation that the limits of all subsequences must satisfy. 
% \begin{lemma}\label{lem:integral_equation}
% If the subsequence $\bar{\psi}_t^{n_k}\Rightarrow \bar{\psi}_t$ as $n_k\to\infty$. Then, $\bar{\psi}_t$ satisfies 
% \begin{align}
%         \nonumber\langle f,\bar{\psi}_{t}\rangle &=\langle f,\bar{\psi}_{0}\rangle + \frac{\bar{\lambda}}{\bar{\mu}}(1+\beta)\langle f\times \iota, F\rangle \int_0^t\II(\xi_{\alpha,s}\leq 0)ds - \int_0^t\langle f\times \iota, \bar{\psi}_{s-}\rangle \II(\xi_{\alpha,s}\leq 0)ds
%         \\
%         &\quad 
%         - \bar{\lambda}\int_0^t \frac{\langle f\times h, \bar{\psi}_{s-}\rangle }{\langle h, \bar{\psi}_{s-}\rangle}\II(\xi_{\alpha,s}\leq 0)ds \mbox{ for all }t\geq 0.
%       \label{eq:limitrandom1}
%     \end{align}
% \end{lemma}

% \begin{lemma}\label{lem:boundedness_instantaneous_idleness}
% For any $f\in \CC_{[\mu_{\min},\mu_{\max}]}^b[0,\infty)$, any solution to equation~\eqref{eq:limitrandom1} is bounded. 
% \end{lemma}
Equation \eqref{eq:limitrandom1} enables us to study the transient behavior of scaled instantaneous allocation processes. Theorem~\ref{thm:stationary_fairness1} shows how \eqref{eq:limitrandom1} can be used to obtain the stationary limiting fairness process.

\begin{theorem}
When $\alpha=1$, the stationary limiting fairness measure $\eta_{1,\infty}^\pi$ under an $h$-random policy is absolutely continuous w.r.t. the service rate distribution $F$ with density 
\begin{equation}
    g(\mu)=
    \big( 1+L_F\tilde{h}(\mu) \big)^{-1}
    \left(\int_{\mu_{\min}}^{\mu_{\max}}(1+L_F\tilde{h}(\mu))^{-1}dF(\mu)\right)^{-1},\quad \mbox{ for all }\mu\in[\mu_{\min},\mu_{\max}],\label{eq:stationaryfairness_hrandom}
\end{equation}
where $\tilde{h}(\mu)=h(\mu)/\mu$ and $L_F$ is the unique solution of
\begin{equation}
    \int_{\mu_{\min}}^{\mu_{\max}}\mu\frac{1+\beta}{\bar{\mu}_F(1+L_F\tilde{h}(\mu))}dF(\mu)=\beta.\label{eq:g_equation}
\end{equation}
\label{thm:stationary_fairness1}
\end{theorem}
\vspace{-0.4in}
The following lemma provides bounds on $L_F$, the solution of \eqref{eq:g_equation}.
\begin{lemma}\label{lem:L_bound}
Define $h_{\min}=\min_{\mu_{\min} \leq \mu\leq \mu_{\max} }\tilde{h}(\mu)$ and $h_{\max}=\max_{\mu_{\min}\leq \mu\leq \mu_{\max}}\tilde{h}(\mu)$ and let $L_F$ be the solution of \eqref{eq:g_equation} for some service rate distribution $F$. Then, $(\beta h_{\max})^{-1}\leq L_F\leq (\beta h_{\min})^{-1}$.
\end{lemma}
% Equation~\eqref{eq:stationaryfairness_hrandom} implies that the density of the stationary limiting fairness process with respect to the service rate distribution $F$ is
% \begin{align*}
%     g(\mu)=\frac{(1+L_F\tilde{h}(\mu))^{-1}}{\int_{\mu_{\min}}^{\mu_{\max}}(1+L_F\tilde{h}(\mu))^{-1}dF(\mu)}.
% \end{align*}

At the beginning of this section, we provide the necessary condition~\eqref{eq:boundqddensity} stating that for a function $g(\mu)$ should satisfy to be the density function of $\eta_{1,\infty}^\pi$ with respect to $F$. Our next result shows that it is possible to find a suitable $h$-random policy to attain any density which satisfies the inequalities in \eqref{eq:boundqddensity} strictly. 
\begin{corollary}\label{cor:attainability_g}
Suppose that $g(\mu)$ satisfies \eqref{eq:boundqddensity} strictly. Then, $g(\mu)$ is the density function of $\eta_{1,\infty}^\pi$ with respect to $F$ under the $h$-random policy where 
\[
h(\mu)=\frac{\left((1+\beta)(\bar{\mu}_F)^{-1}-\beta\langle \iota,\eta_{1,\infty}^\pi \rangle^{-1}g(\mu)\right)\mu}{\beta\bar{\lambda}\langle \iota,\eta_{1,\infty}^\pi \rangle^{-1}g(\mu)}.
\]
\end{corollary}

% We now characterized the stationary limiting fairness measure which characterizes how the idleness is distributed among sets of servers and using the first moment of this measure we can characterize the stationary number of idle servers asymptotically as given in~\eqref{eq:stationary_idle_length}. However, we still need to understand what this implies in terms of the idleness received by an individual server serving with a given rate, which we analyze further in the next section. 

% We now analyze the expected idleness an individual server observes. To do so, we need the following uniform integrability result. 
% \begin{lemma}\label{lem:uniform_integrability_inverse}
% For any $\alpha>1/2$, the collection of random variables $\{(\hat{I}_\alpha^n(\infty))^{-1}\}_{n\in \NN}$ is uniformly integrable. 
% \end{lemma}

% \begin{theorem}\label{thm:convergence_individual}
% As $n\to \infty$, 
% \end{theorem}

\section{Strategic Servers with Individual Preferences} \label{sec:strategic}

In this section, we deploy the results of the previous sections to analyze the strategic behavior of servers in a many-server setting when each server has individual preferences as described in Section~\ref{sec:setting}. As the first step, we show how the stationary limiting fairness process can be used to characterize the idleness experienced by individual servers. This again requires an interchange of limits argument which we show to be valid for $h$-random policies in the purely quality-driven regime ($\alpha=1$) and for any idle-time order based policy when $1/2\leq \alpha\leq 1$. Then, we study the best response of server $k$ with characteristics determined by $\tilde{\mu}_{\min,k}^n, \tilde{\mu}_{\max,k}^n$ and $\tilde{a}_k^n$, and use our analysis to derive a fixed point equation that characterizes the Nash equilibria of the strategic servers game. 
We have the following assumptions on server characteristics. 
\begin{assumption}\label{asm:utility_assumptions}
\begin{enumerate}
    \item The trade-off coefficient $\tilde{a}_k^n$ is bounded away from zero and bounded from above, i.e., there exist $a_{\min}>0$ and $a_{\max}<\infty$ such that $\PP(a_{\min}\leq \tilde{a}_k^n\leq a_{\max})=1$
    \item  The derivative of the cost of effort function is bounded away from zero and bounded from above over the support of the service rate, i.e., there exist $c_{\min}>0$ and $c_{\max}<\infty$ such that $c_{\min}\leq c'(\mu) \leq c_{\max}, \mbox{ for all }\mu_{\min}\leq \mu \leq \mu_{\max}.$
    \item The limiting stationary fairness measure is deterministic and absolutely continuous with respect to $F$ with  a continuously differentiable density function $g(\mu)$ on the closed interval $[\mu_{\min}, \mu_{\max}]$. Both $g(\mu)$ and its derivative $g'(\mu)$ are bounded from above, i.e., $g(\mu)\leq g_{\max}$ and $g'(\mu)\leq g_{\max}'$, and $g'(\mu)$ is bounded away from 0, i.e., $0<g_{\min}'\leq g'(\mu)$ for all $\mu_{\min}\leq \mu \leq \mu_{\max}$.
\end{enumerate}
\end{assumption}

 The last assumption on the limiting stationary fairness measure being deterministic might seem restrictive. 
However, as seen in Proposition~\ref{thm:specialpolicies} and Theorem~\ref{thm:stationary_fairness1}, this property holds for random routing and LISF policies as well as the generalized random routing policies studied in Section~\ref{sec:grr_policy} when $\alpha=1$ (see also Propositions~\ref{thm:interchange_generalized_rr} and \ref{thm:interchange_idle_time_order}). 
On the other hand, the limiting fairness measures for SSF and FSF policies are not absolutely continuous with respect to $F$ and do not possess a density in general thus making the analysis of equilibria for these policies considerably more involved. This is an open problem.

\subsection{The Expected Long Run Proportion of Idleness Experienced by Individual Servers}

Equation~\eqref{eq:idleness_decomposition_fairness} describes the relationship between the idleness received by a server serving with rate $\mu$ and the fairness process. Multiplying both sides by $n^{1-\alpha}$ and taking the limit as $n\to\infty$ yields 

% To relate the stationary limiting fairness measure to the  idleness experienced by individual servers, the routing policy must be indifferent to the index of the servers, and if two servers choose to serve with the same rate their expected long run proportion of idleness should be equal, i.e., 
% \begin{equation}
%     \EE[I_k^n(\infty)|\tilde{\mu}_k^n=\mu]=\EE[I_j^n(\infty)|\tilde{\mu}_j^n=\mu] \mbox{ for all } 1\leq k, j\leq N_\alpha^n\mbox{ and }\mu_{\min}\leq \mu\leq \mu_{\max}.\label{eq:indifference_to_index}
% \end{equation}
% Then, defining the empirical law of service rates for the $n$th system as
% \[
% F^n[\mu,\mu+\Delta) =\frac{\sum_{j=1}^{N_\alpha^n}\delta_{\mu_j^n}[\mu,\mu+\Delta)}{N_\alpha^n} \mbox{ for all }\mu\in [\mu_{\min},\mu_{\max}]\mbox{ and }\Delta>0,
% \]
% then for any $n\in\NN$ can calculate 
% \begin{align*}
%     \EE[n^{1-\alpha}I_k^n(\infty)|\mu_k^n=\mu]&=\lim_{\Delta\to 0}\EE\left[n^{1-\alpha}\frac{\sum_{j=1}^{N_\alpha^n} \delta_{\mu_j^n}[\mu,\mu+\Delta)I_j^n(\infty)}{\sum_{j=1}^{N_\alpha^n}\delta_{\mu_j^n}[\mu,\mu+\Delta)}|\mu_k^n=\mu\right]\\
%     % &=\lim_{\Delta\to 0}\EE\left[\frac{\sum_{j=1}^{N_\alpha^n} \delta_{\mu_j^n}[\mu,\mu+\Delta)\hat{I}_j^n(\infty)}{\hat{I}_\alpha^n(\infty)}\hat{I}_\alpha^n(\infty)\frac{n}{\sum_{j=1}^{N_\alpha^n}\delta_{\mu_j^n}[\mu,\mu+\Delta)}|\mu_k^n=\mu\right]\\
%     &=\lim_{\Delta\to 0}\EE\left[\frac{\eta_{\alpha,\infty}^n[\mu,\mu+\Delta)}{F^n[\mu, \mu+\Delta)}\hat{I}_\alpha^n(\infty)\frac{n}{N_\alpha^n}|\mu_k^n=\mu\right].
% \end{align*}
% Taking the limit as $n\to \infty$ yields 
\begin{equation}
    \lim_{n\to \infty} \EE[n^{1-\alpha}I_k^n(\infty)|\mu_k^n=\mu]=\lim_{n\to \infty}\lim_{\Delta\to 0}\EE\left[\frac{\eta_{\alpha,\infty}^n[\mu,\mu+\Delta)}{F^n[\mu, \mu+\Delta)}\hat{I}_\alpha^n(\infty)\frac{n}{N_\alpha^n}|\mu_k^n=\mu\right].\label{eq:limit_individual}
\end{equation}
Clearly, the stationary limiting fairness measure, $\eta_{\alpha,\infty}^n$, for the $n$th system is absolutely continuous with respect to $F^n$ and the inner limit (in $\Delta$) converges to the density function. 
If the stationary limiting fairness measure is also absolutely continuous w.r.t.~the service rate distribution $F$ with  a deterministic density $g(\mu)$ and if the limits on the right-hand side of \eqref{eq:limit_individual} can be interchanged, i.e.,
\begin{equation}
    \lim_{n\to \infty}\lim_{\Delta\to 0}\EE\left[\frac{\eta_{\alpha,\infty}^n[\mu,\mu+\Delta)}{F^n[\mu, \mu+\Delta)}\hat{I}_\alpha^n(\infty)\frac{n}{N_\alpha^n}|\mu_k^n=\mu\right]=\lim_{\Delta\to 0}\lim_{n\to \infty}\EE\left[\frac{\eta_{\alpha,\infty}^n[\mu,\mu+\Delta)}{F^n[\mu, \mu+\Delta)}\hat{I}_\alpha^n(\infty)\frac{n}{N_\alpha^n}|\mu_k^n=\mu\right],
    \label{eq:interchange_limits}
\end{equation}
 we can then conclude that 
\begin{equation}\label{eq:individual_idleness_approximation}
\lim_{n\to\infty} n^{1-\alpha}\EE[I_k^n(\infty)|\tilde{\mu}_k^n=\mu]=\beta g(\mu)\frac{\bar{\mu}_F}{\bar{\lambda}}\EE[\xi_\alpha^{-}(\infty)].
\end{equation}

Equation~\eqref{eq:individual_idleness_approximation} implies that the idleness received by a server is in the order of $n^{\alpha-1}$ and approaches 0 if $\alpha<1$. When the utility of idleness function is identity, i.e., $u_I(I) = I$ e.g., as in \cites{Gopalakrishnan_etal2016, gopalakrishnan2021, armony2021capacity}, setting $\alpha<1$ does not provide enough incentive for the servers to work harder to increase their service rate and results in all servers working with their minimal service rate $\tilde{\mu}_{\min,k}^n$ for large $n$. Hence, as \cite{Gopalakrishnan_etal2016} prove for identical servers, purely quality-driven regime is asymptotically optimal in this case.

% When the utility of idleness function is the identity function, it is non-negative and the utility is zero when the expected long run proportion of idleness is zero. Even though this is good for modeling the benefits of idleness, it fails to appropriately model the discomfort due to the server not experiencing enough idleness. 
% Similarly, when the utility of idleness function is bounded from below, it can be made equivalent to a non-negative utility of idleness by adding a constant. 
Even though choosing utility of idleness function as identity function is good for modeling the benefits of idleness, it fails to appropriately model the discomfort due to the server not experiencing enough idleness. To model the situations where servers are sensitive to experiencing a low level of idleness and experiencing no idleness is unacceptable by the servers, one needs to have $u_I(I)\to-\infty$ as $I\to 0$. In this case, the servers might be inclined to work faster, even though the proportion of idleness experienced by the server is very close to zero, to reduce the discomfort. This suggests that if the servers are sufficiently sensitive to experiencing low levels of idleness, purely quality-driven regime can be suboptimal. 
Assuming the convergence in~\eqref{eq:individual_idleness_approximation} is uniform, Theorem~\ref{thm:best_response_scaling} shows how the safety staffing level and the behavior of $u_I(\cdot)$ near zero determine the best response rate of a server. 

\begin{theorem}\label{thm:best_response_scaling}
Suppose that for any service rate distribution $F$, the convergence in \eqref{eq:individual_idleness_approximation} is uniform in $\mu_{\min}\leq \mu\leq \mu_{\max}$, i.e.,
\begin{equation}\label{eq:uniform_conv_assumption}
   \sup_{\mu_{\min}\leq \mu \leq \mu_{\max}}\left|n^{1-\alpha}\EE[I_k^n(\infty)|\tilde{\mu}_k^n=\mu]-\beta g_F(\mu)\frac{\bar{\mu}_F}{\bar{\lambda}}\EE[\xi_\alpha^{-}(\infty)]\right|\to 0, \mbox{ as }n\to\infty.  
\end{equation}

\begin{enumerate}
    \item If under the chosen policy, the limiting stationary fairness measure is deterministic with a non-increasing density $g_F(\mu)$ for any service rate distribution $F$, then the optimal service rate for each server uniformly converges to the server's minimum attainable service rate as the system size increases, i.e., for any $\epsilon>0$ there exists an $N_\epsilon$ such that $n>N_\epsilon$ implies that the optimal strategy for any server $k$ in the $n$th system is $\tilde{\mu}_k^{n,*}\in[\tilde{\mu}_{\min,k}^n,\tilde{\mu}_{\min,k}^n+\epsilon)$ w.p.~1.
    
    \item If $n^{\alpha-1}u_I'(n^{\alpha-1}I)\to 0$ for all $I>0$ as $n\to \infty$, then, as in Part 1 above, the optimal service rate for each server uniformly converges to the server's minimum attainable service rate as the system size increases. 

    \item If  $n^{\alpha-1}u_I'(n^{\alpha-1}I)\to \infty$ for all $I>0$ as $n\to \infty$ and $g_F(\mu)$ is strictly increasing and concave for any service rate distribution $F$, then the optimal service rate of each server uniformly converges to server's maximum attainable service rate as the system size increases, i.e., for any $\epsilon>0$ there exists an $N_\epsilon$ such that $n>N_\epsilon$ implies that the optimal strategy for any server $k$ in the $n$th system is $\tilde{\mu}_k^{n,*}\in(\tilde{\mu}_{\max,k}^n-\epsilon,\tilde{\mu}_{\max,k}^n]$ w.p.~1.
\end{enumerate}
\end{theorem}

Theorem~\ref{thm:best_response_scaling} has clear implications regarding the optimal scaling for the safety staffing. We first note that, Theorems~\ref{thm:system_convergence} and \ref{thm:interchangibility_stationary_n} imply that $\EE[(\hat{X}_\alpha^n(\infty))^+]\to 0$ for any $1/2< \alpha \leq 1$ and is finite when $\alpha=1/2$. Hence, using the big-O/little-o notation, we can write the operator cost in \eqref{eq:operator_cost} as
\begin{equation}\label{eq:scale_operator_cost}
    C_O^n(\alpha, \beta,\pi) = c_S \frac{\bar{\lambda}}{\bar{\mu}^*}n + c_S \beta\frac{\bar{\lambda}}{\bar{\mu}^*}n^\alpha + \II(\alpha =1/2)O(n^{1/2}) + \II(\alpha>1/2)o(n^{\alpha}),
\end{equation}
where $\bar{\mu}^*$ is the expected equilibrium service rate for the given staffing and routing policy.  We next define the concept of asymptotical optimality as follows:
\begin{definition}
    Let $(\alpha^{*, n},\beta^{*,n}, \pi^{*,n})$ minimize $C_O^n(\alpha, \beta, \pi)$ for the $n$th system. A policy $(\bar{\alpha},\bar{\beta}, \bar{\pi})$ is said to be asymptotically optimal  if
    \[
    \lim_{n\to \infty}\frac{C_O^n(\bar{\alpha},\bar{\beta},\bar{\pi})}{C_O^n(\alpha^{*,n},\beta^{*,n},\pi^{*,n})} =1.
    \]
\end{definition}

Part 1 of Theorem~\ref{thm:best_response_scaling} implies that if the system operator does not incentivize working faster by providing more idleness to faster servers, i.e., if $g(\mu)$ is non-increasing, then in a large system, all servers will choose to work at their minimum attainable service rate and $\bar{\mu}^*\approx \EE[\tilde{\mu}_{\min,k}]$. For any fixed $\alpha$ and $\beta$, this maximizes the dominating cost term in \eqref{eq:scale_operator_cost} and implies that adopting a policy that yields a non-increasing $g$ when the servers are strategic is (asymptotically) sub-optimal.

Parts 2 and 3 of Theorem~\ref{thm:best_response_scaling} relates the sensitivity of servers to low levels of idleness to the regime coefficient $\alpha$ by quantifying this sensitivity via the derivative of the utility of idleness function $u_I(\cdot)$. To understand the implications on staffing, consider the case where $u_I(I)=\kappa I^p$, for some $\kappa>0$ and $0< p \leq 1$ as an example. For any $\alpha<1$, $n^{\alpha-1}u'(n^{\alpha-1}I)=\kappa p n^{(\alpha-1)(1-p)}I\to 0$ for any $I>0$ as $n\to \infty$ and part 2 of the theorem implies that $\bar{\mu}^*\approx \EE[\tilde{\mu}_{\min,k}]$ for large enough system, again maximizing the dominating term in \eqref{eq:scale_operator_cost}. However, when $\alpha = 1$, the term $n^{\alpha-1} u'(n^{\alpha-1} I)$ simplifies to $\kappa p I$. If there exists a $\beta$ such that the equilibrium mean service rate $\bar{\mu}^*$ satisfies $\bar{\mu}^* / (1+\beta) > \mathbb{E}[\tilde{\mu}_{\min,k}]$, then choosing $\alpha = 1$ is asymptotically optimal. This result aligns with the asymptotic optimality of the purely quality-driven regime established in \cite{Gopalakrishnan_etal2016}. On the other hand, if no such $\beta$ exists, there is no advantage in incentivizing agents to work faster by providing additional idle time.  A complete analysis of the equilibrium for this case is provided in Section~\ref{sec:strategic_grr}

Now, as an example, suppose that the utility of idleness function can be decomposed as $u_I(I) = u_{I,b}(I) + u_{I,d}(I)$, where $u_{I,b}(I)$ is a concave nondecreasing function that describes the benefit of idleness and $u_{I,d}(I)=-\kappa I^{-p}$ for some $\kappa>0$ and $p>0$ which specifically models  the discomfort associated with low levels of idleness. For any $1/2\leq \alpha < 1$, $n^{1-\alpha}u_{I,d}(n^{1-\alpha}I)=\kappa p n^{p(1-\alpha)} I \to \infty$ for any $I>0$ as $n\to \infty$.  Hence, part 3 of Theorem~\ref{thm:best_response_scaling} implies that setting $\alpha < 1$ results in $\bar{\mu}^* \approx \mathbb{E}[\tilde{\mu}_{\max,k}]$ for sufficiently large systems and is asymptotically optimal, as $C_O^n(\alpha, \beta, \pi) \geq \lambda^n / \mathbb{E}[\tilde{\mu}_{\max,k}]$ for all $\alpha$, $\beta$, and $\pi$. In this case, it is preferable to choose $\alpha = 1/2$ as it minimizes the order of the next-leading term.

Under the uniform convergence assumption \eqref{eq:uniform_conv_assumption}, the conclusions of Theorem~\ref{thm:best_response_scaling} only rely on the structure of the utility of idleness function near zero and the routing policy through the density function $g(\mu)$, and do not depend on any other primitive model assumptions such as the distributions of random parameters, safety staffing coefficient $\beta$ and the limiting arrival rate $\bar{\lambda}$. As demonstrated in the above examples, parts 2 and 3 of the theorem covers a wide range of utility idleness functions. However, it is possible to find utility of idleness functions, such as $u_I(I)=\log(I)$, that do not satisfy either of the assumptions.

% When $\alpha<1$, the dominating term in the total staffing cost (where $c_S$ is the unit staffing cost) 
% $c_S\frac{\lambda^n}{\bar{\mu}_F}+c_S\left(\frac{\lambda^n}{\bar{\mu}_F}\right)^{\alpha}$,
% is due to the offered load. Hence, for large $n$, the system operator mainly aims to maximize the average service rate in order to minimize the staffing cost. Suppose that for $\alpha_0<1$, $n^{\alpha_0-1}f'(n^{\alpha_0-1}x)\to \tilde{f}'(x)<\infty$ which is not identically zero. Then, any $\alpha>\alpha_0$ and policy which yields an increasing concave density $g(\mu)$ achieves the maximum possible expected service rate $\bar{\mu}_F=\EE[\tilde{\mu}_{\max,k}^n]$.      

The interchange of limits in~\eqref{eq:interchange_limits} has implications leading to Theorem~\ref{thm:best_response_scaling}. We first note that,  by Assumption~\ref{asm:utility_assumptions}, the stationary fairness measure $\eta_{\alpha,\infty}^n$ is absolutely continuous, and hence, has a density $g^n(\cdot)$  with respect to $F^n$ for all $n\in \NN$ and \eqref{eq:interchange_limits} is equivalent to $g^n(\mu)\to g(\mu)$ as $n\to \infty$. In general, weak convergence of measures does not necessarily imply the convergence of densities. However, we believe that \eqref{eq:interchange_limits} holds for most of the reasonable routing policies. 
Proposition~\ref{thm:interchange_generalized_rr} shows that the limits can be interchanged for any $h$-random policy under a purely quality driven regime. 
\begin{proposition}\label{thm:interchange_generalized_rr}
For $\alpha=1$, under an $h$-random routing policy, Equation \eqref{eq:interchange_limits} holds and the convergence is uniform, i.e., \begin{equation}\sup_{\mu_{\min}\leq \mu\leq\mu_{\max}}\left|\EE[I_k^n(\infty)|\tilde{\mu}_k^n=\mu]- \big( 1+L_F\tilde{h}(\mu) \big)^{-1}\right|\to 0.\label{eq:h_random_individual}
\end{equation}
\end{proposition}

\cite{Gopalakrishnan_etal2016} introduced the \textit{class of idle-time order} based policies, namely, the class of policies where the selected server to which an arrival is routed depends only on the order in which the servers last became idle. 
Some common policies such as longest-idle-server-first and random routing are also in this class (see~\cite{Gopalakrishnan_etal2016} for a formal definition). Using Theorem 9 in~\cite{Gopalakrishnan_etal2016}, we next show that \eqref{eq:interchange_limits} holds for idle-time order based policies under quality-and-efficiency driven regime as well as any quality-driven regime. 

\begin{proposition}
\label{thm:interchange_idle_time_order}
For $1/2\leq \alpha\leq 1$, under any idle-time order based policy,  \eqref{eq:interchange_limits} holds and as $n\to\infty$,
\begin{align*}
&\sup_{\mu_{\min}\leq \mu\leq \mu_{\max}}\left|n^{1/2}\EE[I_k^n(\infty)|\tilde{\mu}_k^n=\mu]- \mu\bar{\lambda}^{-1}\EE\left[\EE[(\xi_{1/2}(\infty))^-|\zeta]\right]\right|\to 0,\quad \mbox{ for }\alpha=1/2,\\
&\sup_{\mu_{\min}\leq \mu\leq \mu_{\max}}\left|n^{1-\alpha}\EE[I_k^n(\infty)|\tilde{\mu}_k^n=\mu]- \mu\beta\bar{\lambda}^{\alpha-2}\bar{\mu}_F^{2-\alpha}(\sigma_F^2+\bar{\mu}_F^{2})^{-1}\right|\to 0,\quad \mbox{ for all }1/2\leq \alpha<1,
 \mbox{  }\\
&\sup_{\mu_{\min}\leq \mu\leq \mu_{\max}}\left|\EE[I_k^n(\infty)|\tilde{\mu}_k^n=\mu]- \left(1+L_f\mu^{-1}\right)^{-1}\right|\to 0,\quad  \mbox{ for }\alpha=1.
\end{align*}
\end{proposition}
It is possible to derive the expectation $\EE[(\xi_{1/2}(\infty))^-|\zeta]$ using the techiques in \cite{BrowneWhitt1995Piecewise} on the piecewise-linear diffusion~\eqref{eq:diffusionlimit}. However, the resulting expression is too complex to gain further insights into our system and hence, is omitted.

\subsection[Characterization of Nash Equilibria for h-Random Policies]{Characterization of Nash Equilibria for $h$-Random Policies}\label{sec:strategic_grr}
In this section, our goal is to characterize the Nash equilibria when the servers are strategic with heterogeneous abilities and preferences, under a given $h$-random policy in the limiting system. To do so, we start by analyzing the best response strategy of a given server $k$ with parameters $\tilde{\mu}_{\min,k},\tilde{\mu}_{\max,k}$ and $\tilde{a}_k$ when the distribution of service rate for all other servers is $F^{(0)}$, i.e., known. Using Proposition~\ref{thm:interchange_generalized_rr}, the utility maximization for server $k$ in the $n$-limit takes the form
\begin{equation}\label{eq:best_response_problem}
    U_k^*:
    =\max_{\tilde{\mu}_{\min,k}\leq \mu\leq \tilde{\mu}_{\max,k}} U_k(\mu)
    =\max_{\tilde{\mu}_{\min,k}\leq \mu\leq \tilde{\mu}_{\max,k}} u_I\left(\big(1+L_{F^{(0)}}\tilde{h}(\mu)\big)^{-1}\right)-\tilde{a}_k c(\mu).
\end{equation}
It is clear that when $\tilde{h}(\mu)$ is non-decreasing, the idleness observed by server $k$ is non-increasing and the best response is to set her service rate to $\mu_k=\tilde{\mu}_{\min,k}$. As standard in the literature, we need concavity properties for the utility function. Hence, we have the following assumption:
\begin{assumption}\label{asm:h_convexity}
The function $\tilde{h}(\mu)$ is convex strictly decreasing satisfying  
\begin{equation}
    2\tilde{h}'(\mu)^2\leq \tilde{h}(\mu)\tilde{h}''(\mu),
    \quad \mbox{ for all }\mu_{\min}\leq \mu \leq \mu_{\max}.
    \label{eq:necessary_concavity}
\end{equation}
If $U_k(\mu_1)=U_k(\mu_2)$ with $\mu_1\leq \mu_2$, the server chooses to serve with $\mu_1$.
\end{assumption}

Condition \eqref{eq:necessary_concavity} is needed to ensure concavity of the limiting utility function without making additional structural assumptions on the utility of idleness function $u_I(\cdot)$. It is satisfied by any $\tilde{h}(\mu)=\mu^{-p}$, where $0<p\leq 1$, which have an interesting managerial interpretation. The case $p=1$ implies $h(\mu)=1$ and corresponds to idle-time order based policies. 
When $p<1$, $h(\mu)=\mu^{1-p}$ and the system operator is more eager to route arrivals to servers with high service rates, but just so that the idleness a server receives is still an increasing function of the service rate. If we have additional information on $u_I(\cdot)$, it is possible to obtain convexity without needing \eqref{eq:necessary_concavity}, e.g., if $u_I(I)=-I^{-1}$, then the utility function is concave for any convex $\tilde{h}(\mu)$. For these cases, the results below still hold. The next lemma shows that the limiting utility function is concave under Assumption~\ref{asm:h_convexity}. 

\begin{lemma}\label{lem:utility_concavity}
Under Assumption~\ref{asm:h_convexity}, the limiting utility function $U_k(\mu)$ is concave for any fixed $L_{F^{(0)}}$.
\end{lemma}

Now, as Lemma~\ref{lem:utility_concavity} ensures that the second order optimality conditions are satisfied, we can concentrate on the first order conditions. Taking the derivative of the utility function and after some algebraic manipulations, first order conditions take the form
\begin{equation}\label{eq:first_order_condition}
    C(\mu,L_{F^{(0)}}):=-\frac{L_{F^{(0)}}u_I'\left((1+L_{F^{(0)}}\tilde{h}(\mu))^{-1}\right)\tilde{h}'(\mu)}{(1+L_{F^{(0)}}\tilde{h}(\mu))^{2}c'(\mu)}=\tilde{a}_k.
\end{equation}
Due to our convexity/concavity assumptions, the left-hand side of \eqref{eq:first_order_condition} is non-increasing in $\mu$ and its solutions  maximize the utility of a server. Based on Assumption~\ref{asm:h_convexity}, letting $\mu_k^{**}$ be the smallest of these solutions, the best response of server $k$, $\mu_k^*$, is given by 
\begin{equation}\label{eq:best_response_server}
    \mu_k^*=
    \left\{\begin{array}{ll}
    \tilde{\mu}_{\min,k} &, \mbox{if } \tilde{a}_k\geq C(\tilde{\mu}_{\min,k}, L_{F^{(0)}})\\
    \mu_k^{**} &, \mbox{if }C(\tilde{\mu}_{\max,k}, L)\leq \tilde{a}_k \leq C(\tilde{\mu}_{\min,k}, L_{F^{(0)}})\\
    \tilde{\mu}_{\max,k} &, \mbox{if } \tilde{a}_k\leq C(\tilde{\mu}_{\max,k}, L_{F^{(0)}})
    \end{array}\right. .
\end{equation}
Considering that the parameters $\tilde{\mu}_{\min,k}$, $\tilde{\mu}_{\max,k}$ and $\tilde{a}_k$ are random with respective distributions described in Section~\ref{sec:systemdynamics}, the distribution function of the optimal service rate of any server $k$ is
\begin{align}\label{eq:distribution_function}
    F^{(1)}\big(\mu|L_{F^{(0)}}\big)
    :=
    \PP\big(\mu_k^*\leq \mu\big)
    =&\PP\big(\tilde{\mu}_{\max,k}\leq \mu\big)
    +\PP\big(\tilde{\mu}_{\max,k}>\mu,\,          
             \tilde{\mu}_{\min,k}<\mu,\, 
             \tilde{a}_k^n\geq C(\mu,L_{F^{(0)}})\big).
\end{align}
We are thus able to characterize the equilibrium service rate distribution as follows.
\begin{definition}
We call $F^{(0)}$ an \textit{equilibrium service rate distribution} if its support is $[\mu_{\min}, \mu_{\max}]$ and  $F^{(0)}(\mu)=F^{(1)}(\mu| L_{F^{(0)}})$, where $F^{(1)}(\mu| L_{F^{(0)}})$ is as given in \eqref{eq:distribution_function} for all $\mu\in[\mu_{\min}, \mu_{\max}]$.  
\end{definition}
Note that the best response of server $k$ depends on $F^{(0)}$ only through $L_{F^{(0)}}$, and $L_{F^{(0)}}$ is uniquely determined by $F^{(0)}$ through \eqref{eq:g_equation}. This $L_{F^{(0)}}$ is then used to build a new distribution $F^{(1)}$ in the form given in \eqref{eq:distribution_function}, and this in turn yields a new $L_{F^{(1)}}$. This process defines an operator $\cL$ which maps $L_{F^{(0)}}$ to $L_{F^{(1)}}$, i.e., $\cL(L_{F^{(0)}})=L_{F^{(1)}}$ whose fixed points, i.e., the set $\{L_F:\cL(L_F)=L_F\}$, characterize the Nash equilibria. In other words, the equilibrium service rate distribution should have the form $F(\mu| L_F)$ where the solution of \eqref{eq:g_equation} is also $L_F$. 
The next theorem summarizes these arguments.
\begin{theorem}\label{thm:fixed_point_equation}
  If $F$ is an equilibrium service rate distribution, then the distribution function $F$ has the form given in \eqref{eq:distribution_function} where $L_F$ is the solution of 
  \begin{equation}
       \int_{\mu_{\min}}^{\mu_{\max}}\mu\frac{1-\beta L_F\tilde{h}(\mu)}{1+L_F\tilde{h}(\mu)}dF(\mu|L_F)=0.\label{eq:g_equilibrium_equation}
  \end{equation}
\end{theorem}

One interesting observation is that the solution $L_F$ of \eqref{eq:g_equilibrium_equation} does not depend on the scaled arrival rate $\bar{\lambda}$ and only depends on the distribution of $(\tilde{a}_k^n, \tilde{\mu}_{\min,k}^n, \tilde{\mu}_{\max,k}^n)$ through Equation \eqref{eq:distribution_function} by determining the measure defining the integral. On the other hand, the integrand only depends on the staffing coefficient  $\beta$ and the routing function $h(\cdot)$. 

To understand how the model assumptions, e.g. the distribution of $\tilde{a}_k^n$, relate to the solution of \eqref{eq:g_equilibrium_equation} we consider the setting in \cite{Gopalakrishnan_etal2016} where the utility of idleness function is the identity function and jobs are routed uniformly at random to the idle servers, i.e., $u_I(I) = I$ and $\tilde{h}(\mu)=\mu^{-1}$. Then, $C(\mu,L_{F})$ defined in \eqref{eq:first_order_condition} becomes 
\begin{equation*}
C(\mu,L_F) = \frac{L_F}{(\mu +L_F)^2c'(\mu)}=\tilde{a}_k^n. 
\end{equation*}
This implies that any server $k$ in the $n$th system with $C(\tilde{\mu}_{\min,k}^n, L_F)\leq \tilde{a}_k^n$ and $C(\tilde{\mu}_{\max,k}^n, L_F)\geq \tilde{a}_k^n$ will work with $\tilde{\mu}_k^n = \tilde{\mu}_{\min,k}^n$ and $\tilde{\mu}_k^n = \tilde{\mu}_{\max,k}^n$, respectively.

In the Appendix~\ref{app:strategic}, we offer an example that allows for an explicit expression of \eqref{eq:g_equilibrium_equation}. We take the trade-off parameters, $\tilde{a}_k^n$, to be continuous random variables and all servers share identical bounds (the computations exploit the proof of Theorem~\ref{thm:fixed_point_equation}). 

Under the additional assumption that $F_{a}$ is a continuous distribution, the next result shows that a solution to Equation~\eqref{eq:g_equilibrium_equation} exists and all solutions are in the interval provided in Lemma~\ref{lem:L_bound}.

\begin{proposition}\label{prop:existence}
Let the distribution of $\tilde{a}_k^n$, $F_a$, be continuous. 
Then, Equation \eqref{eq:g_equilibrium_equation} has at least a solution in the interval $[1/(\beta \tilde{h}(\mu_{\min})),1/(\beta \tilde{h}(\mu_{\max}))]$ and has no solution outside of this interval.
\end{proposition}

Proposition~\ref{prop:existence} does not imply anything about the uniqueness of the solution $L_F$. In fact, there might exist multiple equilibria when the servers are strategic as shown in \cite{gopalakrishnan2021} which is a special case of our model.  While the detailed analysis of systems where there are multiple equilibria remains an important open problem, Proposition~\ref{prop:uniqueness} provides some sufficient conditions for the uniqueness of the solution to \eqref{eq:g_equilibrium_equation}. The proof of the proposition is provided in the Appendix~\ref{app:strategic} along with an example application of the conditions.

\begin{proposition}
    \label{prop:uniqueness}
    Let the distribution of $\tilde{a}_k^n$, $F_a$, be continuous and suppose for any $\mu\in [\mu_{\min},\mu_{\max}]$ and $L_F\in [1/(\beta \tilde{h}(\mu_{\min})),1/(\beta \tilde{h}(\mu_{\max}))]$, the following two conditions hold:
    \begin{enumerate}
        \item The integrand $\mu\frac{1-\beta L_F\tilde{h}(\mu)}{1+L_F\tilde{h}(\mu)}$ is a non-decreasing function of $\mu$.
        \item $C(\mu,L_F)$ is a non-increasing function of $L_F$.
    \end{enumerate}
    Then, the solution of \eqref{eq:g_equilibrium_equation} is unique. 
\end{proposition}

% \begin{assumption}\label{asm:OptimControl}

% We assume $f$ and $c$ to be twice continuously differentiable. Additionally $f$ is concave increasing function ($f'>0$ and $f''<0$) and $c$ is convex and increasing. To guarantee concavity of the objective, we assume that $g(\mu)$ is concave. We also assume that $g(\mu)$ is bounded away from 0, i.e., there exists an $\epsilon_0>0$ such that $g(\mu)>\epsilon_0$ for all $\mu_{\min}\leq \mu \leq \mu_{\max}$

\subsection{Numerical Experiments}\label{sec:numerical}
 
In this section, we numerically analyze the equilibria to gain further insight into the problem.  In
order to gain better insights, we assume that (i) all servers share the same deterministic  minimum and maximum achievable service rates, i.e., $(\tilde{\mu}_{\min,k}^n,\tilde{\mu}_{\max,k}^n)=(\mu_{\min},\mu_{\max})$ w.p.~1 for all $k, n\in  \NN$, (ii) the trade-off parameters, $\tilde{a}_k^n$s, follow a discrete uniform distribution with 1001 scenarios with equally spaced outcomes between lower bound $a_{\min}$ and $a_{\max}$ (additional results for different $\tilde{a}_k^n$ distributions are provided in the Appendix~\ref{sec:additional_numerics}) and (iii) the utility of idleness, cost of effort and the routing functions are power functions, i.e., $u_I(\mu)=\mu^r, c(\mu)=\mu^{q}$ and $\tilde{h}(\mu)=\mu^p$. Table~\ref{tab:basecase} presents the data we use as the base-case of our numerical experiments.  

\begin{table}[h!tb]
    
    \begin{center}
    \caption{The base-case scenario for experiments}
    \begin{tabular}{cccccccccc}\hline
    %\up\down
    $\bar{\lambda}$ &$\beta$ &$\mu_{\min}$ &$\mu_{\max}$ &$\gamma$ &$a_{\min}$ &$a_{\max}$ &$p$ &$q$ &$r$\\
    100 &0.4 &0.5 &1.5 &0.1   &0.05 &0.15 &-1 & 2 &0.5\\\hline
     \end{tabular}
     \label{tab:basecase}
     \end{center}
\end{table}

Theorem~\ref{thm:fixed_point_equation} characterizes the mean-field equilibrium for the fluid limit and Propositions~\ref{prop:existence} and \ref{prop:uniqueness} proves that under mild conditions an equilibrium exists and is unique. We provide further numerical results in this direction in the Appendix~\ref{sec:additional_numerics}. The question of whether this mean-field equilibrium  accurately characterizes the equilibrium for a system with finitely many servers in the pre-limit remains an open problem. For the base-case scenario in Table~\ref{tab:basecase}, we design experiments where all the servers initially assume the service rate distribution to be $F^{(0)}$ and update their optimal service rate iteratively using \eqref{eq:mu_update_equation}, i.e., at iteration $i$, servers set their service rates by replacing $F^{(0)}$ with $F^{(i-1)}$ in \eqref{eq:mu_update_equation}.

To measure the discrepancy between $F^{(i)}$ at each iteration and the mean-field equilibrium service rate distribution in Theorem~\ref{thm:fixed_point_equation}, we use the $L_1$-distance between these two distributions which we calculate as follows. First, we find the equilibrium solution $L_F$ of \eqref{eq:g_equilibrium_equation}, and calculate the mean-field equilibrium service rate $\mu^*(a)$ for any trade-off parameter $\tilde{a}_k^n=a$ by solving \eqref{eq:first_order_condition}. Then, using an accurate representation for the expected idleness to be experienced as a function of $\mu$, $\EE[I_k^n(\infty)|\tilde{\mu}_k^n=\mu]$ given the service rate distribution $F^{(i)}$, we find the maximizer $\mu_i^{*,n}(a)$ of the utility function $U_{k}^n(\mu,F^{(i)})$ in \eqref{eq:utility} when $\tilde{a}_k^n=a$ at iteration $i$. We then calculate the desired $L_1$-distance as
\[
\int_{a_{\min}}^{a_{\max}}|\mu_i^{*,n}(a)-\mu^*(a)|dF_a(a).
\]
To accurately calculate $\EE[I_k^n(\infty)|\tilde{\mu}_k^n=\mu]$ for any $\mu\in[\mu_{\min}, \mu_{\max}]$ at iteration $i$, we propose two approaches. The first approach uses the fluid approximation in Proposition~\ref{thm:interchange_generalized_rr} to approximate the expected idleness for a given service rate $\mu$ replacing $F$ to be $F^{(i)}$, while the second approach estimates this quantity through extensive simulations.

Using three initial distributions, $F^{(0)} = \delta_{\mu_{\max}}$, $F^{(0)} = \delta_{\mu_{\min}}$, and $F^{(0)} = \text{Uniform}(\mu_{\min}, \mu_{\max})$, Figures~\ref{fg:fluid_convergence} and \ref{fg:simulation_convergence} illustrate the discrepancy between the best response rates at each iteration and the mean-field equilibrium for the base-case scenario, based on the fluid and simulation approximations described above, respectively. We see that the service rate distribution converges to the equilibrium in very few iterations for all three initial distributions in both sets of experiments. In Figure~\ref{fg:lambda_convergence}, we compare the equilibrium service rate distributions for finite-server systems, derived from the simulation approximation, with the mean-field equilibrium distribution. The figure demonstrates that the equilibrium service rate distributions converge as $\lambda^n \to \infty$.

 \begin{figure}
\caption{The convergence of best response rates to the equilibrium}
    \centering
    \begin{tabular}{ccc}
    \subfloat[Fluid Approximation \label{fg:fluid_convergence}]{%
  \includegraphics[width=0.3\textwidth]{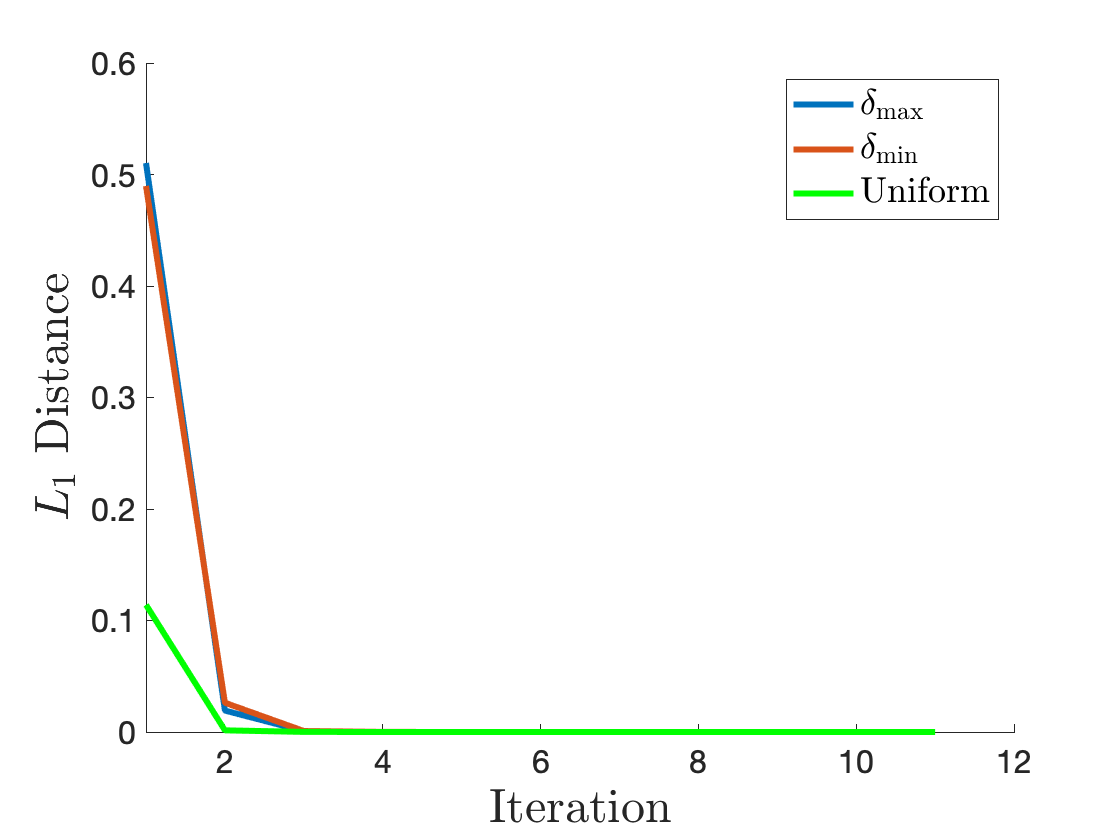}}&
  \subfloat[Simulation Approximation \label{fg:simulation_convergence}]{%
  \includegraphics[width=0.3\textwidth]{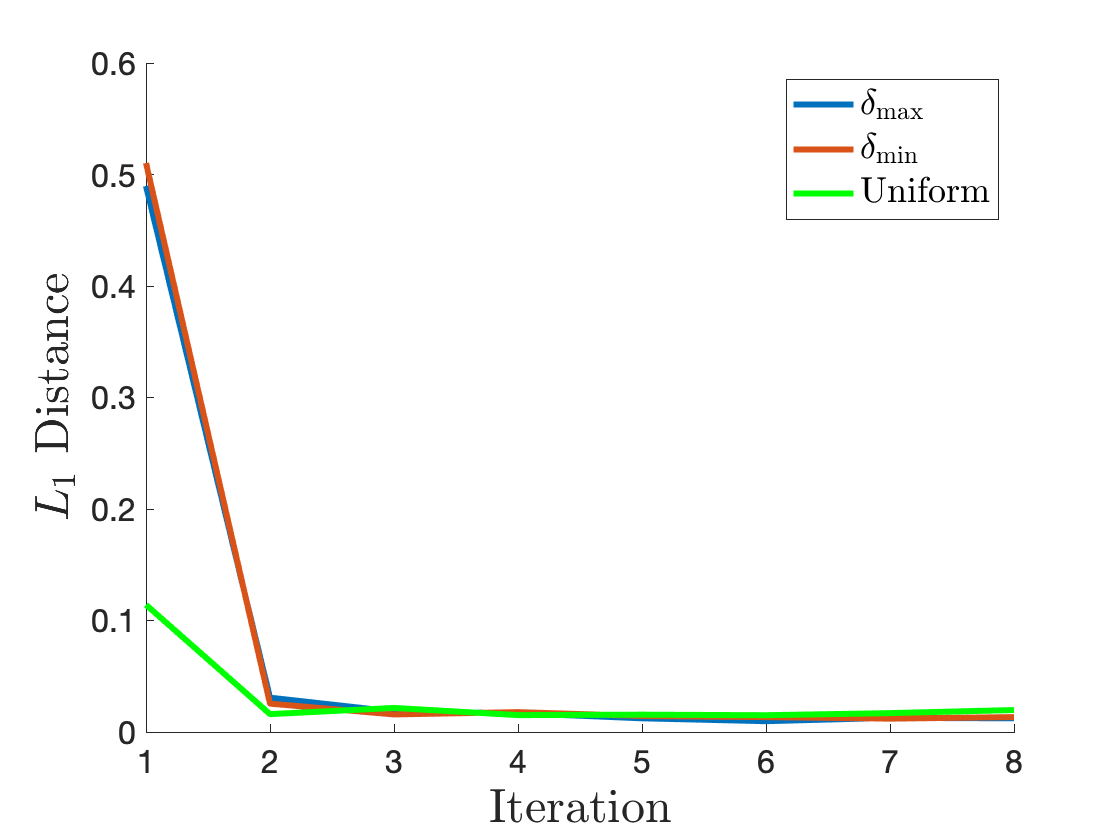}}
  \subfloat[Convergence as $\lambda^n$ increases \label{fg:lambda_convergence}]{%
  \includegraphics[width=0.3\textwidth]{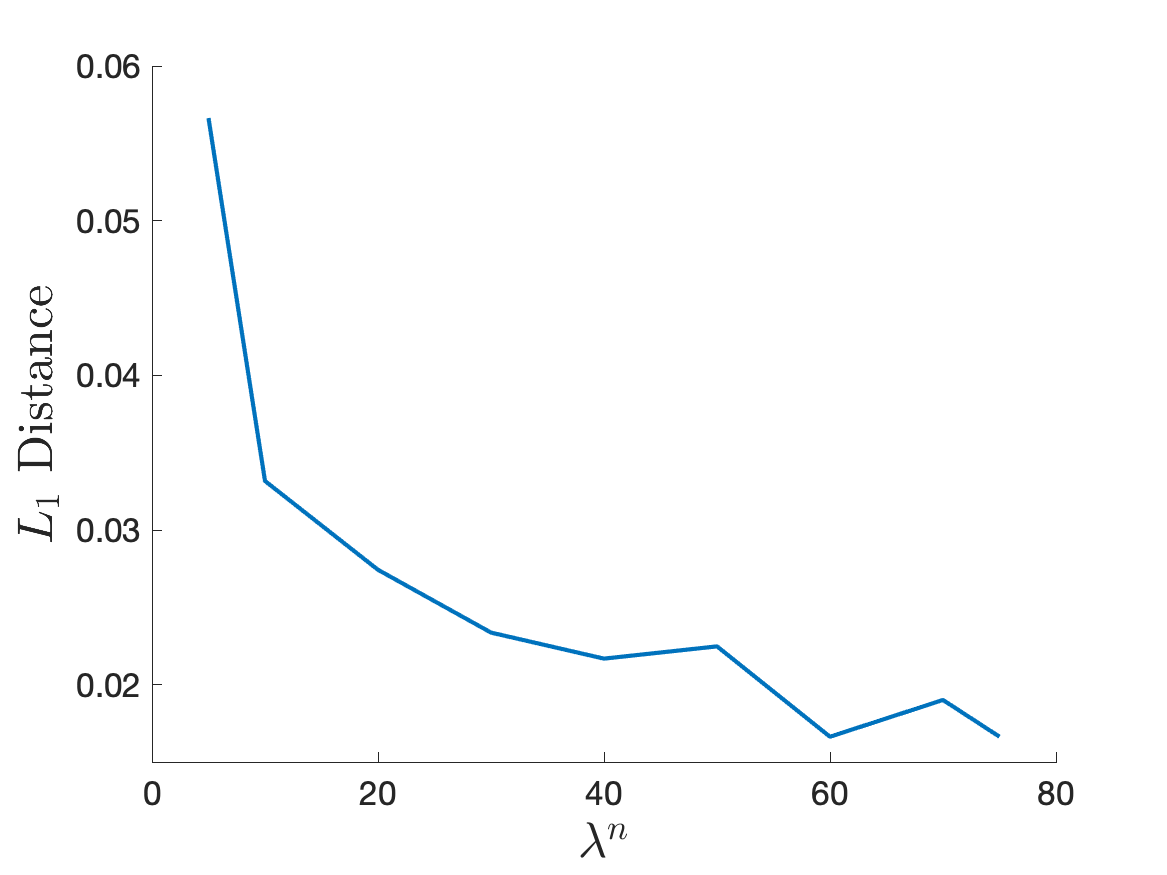}}
  \end{tabular}
  \label{fg:equilibrium_convergence}
 \end{figure}
 %%%%%%%%%%%%%%%%%%%%%%%%%%%%%%%%%%%%%%%%%%%%%%%%%%%%%%%%%%%%%%%%%%

%  \begin{figure}[htp]
% \centering
% % \subfloat[$\rho^2-\theta_1$ and $\rho^2-\theta_2$]{%
% %   \includegraphics[width=0.45\textwidth]{example-image-a}%
% % }%
% % \subfloat[$\delta_1-\theta_1$ and $\delta_1-\theta_2$]{%
% %   \includegraphics[width=0.45\textwidth]{example-image-b}%
% % }
% \caption{Integral on the left-hand side of \eqref{eq:g_equilibrium_equation} as function of $L_F$ for various values of $r$ and $\beta$.}
% \subfloat[Integral vs. $L_F$ varying $r$ in the base-case\label{fg:integral-r}]{
%           \includegraphics[width=0.4\textwidth]{L_integral_var_q.png}
%         }
% \subfloat[\label{fg:L_var}Integral vs. $L_F$ varying $\beta$ in the base-case]{
%           \includegraphics[width=0.4\textwidth]{L_integral_var_beta.png}
%         }
% \label{fg:Fig1}
% \end{figure}
 %%%%%%%%%%%%%%%%%%%%%%%%%%%%%%%%%%%%%%%%%%%%%%%%%%%%%%%%%%%%%%%%%%%%

%  \begin{figure}[htb]
%     \centering
%     \caption{Integral on the left-hand side of \eqref{eq:g_equilibrium_equation} as function of $L_F$ for various values of $r$ and $\beta$.}
% 	\begin{subfigure}{\textwidth}
% 	    \centering
% 	     \includegraphics[width=0.5\textwidth]{L_integral_var_q.png}\\
% 		\caption{Integral vs. $L_F$ varying $r$ values in the base-case}
% 		\label{fg:integral-r}
% 	\end{subfigure}	
% 	\begin{subfigure}{\textwidth}
% 		\centering
% 	  \includegraphics[width=0.5\textwidth]{L_integral_var_beta.png}\\
% 		\caption{Integral vs. $L_F$ varying $\beta$ values in the base-case}
% 		\label{fg:L_var}
% 	\end{subfigure}
% \end{figure}

We now examine how various functional assumptions, namely, the routing policy, cost of effort function, and utility of idleness function, affect the equilibrium mean service rate, $\bar{\mu}^*$, as shown in Figure~\ref{fg:param_vs_N_mu}. When the random routing parameter $p$ approaches 0, the idleness distribution becomes more egalitarian, reducing the incentive for servers to work faster. Consequently, the equilibrium mean service rate $\bar{\mu}^*$ decreases. Similarly, as $q$ increases, the marginal cost of increasing the service rate rises, leading to another decrease in $\bar{\mu}^*$.

In contrast, Figure~\ref{fg:r_vs_mu} shows that the relationship between $\bar{\mu}^*$, $N$, and $r$ is non-monotonic: $\bar{\mu}^*$ initially increases with $r$ up to a certain point before declining. This behavior can be understood by analyzing the derivative of the utility of idleness function, $u_I'(I) = r I^{r-1}$, which first increases and then decreases as $r$ changes. To balance this dynamic and ensure that the marginal benefit of idleness matches the marginal cost of working harder, servers initially aim to experience more idleness as $r$ increases, but beyond a certain threshold, they begin to reduce the idleness they seek.

Figure~\ref{fg:beta_vs_N_mu} illustrates how $\bar{\mu}^*$ and $N$ vary with the safety staffing level, $\beta$. We observe that the equilibrium service rate $\bar{\mu}^*$ initially increases with $\beta$ before decreasing. As a result, the optimal staffing level first decreases and then increases, indicating the existence of a non-trivial optimal staffing level, $\beta$, that minimizes the staffing cost.                
\begin{figure}
\caption{The sensitivity of the equilibrium mean service rate and the number of servers to various parameters}
\centering
\begin{tabular}{ccc}
  \subfloat[$p$ vs $\bar{\mu}^*$ \label{fg:p_vs_mu}]{%
    \includegraphics[width=0.3\textwidth]{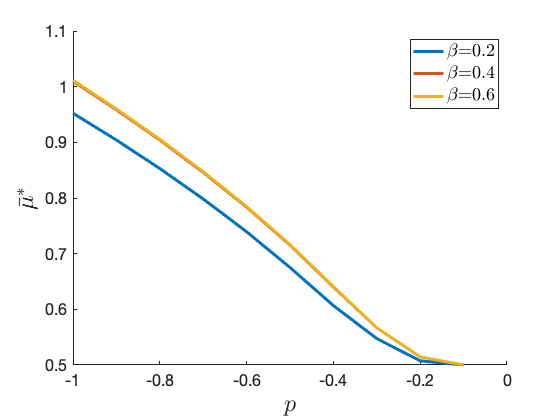}} &
  \subfloat[$q$ vs $\bar{\mu}^*$ \label{fg:q_vs_mu}]{%
    \includegraphics[width=0.3\textwidth]{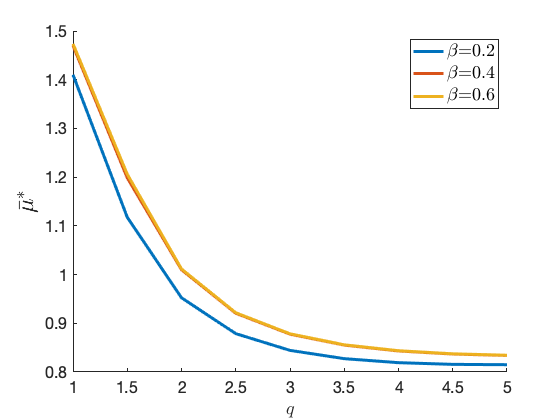}} &
  \subfloat[$r$ vs $\bar{\mu}^*$ \label{fg:r_vs_mu}]{%
    \includegraphics[width=0.3\textwidth]{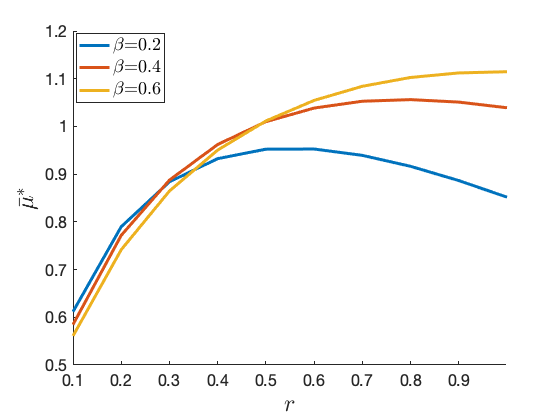}} \\
  % \subfloat[$p$ vs $N$ \label{fg:p_vs_N}]{%
  %   \includegraphics[width=0.3\textwidth]{p_vs_N.png}} &
  % \subfloat[$q$ vs $N$ \label{fg:q_vs_N}]{%
  %   \includegraphics[width=0.3\textwidth]{q_vs_N.png}} &
  % \subfloat[$r$ vs $N$ \label{fg:r_vs_N}]{%
  %   \includegraphics[width=0.3\textwidth]{r_vs_N.png}}
\end{tabular}
\label{fg:param_vs_N_mu}
\end{figure}

  \begin{figure}
   
\caption{The sensitivity of the equilibrium mean service rate and the number of servers to staffing level $\beta$}
    \centering
    \begin{tabular}{cc}
    \subfloat[$\beta$ vs $\bar{\mu}^*$ \label{fg:beta_vs_mu}]{%
  \includegraphics[width=0.3\textwidth]{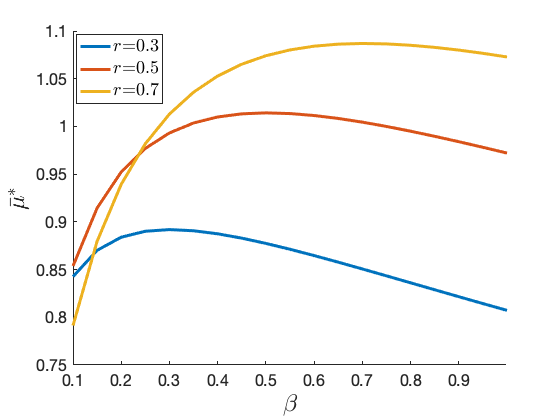}}&
  \subfloat[$\beta$ vs $N$ \label{fg:beta_vs_N}]{%
  \includegraphics[width=0.3\textwidth]{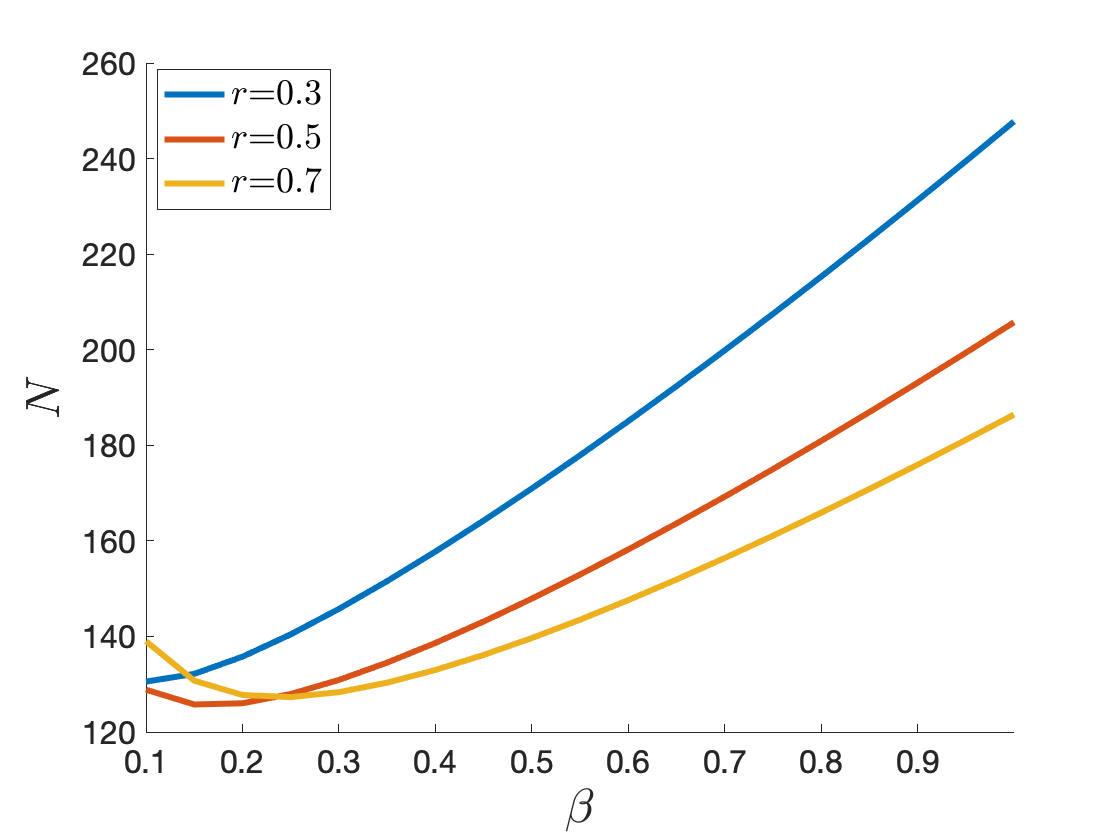}}
  \end{tabular}
   \label{fg:beta_vs_N_mu}
 \end{figure}

As our final experiment, we analyze how the equilibrium distribution is affected by system parameters. Figure~\ref{fg:mu_distributions} illustrates the equilibrium distributions for various values of parameters as histograms. For all our distributions, we observe a sudden spike around the $\mu$ with $C(\mu, L_F) = a_{\max}$ and a decreasing density afterwards. This is a consequence of $C(\mu, L_F)$ having a decreasing derivative with respect to $\mu$ and the distribution of $\tilde{a}_k^n$ being uniform. To see this consider $\mu_1<\mu_2$ in the support of the service rate and take $\Delta>0$. Then, we can find $a_1, a_2, \Delta_1$ and $\Delta_2$ such that 
\begin{align*}
    C(\mu_1,L_F) = a_1, 
    \qquad C(\mu_1 + \Delta) = a_1-\Delta_1, C(\mu_2)=a_2 
    \quad \mbox{ and }\quad 
    C(\mu_2 + \Delta) = a_2-\Delta_2.
\end{align*}
Equation \eqref{eq:first_order_condition} implies $\PP(\tilde{\mu}_{k}^n\in[\mu_i, \mu_i + \Delta]) = \PP(\tilde{a}_k^n\in [a_i-\Delta_i], a_i)$ for $i=1,2$. Hence, the derivative of $C(\mu, L_F)$ being decreasing implies that $\Delta_1>\Delta_2$ and using the uniformity of $\tilde{a}_k^n$, we can conclude that the density of $\tilde{\mu}_k^n$ should be decreasing. Note that the uniformity assumption plays a key role in this conclusion and we present examples where this monotonicity property does not hold when $\tilde{a}_k^n$ follows certain triangular distributions.

As noted above, the routing policy becomes more egalitarian as $p$ approaches 0, reducing the incentive for servers to work faster. Consequently, in Figure~\ref{fg:mu_distributions}, we observe not only a leftward shift in the distribution but also a reduction in its range  indicating that the distribution spreads out more for smaller values of $p$.
When the cost of effort parameter $q$ increases, the rate of change in the cost of effort rises significantly for faster service rates compared to slower ones. 
This results in a leftward shift of the distribution and a less spread out distribution. 
For the utility of idleness parameter $r$, an increase raises the value of idleness, prompting all servers to work faster and causing the distribution to shift left. Additionally, the distribution becomes more spread out, reflecting a greater impact of $r$ on faster servers.

Finally, the staffing parameter $\beta$ does not explicitly appear in \eqref{eq:first_order_condition}, influencing the equilibrium distribution indirectly through changes in the equilibrium $L_F$ that solves \eqref{eq:g_equilibrium_equation}. Hence, the effect of $\beta$ on the distribution's shape is relatively minor compared to other parameters. Nevertheless, as shown in Figure~\ref{fg:beta_vs_mu}, the distribution initially shifts right with increasing $\beta$, indicating that when overall idleness in the system is low due to high safety staffing, servers are motivated to work faster to achieve reasonable levels of idleness. However, as $\beta$ becomes sufficiently large, the idleness provided by high safety staffing becomes adequate, reducing the incentive for servers to work faster.

 \begin{figure}
\caption{The equilibrium distributions for different parametric setups}
    \centering
    \begin{tabular}{ccc}
       \subfloat[$p=-1, q =  2, r = 0.5, \beta = 0.4$ \label{fg:p100q200r050}]{%
  \includegraphics[width=0.3\textwidth]{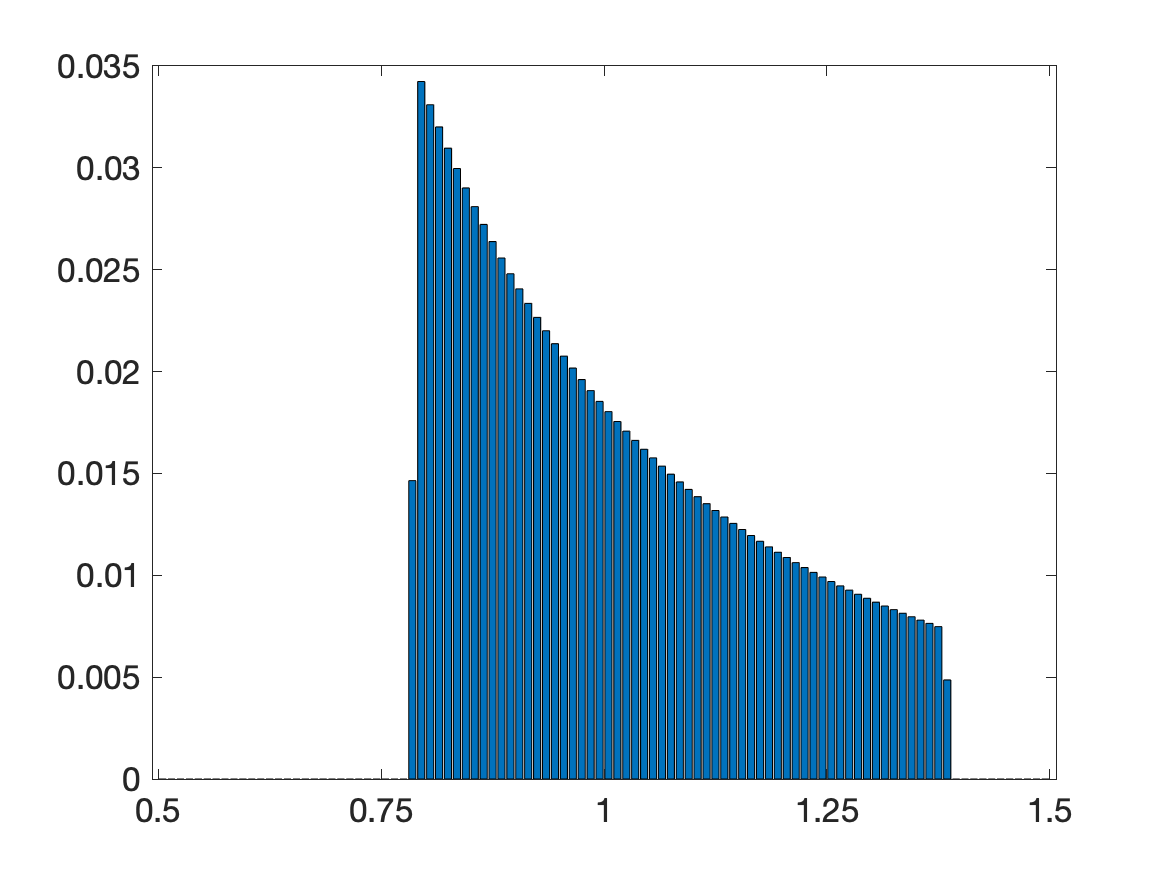}}&
  \subfloat[$p=-0.6, q =  2,r = 0.5, \beta = 0.4$ \label{fg:p060q200r050}]{%
  \includegraphics[width=0.3\textwidth]{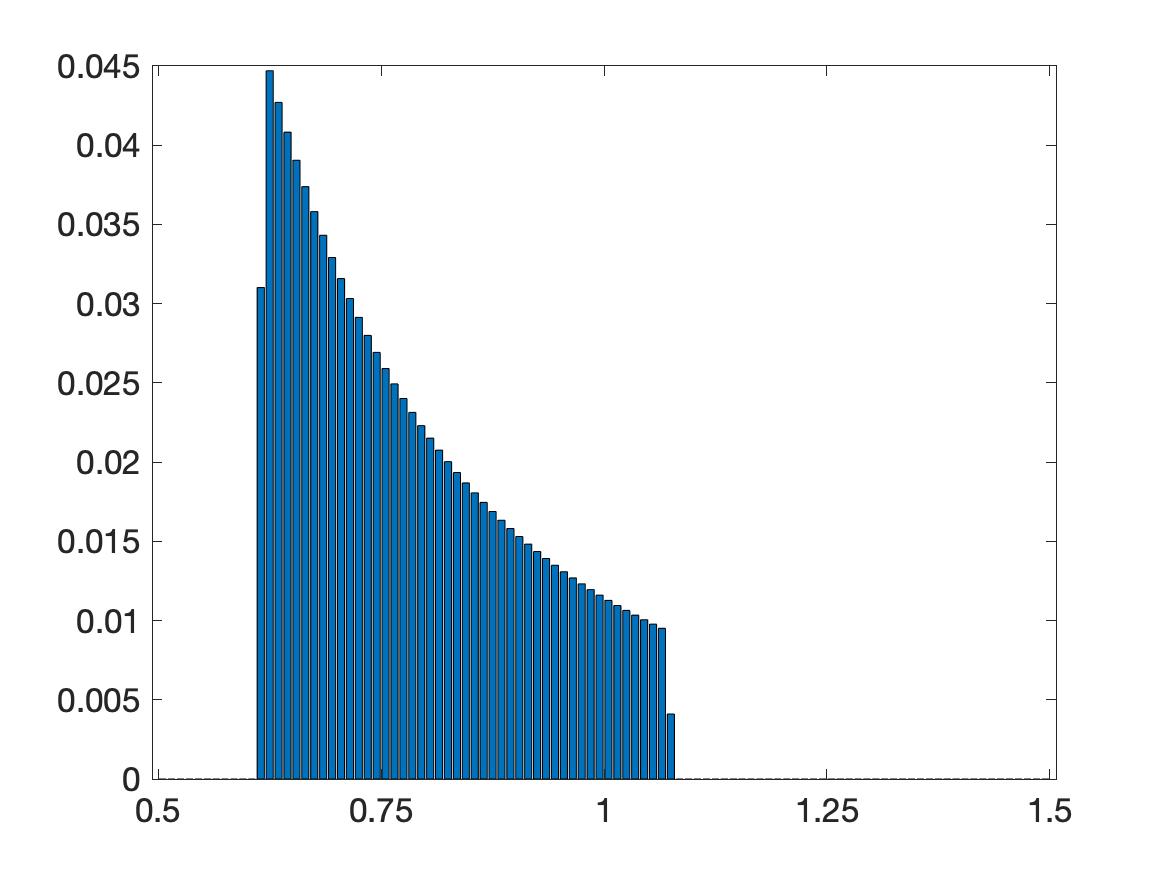}} &
  \subfloat[$p=-0.2, q =  2,r = 0.5, \beta = 0.4$ \label{fg:p020q200r050}]{%
  \includegraphics[width=0.3\textwidth]{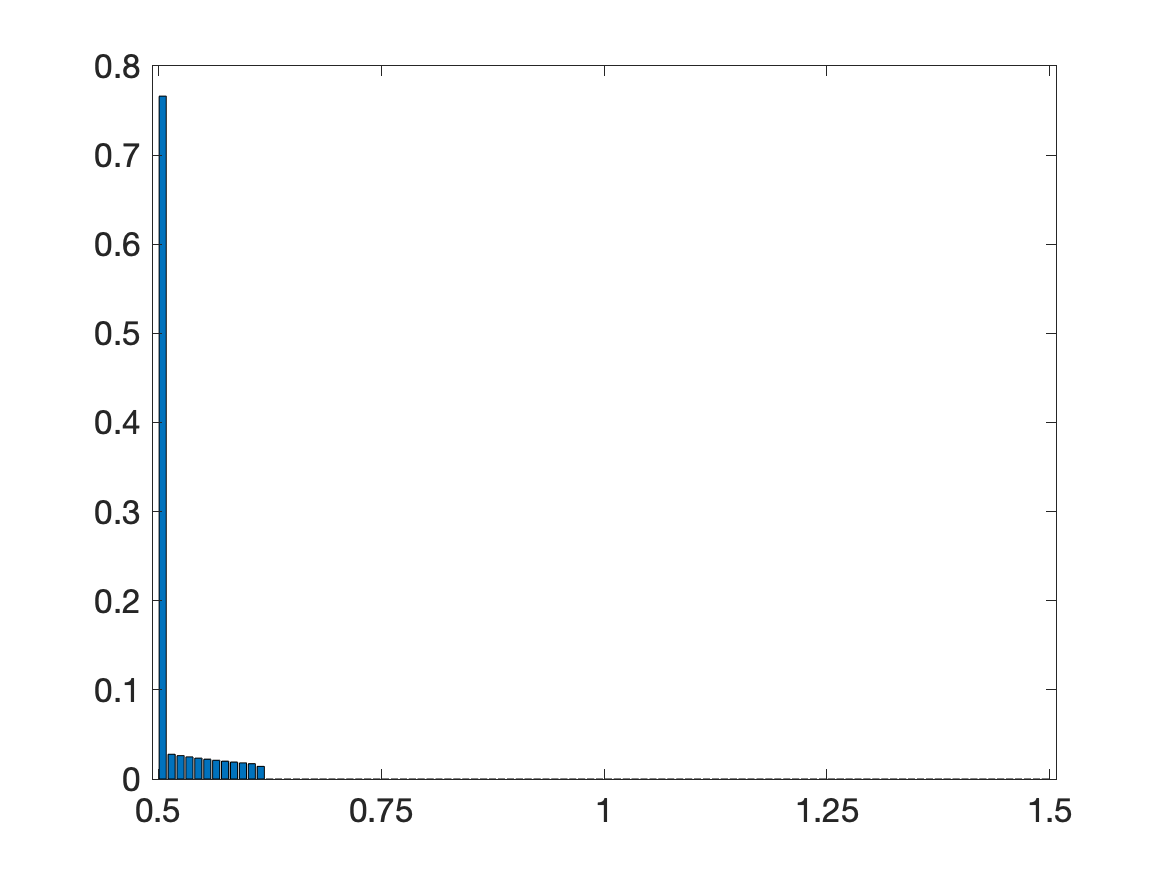}}\\
    \subfloat[$p=-1, q =  3, r = 0.5, \beta = 0.4$\label{fg:p100q300r050}]{%
  \includegraphics[width=0.3\textwidth]{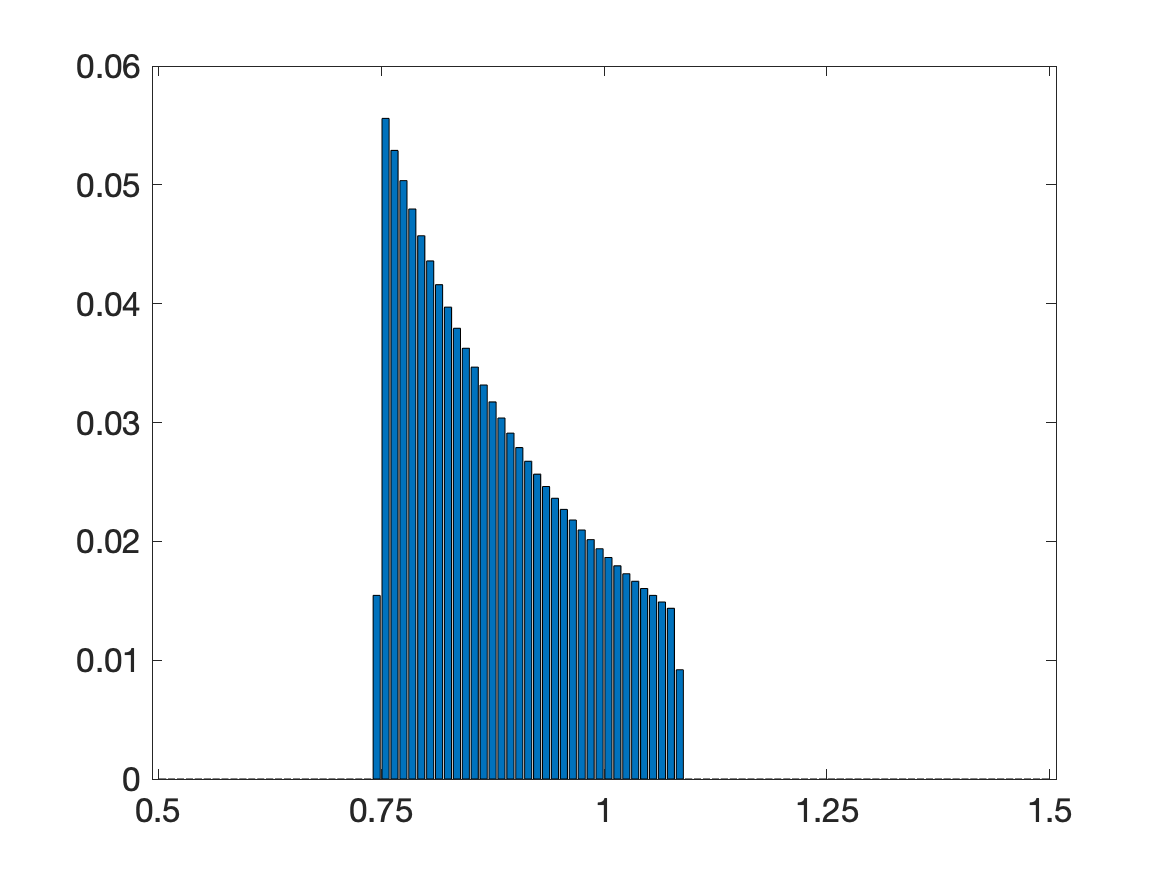}}&
  \subfloat[$p=-1, q =  4, r = 0.5, \beta = 0.4$ \label{fg:p100q400r050}]{%
  \includegraphics[width=0.3\textwidth]{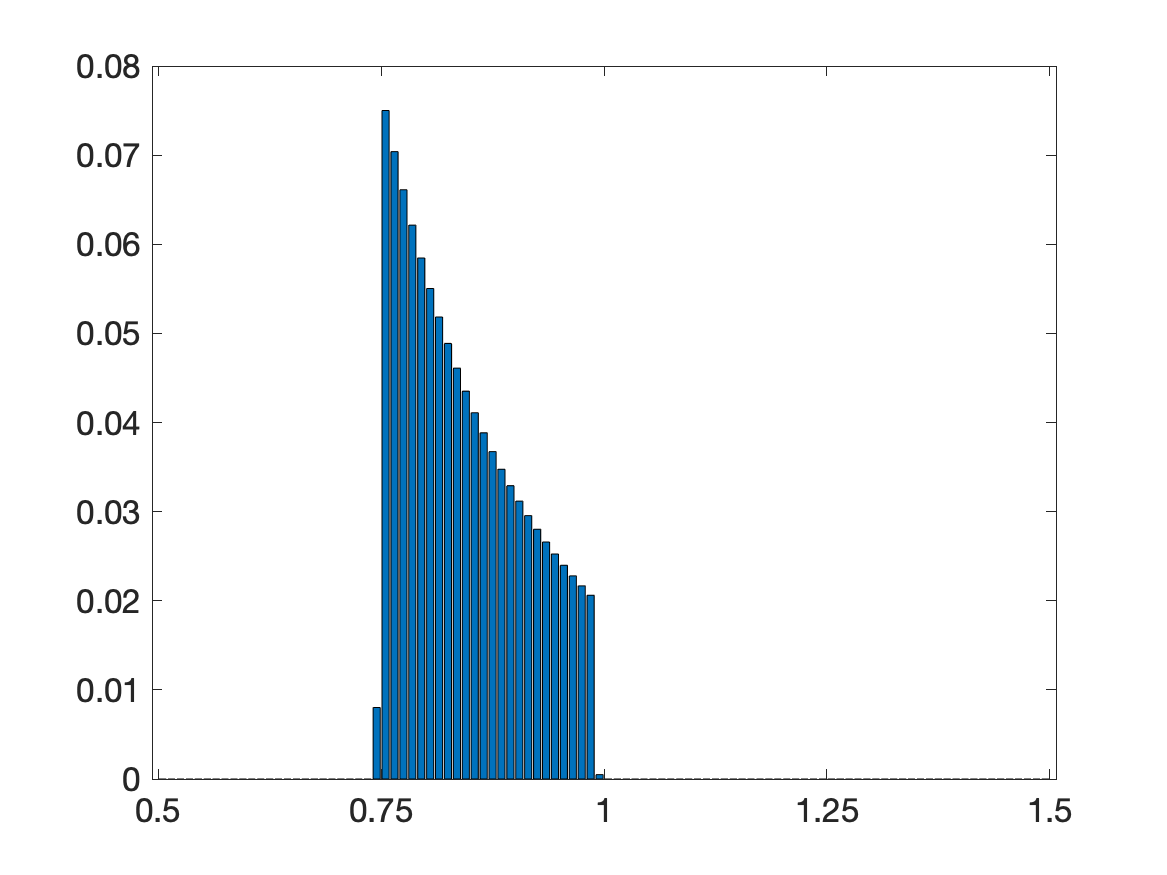}} &
   \subfloat[$p=-1, q =  5, r = 0.5, \beta = 0.4$ \label{fg:p100q500r050}]{%
  \includegraphics[width=0.3\textwidth]{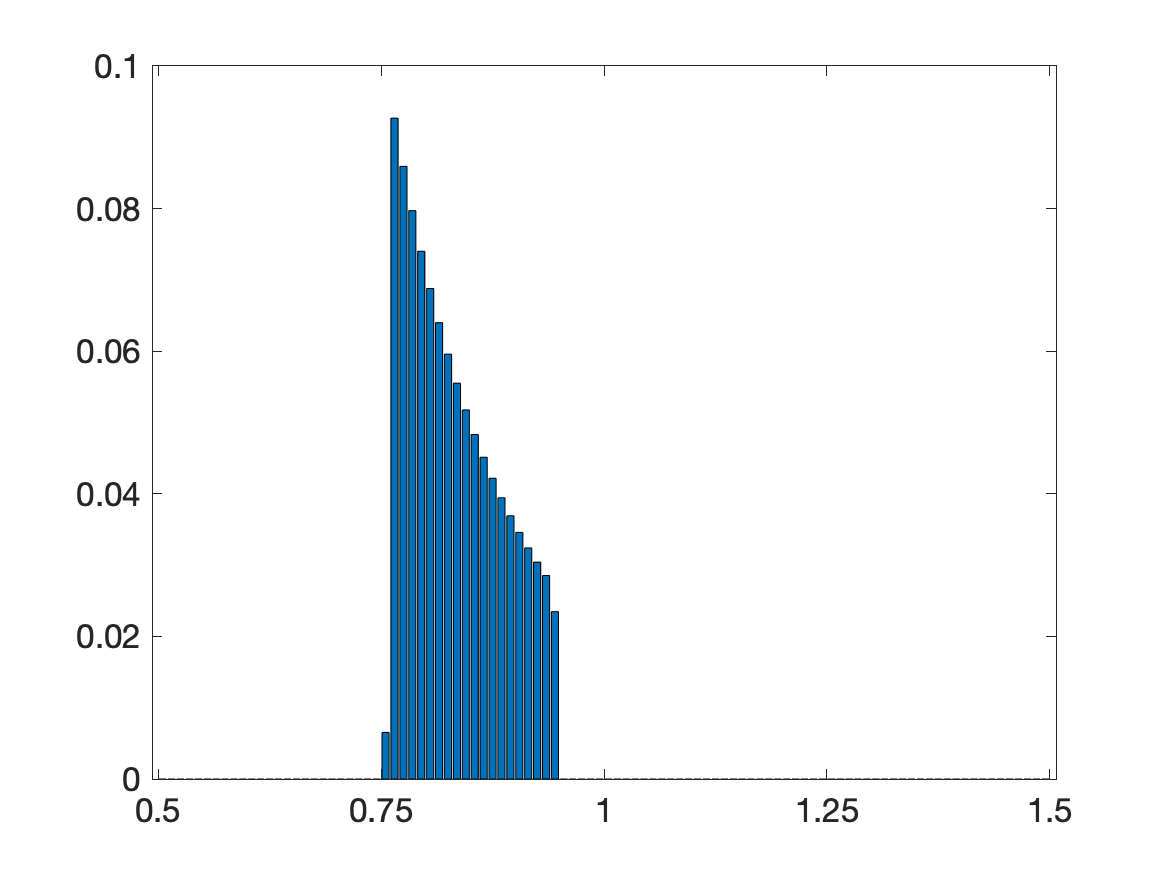}}\\
    \subfloat[$p=-1, q =  2, r = 0.25, \beta = 0.4$ \label{fg:p100q200r025}]{%
  \includegraphics[width=0.3\textwidth]{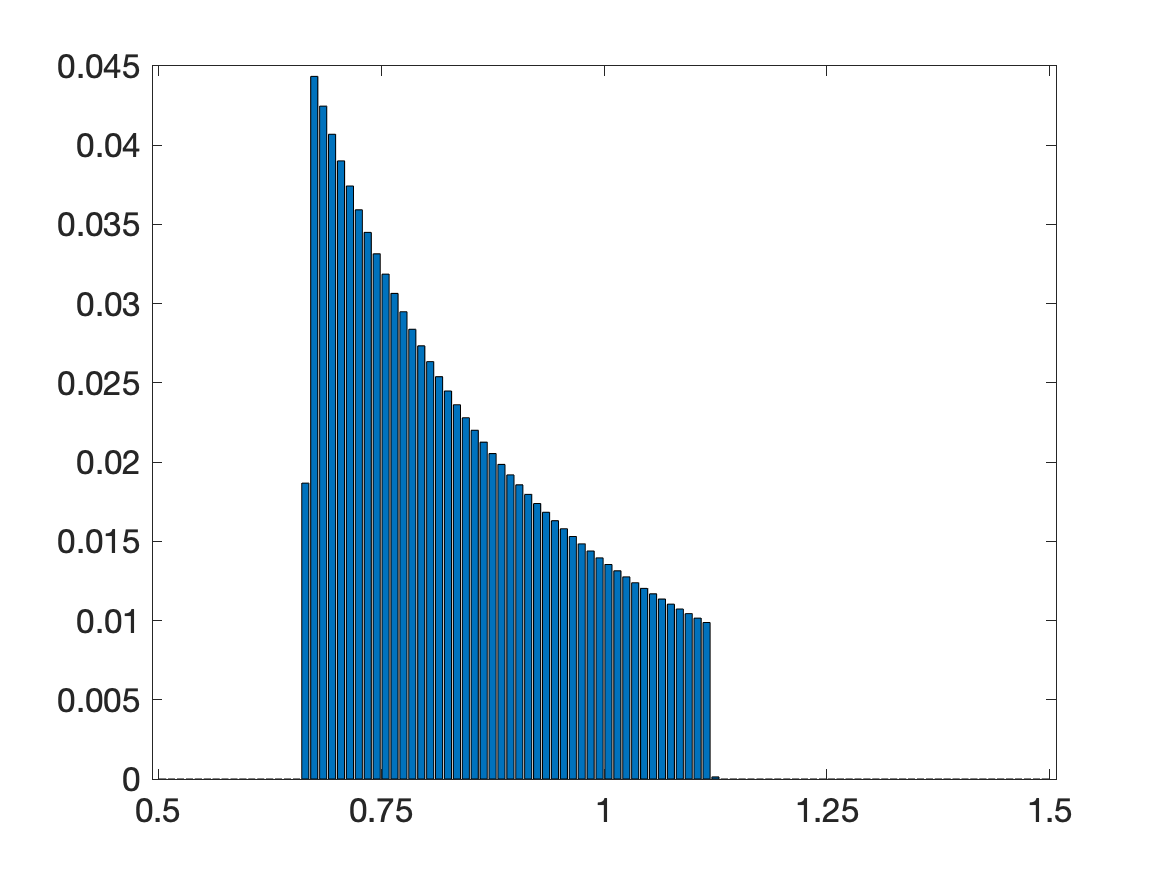}}&
  \subfloat[$p=-1, q =  2, r = 0.75, \beta = 0.4$ \label{fg:p100q200r075}]{%
  \includegraphics[width=0.3\textwidth]{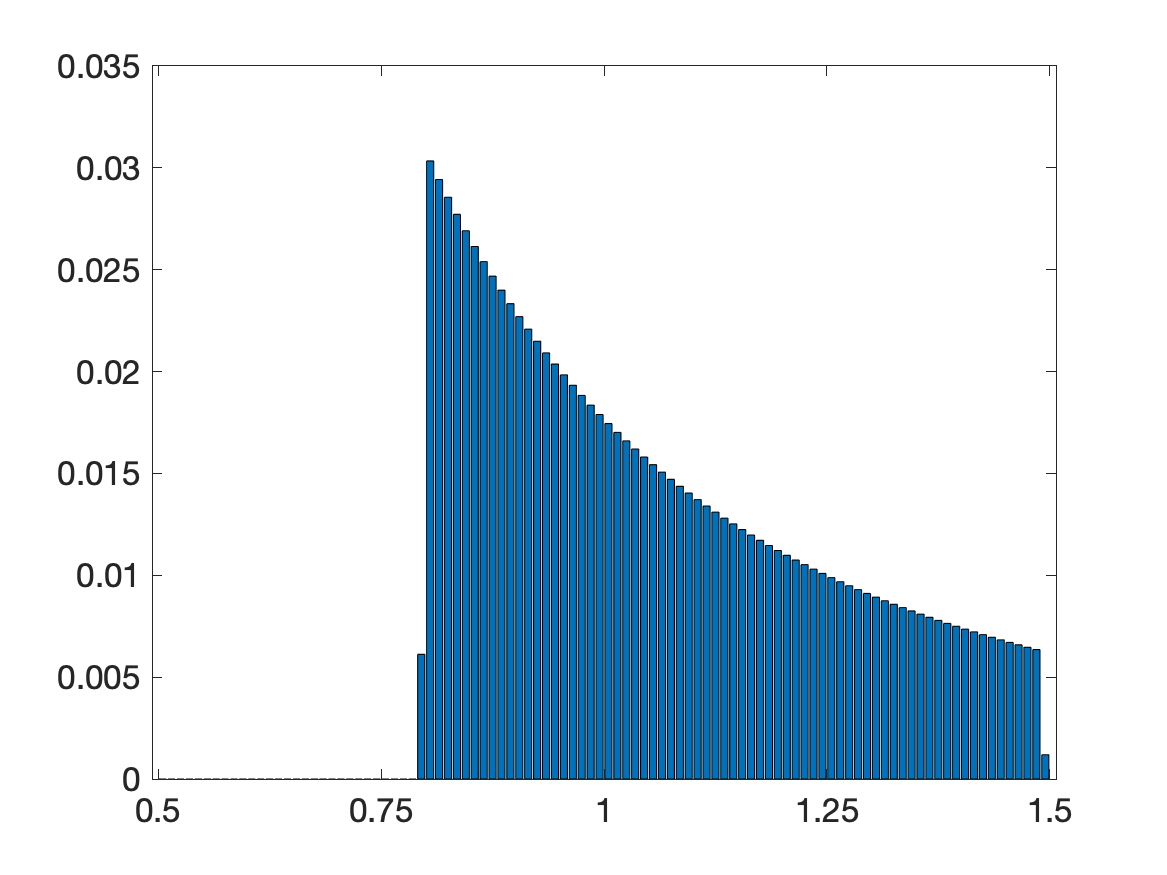}} &
   \subfloat[$p=-1, q =  2, r = 1, \beta = 0.4$ \label{fg:p100q200r100}]{%
  \includegraphics[width=0.3\textwidth]{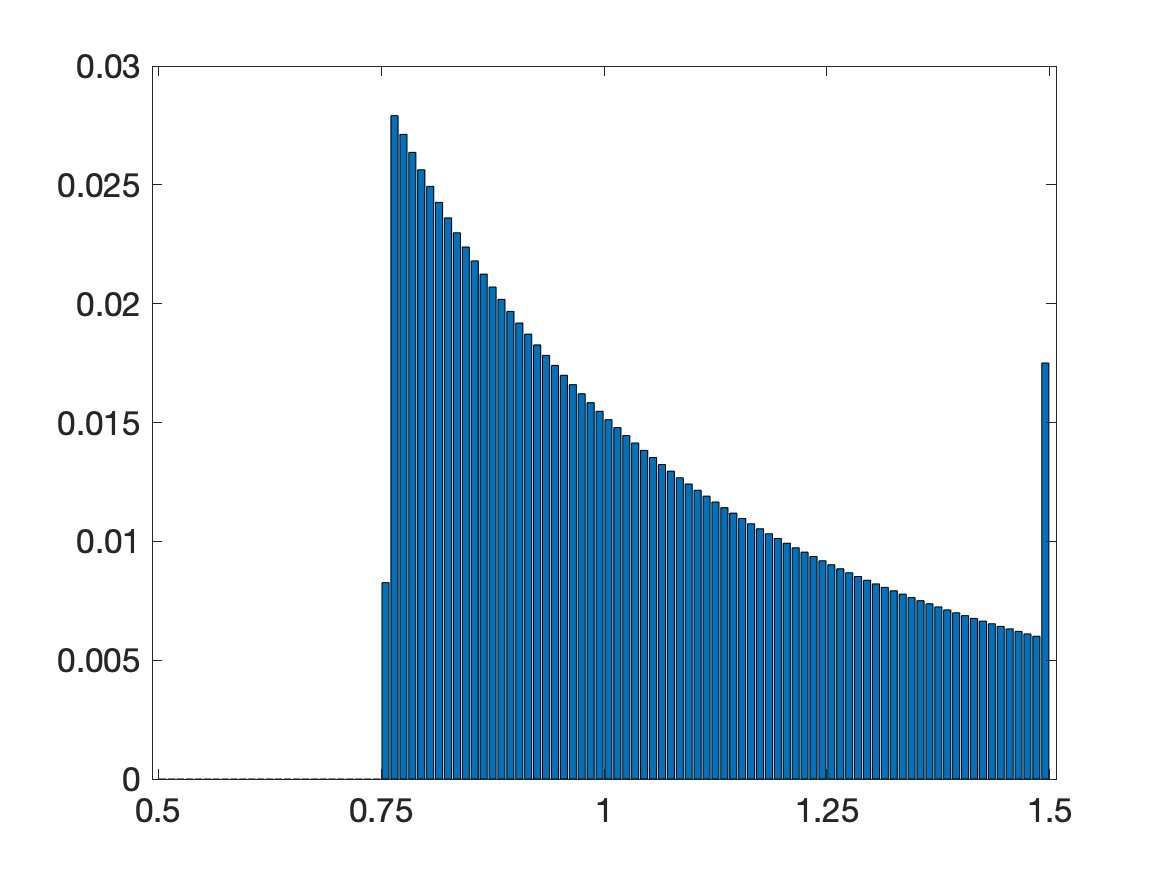}}\\
      \subfloat[$p=-1, q =  2, r = 0.5, \beta = 0.2$ \label{fg:dist_beta02}]{%
  \includegraphics[width=0.3\textwidth]{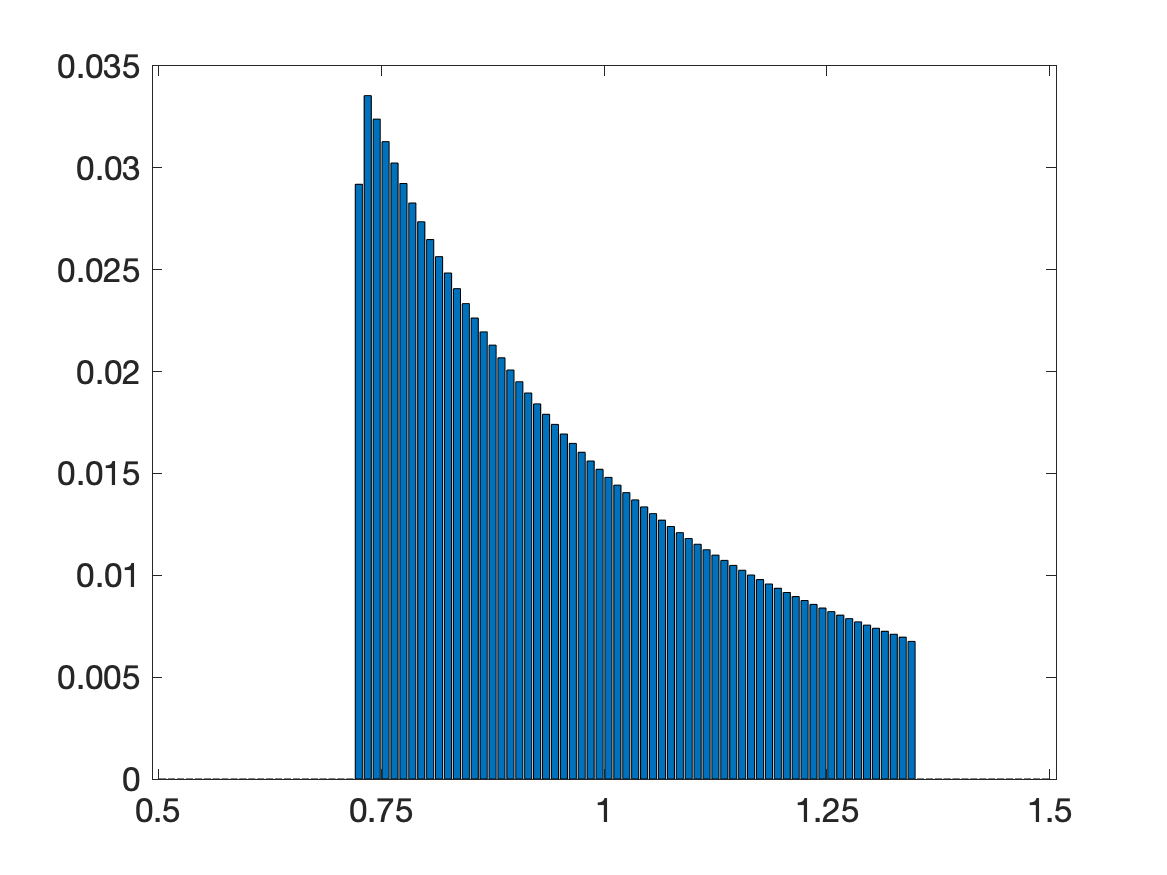}}&
  \subfloat[$p=-1, q =  2, r = 0.5, \beta = 0.6$ \label{fg:dist_beta06}]{%
  \includegraphics[width=0.3\textwidth]{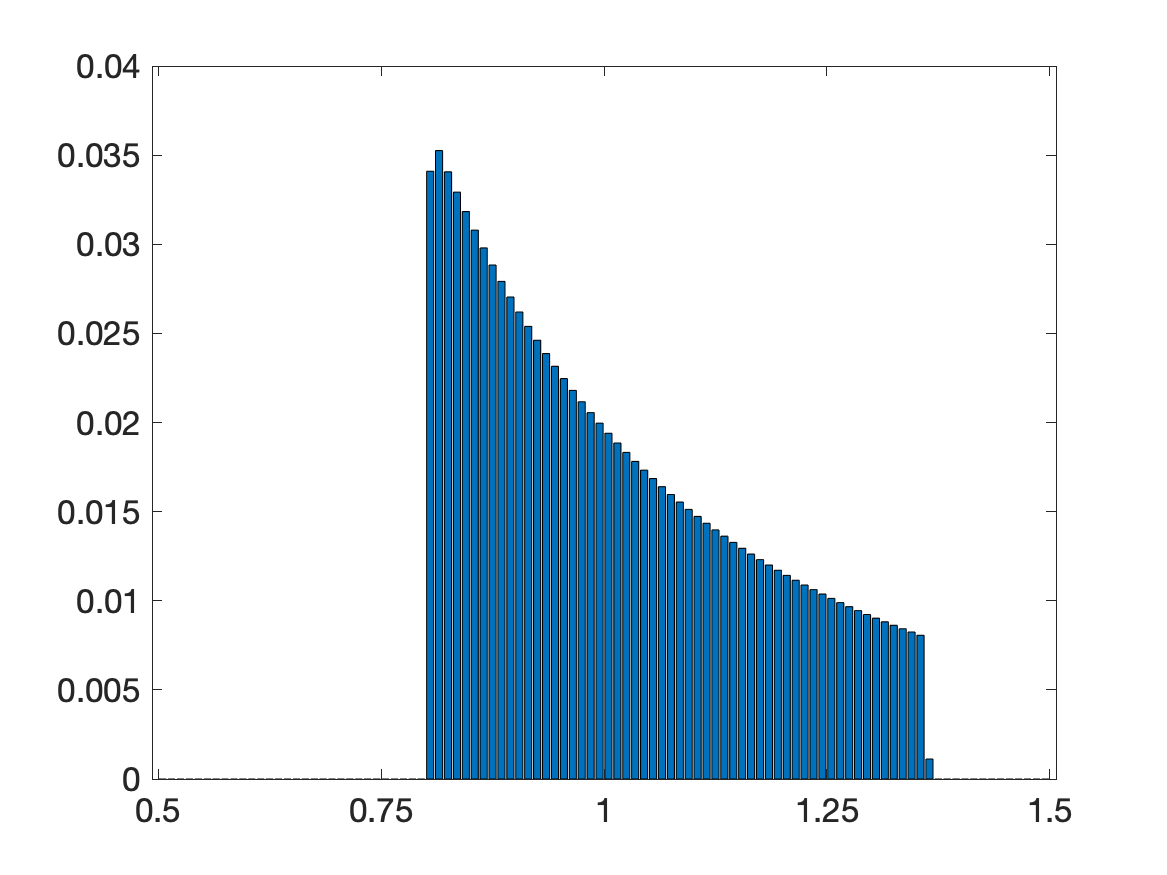}} &
   \subfloat[$p=-1, q =  2, r = 0.5, \beta = 0.8$ \label{fg:dist_beta08}]{%
  \includegraphics[width=0.3\textwidth]{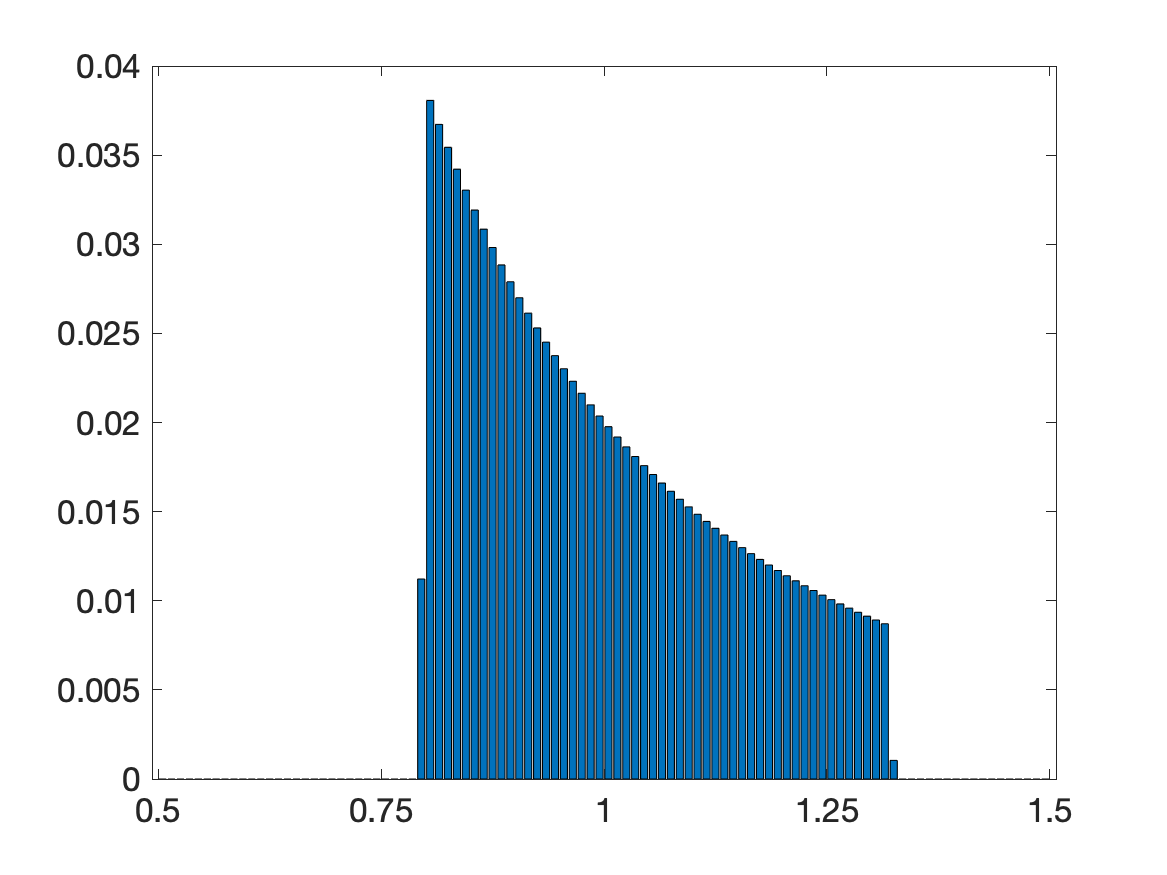}}
  \end{tabular}
  \label{fg:mu_distributions}
 \end{figure}

\section{Concluding Remarks}\label{sec:conclusion}

Servers in most service systems are humans and have inherently diverse preferences and capabilities. This work introduces a framework to analyze strategic server behavior in systems with heterogeneous abilities and preferences. Specifically, it allows concave utility functions of idleness, explicitly modeling extreme discomfort from insufficient idleness, and incorporates individual preferences in balancing the utility of idleness with the cost of working faster.

Unlike prior work, where the utility of idleness is modeled linearly and a purely quality-driven regime is proven to be asymptotically optimal, we demonstrate that when discomfort grows rapidly as idleness approaches zero, the quality-and-efficiency-driven regime becomes optimal. This is the first time this conclusion appears in the literature and indicates that square-root staffing performs well even when servers act strategically. 

We also provide the first detailed analysis of generalized random routing policies in a purely quality-driven regime, examining asymmetric equilibria in this context. Our results include broadly applicable sufficient conditions for the existence and uniqueness of equilibrium service rate distributions, supplemented by extensive numerical analysis of these equilibria.

This work focuses on characterizing asymptotically optimal staffing policies. Identifying optimal routing policies to improve system performance with heterogeneous strategic servers remains an open challenge. Furthermore, our analysis assumes routing policies with fairness processes that possess a density, excluding popular policies such as fastest-server-first and slowest-server-first, which require additional study. To date, research on queues with strategic servers, including this work, has largely addressed systems with a single queue and server pool. However, real-world service systems often involve multiple queues and pools, where customers and servers are categorized by needs and abilities. Investigating strategic server behavior under skill-based routing presents a particularly compelling and challenging direction for future research.

\appendix
\section*{Appendices}
\addcontentsline{toc}{section}{Appendices}
\section{Proofs for Theorems}

In this appendix, we present the proofs of the technical results presented in our manuscript  ``Many-Server Queueing Systems with Heterogeneous Strategic Servers in Heavy Traffic'' and additional numerical results to justify our conclusions. The proofs are organized following the sections in the main document. 

 % with working draft in \url{https://arxiv.org/abs/2211.04158}

\subsection{Additional Notation}
In addition to the notation described in Section~\ref{sec:notation}, we define the short-hand notation $x\wedge y:=\min\{x,y\}$. To have a compact notation, we denote $\vec{\mu} := (\tilde{\mu}_1^n, \ldots, \tilde{\mu}_{N_\alpha^n}^n)$. We define $d_S$ to be the metric for Skorokhod-$J_1$ topology, and for any function $x(\cdot)$ we define $|x(t)|_T^*:=\sup_{0\leq t\leq T}|x(t)|$. 
Lastly, we denote the optional quadratic variation of a stochastic process $Y(t)$ as $[Y]_t$.

\subsection{Formal Definition of the Limiting Fairness Process}\label{app:fairness_definition}

To be able to address strategic servers with heterogeneous service rates, the tool we use is the \emph{fairness process} introduced by \cite{bukeqin2019} and discussed in Section \ref{sec:fairnessprocess}. 
The fairness process is a probability measure-valued process showing how the idleness is distributed among servers with different service rates. The name is motivated by the fact that the distribution of idleness is generally viewed as  a measure of fairness towards servers see, e.g., \cites{armward10, wararm13}.  

Unfortunately, it can be shown that the processes $\eta_{\alpha}^{\pi,n}$ do not converge weakly as $n\to \infty$ in any of the four Skorokhod topologies due to the singularity around $t=\tau_0$. 
Hence, to properly define the limiting fairness process, \cite{bukeqin2019} suggest using shifted versions of the fairness process. For any $\epsilon>0$,  the $\epsilon$-shifted fairness process, $\cS_\epsilon\eta_{\alpha,t}^{\pi,n}$, is defined as 
\begin{equation}
    \cS_\epsilon\eta_{\alpha,t}^{\pi,n}(\AAA):=\left\{\begin{array}{ll}
    \eta_{\alpha,t}^{\pi,n}(\AAA) &, \mbox{if }t>\tau_\epsilon^n,
    \\
    \delta_0(\AAA) &, \mbox{if }t\leq \tau_\epsilon^n, 
    \end{array}\right.
    \label{eq:shifted_fairness}
\end{equation}
 for all $\AAA \in \cB(\RR_+)$ and $\epsilon>0$, where $\tau_\epsilon^n :=\inf\{t>0:\int_0^t\hat{I}_\alpha^n(s)ds>\epsilon\}$. 
 
The following lemma shows that the scaled idleness processes are stochastically bounded. This result is key in enabling us to study the convergence of fairness measures as well as derive fluid and diffusion limits in Section~\ref{sec:convergence_system}.
\begin{lemma}
For $1/2\leq\alpha\leq 1$, the sequence of processes  $\{(\hat{X}_\alpha^n(t))^+\}_{n\in \NN}$, $\{\hat{I}_\alpha^n(t)\}_{n\in\NN}$ are stochastically bounded. % and the sequence of stopping times $\{\tau_\epsilon^n\}_{n\in\NN}$ for $\epsilon>0$ are stochastically bounded. %Moreover, for $1/2<\alpha\leq 1$ there exists a time $t_M>0$ such that  $\sup_{t_M\leq t\leq T}(\hat{X}_\alpha^n(t))^+\toP 0$ for all $T>t_M$ as $n\to \infty$. 
\label{lem:pos_conv}
\end{lemma}

\begin{proof}
Lemma 1 in \cite{bukeqin2019} provides the proof for $\alpha = 1/2$, and hence, we concentrate on the cases where $1/2<\alpha\leq 1$. Let $T>0$, $n\in \NN$, and set $\theta^{n,0}:=\inf\{t\geq 0:(\hat{X}_\alpha^n(t))^+= 0\}$. Then, we have
\begin{align*}
    \PP\left(\left|(\hat{X}_\alpha^n(t))^+\right|_T^*>M\right) 
    & \leq  \PP\left(\left|(\hat{X}_\alpha^n(t))^+\right|_{T\wedge \theta^{n,0}}^*>M\right) 
    + \PP\left(\sup_{\theta^{n,0}\leq t \leq T}\left|(\hat{X}_\alpha^n(t))^+\right|>M, \theta^{n,0}<T\right).
\end{align*}

For all $0\leq t\leq T\wedge \theta^{n,0}$, 
\begin{align}
\nonumber(\hat{X}_\alpha^n(t))^+&\leq (\hat{X}_\alpha^n(0))^+ + n^{-\alpha}A^n(t)-n^{-\alpha}\sum_{k=1}^{N_\alpha^n}S_k(\tilde{\mu}_k^nt),\\
\nonumber&\leq (\hat{X}_\alpha^n(0))^++n^{-\alpha}\left(A^n(t)-\lambda^n t\right)-n^{-\alpha}\sum_{k=1}^{N_\alpha^n}\left(S_k(\tilde{\mu}_k^nt)-\tilde{\mu}_k^nt\right)\\\nonumber&\quad+n^{-\alpha}\left(\lambda^n-N_\alpha^n\bar{\mu}_F\right)t-n^{-\alpha}\sum_{k=1}^{N_\alpha^n}\left(\tilde{\mu}_k^n-\bar{\mu}_F\right)t\\
&\toP \xi_0 - \beta \bar{\lambda}^\alpha\bar{\mu}_F^{1-\alpha}t,\label{eq:conv_xplus}
\end{align}
where we use the martingale central limit theorem c.f.~\cite{ptw07} to prove convergence of the centered Poisson processes along with Assumption~\ref{asm:initialcondition}, the definition of $N_\alpha^n$ in \eqref{eq:staffing_level} and the central limit theorem. 
Defining
$\theta_\epsilon^{n,1}:=\inf\big\{s>\theta^{n,0}:\big(\hat{X}_\alpha^n(s)\big)^+>\epsilon\big\}$  and $\tilde{\theta}_\epsilon^{n,1}:=\sup\big\{\theta^{n,0}\leq s<\theta_\epsilon^{n,1}:\big(\hat{X}_\alpha^n(s)\big)^+<\epsilon/2\big\}$,
\begin{align}
 \nonumber\PP\left(\sup_{\theta^{n,0}\leq s\leq T} \big(\hat{X}_\alpha^n(s)\big)^+>\epsilon\right)
 &\leq\PP\left(\theta^{n,0}<\tilde{\theta}_\epsilon^{n,1}\leq \theta_\epsilon^{n,1}\leq T\right)
 \\
 \nonumber&\leq \PP\left(\sup_{\theta^{n,0}\leq s_1<s_2\leq T} \left\{\left|\frac{A^n(s_2)-A^n(s_1)-\lambda^n(s_2-s_1)}{n^\alpha}\right|\right.\right.
 \\
 \nonumber&\qquad\qquad\qquad+\left.\left
 |\frac{\sum_{k=1}^{N^n_\alpha} \big(S_k(\tilde{\mu}_ks_2)-S_k(\tilde{\mu}_ks_1)-\tilde{\mu}_k(s_2-s_1)\big)}{n^{\alpha}}\right|\right.
 \\
\nonumber&\qquad\qquad\qquad\qquad\qquad\qquad\quad\left.\left.+\left|\frac{(\lambda^n-\sum_{k=1}^{N_\alpha^n}\tilde{\mu}_k)(s_2-s_1)}{n^\alpha}\right|\right\}>\epsilon/2\right)
 \\
 \nonumber&\leq 2\PP\left(n^{-\alpha}|A^n(s)-\lambda^ns|_T^*>\epsilon/16\right)\\
 \nonumber&\qquad+2\PP\left(n^{-\alpha}\left|\sum_{k=1}^{N_\alpha^n}S_k^n(s)-\tilde{\mu}_k^ns\right|_T^*>\epsilon/16\right)\\
 &\qquad+\PP\left(n^{-\alpha}\left(\lambda^n-\sum_{k=1}^{N_\alpha^n}\tilde{\mu}_k\right)>0\right).\label{eq:conv_xplus2}
\end{align}
Again, using the martingale central limit theorem and the central limit theorem, the right-hand side of \eqref{eq:conv_xplus2} converges to $0$, and combining \eqref{eq:conv_xplus} and \eqref{eq:conv_xplus2} we obtain the stochastic boundedness of $\{(\hat{X}_\alpha^n(t))^+\}_{n\in \NN}$. 
%Now, setting $t_M=(M_0+1)/\bar{\lambda}^\alpha \bar{\mu}_F^{1-\alpha}$, we have 
% \begin{align*}
%     \PP\left(\sup_{t_M\leq s\leq T} \big(\hat{X}_\alpha^n(s)\big)^+>\epsilon\right) &\leq \PP\left(\sup_{\theta^{n,0}\leq s\leq T} \big(\hat{X}_\alpha^n(s)\big)^+>\epsilon, \theta^{n,0}\leq t_M\right) + \PP\big(\theta^{n,0}>t_M\big).
% \end{align*}
% Equation~\eqref{eq:conv_xplus2} implies that the first probability term on the right-hand side converges to $0$. For the second term, $\theta^{n,0}>t_M$ implies that $\big(\hat{X}_\alpha^n(t_M)\big)^+>0$. Equation \eqref{eq:conv_xplus} then implies that the probability of this event goes to $0$, which proves the second part of the lemma.

We now prove the stochastic boundedness of $\{\hat{I}_\alpha^n(t)\}_{n\in\NN}$. 
%The case when $\alpha=1/2$ was established in Lemma 1 in \citep{bukeqin2019}. The proof for $1/2<\alpha\leq 1$ is simpler. 
We have 
\begin{align*}
    \left|\hat{I}_\alpha^n(t)\right|_T^*&\leq \left|\hat{X}_\alpha^n(0)\right|+\left|\frac{A^n(t)-\lambda^nt}{n^{\alpha}}\right|_T^*+\left|\frac{\sum_{k=1}^{N_\alpha^n}S_k^n(t)-\mu_k^nt}{n^{\alpha}}\right|_T^*+\left|\frac{\sum_{k=1}^{N_\alpha^n}\mu_k^nt-\lambda^nt}{n^{\alpha}}\right|_T^*\\
    &\quad + \left|\frac{R^n\left(\gamma\int_0^t(X^n(s)-N_\alpha^n)^+ds\right)-\gamma\int_0^t(X^n(s)-N_\alpha^n)^+ds}{n^\alpha}\right|_T^*+\left|\gamma\int_0^t(\hat{X}_\alpha^n(s))^+ds\right|_T^*.
\end{align*}
The stochastic boundedness of $\{(\hat{X}_\alpha^n(t))^+\}_{n\in \NN}$ and the martingale central limit theorem implies the stochastic boundedness of the last two terms. The second and third terms are stochastically bounded again due to the martingale central limit theorem. Finally, the first and the fourth terms are stochastically bounded due to Assumption~\ref{asm:initialcondition} and the central limit theorem. This proves the lemma.
\end{proof}
\endproof

Using Lemma~\ref{lem:pos_conv}, it is straightforward to prove the tightness of the continuous processes $\left\{\int_0^t\hat{I}_\alpha^n(s)ds\right\}_{n\in \NN}$ and that $\{\tau_\epsilon^n\wedge T\}_{n\in \NN}$ are continuous functionals of these processes. 

\begin{definition}
Suppose $\tau_\epsilon^n \wedge T\Rightarrow \tau_\epsilon \wedge T$ for all $\epsilon,T>0$ as $n\to \infty$. A process $\{\eta_{\alpha,t}^\pi\}_{t\geq 0}$ is called the limiting fairness process if for all $\epsilon>0$, it holds that  $\cS_\epsilon\eta_{\alpha}^{\pi,n}\Rightarrow \cS_\epsilon\eta_{\alpha}^\pi$ as $n\to \infty$ in the Skorokhod-$J_1$ topology modified for left-continuous functions, and where $\cS_\epsilon\eta_{\alpha}^\pi$ is defined by replacing $\eta_{\alpha,t}^{\pi,n}$ and $\tau_\epsilon^n$ in \eqref{eq:shifted_fairness} by $\eta_{\alpha,t}^\pi$ and $\tau_\epsilon$. \label{def:limitingfairness}
\end{definition}
The following result generalizes Theorem 2 in \cite{bukeqin2019} from $\alpha=1/2$ to  $1/2\leq\alpha\leq 1$, and is a step towards establishing the existence of a limiting fairness process. The proof is a direct adaptation of that in \cite{bukeqin2019} and is omitted here. 
\begin{proposition}
For $\epsilon>0$ and $1/2\leq \alpha\leq 1$, the shifted fairness processes $\{\cS_\epsilon\eta_\alpha^{\pi,n}\}_{n\in\ZZ_+}$ are tight under any non-idling policy.\label{prop:tightness_fairness}
\end{proposition}

% \section{Proofs for Results Presented in Section~\ref{sec:fairnessprocess}}

% \begin{repeattheorem}[Lemma~\ref{lem:pos_conv}]
% For $1/2\leq\alpha\leq 1$ and $t\geq 0$, the sequence of processes  $\{(\hat{X}_\alpha^n(t))^+\}_{n\in \NN}$ and $\{\hat{I}_\alpha^n(t)\}_{n\in\NN}$ are stochastically bounded. Moreover, for $1/2<\alpha\leq 1$ there exists a time $t_M>0$ such that  $\sup_{t_M\leq t\leq T}(\hat{X}_\alpha^n(t))^+\toP 0$ for all $T>t_M$ as $n\to \infty$. 
% \end{repeattheorem}

\subsection{Proofs for Results Presented in Section~\ref{sec:convergence_system}}\label{app:convergence_system}

\begin{proof}[Proof of Theorem~\ref{thm:system_convergence}.]
Our proof follows the same steps as Theorem 5 in~\cite{bukeqin2019} with the additional modification on the number of servers. To simplify the notation, we suppress the $\pi$ superscript of the fairness process. Adding and subtracting appropriate terms to \eqref{eq:auxiliareqSystemLengthProcess} and normalizing with $n^\alpha$, we get
\begin{align}
\nonumber
    \hat{X}_\alpha^n(t)&=\hat{X}_\alpha^n(0)+\hat{M}_{\alpha,1}^n(t)-\hat{M}_{\alpha,2}^n(t) -\hat{M}_{\alpha,3}^n(t)+n^{-\alpha}(\lambda^n-N_\alpha^n\bar{\mu}_F)t-n^{-\alpha}\left(\sum_{k=1}^{N_\alpha^n}\tilde{\mu}_k^n-N_\alpha^n\bar{\mu}_F\right)t
    \\
    &\quad-\sum_{k=1}^{N_\alpha^n}\tilde{\mu}_{k}^n\int_0^t\hat{I}_{k,\alpha}^n(s)ds-\gamma\int_0^t(\hat{X}_\alpha^n(s))^+ds,
    \label{eq:scaled_decomposition}
\end{align}
where
\begin{align*}
    \displaystyle\hat{M}_{\alpha,1}^n(t)
    &:=\frac{A^n(t)-\lambda^nt}{n^\alpha},
    \\
    \hat{M}_{\alpha,2}^n(t)
    &:=\frac{S^n\left(\sum_{k=1}^{N_\alpha^n}\tilde{\mu}_{k}^n\left(t-\int_0^tI_k^n(s)ds\right)\right)-\sum_{k=1}^{N_\alpha^n}\tilde{\mu}_{k}^n\left(t-\int_0^tI_k^n(s)ds\right)}{n^\alpha},
    \\
    \displaystyle\hat{M}_{\alpha,3}^n(t)
    &:=\frac{R^n\left(\gamma\int_0^t(X^n(s)-N_\alpha^n)^+ds\right)-\gamma\int_0^t(X^n(s)-N_\alpha^n)^+ds}{n^\alpha}.
\end{align*}
Using the martingale central limit theorem, both $\hat{M}_{\alpha,1}^n(t)$ and $\hat{M}_{\alpha,2}^n(t)$ converge weakly to $0$ when $\alpha>1/2$ and to $\sqrt{\bar{\lambda}}W(t)$, where $W(t)$ is a standard Brownian motion when $\alpha=1/2$. To see this, we focus on $\hat{M}_{\alpha,2}^n$. The process $M_{\alpha,2}^n(t)$ is a compensated Poisson process and is a martingale with predictable quadratic variation
\[
\frac{\sum_{k=1}^{N_\alpha^n}\tilde{\mu}_{k}^n\left(t-\int_0^tI_k^n(s)ds\right)}{n^{2\alpha}}=\frac{N_\alpha^n}{n^{2\alpha}}\frac{\sum_{k=1}^{N_\alpha^n}\tilde{\mu}_{k}^nt}{N_{\alpha}^n}-\frac{\sum_{k=1}^{N_\alpha^n}\tilde{\mu}_k^n\int_0^t I_k^n(s)ds}{n^{2\alpha}}.
\]
The stochastic boundedness of $\{\hat{I}_\alpha^n\}_{n\in \NN}$ (Lemma~\ref{lem:pos_conv}) and the boundedness of $\tilde{\mu}_k^n$ imply that the second term on the right-hand side converges to 
$0$ for all $1/2\leq \alpha\leq 1$. Using the law of large numbers, the first term converges to $0$ if $\alpha>1/2$ and to $\bar{\lambda}$ if $\alpha=1/2$. 
The jumps of $M_{\alpha,2}^n(t)$ are bounded by $1/n^{\alpha}$, hence the martingale central limit theorem implies that $\hat{M}_{1/2,2}^n(t)\Rightarrow \sqrt{\bar{\lambda}}W(t)$  and $\hat{M}_{\alpha,2}^n(t)\toP 0$ when $\alpha>1/2$. The proof for $\hat{M}_{\alpha,1}^n(t)$ follows the same steps. Similarly, $\hat{M}_{\alpha,3}^n(t)\Rightarrow 0$ for all $\alpha\geq 1/2$ as a result of the stochastic boundedness of $\{(\hat{X}_\alpha^n)^+\}_{n\in \NN}$ in Lemma~\ref{lem:pos_conv}. 
Plugging in \eqref{eq:staffing_level}, we get $n^{-\alpha}(\lambda^n-N_\alpha^n\bar{\mu}_F)t\to -\beta\bar{\lambda}^\alpha\bar{\mu}_F^{1-\alpha}t$. Finally, using the central limit theorem, we have $n^{-\alpha}\left(\sum_{k=1}^{N_\alpha^n}\tilde{\mu}_k^n-N_\alpha^n\bar{\mu}_F\right)t\toP 0$ for $\alpha>1/2$ and $n^{-\alpha}\left(\sum_{k=1}^{N_\alpha^n}\tilde{\mu}_k^n-N_\alpha^n\bar{\mu}_F\right)t\Rightarrow \zeta\bar{\lambda}^\alpha\bar{\mu}_F^{-\alpha}t$ for $\alpha=1/2$, where $\zeta$ is a normal random variable with mean 0 and variance $\sigma_F^2$. 
As $\cS_\epsilon\eta_{\alpha,t}^n\Rightarrow \cS_\epsilon\eta_{\alpha,t}$ for any $\epsilon>0$, using Lemma 2 in \cite{bukeqin2019}, a modification of the Skorokhod representation theorem, we can assume that all the above processes converge w.p.~1 and to prove the theorem, we need only to prove that for any $\rho>0$, $\PP(d_S(\hat{X}_\alpha^n, \xi_\alpha)>\rho)\to 0$ as $n\to \infty$. 

For any $\varpi>0$, we can find a sequence of homeomorphisms $\Lambda^n(t):[0,T]\to[0,T]$ with derivative $\dot{\Lambda}^n(t)$ and $N_\varpi$ such that for any $n>N_\varpi$,
\[
\left|\hat{M}_{\alpha, 1}^n(t)+\hat{M}_{\alpha, 2}^n(t)+\hat{M}_{\alpha, 3}^n (t)-\sqrt{2\lambda}W(\Lambda^n(t))\right|_T^*\vee \left|\langle\iota, \cS_\epsilon\eta_{\alpha,t}^{n}\rangle- \langle\iota, \cS_\epsilon\eta_{\alpha,\Lambda^n(t)}\rangle\right|_T^*\vee \left|\dot{\Lambda}^n(t)-1\right|_T^*<\varpi,
\]
if $\alpha=1/2$ and
\[
\left|\hat{M}_{\alpha, 1}^n(t)+\hat{M}_{\alpha, 2}^n(t)+\hat{M}_{\alpha, 3}^n (t)\right|_T^*\vee \left|\langle\iota, \cS_\epsilon\eta_{\alpha,t}^{n}\rangle- \langle\iota, \cS_\epsilon\eta_{\alpha,\Lambda^n(t)}\rangle\right|_T^*\vee \left|\dot{\Lambda}^n(t)-1\right|_T^*<\varpi,
\]
if $1/2<\alpha\leq 1$.
We will prove our result by showing 
\begin{equation}
\sup_{0\leq t\leq T}\left|\hat{X}_\alpha^n(t)-\xi_\alpha(\Lambda^n(t))\right|\to 0, \mbox{w.p. }1.  \label{eq:conv_system}
\end{equation}  

Using the tightness of the processes, without loss of generality, we can assume that 
\[
 \sup_{n\in\NN}\left\{\left|\langle \iota, \eta_{\alpha,t}^{n}\rangle\right|_T^*\vee \left|\hat{X}_\alpha^n(t)\right|_T^*\vee\left|\xi_\alpha(t)\right|_T^*
 \right\}<K.
\]
Now, taking $\xi_\alpha(t)$ to be the solution of the appropriate equation in the statement of the theorem, plugging in the definition of fairness process for the seventh term on the right-hand side of \eqref{eq:scaled_decomposition}, for any $\varpi$ we have an $N_\varpi$ such that $n>N_\varpi$ implies  
\begin{align}
   \nonumber |\hat{X}_\alpha^n(t)-\xi(\Lambda^n(t))|&\leq \varpi + \gamma \left|\int_0^t(\hat{X}_\alpha^n(s))^+ds-\int_0^{\Lambda^n(t)}(\xi(s))^+ds\right|\\
    &\qquad + \left|\langle\iota, \eta_{\alpha,t}^{n}\rangle \int_0^t(\hat{X}_\alpha^n(s))^-ds - \langle\iota, \eta_{\alpha,\Lambda^n(t)}\rangle\int_0^{\Lambda^n(t)}(\xi(s))^-ds\right|.
    \label{eq:scaled_difference_decomp}
\end{align}
We can bound the second term on the right-hand side of~\eqref{eq:scaled_difference_decomp} as 
\begin{align*}
    \left|\int_0^t(\hat{X}_\alpha^n(s))^+ds-\int_0^{\Lambda^n(t)}(\xi(s))^+ds\right|&\leq \int_0^t\left|\hat{X}_\alpha^n(s)-\xi_\alpha(\Lambda^n(s))ds\right|+\int_0^t\left|(1-\dot{\Lambda}^n(t))\xi_\alpha(\Lambda^n(s))\right|ds\\
    &\leq \int_0^t\left|\hat{X}_\alpha^n(s)-\xi_\alpha(\Lambda^n(s))ds\right|+\varpi K t.
\end{align*}
To bound the third term on the right-hand side of \eqref{eq:scaled_difference_decomp}, 
\begin{align*}
&\left|\langle\iota, \eta^{n}_{\alpha,t}\rangle\int_0^t(\hat{X}_\alpha^n(s))^
-ds-\langle\iota, \eta_{\alpha, \Lambda^n(t)}\rangle\int_0^{\Lambda(t)}(\xi(s))^-ds\right|
\\
&
\leq \left|\left(\langle\iota, \eta_{\alpha,t}^n\rangle-\langle\iota, \mathcal{S}_\epsilon\eta_{\alpha,t}^n\rangle\right)\int_0^t(\hat{X}_\alpha^n(s))^-ds -\left(\langle\iota, \eta_{\alpha,\Lambda^n(t)}\rangle-\langle\iota, \mathcal{S}_\epsilon\eta_{\alpha,\Lambda^n(t)}\rangle\right)\int_0^{\Lambda(t)}(\xi_\alpha(s))^-ds\right|\\
 &\;+ \left|\left(\langle\iota, \mathcal{S}_\epsilon\eta_{\alpha,t}^n\rangle-\langle\iota, \mathcal{S}_\epsilon\eta_{\alpha,\Lambda^n(t)}\rangle\right)\int_0^t(\hat{X}_\alpha^n(s))^-ds-\langle\iota, \mathcal{S}_\epsilon\eta_{\alpha,\Lambda^n(t)}\rangle\int_0^{t}\left(\Big(\hat{X}_\alpha^n(s)\Big)^--\Big(\xi_\alpha(\Lambda(s))\Big)^-\right)ds\right|\\
 &\;+ \left|\langle\iota, \mathcal{S}_\epsilon\eta_{\alpha,\Lambda^n(t)}\rangle\int_0^{t}(1-\dot{\Lambda}(s))\big(\xi_\alpha(\Lambda(s))\big)^-ds\right|\\
& \leq (2\epsilon+\varpi(1+K)) Kt +K\int_0^{t}\left|\hat{X}^n(s)-\hat{\xi}(\Lambda(s))\right|ds.
 \end{align*}
Plugging these bounds in \eqref{eq:scaled_difference_decomp}, using Gr\"onwall's inequality, and choosing $\epsilon$ and $\varpi$ appropriately, our result follows. \end{proof}
\endproof
Now, we prove the interchangeability of many-server limit and limit as $t\to\infty$. To do so, we need the following uniform integrability result. 
 \begin{lemma}\label{lem:uniform_integrability_system}
For any $1/2\leq \alpha 
\leq 1$, the stationary scaled system lengths $\{\hat{X}^n(\infty)\}_{n\in \NN}$ are uniformly integrable and hence tight. 
\end{lemma}
\begin{proof}
For any $n\in \NN$, we decompose the scaled system length process into its negative and positive parts as
\[\hat{X}_\alpha^n(
\infty)=\left(\hat{X}_\alpha^n(\infty)\right)^+-\left(\hat{X}_\alpha^n(\infty)\right)^-,\]
and prove the uniform integrability of each part separately. 
We consider $\left(\hat{X}_\alpha^n(t)\right)^-$ and use a coupling argument. Now, suppose we fix $n>0$ and  let $\theta_{A,i}^n$ be the occurrence time of the $i$th event for Poisson process $A^n(t)$. Assume that we know $(\tilde{\mu}_1^n,\ldots, \tilde{\mu}_{N_\alpha^n}^n)$ and define $\tilde{S}^n(t)$ to be a Poisson process with rate $\sum_{k=1}^{N^n_\alpha}\tilde{\mu}_k^n$ and let $\theta_{S,i}^n$ be the occurrence time for the $i$th event. We also define a sequence $\{U_i^n\}_{i\in \NN}$ of independent uniform(0,1) random variables. For any given time $t$, we know that the system should be serving with a rate equal to the sum of service rates of busy servers, i.e., $\sum_{k=1}^{N^n_\alpha}\tilde{\mu}_k^n(1-I_k^n(t))$, hence we use the thinning property of Poisson process where event $i$ of $\tilde{S}^n(t)$ is accepted as an actual departure with probability $p_i^n(\theta_{S_i}^n-)=\frac{\sum_{k=1}^{N^n_\alpha}\tilde{\mu}_k^n(1-I_k^n(\theta_{S,i}^n-))}{\sum_{k=1}^{N^n_\alpha}\tilde{\mu}_k^n}$ by checking whether $U_i^n\leq p_i^n(\theta_{S_i}^n-)$ or not.
The processes $I_k^n(t)$ can be rigorously defined by using $U_i^n$s and a routing process based on our routing policy. As this does not play a major role in our proof, we refer the reader to ~\cite{bukeqin2019} for the detailed construction of idleness processes. Now, we can write 
\begin{align*}
    \left(\hat{X}_\alpha^n(t)\right)^-
    =
    \left(\hat{X}_\alpha^n(0)\right)^-
    & 
    +n^{-\alpha}\sum_{i=1}^{\tilde{S}^n(t)}\II\left(\left(\hat{X}_\alpha^n(\theta_{S,i}^n-)\right)^+=0\right)
    \II\Big(U_i^n\leq p_i^n\left(\theta_{S_i}^n-\right)\Big)
    \\
    &
    % \qquad \qquad
    -n^{-\alpha}\sum_{i=1}^{A^n(t)}\II\left( \left(\hat{X}_\alpha^n(\theta_{A,i}^n-)\right)^->0\right).
\end{align*}
 We define a birth-death process $\{Y_1^n(t)\}_{n\in\NN}$ with $Y_1^n(0)=\left(X^n(0)\right)^-$ w.p.1, whose birth rate at $Y^n(t)=i$ is $\sum_{k=1}\tilde{\mu}_k^n-\mu_{\min}i$ and death rate is $\lambda^n$ if $Y_1^n(t)>0$. Then, we can couple the scaled process $\hat{Y}_1^n(t)=n^{-\alpha}Y_1^n(t)$ with the system length process by writing it as 
\begin{align*}
\hat{Y}_1^n(t)
&= 
\hat{Y}_1^n(0)
+n^{-\alpha}\sum_{i=1}^{\tilde{S}^n(t)}\II\Big(U_i^n\leq \tilde{p}_i^n\left(\theta_{S,i}^n-\right)\Big)
-n^{-\alpha}\sum_{i=1}^{A^n(t)}\II\left( \hat{Y}_1^n(\theta_{A,i}^n-)>0\right),
\end{align*}
where $\tilde{p}_i^n\left(\theta_{S_i}^n-\right):=\frac{\sum_{k=1}^{N^n_\alpha}\tilde{\mu}_k^n-\mu_{\min}\hat{Y}_1\left(\theta_{S_i}^n-\right)}{\sum_{k=1}^{N^n}\tilde{\mu}_k^n}$. To see that $\left(\hat{X}_\alpha^n(t)\right)^-\leq \hat{Y}_1^n(t)$ for all $t\geq 0$ with probability 1, define $\vartheta^n=\left\{t:\left(\hat{X}_\alpha^n(t)\right)^->\hat{Y}_1^n(t)\right\}$. As at most one event occurs at any given time $t$ with probability 1, we have $\left(\hat{X}_\alpha^n(\vartheta^n-)\right)^-=\hat{Y}_1^n(\vartheta^n-)$. By definition we have 
\[
\tilde{p}_i^n\left(\vartheta^n-\right)\geq p_i^n(\vartheta^n-).
\]
Hence, if  $\vartheta^n$ is an event epoch for $\tilde{S}^n(t)$, we have $\left(\hat{X}_\alpha^n(\vartheta^n)\right)^-\leq \hat{Y}_1^n(\vartheta^n)$ and if $\vartheta^n$ is an event epoch for $A^n(t)$, we have $\left(\hat{X}_\alpha^n(\vartheta^n)\right)^-= \hat{Y}_1^n(\vartheta^n)$, both contradicts with the definition of $\vartheta^n$. Hence, we conclude that $\left(\hat{X}_\alpha^n(t)\right)^-\leq \hat{Y}_1^n(t)$ for all $t\geq 0$ with probability 1. 

Now, define $\zeta^n:=n^{-\alpha}\left(\sum_{k=1}^{N^n_\alpha}\tilde{\mu}_k^n-N^n_\alpha \bar{\mu}\right)$ and by re-arranging the terms we have
\[
\sum_{k=1}^{N^n_\alpha}\tilde{\mu}_k^n=\lambda^n+\beta\left(\lambda^n\right)^\alpha\bar{\mu}^{1-\alpha}+n^\alpha \zeta^n.
\]
Take $M_1:=(\beta\left(\lambda^n\right)^\alpha\bar{\mu}^{1-\alpha}+n^\alpha \zeta^n)^++2n^\alpha,$ and define a new birth-death process $Y_2^n(t)$ with $Y_2^n(0)=Y_1^n(0)$ with probability 1, whose birth rate at $Y_2^n(t)=i$ is $\lambda^n-\mu_{\min}\min\{i,M_1\}$ and death rate is $\lambda^n$ at $Y_2^n(t)\neq 0$. By definition, $Y_2(t)$ is stochastically greater than $Y_1^n(t)$. Using a similar argument as above, we can couple $Y_2(t)$ with a simple birth death process  $Y_3(t)$ where $Y_3^n(0)=(Y_2^n(0)-M_1)^+$ with birth rate $\lambda^n-\mu_{\min}M_1$, death rate $\lambda^n$ and $Y_2^n(t)\leq M_1+Y_3^n(t)$ for all $t\geq 0$. As birth and death rates are constants, $Y_3(t)$ is equivalent to an $M/M/1$ queue. Hence, we can prove that $\left(\hat{X}^n(\infty)\right)^-$ is uniformly integrable by showing that $\sup\EE\left[\left(n^{-\alpha}(M_1+Y_3^n(\infty))\right)^2\right]<\infty$. For $n\in \NN$,
\begin{align*}
    \EE\Big[\Big(n^{-\alpha}(M_1+
    &Y_3^n(\infty))\Big)^2\Big]
    \\    &=n^{-2\alpha}\left(\EE[M_1^2]+2\EE[M_1Y_3^n(\infty)]+\EE\left[\left(Y_3^n(\infty)\right)^2\right]\right)
    \\    &=n^{-2\alpha} \left(\EE\left[M_1^2\right]+2\EE\left[M_1\EE[Y_3^n(\infty)|\vec{\mu}]\right]+\EE\left[\EE\left[\left(Y_3^n(\infty)\right)^2|\vec{\mu}\right]\right]\right)\\
    &=n^{-2\alpha}\left(\EE\left[M_1^2\right]+2\EE\left[\frac{\lambda^n-\mu_{\min}M_1}{\mu_{\min}}\right]+\EE\left[\frac{(\lambda^n-\mu_{\min} M_1)(2\lambda^n-\mu_{\min}M_1)
    }{\mu_{\min}^2M_1^2}\right]\right)\\
    &=n^{-2\alpha}\left(\EE\left[M_1^2\right]+\frac{2\lambda^n}{\mu_{\min}}-2\EE\left[M_1\right]+2\EE\left[\frac{(\lambda^n)^2}{\mu_{\min}^2M_1^2}\right]-3\EE\left[\frac{\lambda^n}{\mu_{\min}M_1}\right]+1\right)\\
    &\leq n^{-2\alpha}\left(\EE\left[M_1^2\right]+\frac{2\lambda^n}{\mu_{\min}}+2\EE\left[M_1\right]+\frac{2(\lambda^n)^2}{\mu_{\min}^2n^{2\alpha}}+\frac{3\lambda^n}{\mu_{\min}n^\alpha}+1\right).
\end{align*}
Assumption~\ref{asm:lambdascaling} implies that the second, fourth and fifth terms converge to a finite number for any $\alpha\geq 1/2$ and hence can be bounded uniformly for any $n$. Also, using the identity $(a+b)^2\leq 2a^2+2b^2$, we have
\[
n^{-2\alpha}\EE\left[M_1^2\right]\leq 4n^{-2\alpha}\beta^2\left(\lambda^n\right)^{2\alpha}\bar{\mu}^{2-2\alpha}+4\EE[(\zeta^n)^2] +2.
\]
Again we can use Assumption~\ref{asm:lambdascaling} to show that the first term converges to a finite number. Using independence we have 
\[
\EE[(\zeta^n)^2]=n^{-2\alpha}\EE\left[\left(\sum_{k=1}^{N^n_\alpha}\tilde{\mu}_k^n-N^n_\alpha\bar{\mu}\right)^2\right]=n^{-2\alpha}\EE\left[\sum_{k=1}^{N^n_\alpha}\left(\tilde{\mu}_k^n-\bar{\mu}\right)^2\right]\leq n^{-2\alpha}N^n_\alpha(\mu_{\max}-\mu_{\min})^2,
\]
which also converges and hence can be uniformly bounded for all $n$. Similarly, 
\[
n^{-2\alpha}\EE[M_1]\leq n^{-2\alpha}\EE[\beta\left(\lambda^n\right)^\alpha\bar{\mu}^{1-\alpha} |\mu_{\max}-\mu_{\min}|+2n^{-\alpha}] -n^{-2\alpha}N^n_\alpha,
\]
which also converges and proves the uniform integrability of $\left\{\left(\hat{X}_\alpha^n(\infty)\right)^-\right\}_{n\in \NN}$. The proof of the uniform integrability of $\left\{\left(\hat{X}_\alpha^n(\infty)\right)^+\right\}_{n\in \NN}$ follows the same lines and hence is omitted. 
\end{proof}

\begin{proof}[Proof of Theorem~\ref{thm:interchangibility_stationary_n}.]
Suppose $\hat{X}_\alpha^n(0)$, $(\tilde{\mu}_1^n, \ldots, \tilde{\mu}_{N_{\alpha}^n}^n)$ are distributed according to the stationary measure of the $n$th system. Lemma~\ref{lem:uniform_integrability_system} ensures that $\hat{X}_\alpha^n(0)$ satisfies the uniform integrability in Assumption~\ref{asm:initialcondition}. Then, $\hat{X}_\alpha^n(t)$, $(\tilde{\mu}_1^n, \ldots, \tilde{\mu}_{N_{\alpha}^n}^n)$ are also distributed according to $\Pi^n$. As we know that the $\{\Pi^n\}$ are tight and $\hat{X}_\alpha^n(t)\Rightarrow \xi_\alpha(t)$, the first convergence holds and the second convergence holds as a result of the uniform integrability. To see the third one, we need to see that 
\[\lim_{T\to\infty}\frac{1}{T}\sum_{k=1}^{N^n_\alpha}\delta_{\tilde{\mu}_k^n}(\AAA)\int_0^T\hat{I}_k^n(s)ds=\sum_{k=1}^{N^n_\alpha}\delta_{\tilde{\mu}_k^n}(\AAA)\EE[I_k^n(\infty)|\vec{\mu}].\]
Again, starting with a stationary system and using the uniform integrability, the result follows. 
\end{proof}

\subsection{Proofs for Results Presented in Section~\ref{sec:grr_policy}}\label{app:grr_policy}
\begin{proof}[Proof of Lemma~\ref{lem:integral_equation}.]
We define the discrete-time measure-valued stochastic process $$
\cU_{A,i}^n(\AAA)=\delta_{\tilde{\mu}_k^n}(\AAA)\II(X^n(\theta_{A,i}^n-)\leq 0)$$ for all $\AAA\in \cB(\RR_+)$ if the $i$th incoming arrival in the $n$th system is immediately routed to server $k$ and $0$ (thought as a measure) otherwise. Similarly, we define $\cU_{S,i}^n(\AAA)=\delta_{\tilde{\mu}_k^n}(\AAA)\II(X^n(\theta_{S,i}^n-)\leq 0)$ for all $\AAA\in \cB(\RR_+)$, if the $i$th potential service completion is from server $k$, if it is not an actual service completion $\cU_{S,i}^n=0$.  Then, for any $f\in C_{[\mu_{\min},\mu_{\max}]}^b[0,\infty)$ we have the following balance 
 \begin{align}
    \langle f, \bar{\psi}^{n}_t\rangle = \langle f, \bar{\psi}_0^{n}\rangle + n^{-1}\sum_{i=1}^{S^n(t)} \langle f, \cU_{S,i}^n \rangle -n^{-1}\sum_{i=1}^{A^n(t)}\langle f, \cU_{A,i}^n \rangle
    \label{eq:idle_balance}.
 \end{align}
To prove tightness, we use the conditions introduced by~\cite{jak86} that we recall next.
 \begin{theorem}[\cite{jak86}]
   A sequence of stochastic processes $\{\bar{\psi}_t^{n}\}_{n\in\NN}$ taking values in $\DD_\cP[0,T]$ is tight if and only if:
   \begin{description}
   \item[J1. (Compact Containment Condition)] For each $\rho, T>0$, there exists a compact set $\mathcal{K}_\rho\subset \cP$ such that
   \[
 	\liminf_{n\to \infty}\PP(\bar{\psi}^{n}_t\in \mathcal{K}_\rho, \text{ for all }t\in[0,T])>1-\rho.  
   \]
   \item[J2.] There exists a family of functions $\mathbb{F}$ such that 
   \begin{enumerate}[i.]
   \item $H\in \FF: \cP\to \RR$, $\FF$ separates points in $\cP$ and $\FF$ is closed under addition.
   \item For any fixed $H\in\mathbb{F}$, the sequence of functions $\{h^n(t):=H(\bar{\psi}^{n}_t), \mbox{ for all }t \in \RR\}_{n\in \NN}$ is tight in $\DD_{\RR}[0,\infty)$ endowed with Skorokhod-$J_1$ topology. 
   \end{enumerate}
   \end{description}
   \end{theorem}
   
   We now return to our proof. Using Lemma~\ref{lem:pos_conv}, we know that, for all $\epsilon>0$, there exists a $K_\epsilon$ such that 
   $\PP(|\hat{I}_1^n(t)|_T^*>K_\epsilon)<\epsilon$ for all $n\in \NN$. Define $\cK_\epsilon$ as the set of measures bounded by $K_\epsilon$ on the support $[\mu_{\min}, \mu_{\max}]$. The set $\cK_\epsilon$ is compact and Lemma~\ref{lem:pos_conv} implies J1. 
   
   To show that J2 holds, we define $$\FF=\{H:\cM_F[\mu_{\min}, \mu_{\max}]\to \RR:\exists f\in C_{[\mu_{\min},\mu_{\max}]}^b[0,\infty) \mbox{ such that } H(\psi)=\langle f, \psi\rangle \mbox{ for all }\psi\in\cM_F\}, $$
   where $\cM_F[\mu_{\min}, \mu_{\max}]$ is the set of finite measures on $[\mu_{\min}, \mu_{\max}]$. The set $\FF$ separates the points in $\cM_F$ and is closed under addition. Take $H(\psi)=\langle f, \psi \rangle$ that corresponds to $f(\mu)\leq K_f$ for all $\mu\in [\mu_{\min}, \mu_{\max}]$. To show tightness of $\{\langle f,\bar{\psi}_t^n\rangle \}_{n\in \NN}$, we need to show that for all $\epsilon,\rho>0$
   \begin{enumerate}
   \item there exists an $M_{f,\epsilon}$ such that $\PP\left(\sup_{0\leq t\leq T}|\langle f,\bar{\psi}_t^n\rangle|>M_{f,\epsilon}\right)<\epsilon$ and 
   \item there exists a $\rho$ and an $N_\rho$ such that for all $n>N_\rho$, $ \PP\left(w(\langle f, \bar{\psi}_t^n \rangle, \rho)\geq \epsilon\right)<\epsilon$, where $$w(\langle f, \bar{\psi}_t^n\rangle, \rho)=\inf_{\{t_i\}}\max_{i}\sup_{t_i\leq t,s\leq t_{i+1}}|\langle f, \bar{\psi}_t^n \rangle-\langle f, \bar{\psi}_s^n \rangle|$$ and $\{t_i\}_{0\leq i \leq \nu}$ is any $\rho$-sparse set, i.e., $0=t_0<t_1<\cdots<t_\nu=T$ with $\min_{i}|t_{i+1}-t_i|>\rho$. 
   \end{enumerate} 
   Again taking $M_{f,\rho}=K_\rho K_f$, we have \[\PP\left(\sup_{0\leq t\leq T}|\langle f,\bar{\psi}_t^n\rangle|>M_{f,\rho}\right)\leq \PP\left(K_f\sup_{0\leq t\leq T}\hat{I}^{n}(t)>M_{f,\rho}\right)\leq \rho,\]
   which implies the first condition. Now, we prove the second condition. Using \eqref{eq:idle_balance}, for any $0\leq s< t\leq T$ we have
   \begin{align}
       \nonumber|\langle f, \bar{\psi}^{n}_t\rangle - \langle f, \bar{\psi}^{n}_s\rangle| &= \left |n^{-1}\sum_{i=1}^{S^n(t)} \langle f, \cU_{S,i}^n \rangle -n^{-1}\sum_{i=1}^{A^n(t)}\langle f, \cU_{A,i}^n \rangle - n^{-1}\sum_{i=1}^{S^n(s)} \langle f, \cU_{S,i}^n \rangle -n^{-1}\sum_{i=1}^{A^n(s)}\langle f, \cU_{A,i}^n \rangle\right| \\
       \nonumber& \leq  n^{-1}K_f \left|S^n(t) - S^n(s)-\sum_{k=1}^{N^n_\alpha}\mu_kt +\sum_{k=1}^{N^n_\alpha}\mu_ks\right| + n^{-1}K_f\sum_{k=1}^{N^n_\alpha}\mu_k  \left|t-s\right|\\
       \nonumber&\quad+ n^{-1}K_f \left|A^n(t) - A^n(s)-\lambda^nt +\lambda^ns\right| + n^{-1}K_f\lambda^n  \left|t-s\right| \\
       \nonumber& \leq  n^{-1}K_f \left|S^n(t) - \sum_{k=1}^{N^n_\alpha}\mu_kt\right| +n^{-1}K_f\left|S^n(s) -\sum_{k=1}^{N^n_\alpha}\mu_ks\right|\\
       \nonumber&\quad + K_f\left|n^{-1}\sum_{k=1}^{N^n_\alpha}\mu_k-\bar{\lambda}(1+\beta)\right|  \left|t-s\right|+ n^{-1}K_f \left|A^n(t) - \lambda^nt\right| +\left|A^n(s)-\lambda^ns\right|\\
       \nonumber&\quad + n^{-1}K_f\lambda^n  \left|t-s\right| + \bar{\lambda}(1+\beta) \left|t-s\right|\\
       \nonumber& \leq  n^{-1}2K_f \left|S^n(t) - \sum_{k=1}^{N^n_\alpha}\mu_kt\right|_T^* +n^{-1}2K_f\left|A^n(t) - \lambda^nt\right|_T^*\\
       &\quad + K_f\left|n^{-1}\sum_{k=1}^{N^n_\alpha}\mu_k-\bar{\lambda}(1+\beta)\right|  \left|t-s\right| + n^{-1}K_f\lambda^n  \left|t-s\right| + \bar{\lambda}(1+\beta) \left|t-s\right|.
       \label{eq:equicontinuity_measures}
   \end{align}
   Using the martingale central limit theorem, the first and second terms on the right-hand side converges to 0 in probability. Similarly, using the law of large numbers, we can show that the third term also converges to 0 in probability. Finally, from Assumption~\ref{asm:lambdascaling}, we know that $n^{-1}\lambda^n\to \bar{\lambda}$. Hence, choosing $\rho<\epsilon/2(K_f\bar{\lambda}(2+\beta))$ and $N$ large enough the second condition follows. Moreover, by examining $\eqref{eq:equicontinuity_measures}$, one can conclude that any limit is continuous. 

Now, we are ready to prove that any convergent subsequence should satisfy the provided integral equations. 
The martingale central limit theorem (c.f., \cite{ptw07}) implies $n^{-1}|S^n(t)-\sum_{k=1}^{N^n_\alpha}\mu_kt|_T^*\toP 0$ and $n^{-1}|A^n(t)-\lambda^nt|_T^*\toP 0$. Using the Skorokhod representation theorem~(see, e.g., Theorem 6.7 in~\cite{bil99}), we can assume that these and the subsequence in the statement of the lemma converge almost surely. Also, as mentioned in the proof of tightness, for any $f\in C_{\RR_+}[0,T]$, the limit $\langle f, \bar{\psi}_t\rangle$ is continuous and hence, the convergence holds in the supremum norm as well as the Skorokhod $d_S$ metric. We define the processes
\begin{align*}
    M_1^n(t) &:= \sum_{i=1}^{S^n(t)}\langle f, \cU_{S,i}^n\rangle - \sum_{i=1}^{S^n(t)}\left\langle f, \frac{\sum_{k=1}^{N_\alpha^n}\mu_k(1-I_k^n(\theta_{S,i}^n-))\delta_{\mu_k}\II(X^n(\theta_{S,i}^n-)\leq 0)}{\sum_{k=1}^{N_\alpha^n}\mu_k}\right\rangle,
    \\
    M_2^n(t)&:=\sum_{i=1}^{A^n(t)}\langle f, \cU_{A,i}^n \rangle - \sum_{i=1}^{A^n(t)}\left\langle f, \frac{\sum_{k=1}^{N_\alpha^n} h(\mu_k)I_k^n(\theta_{A,i}^n-)\delta_{\mu_k}\II(X^n(\theta_{A,i}^n-)\leq 0)}{\sum_{k=1}^{N_\alpha^n}h(\mu_k)I_k^n(\theta_{A,i}^n-)}\right\rangle, \end{align*}
where 0/0 is assumed to be 0. It is easy to see that both $M_1^n$ and $M_2^n$ are $\cF_t$ martingales. After some algebraic manipulations, Equation \eqref{eq:idle_balance} becomes 
\begin{align}
    \nonumber \langle f, \bar{\psi}^{n}_t\rangle &= \langle f, \bar{\psi}_0^{n}\rangle + n^{-1}M_1^n(t) + n^{-1}\sum_{i=1}^{S^n(t)}\left\langle f, \frac{\sum_{k=1}^{N_\alpha^n}\mu_k(1-I_k^n(\theta_{S,i}^n-))\delta_{\mu_k}\II(X^n(\theta_{S,i}^n-)\leq 0)}{\sum_{k=1}^{N_\alpha^n}\mu_k}\right\rangle\\
    \nonumber&\quad - n^{-1}M_2^n(t) - n^{-1}\sum_{i=1}^{A^n(t)}\left\langle f, \frac{\sum_{k=1}^{N_\alpha^n} h(\mu_k)I_k^n(\theta_{A,i}^n-)\delta_{\mu_k}\II(X^n(\theta_{A,i}^n-)\leq 0)}{\sum_{k=1}^{N_\alpha^n}h(\mu_k)I_k^n(\theta_{S,i}^n-)}\right\rangle \\
    \nonumber &=\langle f, \bar{\psi}_0^{n}\rangle + n^{-1}M_1^n(t) + n^{-1}\frac{\langle f\times \iota, \sum_{k=1}^{N^n_\alpha}\delta_{\mu_k} \rangle}{\sum_{k=1}^{N^n_\alpha}\mu_k}\int_0^t\II(X^n(s-)\leq 0)dS^n(s) \\
  \nonumber &\quad- n^{-1}\int_0^t\frac{\langle f\times \iota, \bar{\psi}_{s-}^n\rangle}{n^{-1}\sum_{k=1}^{N^n_\alpha}\mu_k}\II(X^n(s-)\leq 0)dS^n(s) - n^{-1}M_2^n(t) \\
  \nonumber&\quad- n^{-1}\int_0^t \frac{\langle f\times h, \bar{\psi}_{s-}^n\rangle}{\langle h, \bar\psi_{s-}^n\rangle}\II(X^n(s-)\leq 0)dA^n(s)\\
%\nonumber &= \langle f, \bar{\psi}_0^{n}\rangle + n^{-1}M_1^n(t) + \frac{\langle f\times \iota, n^{-1}\sum_{k=1}^{N^n}\delta_{\mu_k} \rangle}{n^{-1}\sum_{k=1}^{N^n}\mu_k}\frac{S^n(t)}{n} - \int_0^t\frac{\langle f\times \iota, \psi_{s-} \rangle}{\sum_{k=1}^{N^n}\mu_k}dS^n(s)\\
    % %\nonumber &\quad  - M_2^n(t) - \int_0^t \frac{\langle f\times h, \psi_{s-}\rangle}{\langle h, \psi_{s-}\rangle}dA^n(s)\\
      \nonumber &=\langle f, \bar{\psi}_0^{n}\rangle + \frac{M_1^n(t)}{n} + \frac{\langle f\times \iota, n^{-1}\sum_{k=1}^{N^n_\alpha}\delta_{\mu_k}\rangle}{n^{-1}\sum_{k=1}^{N^n_\alpha}\mu_k}\int_0^t\II\left(\frac{X^n(s-)}{n}\leq 0\right)d\left(\frac{S^n(s)-\sum_{k=1}^{N^n_\alpha}\mu_ks}{n}\right)\\
      \nonumber &\quad+\langle f\times \iota, n^{-1}\sum_{k=1}^{N^n_\alpha}\delta_{\mu_k} \rangle \int_0^t\II\left(\frac{X^n(s-)}{n}\leq 0\right)ds 
      \\
      \nonumber&\quad-\int_0^t\frac{\langle f\times \iota, \bar{\psi}_{s-}^n \rangle}{n^{-1}\sum_{k=1}^{N^n_\alpha}\mu_k}\II\left(\frac{X^n(s-)}{n}\leq 0\right)d\left(\frac{S^n(s)-\sum_{k=1}^{N^n_\alpha}\mu_ks}{n}\right)
      \\
      \nonumber&\quad 
      -\int_0^t\langle f\times \iota, \bar{\psi}_{s-}^n \rangle \II\left(\frac{X^n(s-)}{n}\leq 0\right)ds
      \\
      \nonumber&\quad  
      - \frac{M_2^n(t)}{n} - \int_0^t \frac{\langle f\times h, 
      \bar{\psi}_{s-}^n\rangle}{\langle h, \bar{\psi}_{s-}^{n}\rangle}\II\left(\frac{X^n(s-)}{n}\leq 0\right)d\left(\frac{A^n(s)-\lambda^ns}{n}\right)\\
      &\quad-\int_0^t \frac{\langle f\times h, \bar{\psi}_{s-}^n\rangle}{\langle h, \bar{\psi}_{s-}^n\rangle}\frac{\lambda^n}{n}\II\left(\frac{X^n(s-)}{n}\leq 0\right)ds.
      \label{eq:idle_balance_decomp}
\end{align} 

Since $f\in C_{[\mu_{\min},\mu_{\max}]}^b[0,\infty)$, we assume that $f(\mu)\leq K_f$ for all $\mu\in [\mu_{\min},\mu_{\max}]$. By our assumption, we know that $\sup_{0\leq t\leq T}|\langle f, \bar{\psi}^{n_k}_t\rangle- \langle f, \bar{\psi}_t\rangle|\to 0$ almost surely, along the subsequence $\{\bar{\psi}_t^{n_k}\}_{k=1}^\infty$. The martingales $M_1^n(t)$ and $M_2^n(t)$ can be written as 
\begin{align*}
    n^{-1}M_1^n(t)&=n^{-1}\sum_{i=1}^{S^n(t)}(\langle f, \cU_{S,i}^n\rangle-\EE[\langle f, \cU_{S,i}^n\rangle|\cF_{\theta_{S,i}^n-}]),
    \\
    \textrm{ and}\quad n^{-1}M_2^n(t)&=n^{-1}\sum_{i=1}^{A^n(t)}(\langle f, \cU_{A,i}^n\rangle-\EE[\langle f, \cU_{A,i}^n\rangle|\cF_{\theta_{A,i}^n-}]).
\end{align*}
Since both are pure jump martingales, we can write the optional quadratic variation of these martingales as 
\begin{align*}
    [n^{-1}M_1^n(t)]&=n^{-2}\sum_{i=1}^{S^n(t)}(\langle f, \cU_{S,i}^n\rangle-\EE[\langle f, \cU_{S,i}^n\rangle|\cF_{\theta_{S,i}^n-}])^2\leq \frac{K_f^2}{n^2}\sup_{0\leq t\leq T}S^n(t),\\
    [n^{-1}M_2^n(t)]&=n^{-2}\sum_{i=1}^{A^n(t)}(\langle f, \cU_{A,i}^n\rangle-\EE[\langle f, \cU_{A,i}^n\rangle|\cF_{\theta_{A,i}^n-}])^2\leq \frac{K_f^2}{n^2}\sup_{0\leq t\leq T}A^n(t),
\end{align*}
which converges to 0 almost surely. We know that $h(\mu)>0$  and it is continuous on the closed interval $[\mu_{\min},\mu_{\max}]$ and hence, there exists $\epsilon_h$ and $K_h$ such that $0<\epsilon_h\leq h(\mu)\leq K_h$ for all $\mu\in[\mu_{\min},\mu_{\max}]$. Also, Assumption~\ref{asm:lambdascaling} implies the existence of $K_N$ and $k_n$ such that $0<\epsilon_N\leq n^{-1}N^n_\alpha\leq K_N<\infty$. Hence, we have the following bounds: 
\begin{align*}
   \langle f\times\iota, \bar{\psi}_t^n\rangle&\leq \langle f\times \iota, n^{-1}\sum_{k=1}^{N_\alpha^n}\delta_{\mu_k}\rangle\leq K_fK_N\mu_{\max}, \langle f\times h, 
    \bar{\psi}_{s-}^n\rangle \leq \langle f\times h, n^{-1}\sum_{k=1}^{N_\alpha^n}\delta_{\mu_k}\rangle\leq K_fK_hK_N,
    \\
    \frac{\langle f\times\iota, \bar{\psi}_t^n\rangle}{n^{-1}\sum_{k=1}^{N_\alpha^n}\mu_k}&\leq  \frac{\langle f\times \iota, n^{-1}\sum_{k=1}^{N_\alpha^n}\delta_{\mu_k}\rangle}{n^{-1}\sum_{k=1}^{N^n_\alpha}\mu_k}\leq \frac{K_fK_N\mu_{\max}}{\epsilon_N}\mbox{ and }
    \frac{\langle f\times h, 
    \bar{\psi}_{s-}^n\rangle}{\langle h, \bar{\psi}_{s-}^{n}\rangle}\leq \frac{K_fK_h}{\epsilon_h}.
\end{align*}
These bounds along with  $\sup_{0\leq t\leq T}n^{-1}|S^n(t)-\sum_{k=1}^{N^n_\alpha}\mu_kt|\to 0$ and $\sup_{0\leq t\leq T} n^{-1} |A^n(t)-\lambda^nt|\to 0$ almost surely as $n\to \infty$, the third, fifth and the eighth terms on the right-hand side of \eqref{eq:idle_balance_decomp} converges to 0 almost surely. Using the dominated convergence theorem, we deduce that 
\begin{align*}
    \int_0^t\langle f\times \iota, \bar{\psi}_{s-}^n \rangle \II\left(\frac{X^n(s-)}{n}\leq 0\right)ds &\to \int_0^t\langle f\times \iota, \bar{\psi}_{s-} \rangle \II\left(\xi(s-)\leq 0\right)ds,
    \\
    \int_0^t \frac{\langle f\times h, \bar{\psi}_{s-}^n\rangle}{\langle h, \bar{\psi}_{s-}^n\rangle}\frac{\lambda^n}{n}\II\left(\frac{X^n(s-)}{n}\leq 0\right)ds &\to \int_0^t \frac{\langle f\times h, \bar{\psi}_{s-}\rangle}{\langle h, \bar{\psi}_{s-}\rangle}\bar{\lambda}\II\left(\xi(s-)\leq 0\right)ds.
\end{align*}
Hence, our lemma follows. \end{proof}
\endproof

% \begin{repeattheorem}[Lemma~\ref{lem:boundedness_instantaneous_idleness}]
% For any $f\in \CC_{[\mu_{\min},\mu_{\max}]}^b[0,\infty)$, any solution to equation~\eqref{eq:limitrandom1} is bounded. 
% \end{repeattheorem}

% \begin{proof}
% By definition, $\langle f, \bar{\psi}^{n}_t\rangle\leq 0$ for all $t\geq 0$ and we only need to be concerned with being bounded from above. Taking the derivative, we have
% \[
% \frac{d\langle f, \bar{\psi}^{n}_t\rangle}{dt}\leq \frac{\bar{\lambda}}{\bar{\mu}}(1+\beta)\langle f\times \iota, F\rangle \II(\xi_{\alpha,t}\leq 0) -\mu_{\min} \langle f, \bar{\psi}_{1,t}\rangle \II(\xi_{\alpha,t}\leq 0)
% \]
% Hence, the derivative is negative for any $t$ such that $\langle f, \bar{\psi}_{1,t}\rangle\geq \frac{\bar{\lambda}}{\mu_{\min}\bar{\mu}_F}(1+\beta)\langle f\times \iota, F\rangle$ and hence, we conclude that $\langle f, \bar{\psi}_{1,t}\rangle\leq \frac{\bar{\lambda}}{\mu_{\min}\bar{\mu}_F}(1+\beta)\langle f\times \iota, F\rangle$.
% \end{proof}

\begin{proof}[Proof of Theorem~\ref{thm:stationary_fairness1}.]
We know that $\xi_{1,\infty}>0$, and hence any fixed point of~\eqref{eq:limitrandom1} satisfies 
\begin{align}
\frac{\bar{\lambda}}{\bar{\mu}}(1+\beta)\langle f\times \iota, F\rangle \II(\xi_{1}(t)\leq 0)&= \langle f\times \iota, \bar{\psi}_{1,t}\rangle \II(\xi_{1}(t)\leq 0) +\bar{\lambda}\frac{\langle f\times h, \bar{\psi}_{1,t}\rangle }{\langle h, \bar{\psi}_{1,t}\rangle}\II(\xi_{1}(t)\leq 0).
\label{eq:stationarypoint}
\end{align}
%\bar{g} is the density of the instantaneous and g is the density of fairness measures. 
As \eqref{eq:boundqdregime} implies that $
\bar{\psi}_{1,\infty}$ is absolutely continuous, it possesses a Radon-Nikodym derivative $\bar{g}(\mu)$ with respect to $F$. Setting $f(\mu)=1$ for all $\mu$, we have
\begin{equation}
\int_{\mu_{\min}}^{\mu_{\max}}\mu \bar{g}(\mu)dF(\mu)=\bar{\lambda}\beta.
\label{eq:gbarmean}
\end{equation}
As the equation holds for any $f\in \CC_{[\mu_{\min},\mu_{\max}]}^b[0,\infty)$, defining $c_g=\int_{\mu_{\min}}^{\mu_{\max}}h(\mu) \bar{g}(\mu)dF(\mu)$ to simplify the notation, for $F$-almost all $\mu$
\[
\mu \bar{g}(\mu) +\frac{\bar{\lambda}h(\mu)\bar{g}(\mu)}{c_g}=\frac{\bar{\lambda}}{\bar{\mu}_F}(1+\beta)\mu.
\]
Defining $L_F=\bar{\lambda}/c_g$ and re-organizing terms, we obtain 
\begin{equation}
    \bar{g}(\mu)=\frac{\bar{\lambda}}{\bar{\mu}_F}(1+\beta)(1+L_F\tilde{h}(\mu))^{-1}.\label{eq:gbar_equation}
\end{equation}
Plugging in \eqref{eq:gbarmean} we obtain~\eqref{eq:g_equation}, and since the integrand is decreasing in $L_F$ then the $L_F$ that satisfies this equation must be unique. This proves that for any limit point $\bar{\psi}_{1,t}$, we have $\langle f, \bar{\psi}_{1,t}\rangle \to \int_{\mu_{\min}}^{\mu_{\max}} f(\mu)\bar{g}(\mu)dF(\mu)$ as $t\to\infty$ for all $f\in \CC_{[\mu_{\min},\mu_{\max}]}^b[0,\infty)$. Hence, using the fact that the indicator function of any Borel set $\AAA$ can be approximated by functions in $\CC_{[\mu_{\min},\mu_{\max}]}^b[0,\infty)$, we have 
$\bar{\psi}_{1,t}(\AAA)\to \int_\AAA \bar{g}(\mu)dF(\mu)$. Using Theorem~\ref{thm:interchangibility_stationary_n}, this implies that $\bar{\psi}_\infty^n(\AAA)\Rightarrow \int_\AAA \bar{g}(\mu)dF(\mu)$ as $n\to \infty$. Plugging $
\bar{\psi}_{1,t}$ into the definition of the fairness process, the result follows. \end{proof}

\begin{proof}[Proof of Lemma~\ref{lem:L_bound}.]
Dividing both sides of \eqref{eq:gbarmean} by $\bar{\lambda}\beta$, this equation defines a probability measure $\bar{F}$ on $[\mu_{\min},\mu_{\max}]$ and $c_g$ in the proof of Theorem~\ref{thm:stationary_fairness1} can be written as $c_g=\bar{\lambda}\beta \int_{\mu_{\min}}^{\mu_{\max}}\tilde{h}(\mu) \frac{\mu\bar{g}(\mu)}{\bar{\lambda}\beta}dF(\mu)$. The integral is the expectation of $h(\tilde{\mu})$ where $\tilde{\mu}$ has distribution $\bar{F}$, and hence, is bounded from below by $h_{\min}$ and from above by $h_{\max}$. 
The result follows as $L_F=\bar{\lambda}/c_g$.
\end{proof}

\begin{proof}[Proof of Corollary~\ref{cor:attainability_g}.]
Let $\bar{g}(\mu)$ be the density function of $\bar{\psi}_{1,\infty}$ with respect to $F$. The definition of the limiting fairness measure and \eqref{eq:stationary_idle_length} implies that $\bar{g}(\mu)=\beta\bar{\lambda}\langle \iota,\eta_{1,\infty}\rangle^{-1} g(\mu)$. Plugging into the definition of $c_g$, we can see that for the proposed $h(\mu)$ we have
\begin{align*}
    c_g&=\int_{\mu_{\min}}^{\mu_{\max}}\left(\frac{1+\beta}{\bar{\mu}_F}-\beta\langle \iota,\eta_{1,\infty} \rangle^{-1}g(\mu)\right)\mu dF(\mu)\\
    &=\frac{1+\beta}{\bar{\mu}_F}\int_{\mu_{\min}}^{\mu_{\max}}\mu dF(\mu)-\beta\langle \iota,\eta_{1,\infty} \rangle^{-1}\int_{\mu_{\min}}^{\mu_{\max}}\mu g(\mu)dF(\mu)\\
    &=1.
\end{align*}
Hence, $L_F=\bar{\lambda}$. Now, plugging this in \eqref{eq:gbar_equation}, we get the desired result. \end{proof}

\subsection{Proofs for Results Presented in Section~\ref{sec:strategic}}\label{app:strategic}

\begin{proof}[Proof of Theorem~\ref{thm:best_response_scaling}.]
To simplify notation, define $C=\beta\bar{\mu}_F\bar{\lambda}^{-1}\EE[\xi_\alpha^{-}(\infty)]$. We prove part 2 first, as it is slightly more complicated and part 1 follows similar steps. 
The uniform convergence assumption implies that for any $\rho>0$, there exists an $N_\rho$ such that for all $\mu_{\min}\leq \mu\leq \mu_{\max}$
\begin{equation}\label{eq:approximation_idleness_uniform}
\left|\EE[I_k^n(\infty)|\tilde{\mu}_k^n=\mu]-n^{\alpha-1}Cg(\mu)\right|\leq n^{\alpha-1}\rho.
\end{equation}
First, we analyze the best response for server $k$ for all large $n$ for the approximating problem
\begin{equation}\label{eq:approximating_best_response}
 \max_{\tilde{\mu}_{\min,k}^n\leq \mu\leq \tilde{\mu}_{\max,k}^n}u_I(n^{\alpha-1}Cg(\mu))-a_k^nc(\mu).
\end{equation}
%Since $g(\mu)$ is assumed to be concave with probability 1, $f(n^{\alpha-1}Cg(\mu))$ is also concave in $\mu$. Hence, this is a maximization problem with concave objective function. 
Taking the derivative of the objective, we get
\begin{align}\label{eq:asymptotic_utility_derivative}
n^{\alpha-1}g'(\mu)u_I'(n^{\alpha-1}Cg(\mu))-\tilde{a}_k^nc'(\mu).
\end{align}

First, we analyze part 1, i.e., the situation when $n^{\alpha-1}u_I'(n^{\alpha-1}x)\to 0$ for all $x>0$. We know that $u_I(\cdot)$ is concave and thus its derivative is decreasing. Since $g(\mu)$ is continuous, it attains its minimum at $\mu_g^*\in[\mu_{\min}, \mu_{\max}]$. Then, we have $u_I'(n^{\alpha-1}Cg(\mu_g^*))\geq u_I'(n^{\alpha-1}Cg(\mu))$ for all $\mu_{\min}\leq \mu \leq \mu_{\max}$. Hence, there exists an $N_{1}$ such that $n^{\alpha-1}u_I'(n^{\alpha-1}Cg(\mu_g^*))<\frac{a_{\min}c_{\min}}{2Cg_{\max}'}$ for all $n>N_{1}$ and hence, the minimizer of \eqref{eq:approximating_best_response} is $\mu_k^{n,*}=\tilde{\mu}_{\min,k}^n$ for all $k$ and $n>N_1$. Using the gradient inequality for concave functions on $u_I(n^{\alpha-1}x)$, we get for any $\mu_{\min}\leq \mu\leq \mu_{\max}$ that  
\begin{align*}
u_I\big(n^{\alpha-1}Cg(\mu)\big)
-\tilde{a}_k^nc(\mu)
&
\leq 
u_I\big(n^{\alpha-1}Cg(\mu_{\min,k}^n)\big)
+n^{\alpha-1}Cu_I'\big(n^{\alpha-1}Cg(\mu_g^*)\big)\big(g(\mu)-g(\mu_{\min,k}^*)\big)
\\
&\qquad-\tilde{a}_k^nc(\mu_{\min,k}^n)-a_{\min}c_{\min}(\mu-\mu_{\min,k}^*).
\end{align*}
Hence, for any $n>N_1$, we have 
\begin{align*}
u_I\big(n^{\alpha-1}Cg(\mu)\big)-\tilde{a}_k^nc(\mu)
&
\leq u_I\big(n^{\alpha-1}Cg(\mu_{\min,k}^n)\big) -\tilde{a}_k^nc(\mu_{\min,k}^n)-\frac{a_{\min}c_{\min}}{2}(\mu-\mu_{\min,k}^*),
\end{align*}
and 
for any $\mu>\mu_{\min,k}^n + 6\epsilon(a_{\min}c_{\min})^{-1}$, 
\begin{align}\label{eq:bound_f}
    u_I\big(n^{\alpha-1}Cg(\mu)\big)
    -\tilde{a}_k^nc(\mu)
    \leq u_I\big(n^{\alpha-1}Cg(\mu_{\min,k}^n)\big) -\tilde{a}_k^nc(\mu_{\min,k}^n)-3\epsilon.
\end{align}
Again using \eqref{eq:approximation_idleness_uniform}, there exists an $N_2$ such that $n>N_2$ implies
\begin{align*}
    \EE[I_k^n(\infty)|\tilde{\mu}_k^n=\mu]\geq n^{\alpha-1}\frac{Cg(\mu_g^*)}{2}.
\end{align*}
Using the mean-value theorem and the concavity of $u_I(\cdot)$, for any $\mu_{\min}\leq \mu \leq \mu_{\max}$ and  $n>N_2\vee N_\rho$ 
\begin{align*}
    \Big|u_I\big(\EE[I_k^n(\infty)|\tilde{\mu}_k^n=\mu]\big)
    -u_I\big ( n^{\alpha-1}Cg(\mu)\big) \Big|
    \leq 
    n^{\alpha-1}u_I'\left(\frac{n^{\alpha-1}Cg(\mu_g^*)}{2}\right)\rho.
\end{align*}
Now, choosing $N_3$ such that $n>N_3$ implies
\[
n^{\alpha-1}u_I'\left(\frac{n^{\alpha-1}Cg(\mu_g^*)}{2}\right)\leq \frac{\epsilon}{\rho},
\]
and using \eqref{eq:bound_f}, we have 
\begin{align*}
    u_I\big(\EE[I_k^n(\infty)|\tilde{\mu}_k^n=\mu]\big)-\tilde{a}_k^nc(\mu)\leq u_I\big(\EE[I_k^n(\infty)|\tilde{\mu}_k^n=\tilde{\mu}_{\min,k}^n]\big) -\tilde{a}_k^nc(\tilde{\mu}_{\min,k}^n)-\epsilon,
\end{align*}
for all $\mu>\mu_{\min,k}^n + 6\epsilon(a_{\min}c_{\min})^{-1}$ and $n>\max\{N_1, N_2, N_3, N_\rho\}$.

Now, we consider the second case where $g(\mu)$ is decreasing. As \eqref{eq:asymptotic_utility_derivative} is negative for all $\mu$, the minimizer of \eqref{eq:approximating_best_response} is $\mu_{k}^{n,*}=\tilde{\mu}_{\min,k}^n$ for all $n$. 

Now, we analyze the part 3. As $u_I(\cdot)$ is concave, $u_I'(n^{\alpha-1}Cg(\mu))\geq u_I'(n^{\alpha-1}Cg(\mu_{\max}))$ for all $\mu_{\min}\leq \mu\leq \mu_{\max}$. Using this and the convexity of $c(\mu)$, there exists an $N_4$, such that for $n>N_4$ the derivative in \eqref{eq:asymptotic_utility_derivative} is positive for all $\mu_{\min}\leq \mu \leq \mu_{\max}$, which in turn implies the maximizer of \eqref{eq:approximating_best_response} is $\mu_{k}^{n,*}=\tilde{\mu}_{\max, k}^n$. Now, choosing $N_1$ such that $n>N_1$ implies $n^{\alpha-1}u_I'(n^{\alpha-1}Cg(\mu_{\max}))>\frac{2a_{\min}c_{\min}}{Cg_{\min}'}$ and following the same steps as in part 2, the theorem follows.
\end{proof}
\endproof

To prove Proposition~\ref{thm:interchange_generalized_rr}, we need the following uniform integrability result:
\begin{lemma}\label{lem:uniform_integrable_inverse}
For any $\alpha>1/2$, the collection of random variables $\{(\hat{I}_\alpha^n(\infty))^{-1}\II(I^n(\infty)>0)\}_{n\in \NN}$ is uniformly integrable. 
\end{lemma}
\begin{proof}
We need to prove that for any $\rho>0$, there exists an $M>0$ such that 
\begin{align}
\sup_n \EE\left[\frac{n^{\alpha}}{I^n(\infty)}\II(I^n(\infty)>0)\II\left(\frac{n^{\alpha}}{I^n(\infty)}>M\right)\right]<\rho.
\label{eq:uniform_integrability_def}
\end{align}
For any $n\in \NN$, we have 
\begin{align}
    \nonumber&\EE\left[\frac{n^{\alpha}}{I^n(\infty)}\II(I^n(\infty)>0)\II\left(\frac{n^{\alpha}}{I^n(\infty)}>M\right)\right]\\
    \nonumber&\qquad\qquad=\EE\left[\frac{n^{\alpha}}{I^n(\infty)}\II(I^n(\infty)>0)\II\left(\frac{n^{\alpha}}{I^n(\infty)}>M\right)\II\left(\left|\sum_{k=1}^{N^n_\alpha}\mu_k^n-N^n_\alpha\bar{\mu}_F\right|\leq \frac{\beta}{2}(\lambda^n)^\alpha(\bar{\mu}_F)^{1-\alpha}\right)\right]\\
    \nonumber&\qquad\qquad\quad + \EE\left[\frac{n^{\alpha}}{I^n(\infty)}\II(I^n(\infty)>0)\II\left(\frac{n^{\alpha}}{I^n(\infty)}>M\right)\II\left(\left|\sum_{k=1}^{N^n_\alpha}\mu_k^n-N^n_\alpha\bar{\mu}_F\right|> \frac{\beta}{2}(\lambda^n)^\alpha(\bar{\mu}_F)^{1-\alpha}\right)\right]\\
    \nonumber&\qquad\qquad\leq \EE\left[\frac{n^{\alpha}}{I^n(\infty)}\II(I^n(\infty)>0)\II\left(\frac{n^{\alpha}}{I^n(\infty)}>M\right)|\left|\sum_{k=1}^{N^n_\alpha}\mu_k^n-N^n_\alpha\bar{\mu}_F\right|\leq \frac{\beta}{2}(\lambda^n)^\alpha(\bar{\mu}_F)^{1-\alpha}\right]
    \\&\qquad\qquad\quad + n^{\alpha}\PP\left(\left|\sum_{k=1}^{N^n_\alpha}\mu_k^n-N^n_\alpha\bar{\mu}_F\right|> \frac{\beta}{2}(\lambda^n)^\alpha(\bar{\mu}_F)^{1-\alpha}\right).
    \label{eq:bound_uniform_integrability_I}
\end{align}
First, we concentrate on the first term of the right-hand side. From Assumption~\ref{asm:lambdascaling}, for any $\epsilon_1>0$, there exists an $N_1$ such that for any $n>N_1$, $n^{-1}\lambda^n\leq \bar{\lambda}+\epsilon_1$. Hence, with  $n>N_1$, we define 
\[
K_\alpha=\left\lceil \frac{\beta(\bar{\lambda}+\epsilon_1)^\alpha\bar{\mu}_F^{1-\alpha}}{4\mu_{\max}} \right\rceil. 
\]
Now, we consider a sequence of birth-death processes $\{Y^n(t)\}_{n\in \NN}$. The birth rate in the $n$th system is uniformly equal to $\lambda^n$ for any state. When the system is at state $Y^n(t)=i$, the death rate is given by
\begin{equation*}
\nu_i =\left\{\begin{array}{ll} i\mu_{\min}&\mbox{if }i\leq N^n_\alpha-n^{\alpha}K_\alpha\\
\lambda^n+\frac{\beta}{4}(\lambda^n)^{\alpha}\bar{\mu}_F^{1-\alpha} &\mbox{if }i>N^n_\alpha-n^{\alpha}K_\alpha
\end{array}\right..
\end{equation*}
Now, we show that the process $N^n_\alpha-Y^n(t)$ is stochastically smaller than $\hat{I}^n(t)$. Remembering that due to non-idling property, we have  $\hat{I}^n(t)=(\hat{X}_\alpha^n(t))^-$, we can write
\begin{align*}
    \hat{I}_\alpha^n&=  \left(\hat{X}^n(0)\right)^-+n^{-\alpha}\sum_{i=1}^{S^n(t)}\II\left(\left(\hat{X}^n(\theta_{S,i}^n-)\right)^+=0\right)\II\left(U_i^n\leq p_i^n\left(\theta_{S,i}^n-\right)\right)\\&\qquad\qquad\qquad-n^{-\alpha}\sum_{i=1}^{A^n(t)}\II\left( \left(\hat{X}^n(\theta_{A,i}^n-)\right)^->0\right).
\end{align*}
Similarly, we can couple the birth-death process with the idleness process as
\begin{align*}
Y^n(t)&= Y^n(0)-n^{-\alpha}\sum_{i=1}^{S^n(t)}\II\left(U_i^n\leq \tilde{p}_i^n\left(\theta_{S,i}^n-\right)\right)%\\&\qquad\qquad\qquad
+n^{-\alpha}\sum_{i=1}^{A^n(t)}\II\left( \left(\hat{Y}_1^n(\theta_{A,i}^n-)\right)^->0\right),
\end{align*}
where $\tilde{p}_i^n\left(\theta_{S,i}^n-\right)=\frac{\nu_i}{\sum_{k=1}^{N^n_\alpha}\tilde{\mu}_k^n}$. 
Suppose $N^n_\alpha-Y^n(0)=I^n(0)$ and define $\vartheta^n=\inf\{t: N^n_\alpha-Y^n(t)>I^n(t)\}$. 
Then, with probability 1, $N^n_\alpha Y^n(\vartheta-)=I^n(\vartheta-)$. 
As for any state where $I^n(t)=i$, $\frac{\nu_i}{\sum_{k=1}^{N^n_\alpha}\tilde{\mu}_k^n}\leq \frac{\sum_{k=1}^{N^n_\alpha}\mu_k(1-I_k^n(t))}{\sum_{k=1}^{N^n_\alpha}\mu_k}$, if $\vartheta$ is an event epoch for $S^n(t)$, we have $N^n_\alpha-Y^n(\vartheta-)\leq N^n_\alpha-Y^n(\vartheta-)$. If $\vartheta$ is an event epoch for $A^n$, then $N^n_\alpha-Y^n(\vartheta-)\leq N^n_\alpha-Y^n(\vartheta-)$. This leads to a contradiction and we conclude that $N^n_\alpha-Y^n(t)$ is stochastically less than $I^n(t)$. 
Now, 
\begin{align*}
   \EE\left[\frac{n^{\alpha}}{I^n(\infty)}\II(I^n(\infty)>0)\II\left(\frac{n^{\alpha}}{I^n(\infty)}>\frac{2}{K_\alpha}\right)\right]
   & 
   \leq n^\alpha \PP\left(I^n(\infty)<\frac{K_\alpha n^{\alpha}}{2}\right)
  \\
  &
  \leq n^\alpha \PP\left(Y^n(\infty)\geq N^n_\alpha-\frac{K_\alpha n^{\alpha}}{2}\right).
\end{align*}
Now, as $\nu_i> \lambda^n$ for all $i\geq N^n_\alpha-n^{\alpha}K_\alpha$, for each $n$, the birth-death process $Y^n(t)$ is positive recurrent and for all $i\geq N^n_\alpha-n^{\alpha}K_\alpha$, we have
\begin{align*}
\PP(Y^n(\infty)=i)
&=\left(\frac{\lambda^n}{\lambda^n+\frac{\beta}{4}(\lambda^n)^{\alpha}\bar{\mu}_F^{1-\alpha}}\right)^{i-N^n_\alpha+n^\alpha K_\alpha}
\times \PP(Y^n(\infty)=N^n_\alpha-n^\alpha K_\alpha)
\\
&\leq \left(\frac{\lambda^n}{\lambda^n+\frac{\beta}{4}(\lambda^n)^{\alpha}\bar{\mu}_F^{1-\alpha}}\right)^{i-N^n_\alpha+n^\alpha K_\alpha}.
\end{align*}
Hence, 
\[
n^\alpha\PP\left(Y^n(\infty)
\geq 
N^n_\alpha-\frac{ n^\alpha K}{2}\right)
\leq \left(\frac{\lambda^n}{\lambda^n+\frac{\beta}{4}(\lambda^n)^{\alpha}\bar{\mu}_F^{1-\alpha}}\right)^{n^\alpha K_\alpha/2}\frac{(\lambda^n+\frac{\beta}{4}(\lambda^n)^{\alpha}\bar{\mu}_F^{1-\alpha})n^\alpha}{\frac{\beta}{4}(\lambda^n)^{\alpha}\bar{\mu}_F^{1-\alpha}}.
\]
When $\alpha>1/2$, we see that the first-term on the right-hand side approaches 0 exponentially fast, whereas the second term increases linearly as $n\to \infty$. This implies that there exists an $N_2$ where the first term on the right-hand side of~\eqref{eq:bound_uniform_integrability_I} is less than $\rho/2$ for $n>N_2$.

To address the second term on the right-hand side of~\eqref{eq:bound_uniform_integrability_I}, choose $p$ to be the smallest even number such that $p\alpha>p/2+\alpha$. Then, using Markov's inequality
\begin{align*}
    n^{\alpha}\PP\left(\left|\sum_{k=1}^{N^n_\alpha}\mu_k^n-N^n_\alpha\bar{\mu}_F\right|> \frac{\beta}{2}(\lambda^n)^\alpha(\bar{\mu}_F)^{1-\alpha}\right)&\leq \frac{n^\alpha \sum_{j=1}^{p/2}
    \left(\begin{array}{c}N^n_\alpha\\2j\end{array}\right)(\mu_{\max}-\mu_{\min})^p}{\left(\frac{\beta}{2}(\lambda^n)^\alpha(\bar{\mu}_F)^{1-\alpha}\right)^p}.
\end{align*}
It is now easy to see that the numerator scales with $p/2+\alpha$ where as the denominator scales with $p\alpha$. Hence, the right-hand side approaches 0 as $n\to \infty$ and there exists an $N_3$ such that  $n>N_3$ implies 
\[
n^{\alpha}\PP\left(\left|\sum_{k=1}^{N^n_\alpha}\mu_k^n-N^n_\alpha\bar{\mu}_F\right|> \frac{\beta}{2}(\lambda^n)^\alpha(\bar{\mu}_F)^{1-\alpha}\right)<\rho/2.
\]
 Choosing $M=\max\{2/K_\alpha, N_1, N_2, N_3\}$,~\eqref{eq:uniform_integrability_def} holds and our result follows. \end{proof}

\begin{proof}[Proof of Proposition~\ref{thm:interchange_generalized_rr}.]
We define $U_{A,k,i}^n=1$ if the $i$th event epoch of $A^n(t)$ corresponds to an arrival directly route to server $k$ and 0 otherwise. Similarly, we define $U_{S,k,i}^n=1$ if the event epoch of $S^n(t)$ correpondes to an actual service completion at server $k$. Then, we have the following balance equation:
\begin{align*}
    I_k^n(t) &= I_k^n(0) + \sum_{i=1}^{S^n(t)}U_{S,k,i}^n- \sum_{i=1}^{A^n(t)}U_{A,k,i}^n\\
    &=I_k^n(0) + \left(\sum_{i=1}^{S^n(t)}U_{S,k,i}^n-\mu_k\int_0^t(1-I_k^n(s-))ds\right) + \mu_k\int_0^t(1-I_k^n(s-))ds\\ &\quad -\left(\sum_{i=1}^{A^n(t)}U_{A,k,i}^n - \frac{\lambda^n}{n}\int_0^t\frac{h(\tilde{\mu}_k^n)I_k^n(s-)}{\langle h,\bar{\psi}_{1,s-}^n \rangle }ds\right)-\frac{\lambda^n}{n}\int_0^t\frac{h(\tilde{\mu}_k^n)I_k^n(s-)}{\langle h,\bar{\psi}_{1,s-}^n \rangle }ds.
\end{align*}
The second and fourth terms on the right-hand side are Poisson martingales with initial value 0. Assuming that the system is in stationarity, taking the expectation of both sides and using Fubini's theorem, we have
\begin{align*}
 \frac{\lambda^nh(\mu)}{n}\EE\left[\frac{I_k^n(\infty)}{\langle h,\bar{\psi}_{1,\infty}^n \rangle }|\mu_k=\mu\right]+\mu\EE[I_k^n(\infty)|\mu_k=\mu])=\mu.
\end{align*}
Theorem~\ref{thm:stationary_fairness1} implies that $\langle h,\bar{\psi}_\infty^n\rangle\toP \langle h,\bar{\psi}_\infty\rangle$ where the limit is deterministic. Using Lemma~\ref{lem:uniform_integrable_inverse} and the notation $L_F=\bar{\lambda}\langle h,\bar{\psi}_\infty\rangle^{-1}$, we can find an $N_\epsilon$ such that $n>N_\epsilon$ implies
\[
\left|L_Fh(\mu)\EE[I_k^n(\infty)|\mu_k=\mu]+\mu\EE[I_k^n(\infty)|\mu_k=\mu])-\mu\right|=\epsilon.
\]
Re-arranging the equation, the desired uniform convergence result holds.  
\end{proof}

\begin{proof}[Proof of Proposition~\ref{thm:interchange_idle_time_order}.]
Theorem 9 in \cite{Gopalakrishnan_etal2016} ensures that all the idle-time order based policies have the same stationary distribution. Hence, it is enough to prove the theorem only for one idle-time order based policy for each $\alpha$. The case when $\alpha=1$ follows from our Proposition~\ref{thm:interchange_generalized_rr} when taking $h(\mu)=1$ for all $\mu_{\min}\leq \mu\leq \mu_{\max}$. To prove the result for $1/2\leq \alpha<1$, we concentrate on the longest-idle-server-first policy and follow a similar approach to the proofs of Lemma 4 and Theorem 6 in \cite{bukeqin2019}. 

As in the proof of Lemma~\ref{lem:pos_conv}, we define the $i$th event epoch of the arrival process $A^n(t)$ as $\theta_{A,i}^n$ and the inter-arrival time between arrival $i-1$ and $i$ as $u_{i}^n=\theta_{A, i}^n-\theta_{A,i-1}^n$. As the arrival process is a Poisson process, $u_i^n$s are independent exponential random variables with rate $\lambda^n$. Similarly, we define the $i$th epoch of the potential service completion process, $S_k^n(t)$, of server $k$ as $\theta_{S,k,i}^n$ and the inter-event time between $(i-1)$st and $i$th epoch as $v_{k,i}^n=\theta_{S,k,i}^n-\theta_{S,k,i-1}^n$ for all $i\in \NN$, where $v_{k,i}^n$ are independent exponential random variables with rate $\tilde{\mu}_{k}^n$.  We also denote the $i$th epoch of the actual service completion process, $D_k^n(t)$, of server $k$ as $\bar{\theta}_{S,k,i}^n$. We define $\bar{\phi}_{k,i}^n$ the idling time of server $k$ after the $i$th server completion as $\bar{\phi}_{k,i}^n:=\inf\{t-\bar{\theta}_{S,k,i}^n:I_k^n(t)=0, t>\bar{\theta}_{S,k,i}^n\}$. For longest-idle-server-first, we have 
\begin{align*}
    \bar{\phi}_{k,i}^n=\sum_{j=2}^{I^n(\bar{\theta}_{S,k,i}^n)}u_{A(\bar{\theta}_{S,k,i}^n)+j}^n+(u_{A(\bar{\theta}_{S,k,i}^n)+1}-\bar{\theta}_{S, k,i}^n).
\end{align*}
Using the expression as the start point, for the $i$th event epoch of the potential service completion process $S_k^n(t)$, we associate the potential idling time of server $k$, $\phi_{k,i}^n$, defined as 
\begin{align*}
    \phi_{k,i}^n=\sum_{j=2}^{I^n(\theta_{S,k,i}^n)}u_{A(\theta_{S,k,i}^n)+j}^n+(u_{A(\theta_{S,k,i}^n)+1}-\theta_{S,k,i}^n).
\end{align*}
We also define $\phi_{-k}^n:=\inf\{t\geq 0:I_{k}^n(t)=0\}$ as the first time server $k$ is busy and 
\begin{align*}
    \phi_{-k}^n\leq \sum_{j=1}^{I^n(0)}u_{j}^n.
\end{align*} 
Similar to equation (14) in \cite{bukeqin2019}, we can write 
\begin{align}
    \nonumber\int_0^t I_k^n(s)ds&=(\phi_{-k}^n\wedge t)+\sum_{i=1}^{D_k^n(t)}(\bar{\phi}_{k,i}^n \wedge (t-\bar{\theta}_{S,k,i}^n))\\
    &=(\phi_{-k}^n\wedge t)+\sum_{i=1}^{D_k^n(t)}\bar{\phi}_{k,i}^n - \sum_{i=1}^{D_k^n(t)}(\bar{\phi}_{k,i}^n - t+\bar{\theta}_{S,k,i}^n)^+.\label{eq:idle_expansion}
\end{align}

For the proofs in this section, we assume that the system starts in stationarity. Hence, for all $t\geq 0$, we have
\[
\EE[I_k^n(\infty)|\tilde{\mu}_k^n=\mu] = \EE[I_k^n(t)|\tilde{\mu}_k^n=\mu].
\]
We need the following lemma to prove our theorem. 
\begin{lemma}\label{lem:idleness_conv_zero}
For any $k$, $I_k^n(\infty)\toP 0$ and $\EE[I_k^n(\infty)|\mu_k^n=\mu]\to 0$ as $n\to \infty$.
\end{lemma}
\begin{proof}
We have 
\begin{align*}
 \EE\Big[\int_0^t & I_k^n(s)  ds|\tilde{\mu}_k^n=\mu\Big]
=\int_0^t\EE[I_k^n(s)|\tilde{\mu}_k^n=\mu]ds
\\
&
=\EE[(\phi_{-k}^n\wedge t)|\tilde{\mu}_k^n=\mu]+\EE\left[\sum_{i=1}^{D_k^n(t)}\bar{\phi}_{k,i}^n\wedge (t-\bar{\theta}_{S,k,i}^n)|\tilde{\mu}_k^n=\mu\right]\\
&\leq \EE[\phi_{-k}^n|\tilde{\mu}_k^n=\mu]+ \EE\left[\sum_{i=1}^{S_k^n(t)}\phi_{k,i}^n|\tilde{\mu}_k^n=\mu\right]\\
&\leq \frac{\EE[I^n(\infty)]}{\lambda^n}\\&\quad+\EE\left[\sum_{i=1}^{\infty}\II(\theta_{S,k,i}^n\leq t)\left(\sum_{j=2}^{I^n(\theta_{S,k,i}^n)}u_{A(\theta_{S,k,i}^n)+j}^n+(u_{A(\theta_{S,k,i}^n)+1}-\theta_{S,k,i}^n)\right)|\tilde{\mu}_k^n=\mu\right]\\
&\leq \frac{\EE[I^n(\infty)]}{\lambda^n}\\&
\quad+\sum_{i=1}^{\infty}\EE\left[\II(\theta_{S,k,i}^n\leq t)\EE\left[\left(\sum_{j=2}^{I^n(\theta_{S,k,i}^n)}u_{A(\theta_{S,k,i}^n)+j}^n+(u_{A(\theta_{S,k,i}^n)+1}-\theta_{S,k,i}^n)\right)|\cF_{\theta_{S,k,i}^n}-\right]|\tilde{\mu}_k^n=\mu\right]\\
&\leq \frac{\EE[I^n(\infty)]}{\lambda^n}+\EE\left[\sum_{i=1}^{S_k^n(t)}\frac{(I^n(\theta_{S,k,i}^n-)+1)}{\lambda^n}|\tilde{\mu}_k^n=\mu\right]\\
&\leq \frac{\EE[I^n(\infty)]}{\lambda^n}+\frac{\mu t\EE[I^n(\infty)+1|\tilde{\mu}_k^n=\mu]}{\lambda^n}\to 0.
\end{align*}
As $\mu\leq \mu_{\max}$, the convergence is uniform. The convergence in probability follows using Markov's inequality, which concludes the proof.\end{proof}

Our stationarity assumption combined with \eqref{eq:idle_expansion} implies that
\begin{align*}
    n^{1-\alpha}\EE[I_k^n(\infty)|\tilde{\mu}_k^n=\mu]&=\frac{n^{1-\alpha}}{t}\int_0^t\EE[I_k^n(s)|\tilde{\mu}_k^n=\mu]ds\\
    &=\frac{n^{1-\alpha}}{t}\EE[(\phi_{-k}^n\wedge t)|\tilde{\mu}_k^n=\mu]+\frac{n^{1-\alpha}}{t}\EE\left[\sum_{i=1}^{D_k^n(t)}\bar{\phi}_{k,i}^n|\tilde{\mu}_k^n=\mu\right]\\
    &\qquad \qquad-\frac{n^{1-\alpha}}{t}\EE\left[\sum_{i=1}^{D_k^n(t)}(\bar{\phi}_{k,i}^n - t+\bar{\theta}_{S,k,i}^n)^+|\tilde{\mu}_k^n=\mu\right].
\end{align*}
Hence, for all $t\geq 0$, we can bound $n^{1-\alpha}\EE[I_k^n(\infty)|\tilde{\mu}_k^n=\mu]$ as
\begin{align}
    \label{eq:individual_idle_LISF_lb}n^{1-\alpha}\EE[I_k^n(\infty)|\tilde{\mu}_k^n=\mu]&\geq \frac{n^{1-\alpha}}{t}\EE\left[\sum_{i=1}^{D_k^n(t)}\phi_i^n|\tilde{\mu}_k^n=\mu\right]-\frac{n^{1-\alpha}}{t}\EE\left[\sum_{i=1}^{D_k^n(t)}(\phi_i^n-(t-\bar{\theta}_{S,k,i}^n))^+|\tilde{\mu}_k^n=\mu\right],
    \\
    \label{eq:individual_idle_LISF_ub}n^{1-\alpha}\EE[I_k^n(\infty)|\tilde{\mu}_k^n=\mu]&\leq\frac{n^{1-\alpha}}{t}\EE\left[\sum_{i=1}^{D_k^n(t)}\phi_i^n|\tilde{\mu}_k^n=\mu\right] + \frac{n^{1-\alpha}}{t}\EE[\phi_{-k}^n|\tilde{\mu}_k^n=\mu].
\end{align}
 The first terms on the right-hand sides of \eqref{eq:individual_idle_LISF_lb} and \eqref{eq:individual_idle_LISF_ub} are the same and we first concentrate on this. We have then 
 \begin{align*}
 \frac{n^{1-\alpha}}{t}\EE\left[\sum_{i=1}^{D_k^n(t)}\phi_i^n|\tilde{\mu}_k^n=\mu\right] &= \frac{n^{1-\alpha}}{t}\EE\left[\sum_{i=1}^{S_k^n(t)}\sum_{j=1}^{I^n(\bar{\theta}_{S,k,i}^n)}u_{A(\bar{\theta}_{S,k,i}^n)+j}^n|\tilde{\mu}_k^n=\mu\right]\\
 &= \frac{n^{1-\alpha}}{t}\EE\left[\sum_{i=1}^{S_k^n(t)}(1-I_k^n(\theta_{S,k,i}^n-))\sum_{j=1}^{I^n(\theta_{S,k,i}^n)}u_{A(\theta_{S,k,i}^n)+j}^n|\tilde{\mu}_k^n=\mu\right]\\
 &=\frac{n^{1-\alpha}}{t}\EE\left[\sum_{i=1}^{S_k^n(t)}\sum_{j=1}^{I^n(\theta_{S,k,i}^n)}u_{A(\theta_{S,k,i}^n)+j}^n|\tilde{\mu}_k^n=\mu\right]\\&\qquad-\frac{n^{1-\alpha}}{t}\EE\left[\sum_{i=1}^{S_k^n(t)}I_k^n(\theta_{S,k,i}^n-)\sum_{j=1}^{I^n(\theta_{S,k,i}^n)}u_{A(\theta_{S,k,i}^n)+j}^n|\tilde{\mu}_k^n=\mu\right]\\
 &=\frac{n^{1-\alpha}}{t}\EE\left[\int_0^t\sum_{j=1}^{I^n(s-)+1}u_{A(s-)+j}^ndS_k^n(s)|\tilde{\mu}_k^n=\mu\right]\\
 &\qquad- \frac{n^{1-\alpha}}{t}\EE\left[\int_0^t\left(I_k^n(s-)\sum_{j=1}^{I^n(s-)}u_{A(s-)+j}^n\right)dS_k^n(s)|\tilde{\mu}_k^n=\mu\right]\\
 &=\frac{n\mu}{\lambda^n t}\EE\left[\int_0^t(\hat{I}_\alpha^n(s-)+n^{-\alpha})ds|\tilde{\mu}_k^n=\mu\right]\\
 &\qquad -\frac{n\mu}{\lambda^n t}\EE\left[\int_0^tI_k^n(s-)(\hat{I}_\alpha^n(s-)+n^{-\alpha})ds|\tilde{\mu}_k^n=\mu\right].
 \end{align*}
Using the stochastic boundedness of $\hat{I}_\alpha^n$ and Lemma~\ref{lem:idleness_conv_zero}, the second term converges to 0. Using the stationarity assumption, 
\begin{equation*}
    \sup_{\mu_{\min}\leq \mu\leq \mu_{\max}}\left|\frac{n\mu}{\lambda^n t}\EE\left[\int_0^t(\hat{I}_\alpha^n(s-)+n^{-\alpha})ds|\tilde{\mu}_k^n=\mu\right]- \frac{\mu}{\bar{\lambda}}\EE[\hat{I}_\alpha^n(\infty)|\tilde{\mu}_k^n=\mu]\right|\to 0.
\end{equation*}
 Next, we concentrate on the second term on the right-hand side of \eqref{eq:individual_idle_LISF_lb}. Each summand corresponds to the remaining idling time from the service completion at $\bar{\theta}_{S,k,i}^n$. This implies only the summand corresponding to the last service completion is positive and 
 \[
 (\phi_i^n-(t-\bar{\theta}_{S,k,i}^n))^+\leq \sum_{j=1}^{I^n(t)}u_{A^n(t)+j}.
 \]
 Hence,
 \begin{align*}
     \frac{n^{1-\alpha}}{t}\EE\left[\sum_{i=1}^{D_k^n(t)}(\phi_i^n-(t-\bar{\theta}_{S,k,i}^n))^+|\mu_k^n=\mu\right]&\leq \frac{n^{1-\alpha}}{t}\EE\left[\sum_{j=1}^{I^n(t)}u_{A^n(t)+j}|\mu_k^n=\mu\right]\\
     &=\frac{n^{1-\alpha}}{t}\EE\left[\EE\left[\sum_{j=1}^{I^n(t)}u_{A^n(t)+j}|\cF_t\right]|\mu_k^n=\mu\right]\\
     &=\frac{n}{\lambda^nt}\EE\left[\hat{I}_\alpha^n(t)|\mu_k^n=\mu\right]\\
     &=\frac{n}{\lambda^nt}\EE\left[\hat{I}_\alpha^n(\infty)|\mu_k^n=\mu\right].
 \end{align*}
 Similarly, we can bound the second term on the right-hand side of \eqref{eq:individual_idle_LISF_ub} as 
 \begin{align*}
     \frac{n^{1-\alpha}}{t}\EE[\phi_{-k}^n|\tilde{\mu}_k^n=\mu]&\leq \frac{n^{1-\alpha}}{t}\EE\left[\sum_{j=1}^{I^n(\infty)}u_{j}^n|\tilde{\mu}_k^n=\mu\right]
    %  \\
    %  &
     =\frac{n}{\lambda^nt}\EE\left[\hat{I}_\alpha^n(\infty)|\mu_k^n=\mu\right].
 \end{align*}
Taking $t$ large enough, both terms can be made arbitrarily small and hence \eqref{eq:individual_idle_LISF_lb} and \eqref{eq:individual_idle_LISF_ub} implies
\[
 \sup_{\mu_{\min}\leq \mu\leq \mu_{\max}}\left|n^{1-\alpha}\EE[I_k^n(\infty)|\tilde{\mu}_k^n=\mu]- \frac{\mu}{\bar{\lambda}}\EE[\hat{I}_\alpha(\infty)|\tilde{\mu}_k^n=\mu]\right|\to 0.
\]
A careful investigation of the proof of Theorem~\ref{thm:system_convergence}, i.e., repeating the same proof with replacing $\tilde{\mu}_k^n=\mu$ for a specific $k$, reveals that $\EE[\hat{I}_\alpha(\infty)|\tilde{\mu}_k^n=\mu] = \EE[\hat{I}_\alpha(\infty)]$, which implies the result for $\alpha=1/2$. Proposition~\ref{thm:specialpolicies} implies that $\langle \iota, \eta_{LISF,\infty}\rangle =\frac{\sigma_F^2 + \mu_F^2}{\mu_F}$ for $\alpha<1$. Replacing this to equation~\eqref{eq:stationary_idle_length}, we get
\[
\sup_{\mu_{\min}\leq \mu\leq \mu_{\max}}\left|n^{1-\alpha}\EE[I_k^n(\infty)|\tilde{\mu}_k^n=\mu]- \mu\beta\bar{\lambda}^{\alpha-1}(\bar{\mu}_F^{-\alpha}\sigma_F^2+\bar{\mu}_F^{2-\alpha})\right|\to 0,\quad \mbox{ for all }1/2< \alpha<1
\]
Hence, the result follows. \end{proof}

\begin{proof}[Proof of Lemma~\ref{lem:utility_concavity}.]
The utility of idleness function $u_I(\cdot)$ is increasing and concave.  Hence, to prove the concavity of the first part, it is enough to prove that the idleness experienced by a server with rate $\mu$, $(1+L_{F^{(0)}}\tilde{h}(\mu))^{-1}$, is concave see, e.g., page 86 in \cite{boydvandenberghe2004}. Taking the second derivative yields 
\begin{align*}
    \frac{d^2(1+L_{F^{(0)}}\tilde{h}(\mu))^{-1}}{d\mu^2}=\frac{L_{F^{(0)}}(L_{F^{(0)}}(2\tilde{h}'(\mu)^2-\tilde{h}(\mu)\tilde{h}''(\mu))-\tilde{h}''(\mu))}{(1+L_{F^{(0)}}\tilde{h}(\mu))^3}.
\end{align*}
Using the convexity of $\tilde{h}(\mu)$, we can conclude that the second derivative is negative for the idleness experienced by a server with rate $\mu$ is concave under \eqref{eq:necessary_concavity} and hence, the result follows.   \end{proof}

\begin{proof}[Proof of Theorem~\ref{thm:fixed_point_equation}.]
As $L_F$ uniquely characterizes the distribution of idleness among servers, our result will follow if the distribution of service rates resulting from $L_F$ also yields the same $L_F$ for the $h$-random policy under consideration. Plugging in \eqref{eq:g_equation}, this can be written as
\[
\int_{\mu_{\min}}^{\mu_{\max}}\mu\frac{1+\beta}{\bar{\mu}_F(1+L_F\tilde{h}(\mu))}dF(\mu|L_F)=\beta.
\]
Taking $\bar{\mu}_F$ to the right-hand side and writing it as an integral, we get
\[
\int_{\mu_{\min}}^{\mu_{\max}}\mu\frac{1+\beta}{1+L_F\tilde{h}(\mu)}dF(\mu)=\beta\int_{\mu_{\min}}^{\mu_{\max}}\mu dF(\mu|L_F).
\]
Re-arranging the terms, we get the desired result.
\end{proof}

We now provide an explicit expression for \eqref{eq:g_equilibrium_equation} when the random trade-off parameter $\tilde{a}_k^n$ is a continuous random variable with density $f_a(\cdot)$ and cumulative distribution $F_a(\cdot)$ and  all servers have identical lower and upper bounds, i.e., $\tilde{\mu}_{\min,k}^n = \mu_{\min}$ and $\tilde{\mu}_{\max,k}^n = \mu_{\max}$ w.p.~1 for any $k,n\in\NN$. 
Using~\eqref{eq:distribution_function} along with the change of variables formula, the equilibrium equation \eqref{eq:g_equilibrium_equation} for this special case reduces to 
\begin{align*}
&\mu_{\min}\frac{1-\beta L_F\tilde{h}(\mu_{\min})}{1+L_F\tilde{h}(\mu_{\min})}\left(1-F_a\left(\frac{L_F}{(\mu_{\min} +L_F)^2c'(\mu_{\min})}\right)\right)\\&\qquad + 
\int_{\mu_{\min}}^{\mu_{\max}}\mu\frac{(1-\beta L_F\tilde{h}(\mu))(c'(\mu)+(\mu+ L_F)c''(\mu))}{(1+L_F\tilde{h}(\mu))^4(c'(\mu))^2}f_a\left(\frac{L_F}{(\mu +L_F)^2c'(\mu)}\right)d\mu\\
&\qquad\qquad\qquad\qquad\qquad\qquad\qquad\qquad + \mu_{\max}\frac{1-\beta L_F\tilde{h}(\mu_{\max})}{1+L_F\tilde{h}(\mu_{\max})}F_a\left(\frac{L_F}{(\mu_{\max} +L_F)^2c'(\mu_{\max})}\right) = 0.
\end{align*}

\begin{proof}[Proof of Proposition~\ref{prop:existence}.]
As $a_k^n$ is a continuous random variable, using the continuity of $C(\mu,L_{F^{(0)}})$ we can conclude that if $L_{i}\to L_{F^{(0)}}$, we have $F(\mu|L_{i})\to F(\mu|L_{F^{(0)}})$ for all $\mu\in[\mu_{\min}, \mu_{\max}]$. Hence, 
\begin{align*}
&\left|\int_{\mu_{\min}}^{\mu_{\max}}\mu\frac{1-\beta L_i\tilde{h}(\mu)}{1+L_i\tilde{h}(\mu)}dF(\mu|L_i)-\int_{\mu_{\min}}^{\mu_{\max}}\mu\frac{1-\beta L_{F^{(0)}}\tilde{h}(\mu)}{1+L_{F^{(0)}}\tilde{h}(\mu)}dF(\mu|L_{F^{(0)}})\right|\\
&\qquad\qquad\leq \left|\int_{\mu_{\min}}^{\mu_{\max}}\mu\frac{1-\beta L_i\tilde{h}(\mu)}{1+L_i\tilde{h}(\mu)}dF(\mu|L_i)-\int_{\mu_{\min}}^{\mu_{\max}}\mu\frac{1-\beta L_{F^{(0)}}\tilde{h}(\mu)}{1+L_{F^{(0)}}\tilde{h}(\mu)}dF(\mu|L_{F_i})\right|\\
&\qquad\qquad\quad + \left|\int_{\mu_{\min}}^{\mu_{\max}}\mu\frac{1-\beta L_{F^{(0)}}\tilde{h}(\mu)}{1+L_{F^{(0)}}\tilde{h}(\mu)}dF(\mu|L_i)-\int_{\mu_{\min}}^{\mu_{\max}}\mu\frac{1-\beta L_{F^{(0)}}\tilde{h}(\mu)}{1+L_{F^{(0)}}\tilde{h}(\mu)}dF(\mu|L_{F^{(0)}})\right|.
\end{align*}
As $i\to \infty$, the first term on the right-hand side converges due to continuity of the integrand, and the second term converges using the definition of weak convergence, which implies the left-hand side of \eqref{eq:g_equilibrium_equation} is continuous with respect to $L_F$. Now, if $L_F=1/(\beta \tilde{h}(\mu_{\min}))$, the integrand of \eqref{eq:g_equilibrium_equation} is non-negative for all $\mu\in[\mu_{\min}, \mu_{\max}]$ and $\tilde{h}(\mu)$ being strictly decreasing, implies that the integral is positive for all $L_F\leq 1/(\beta \tilde{h}(\mu_{\min}))$. Similarly, if $L_F=1/(\beta \tilde{h}(\mu_{\max}))$, the integrand is non-positive and the integral is negative. Using the intermediate value theorem, the result follows. \end{proof} 

\begin{proof}[Proof of Proposition~\ref{prop:uniqueness}.]
The proof of Proposition~\ref{prop:existence} relies on left-hand side of \eqref{eq:g_equilibrium_equation} being continuous and positive when $L_F = 1/(\beta \tilde{h}(\mu_{\min}))$ and negative when $L_F = 1/(\beta \tilde{h}(\mu_{\max}))$. The uniqueness follows if we can show that the left-hand side is strictly decreasing with respect to $L_F$. We first show that the integrand is strictly decreasing with respect to $L_F$ for any given $\mu\in [\mu_{\min},\mu_{\max}]$. Consider $L_1$ and $L_2$ such that $1/(\beta \tilde{h}(\mu_{\min}))\leq L_1<L_2\leq 1/(\beta \tilde{h}(\mu_{\max}))$. 
\begin{align}
    \nonumber \mu\frac{1-\beta L_2\tilde{h}(\mu)}{1+L_2\tilde{h}(\mu)} - \mu\frac{1-\beta L_1\tilde{h}(\mu)}{1+L_1\tilde{h}(\mu)} &= \mu \frac{\tilde{h}(\mu)(L_1 - L_2)(1 + \beta)}{\big(1 + L_2\tilde{h}(\mu)\big)\big(1 + L_1\tilde{h}(\mu)\big)}\\
     &\leq \mu_{\min} \frac{\tilde{h}(\mu_{\max})(L_1 - L_2)(1 + \beta)}{\big(1 + L_2\tilde{h}(\mu_{\min})\big)\big(1 + L_1\tilde{h}(\mu_{\min})\big)}<0. \label{eq:monotonicity_L}
\end{align}
Now, consider two random variables $\tilde{\mu}^1$ and $\tilde{\mu}^2$ distributed with $F(\mu|L_1)$ and $F(\mu|L_2)$, respectively. The definition of $F(\mu|L_2)$ in \eqref{eq:distribution_function} and $C(\mu,L_F)$ being non-increasing with respect to $L_F$ imply that $\tilde{\mu}^1\geq_{st} \tilde{\mu}^2$. Defining
\[
\tilde{Y}(\mu, L_F) = \mu\frac{1-\beta L_F\tilde{h}(\mu)}{1+L_F\tilde{h}(\mu)},
\]
the integrand being non-decreasing imply $Y(\tilde{\mu}^1, L_F)\geq_{st} Y(\tilde{\mu}^2, L_F)$ for any $L_F$. Hence,
\begin{align*}
       \int_{\mu_{\min}}^{\mu_{\max}}\mu\frac{1-\beta L_1\tilde{h}(\mu)}{1+L_1\tilde{h}(\mu)}dF(\mu|L_1) = \EE[\tilde{Y}(\tilde{\mu}^1, L_1)]&>\EE[\tilde{Y}(\tilde{\mu}^1, L_2)]\\&\geq \EE[\tilde{Y}(\tilde{\mu}^2, L_2)] =  \int_{\mu_{\min}}^{\mu_{\max}}\mu\frac{1-\beta L_2\tilde{h}(\mu)}{1+L_2\tilde{h}(\mu)}dF(\mu|L_2),
\end{align*}
where the strict inequality follows from \eqref{eq:monotonicity_L} and this proves that there exists a unique solution. 
\end{proof}

One can easily check the condition in Part 1 of Proposition~\ref{prop:uniqueness} by taking derivative and checking
\begin{align}
    \frac{\partial}{\partial \mu}\mu\frac{1-\beta L_F\tilde{h}(\mu)}{1+L_F\tilde{h}(\mu)} &= \frac{ 1 + (1 - \beta) L_F \tilde{h}(\mu) - \beta (L_F \tilde{h}(\mu))^2 - \mu (\beta + 1) L_F \tilde{h}'(\mu)
}{(1+L_F\tilde{h}(\mu))^2} \geq 0.
\label{eq:uniqueness_c1_modified}
\end{align}
The denominator is always positive and hence, it is enough to ensure that the numerator is non-negative. Now, we check when \eqref{eq:uniqueness_c1_modified} holds when $\tilde{h}(\mu)=\mu^{-p}$ for $0<p\leq 1$, which satisfies Assumption~\ref{asm:h_convexity}. Replacing in \eqref{eq:uniqueness_c1_modified} and multiplying both sides with $\mu^{2p}/(L_F^2)$, we get 
\[
 \frac{\mu^{2p}}{L_F^2} + \Big( (1 - \beta) + (\beta + 1)p \Big) \frac{\mu^p}{L_F} - \beta \geq 0.
\]
For any fixed $L_F$, the left-hand side is a quadratic function of $\mu^p/L_F$ increasing to $\infty$ as $\mu\to\infty$ with two real roots. Hence, if both roots of this quadratic function is less than the lowest possible value of $\mu^p/L_F$ which is $\beta(\mu_{\min}/\mu_{\max})^p$, i.e.,
\[
 \beta\left(\frac{\mu_{\min}}{\mu_{\max}}\right)^p \geq  \frac{- \Big( (1 - \beta) + (\beta + 1)p \Big) + \sqrt{\Big( (1 - \beta) + (\beta + 1)p \Big)^2 + 4\beta}}{2}.
\]
For example, for $p=1$, which corresponds to the uniformly random routing, the condition reduces to 
\[
\frac{\mu_{\min}}{\mu_{\max}}\geq \frac{-1 + \sqrt{1+\beta}}{\beta}.
\]

The term $C(\mu,L_F)$ involves the utility of idleness function $u_I(I)$ and, we now demonstrate how Part 2 of Proposition~\ref{prop:uniqueness} can be verified for $u_I(I) = I$, which is the case in ~\cite{Gopalakrishnan_etal2016}. Replacing these in $C(\mu,L_F)$ and noting that $\tilde{h}(\mu)$ and $c(\mu)$ are strictly decreasing and increasing functions with respect to $\mu$, respectively, we need to show that 
\[
\frac{\partial}{\partial L_F}\frac{L_F}{(1 + L_F \tilde{h}(\mu))^2} = \frac{1 - L_F \tilde{h}(\mu)}{(1 + L_F \tilde{h}(\mu))^3}\leq 0,
\]
for all $\mu\in[\mu_{\min}, \mu_{\max}]$ and $L_F\in[1/\beta \tilde{h}(\mu_{\min}), 1/\beta \tilde{h}(\mu_{\max})]$, which is satisfied if
\[
\frac{h(\mu_{\max})}{h(\mu_{\min})}\geq \beta.
\]

\section{Additional Numerical Results}\label{sec:additional_numerics}

In this section, we provide further numerical results to demonstrate robustness of the results presented in Section~\ref{sec:numerical}. We again use the base-case scenario presented therein and state the changing parameter.

As our first experiment, we study the existence of uniqueness of a solution for \eqref{eq:g_equilibrium_equation}. Proposition~\ref{prop:existence} shows that Equation \eqref{eq:g_equilibrium_equation} has a solution, possibly non-unique,  in the interval suggested by Lemma~\ref{lem:L_bound}.
Figure~\ref{fg:L_var} illustrates the value of the integral in~\eqref{eq:g_equilibrium_equation} as a function of $L_F$, and shows all these functions to be equal to $0$ for a unique value of $L_F$ for various values of $p,q,r$ and $\beta$. 
% \begin{figure}
%     \caption{Integral on the left-hand side of \eqref{eq:g_equilibrium_equation} as function of $L_F$ for various values of $r$ and $\beta$.} 
%     "Fig1a=\ref{fg:integral-r}" and "Fig1=\ref{fg:L_var}"}\\
%     \subfloat[\label{fg:integral-r}Integral vs. $L_F$ varying $r$ values in the base-case]{%
%   \includegraphics[width=0.5\textwidth]{L_integral_var_q.png}}
%   \subfloat[\label{fg:L_var}Integral vs. $L_F$ varying $\beta$ values in the base-case]{%
%   \includegraphics[width=0.5\textwidth]{L_integral_var_beta.png}}
% \end{figure}

 \begin{figure}
\caption{Integral on the left-hand side of \eqref{eq:g_equilibrium_equation} as function of $L_F$ for various values of $r$ and $\beta$.}
    \centering
    \begin{tabular}{cc}
    \subfloat[Integral vs. $L_F$ varying $\beta$ \label{fg:integral-beta}]{%
  \includegraphics[width=0.35\textwidth]{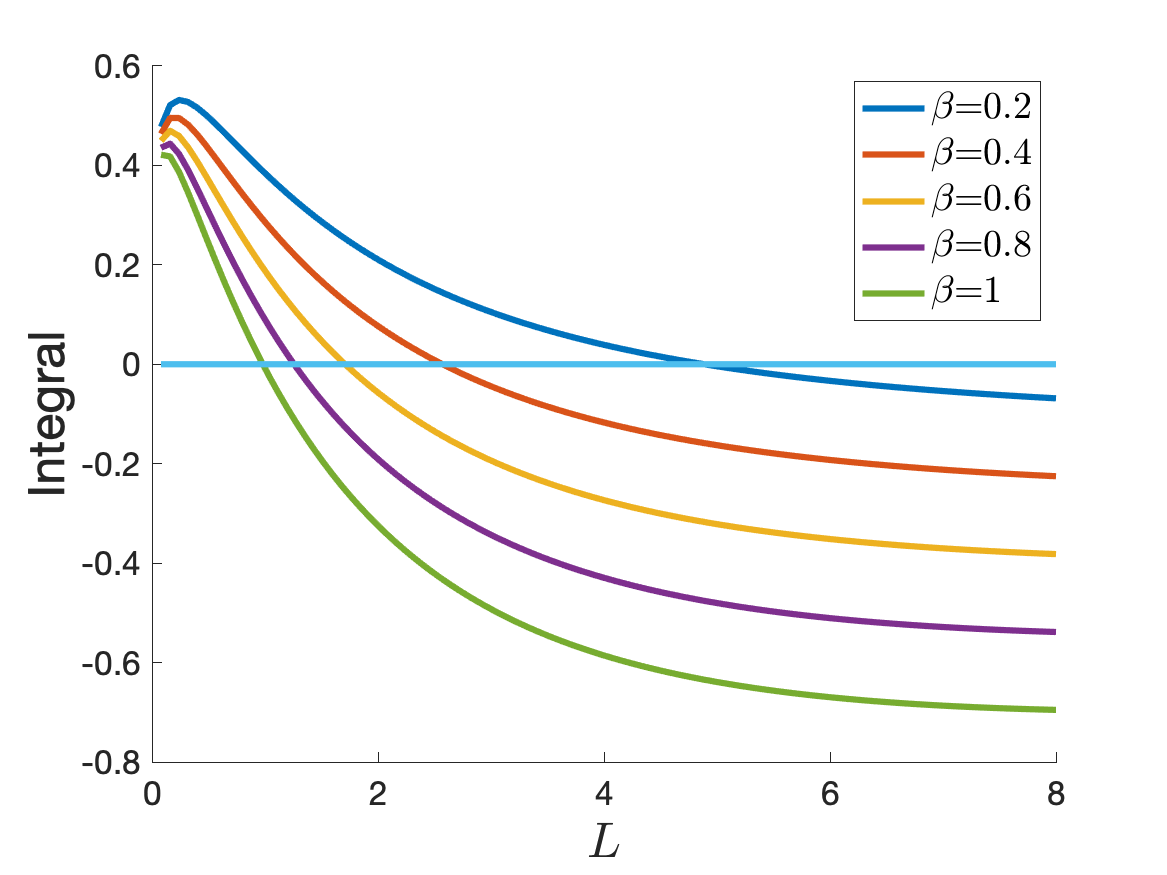}}&
  \subfloat[Integral vs. $L_F$ varying $p$ \label{fg:integral-p}]{%
  \includegraphics[width=0.35\textwidth]{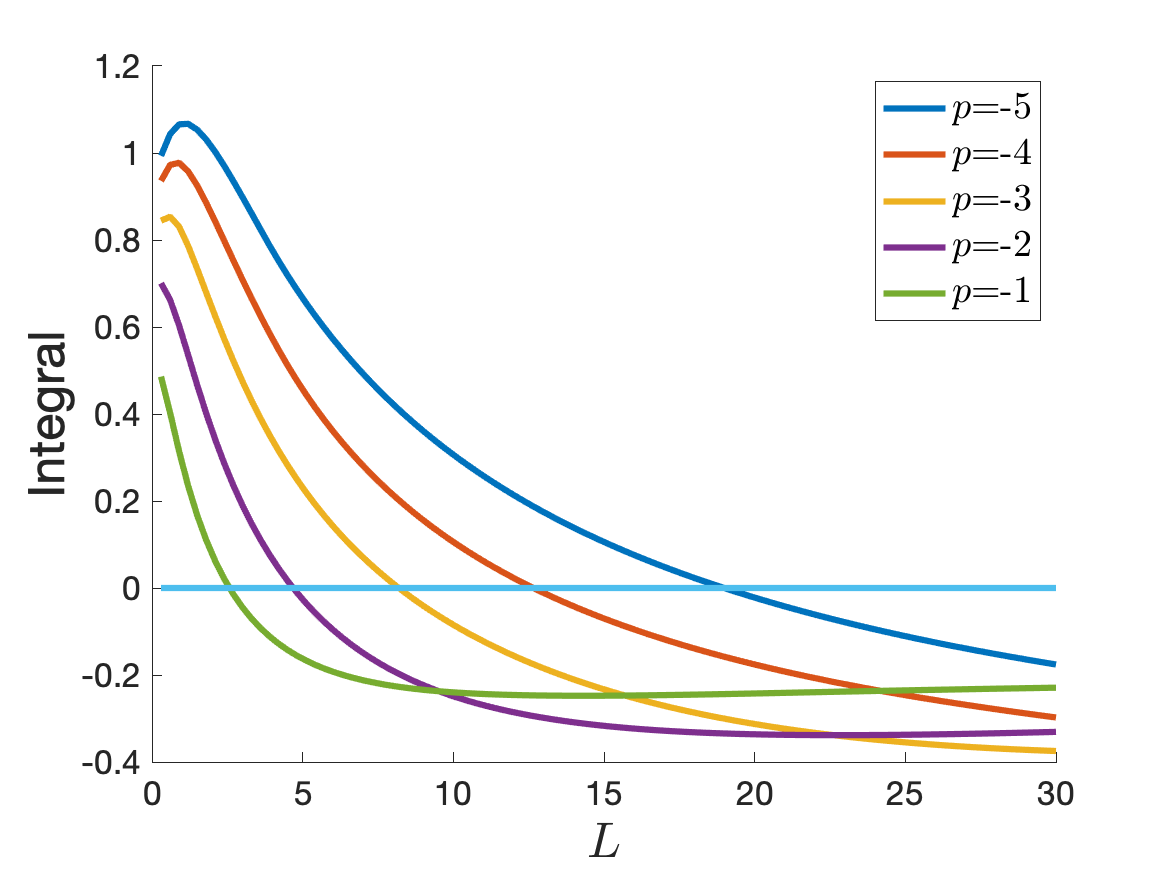}}\\
    \subfloat[Integral vs. $L_F$ varying $q$ \label{fg:integral-q}]{%
  \includegraphics[width=0.35\textwidth]{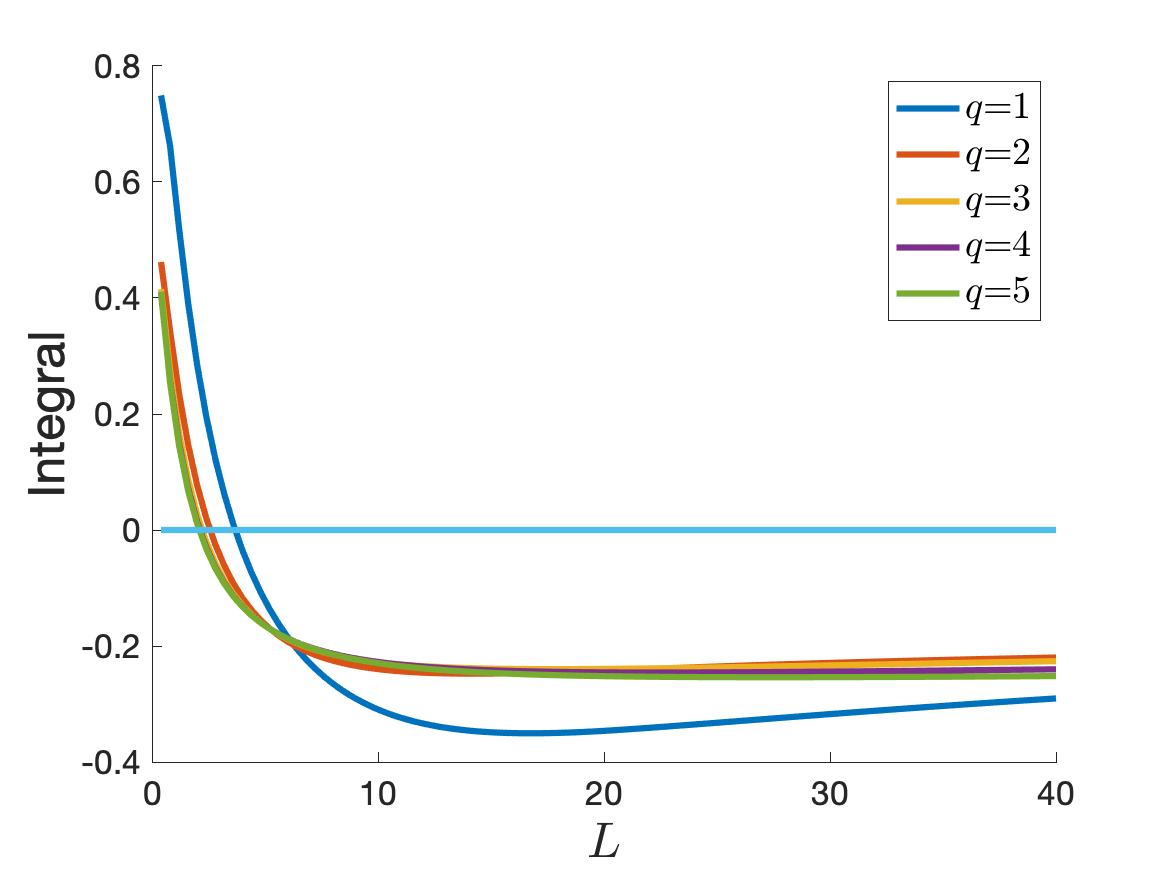}}&
   \subfloat[Integral vs. $L_F$ varying $r$ \label{fg:integral-r}]{%
  \includegraphics[width=0.35\textwidth]{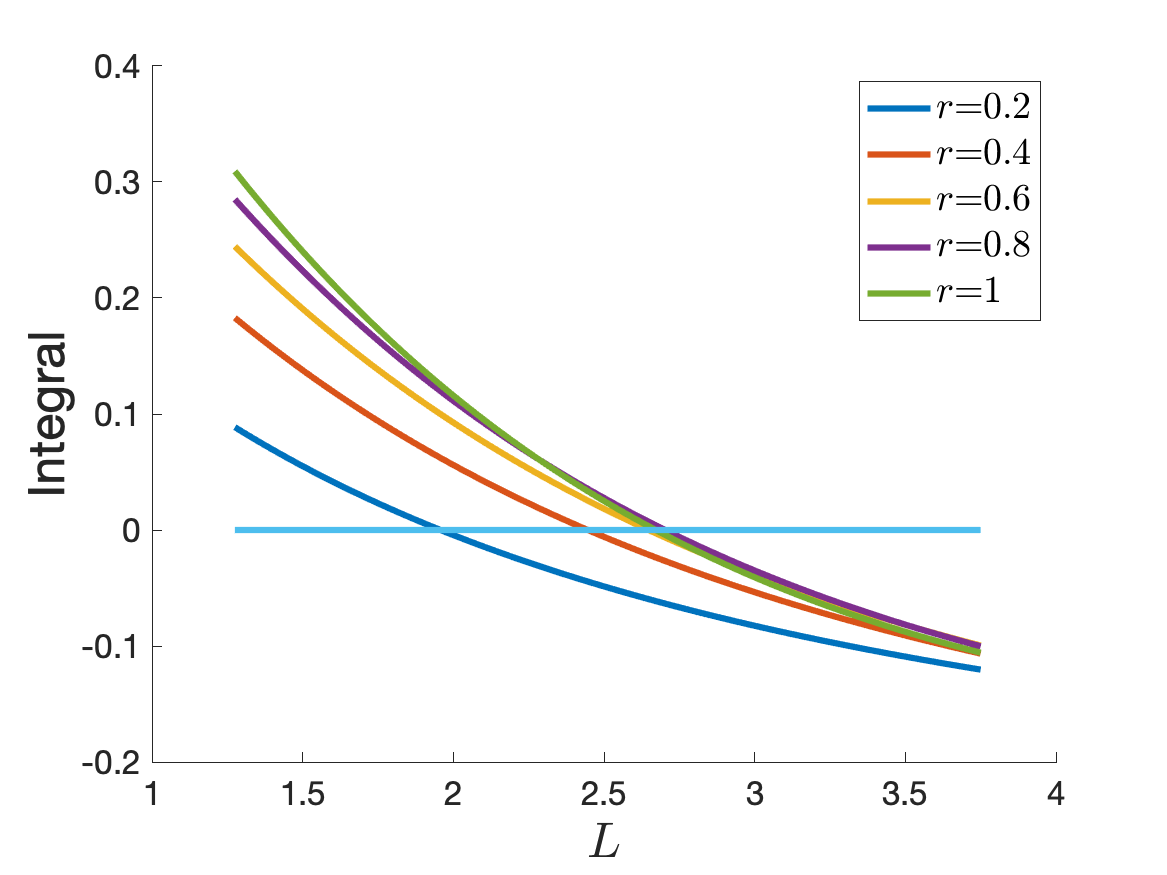}}
  \end{tabular}
  \label{fg:L_var}
 \end{figure}

We now show that the results presented in Figure~\ref{fg:L_var} is robust for other distributions of the trade-off parameter $\tilde{a}_k^n$. We consider three triangular distributions: (a) a symmetric triangular distribution (b) a triangular distribution with increasing density and (c) a triangular distribution with a decreasing density as presented in Figure~\ref{fg:triangular_distributions}. Figures~\ref{fg:tri_L_var}-\ref{fg:tri2_beta_vs_N_mu} are very similar in nature to Figure~\ref{fg:L_var}-\ref{fg:beta_vs_N_mu} and shows that uniform distribution assumption is not restrictive. On the other hand, the distributions illustrated in Figures~\ref{fg:tri_mu_distributions}-\ref{fg:tri2_mu_distributions} differ in shape compared to Figure~\ref{fg:mu_distributions}. However, the effect of parameters on the scale and location of the equilibrium distributions do not exhibit sensitivity to the underlying distribution. 
\begin{figure}
    \caption{A symmetric triangular distribution.}
    \centering
    \begin{tabular}{ccc}
    \subfloat[A symmetric triangular distribution\label{fg:sym_triangular}]{%
  \includegraphics[width=0.3\textwidth]{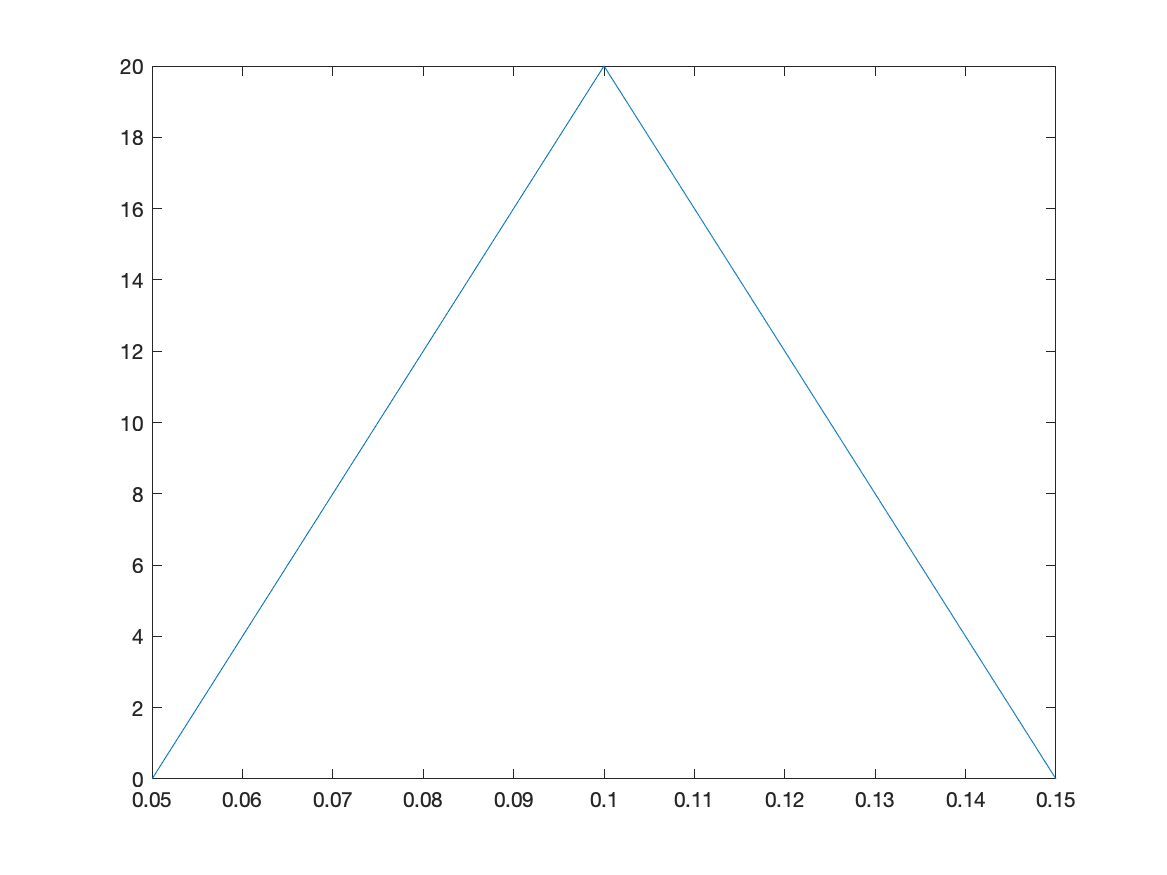}} &  \subfloat[An increasing triangular distribution\label{fg:triangular}]{%
  \includegraphics[width=0.3\textwidth]{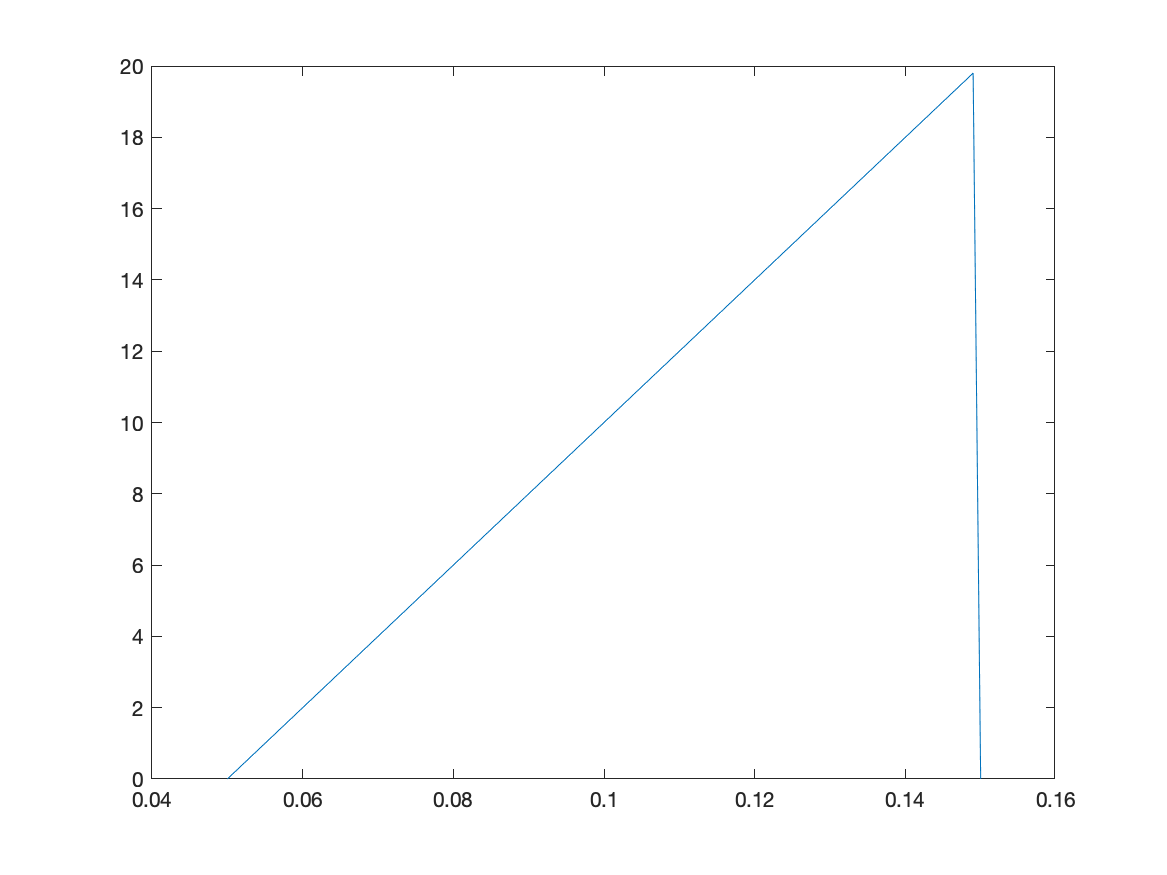}} &\subfloat[A decreasing triangular distribution\label{fg:dec_triangular}]{%
  \includegraphics[width=0.3\textwidth]{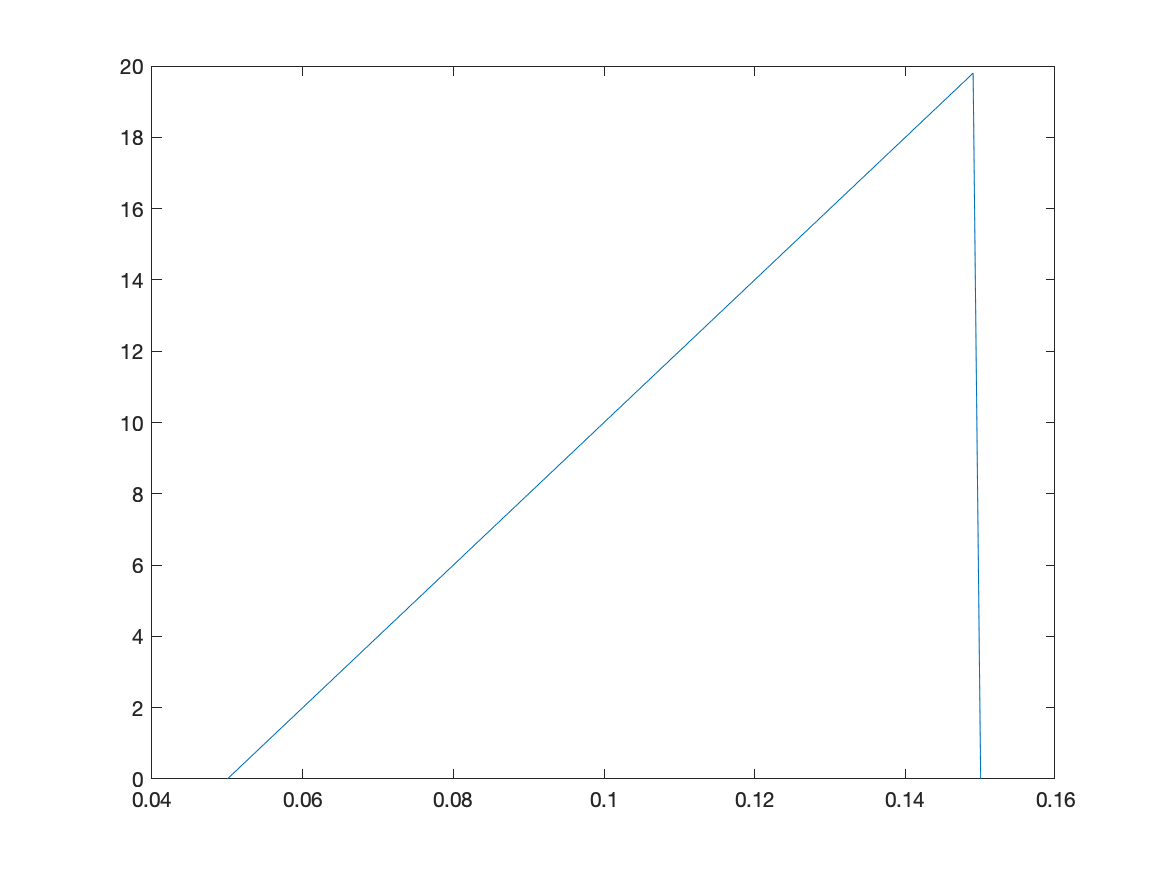}}
  \end{tabular}
  \label{fg:triangular_distributions}
 \end{figure}

  \begin{figure}
\caption{Integral on the left-hand side of \eqref{eq:g_equilibrium_equation} as function of $L_F$ for various values of $r$ and $\beta$ for the symmetric triangular distribution in Figure~\ref{fg:sym_triangular}}
    \centering
    \begin{tabular}{cc}
    \subfloat[Integral vs. $L_F$ varying $\beta$ in the base-case\label{fg:tri_integral-beta}]{%
  \includegraphics[width=0.4\textwidth]{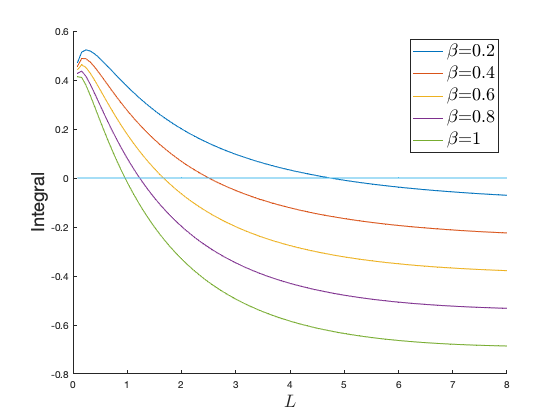}}&
  \subfloat[Integral vs. $L_F$ varying $p$ in the base-case\label{fg:tri_integral-p}]{%
  \includegraphics[width=0.4\textwidth]{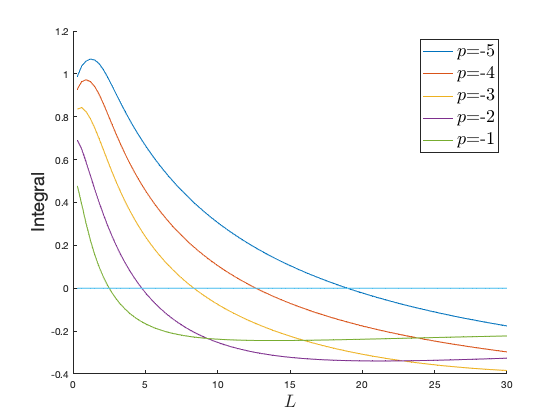}}\\
    \subfloat[Integral vs. $L_F$ varying $q$ in the base-case\label{fg:tri_integral-q}]{%
  \includegraphics[width=0.4\textwidth]{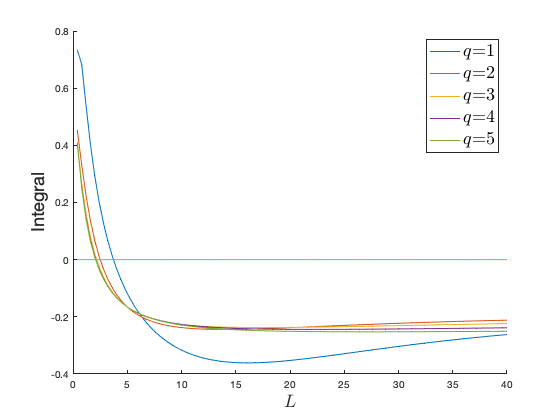}}&
   \subfloat[Integral vs. $L_F$ varying $r$ in the base-case\label{fg:tri_integral-r}]{%
  \includegraphics[width=0.4\textwidth]{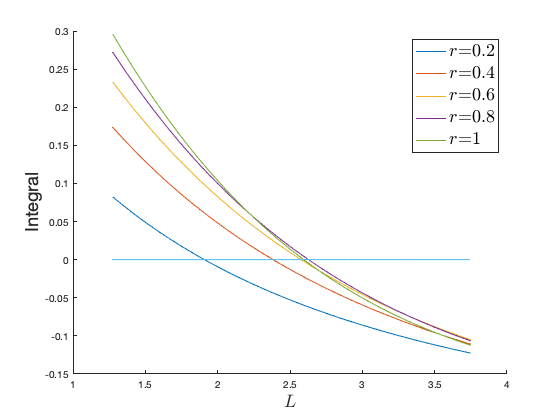}}
  \end{tabular}
  \label{fg:tri_L_var}
 \end{figure}

  \begin{figure}
\caption{Integral on the left-hand side of \eqref{eq:g_equilibrium_equation} as function of $L_F$ for various values of $r$ and $\beta$ for the triangular distribution with increasing density in Figure~\ref{fg:triangular}}
    \centering
    \begin{tabular}{cc}
    \subfloat[Integral vs. $L_F$ varying $\beta$ in the base-case\label{fg:tri1_integral-beta}]{%
  \includegraphics[width=0.4\textwidth]{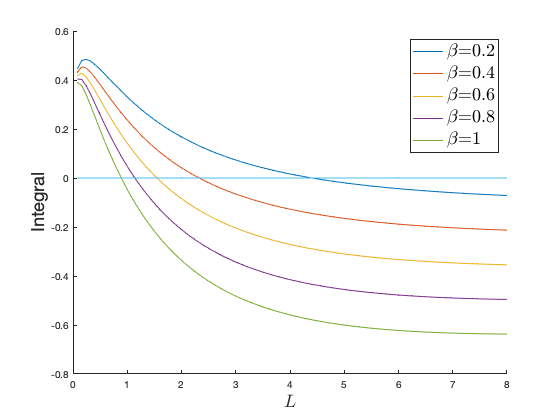}}&
  \subfloat[Integral vs. $L_F$ varying $p$ in the base-case\label{fg:tri1_integral-p}]{%
  \includegraphics[width=0.4\textwidth]{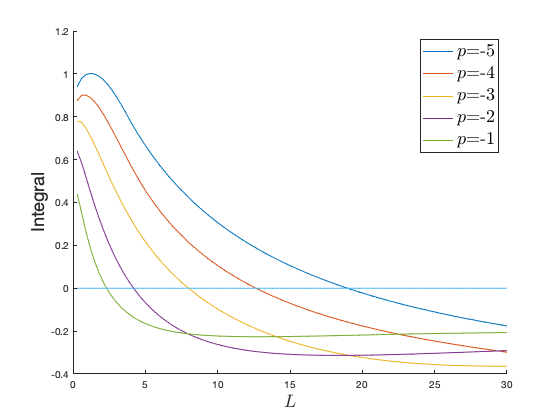}}\\
    \subfloat[Integral vs. $L_F$ varying $q$ in the base-case\label{fg:tri1_integral-q}]{%
  \includegraphics[width=0.4\textwidth]{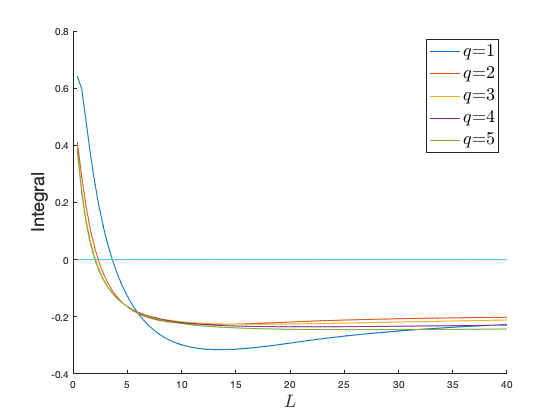}}&
   \subfloat[Integral vs. $L_F$ varying $r$ in the base-case\label{fg:tri1_integral-r}]{%
  \includegraphics[width=0.4\textwidth]{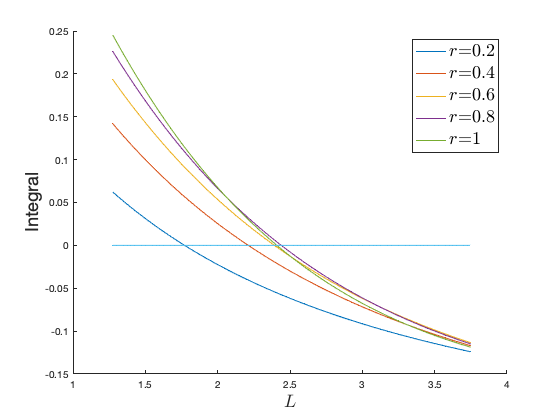}}
  \end{tabular}
  \label{fg:tri1_L_var}
 \end{figure}

  \begin{figure}
\caption{Integral on the left-hand side of \eqref{eq:g_equilibrium_equation} as function of $L_F$ for various values of $r$ and $\beta$ for the triangular distribution with decreasing density in Figure~\ref{fg:dec_triangular}}
    \centering
    \begin{tabular}{cc}
    \subfloat[Integral vs. $L_F$ varying $\beta$ in the base-case\label{fg:tri2_integral-beta}]{%
  \includegraphics[width=0.4\textwidth]{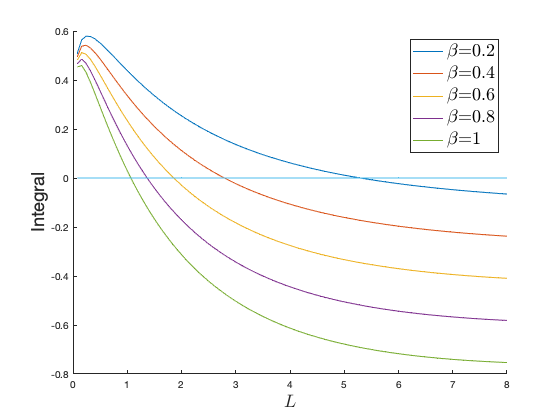}}&
  \subfloat[Integral vs. $L_F$ varying $p$ in the base-case\label{fg:tri2_integral-p}]{%
  \includegraphics[width=0.4\textwidth]{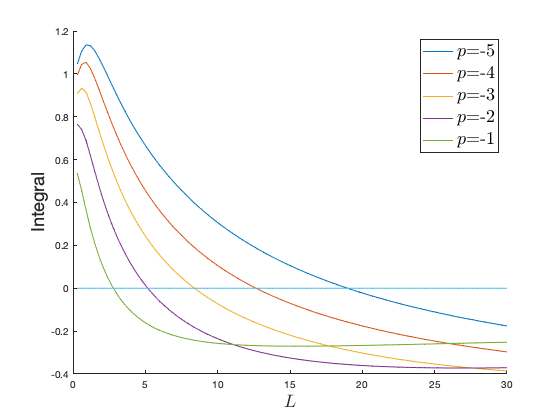}}\\
    \subfloat[Integral vs. $L_F$ varying $q$ in the base-case\label{fg:tri2_integral-q}]{%
  \includegraphics[width=0.4\textwidth]{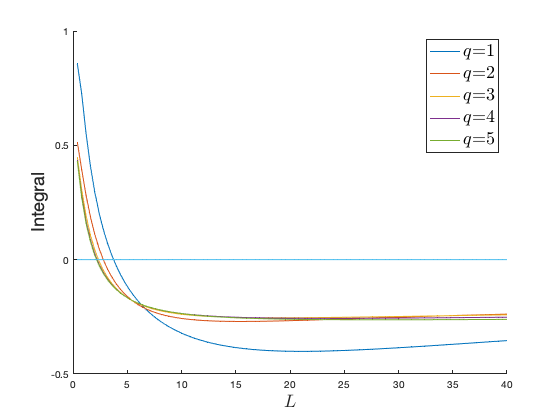}}&
   \subfloat[Integral vs. $L_F$ varying $r$ in the base-case\label{fg:tri2_integral-r}]{%
  \includegraphics[width=0.4\textwidth]{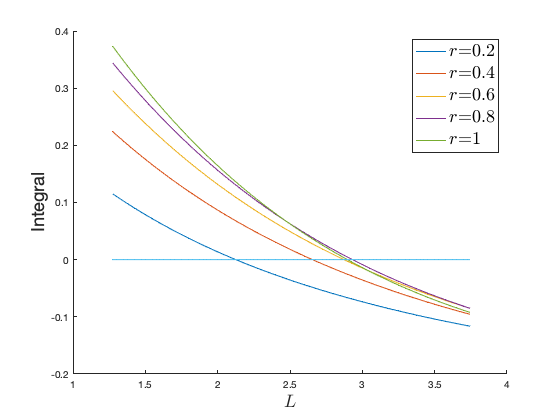}}
  \end{tabular}
  \label{fg:tri2_L_var}
 \end{figure}

 \begin{figure}
\caption{The sensitivity of the equilibrium mean service rate and the number of servers to various parameters for a symmetric triangular distribution}
    \centering
    \begin{tabular}{ccc}
    \subfloat[$p$ vs $\bar{\mu}$ \label{fg:tri_p_vs_mu}]{%
  \includegraphics[width=0.3\textwidth]{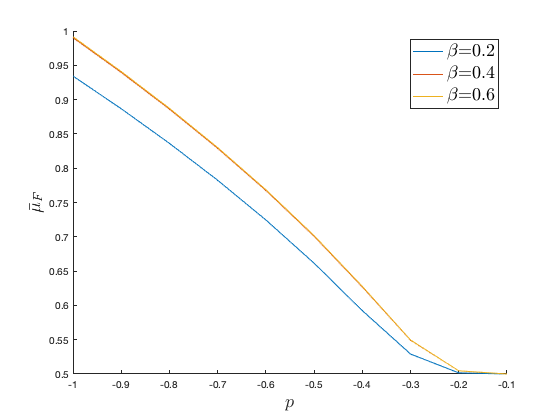}}&
  \subfloat[$q$ vs $\bar{\mu}$ \label{fg:tri_q_vs_mu}]{%
  \includegraphics[width=0.3\textwidth]{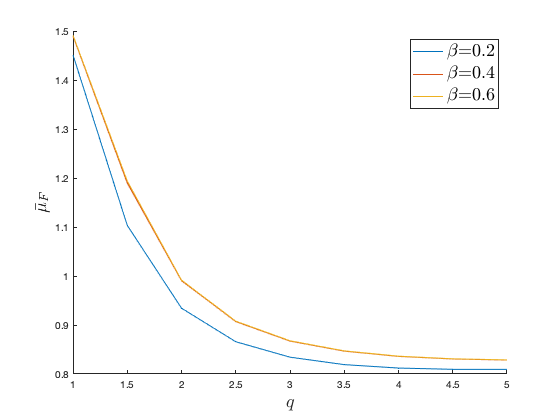}} &
   \subfloat[$r$ vs $\bar{\mu}$ \label{fg:tri_r_vs_mu}]{%
  \includegraphics[width=0.3\textwidth]{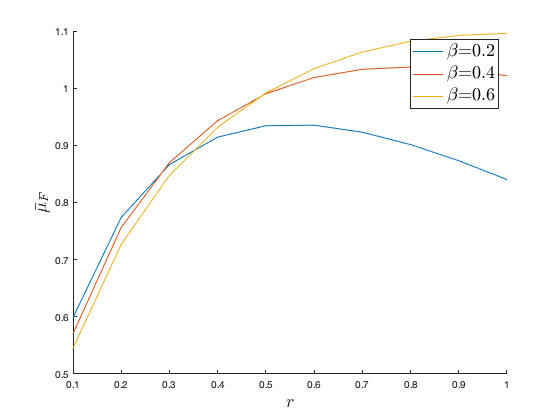}}\\
    \subfloat[$p$ vs $N$ \label{fg:tri_p_vs_N}]{%
  \includegraphics[width=0.3\textwidth]{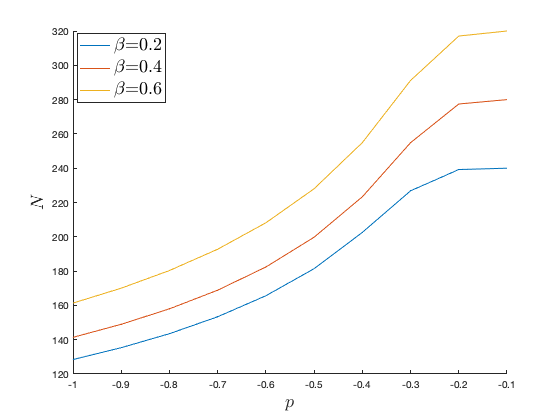}}&
  \subfloat[$q$ vs $N$ \label{fg:tri_q_vs_N}]{%
  \includegraphics[width=0.3\textwidth]{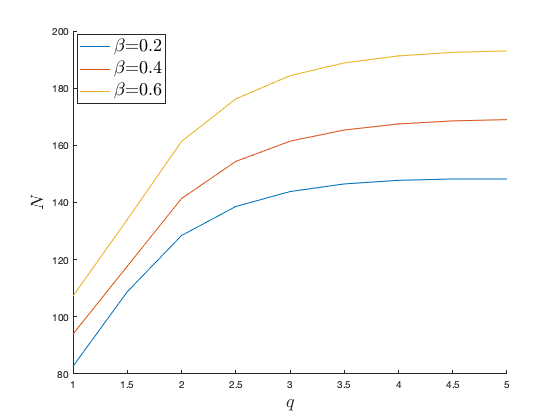}} &
   \subfloat[$r$ vs $N$ \label{fg:tri_r_vs_N}]{%
  \includegraphics[width=0.3\textwidth]{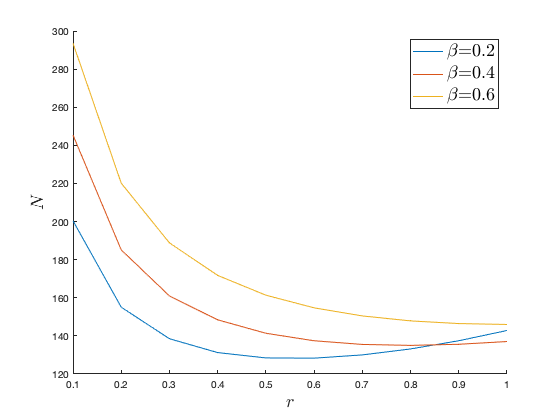}}
  \end{tabular}
  \label{fg:tri_param_vs_N_mu}
 \end{figure}

 \begin{figure}
\caption{The sensitivity of the equilibrium mean service rate and the number of servers to various parameters for a triangular distribution with increasing density}
    \centering
    \begin{tabular}{ccc}
    \subfloat[$p$ vs $\bar{\mu}$ \label{fg:tri1_p_vs_mu}]{%
  \includegraphics[width=0.3\textwidth]{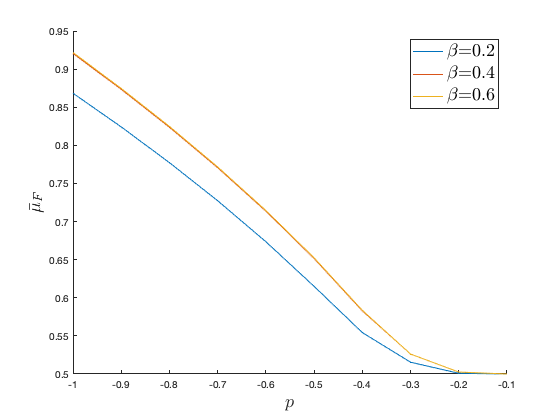}}&
  \subfloat[$q$ vs $\bar{\mu}$ \label{fg:tri1_q_vs_mu}]{%
  \includegraphics[width=0.3\textwidth]{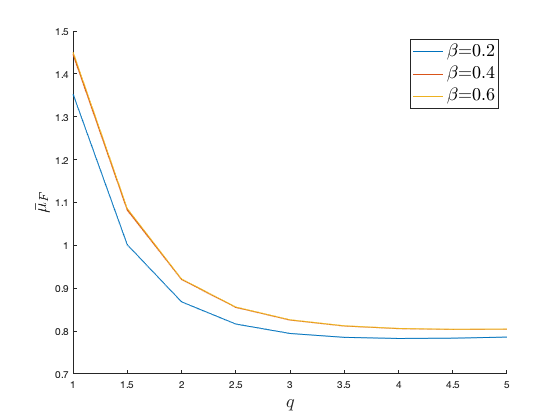}} &
   \subfloat[$r$ vs $\bar{\mu}$ \label{fg:tri1_r_vs_mu}]{%
  \includegraphics[width=0.3\textwidth]{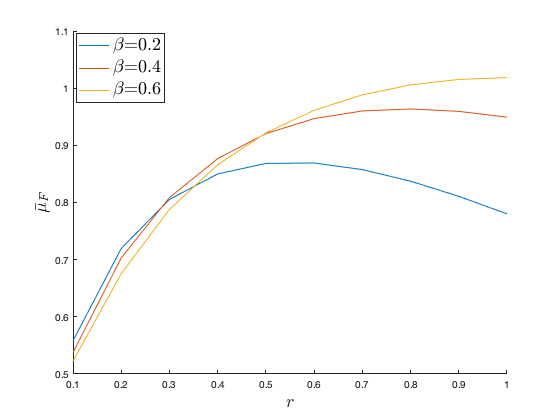}}\\
    \subfloat[$p$ vs $N$ \label{fg:tri1_p_vs_N}]{%
  \includegraphics[width=0.3\textwidth]{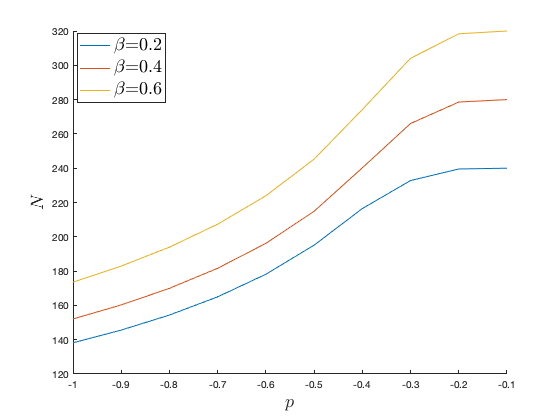}}&
  \subfloat[$q$ vs $N$ \label{fg:tri1_q_vs_N}]{%
  \includegraphics[width=0.3\textwidth]{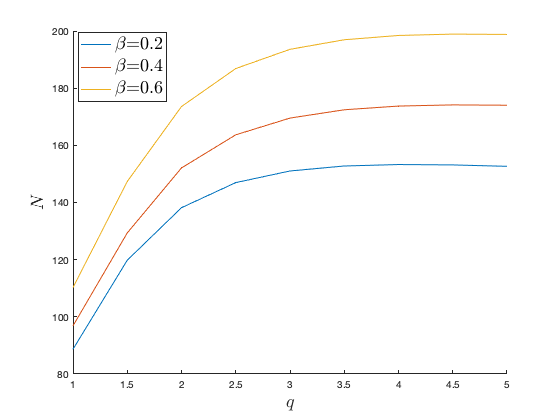}} &
   \subfloat[$r$ vs $N$ \label{fg:tri1_r_vs_N}]{%
  \includegraphics[width=0.3\textwidth]{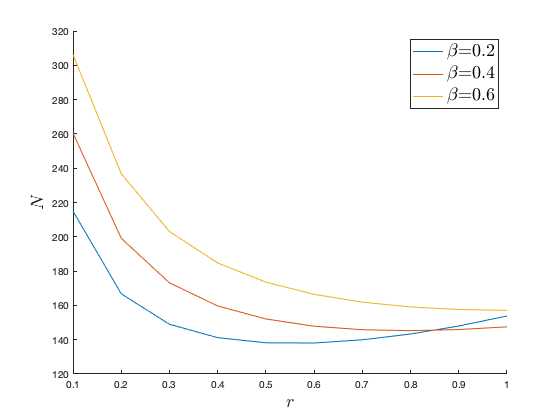}}
  \end{tabular}
  \label{fg:tri1_param_vs_N_mu}
 \end{figure}

  \begin{figure}
\caption{The sensitivity of the equilibrium mean service rate and the number of servers to various parameters for a triangular distribution with decreasing density}
    \centering
    \begin{tabular}{ccc}
    \subfloat[$p$ vs $\bar{\mu}$ \label{fg:tri2_p_vs_mu}]{%
  \includegraphics[width=0.3\textwidth]{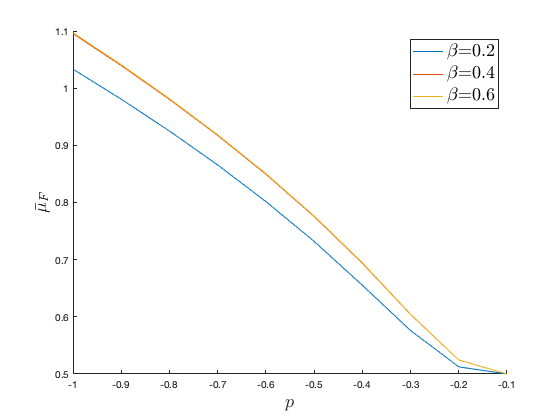}}&
  \subfloat[$q$ vs $\bar{\mu}$ \label{fg:tri2_q_vs_mu}]{%
  \includegraphics[width=0.3\textwidth]{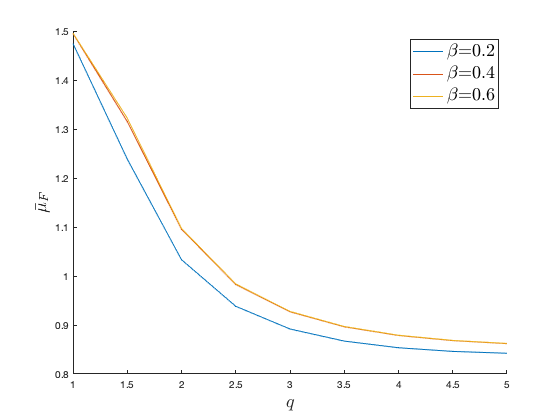}} &
   \subfloat[$r$ vs $\bar{\mu}$ \label{fg:tri2_r_vs_mu}]{%
  \includegraphics[width=0.3\textwidth]{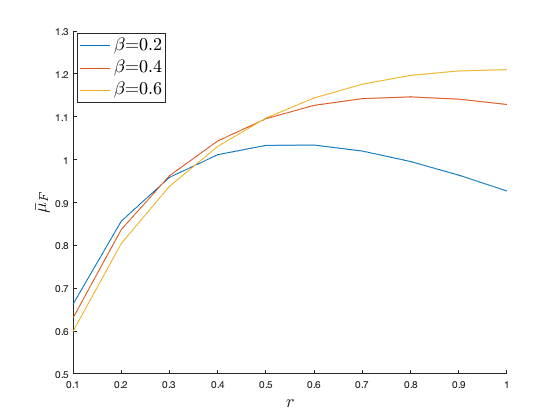}}\\
    \subfloat[$p$ vs $N$ \label{fg:tri2_p_vs_N}]{%
  \includegraphics[width=0.3\textwidth]{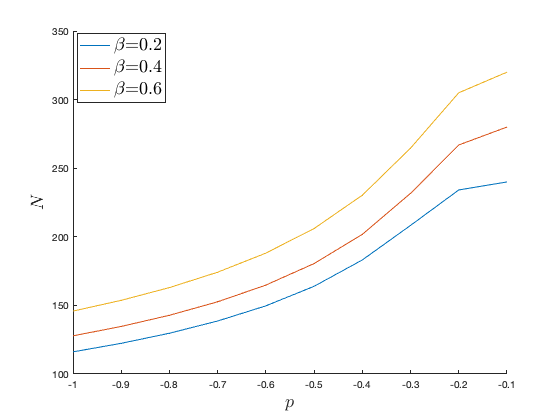}}&
  \subfloat[$q$ vs $N$ \label{fg:tri2_q_vs_N}]{%
  \includegraphics[width=0.3\textwidth]{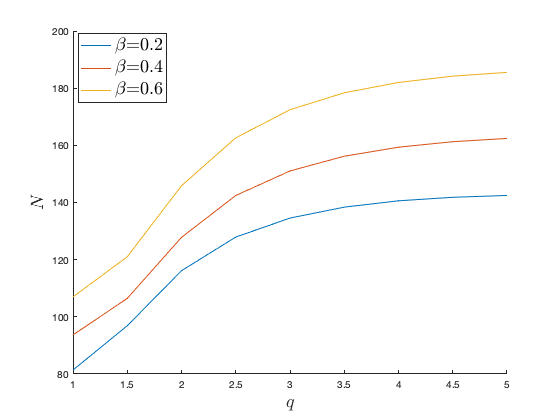}} &
   \subfloat[$r$ vs $N$ \label{fg:tri2_r_vs_N}]{%
  \includegraphics[width=0.3\textwidth]{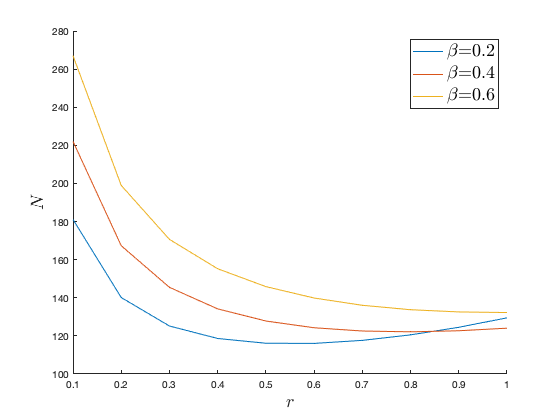}}
  \end{tabular}
  \label{fg:tri2_param_vs_N_mu}
 \end{figure}

  \begin{figure}
\caption{The sensitivity of the equilibrium mean service rate and the number of servers to staffing level $\beta$ for a symmetric triangular distribution}
    \centering
    \begin{tabular}{cc}
    \subfloat[$\beta$ vs $\bar{\mu}$ \label{fg:tri_beta_vs_mu}]{%
  \includegraphics[width=0.3\textwidth]{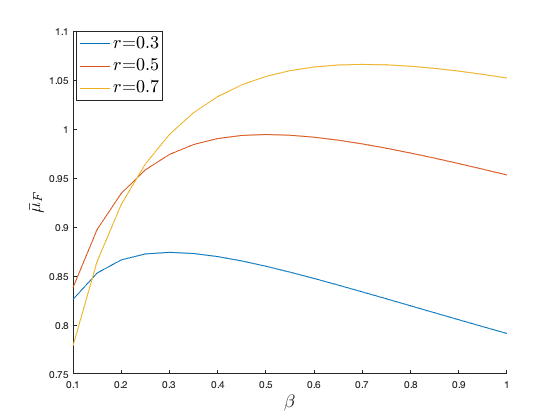}}&
  \subfloat[$\beta$ vs $N$ \label{fg:tri_beta_vs_N}]{%
  \includegraphics[width=0.3\textwidth]{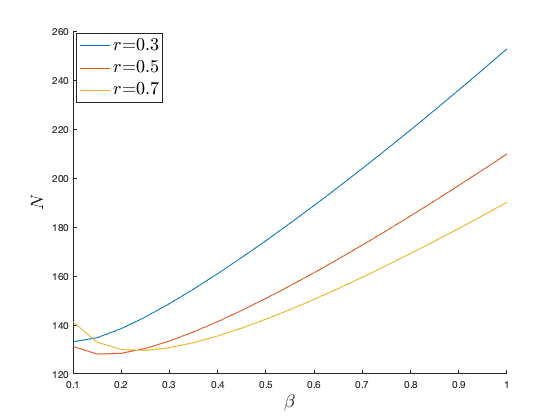}}
  \end{tabular}
  \label{fg:tri_beta_vs_N_mu}
 \end{figure}

  \begin{figure}
\caption{The sensitivity of the equilibrium mean service rate and the number of servers to staffing level $\beta$ for a triangular distribution with increasing density}
    \centering
    \begin{tabular}{cc}
    \subfloat[$\beta$ vs $\bar{\mu}$ \label{fg:tri1_beta_vs_mu}]{%
  \includegraphics[width=0.3\textwidth]{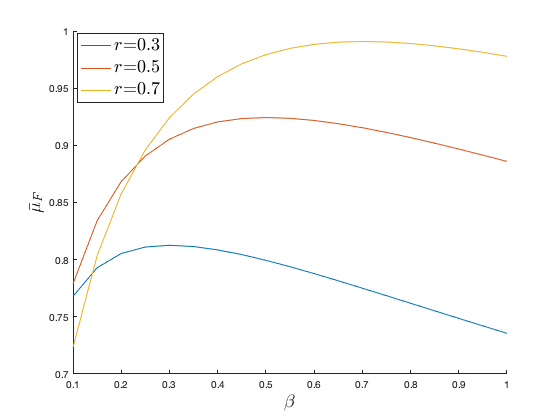}}&
  \subfloat[$\beta$ vs $N$ \label{fg:tri1_beta_vs_N}]{%
  \includegraphics[width=0.3\textwidth]{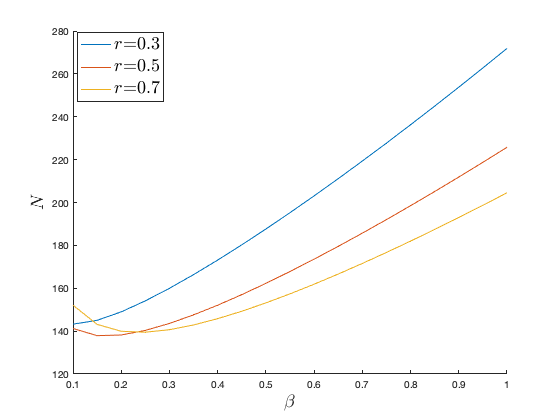}}
  \end{tabular}
  \label{fg:tri1_beta_vs_N_mu}
 \end{figure}

   \begin{figure}
\caption{The sensitivity of the equilibrium mean service rate and the number of servers to staffing level $\beta$ for a triangular distribution with decreasing density}
    \centering
    \begin{tabular}{cc}
    \subfloat[$\beta$ vs $\bar{\mu}$ \label{fg:tri2_beta_vs_mu}]{%
  \includegraphics[width=0.3\textwidth]{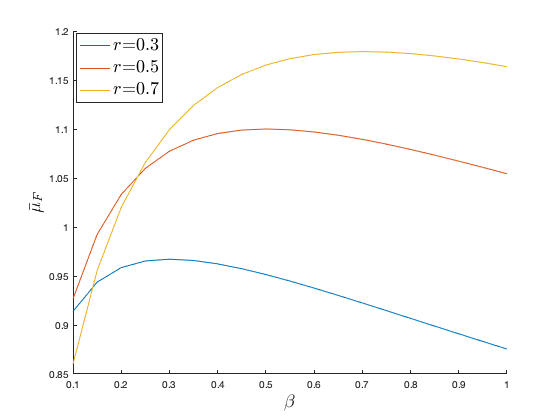}}&
  \subfloat[$\beta$ vs $N$ \label{fg:tri2_beta_vs_N}]{%
  \includegraphics[width=0.3\textwidth]{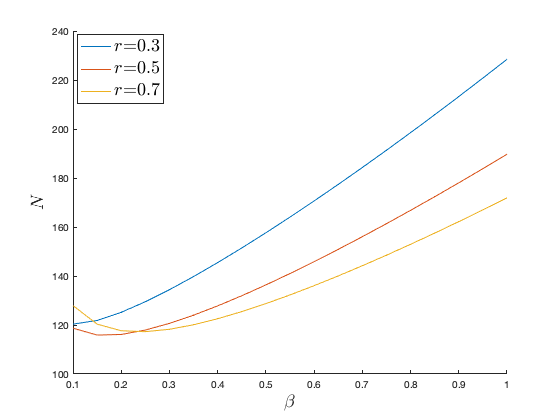}}
  \end{tabular}
  \label{fg:tri2_beta_vs_N_mu}
 \end{figure}

 \begin{figure}
\caption{The equilibrium distributions for different parametric setups with a symmetric triangular distribution}
    \centering
    \begin{tabular}{ccc}
       \subfloat[$p=-1, q =  2, r = 0.5, \beta = 0.4$ \label{fg:tri_p100q200r050}]{%
  \includegraphics[width=0.3\textwidth]{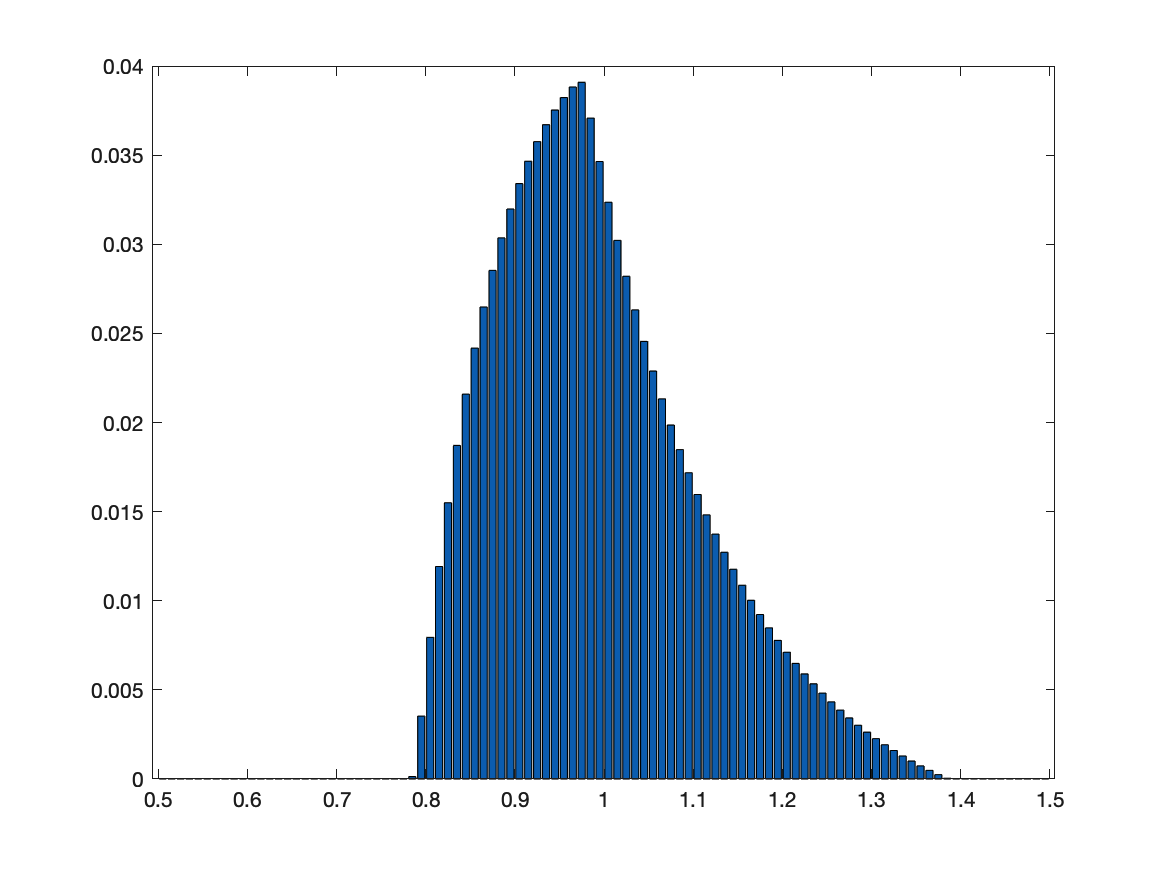}}&
  \subfloat[$p=-0.6, q =  2,r = 0.5, \beta = 0.4$ \label{fg:tri_p060q200r050}]{%
  \includegraphics[width=0.3\textwidth]{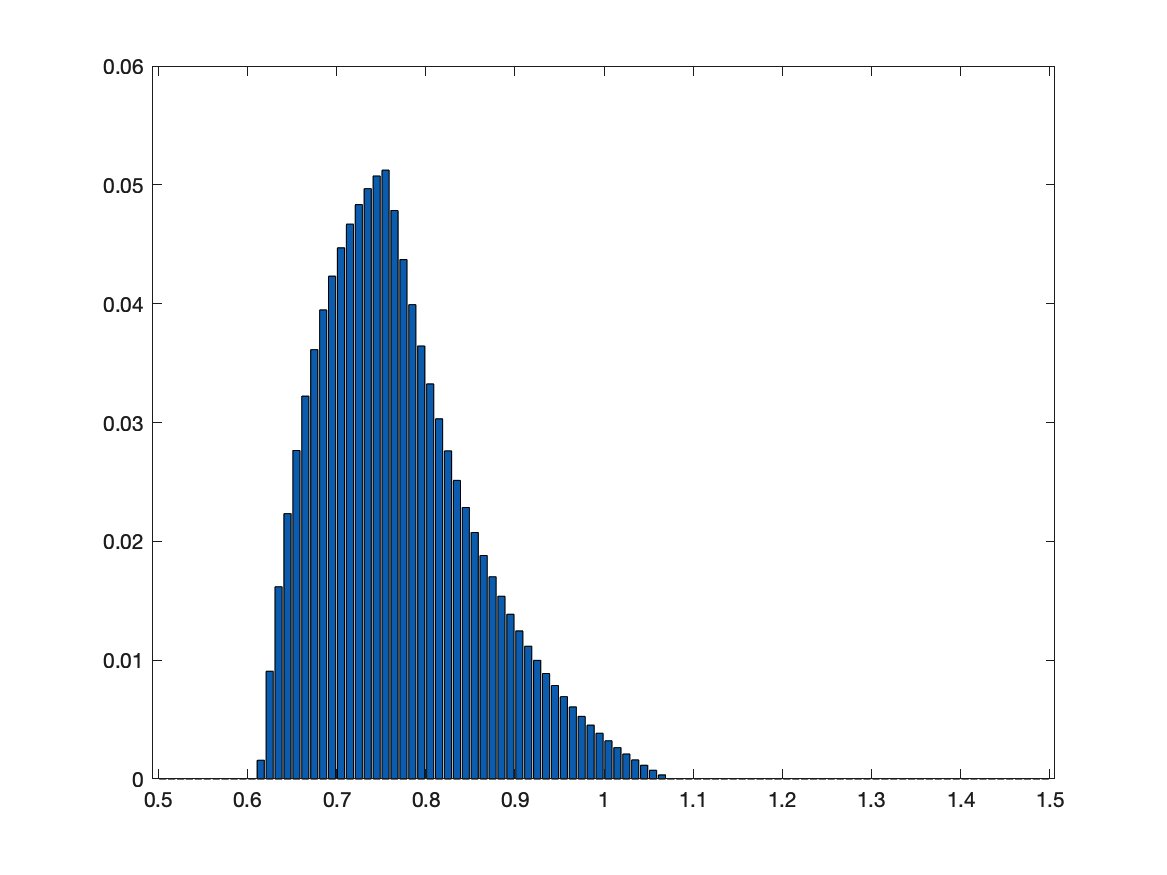}} &
  \subfloat[$p=-0.2, q =  2,r = 0.5, \beta = 0.4$ \label{fg:tri_p020q200r050}]{%
  \includegraphics[width=0.3\textwidth]{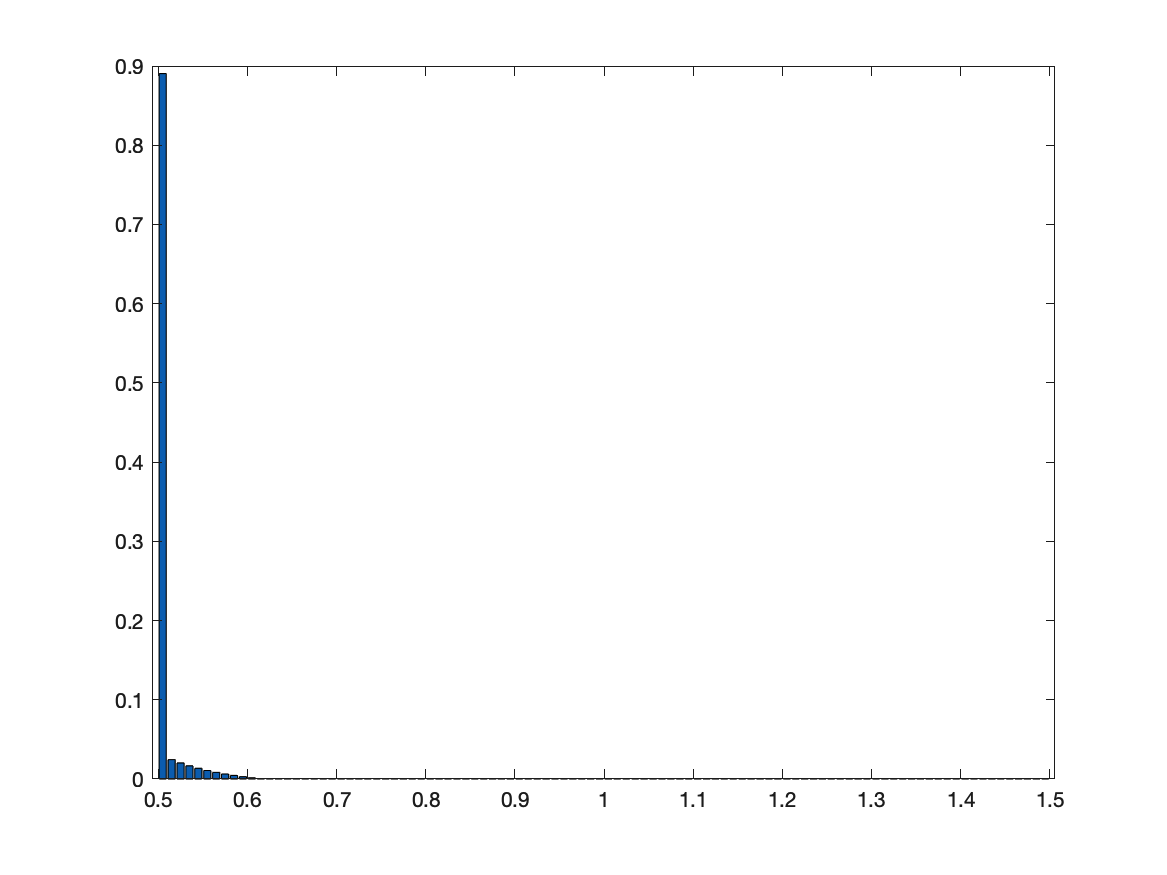}}\\
    \subfloat[$p=-1, q =  3, r = 0.5, \beta = 0.4$\label{fg:tri_p100q300r050}]{%
  \includegraphics[width=0.3\textwidth]{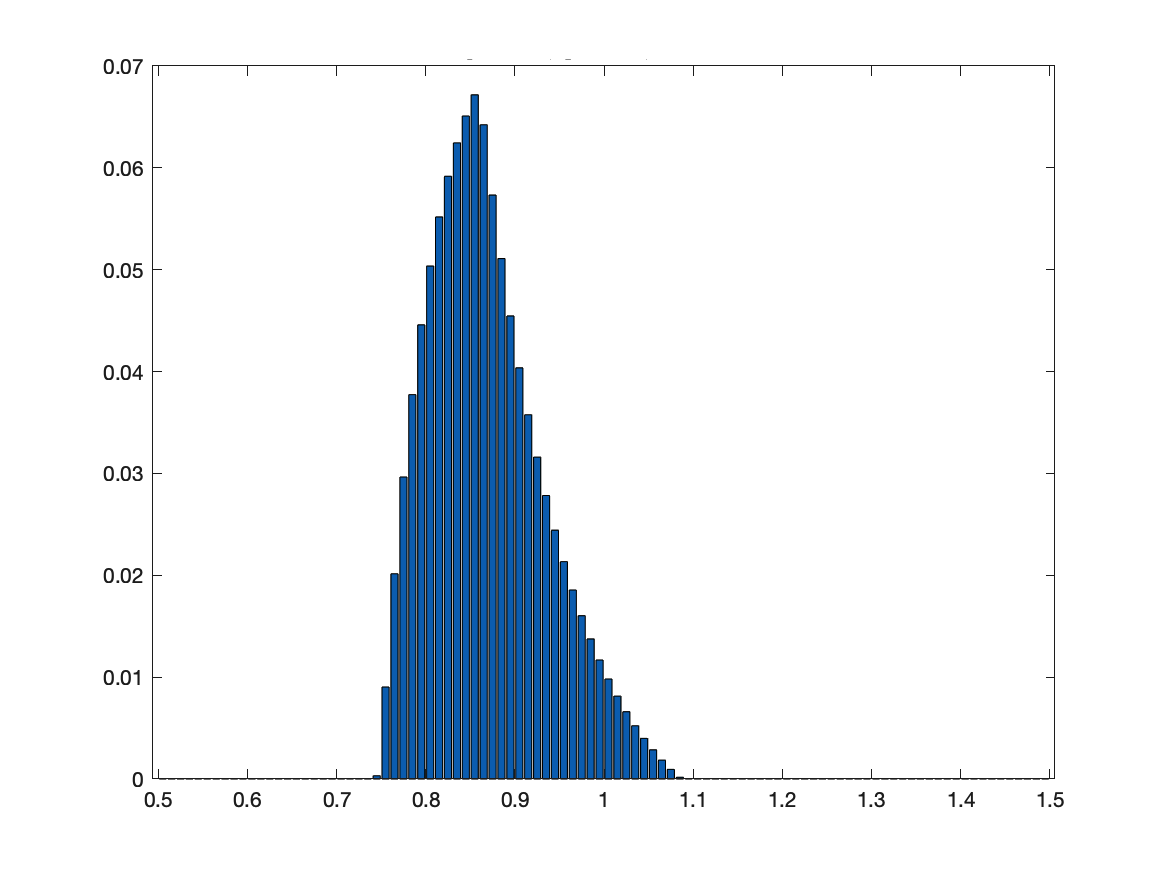}}&
  \subfloat[$p=-1, q =  4, r = 0.5, \beta = 0.4$ \label{fg:tri_p100q400r050}]{%
  \includegraphics[width=0.3\textwidth]{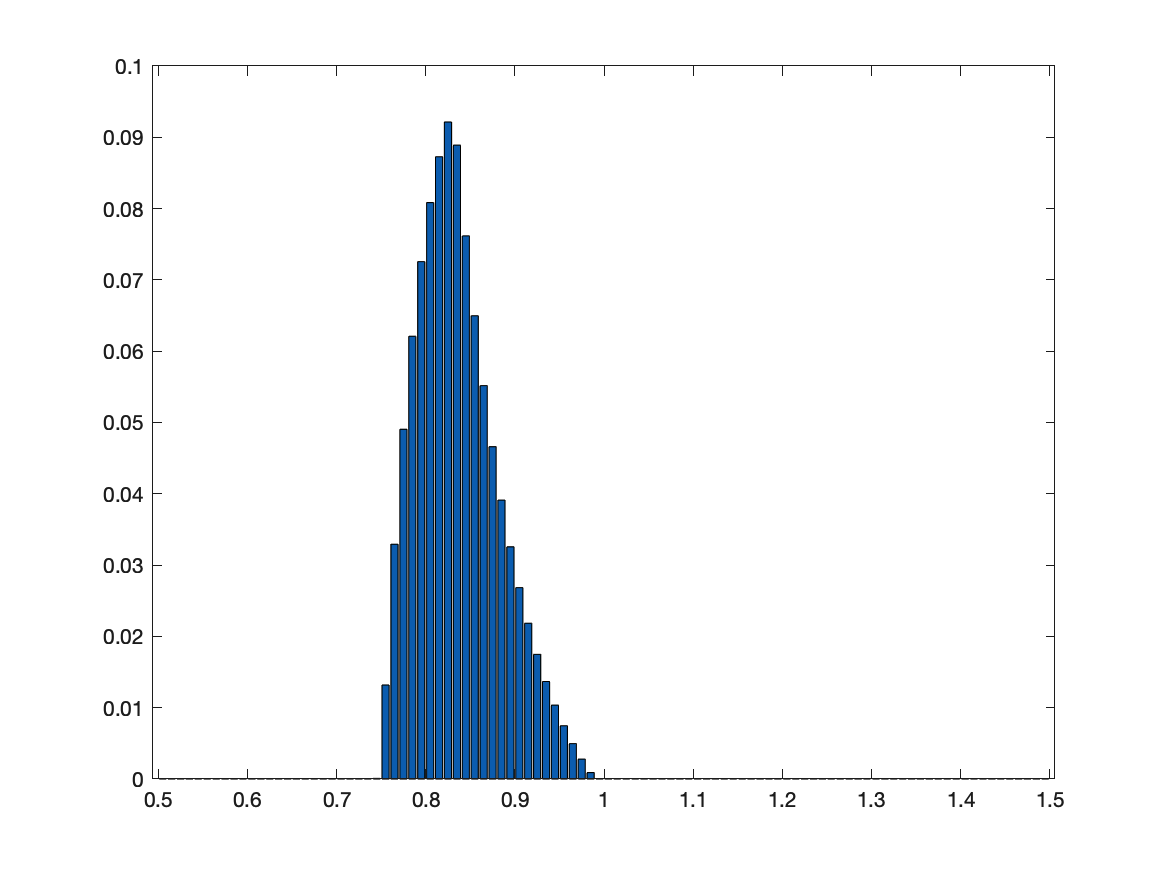}} &
   \subfloat[$p=-1, q =  5, r = 0.5, \beta = 0.4$ \label{fg:tri_p100q500r050}]{%
  \includegraphics[width=0.3\textwidth]{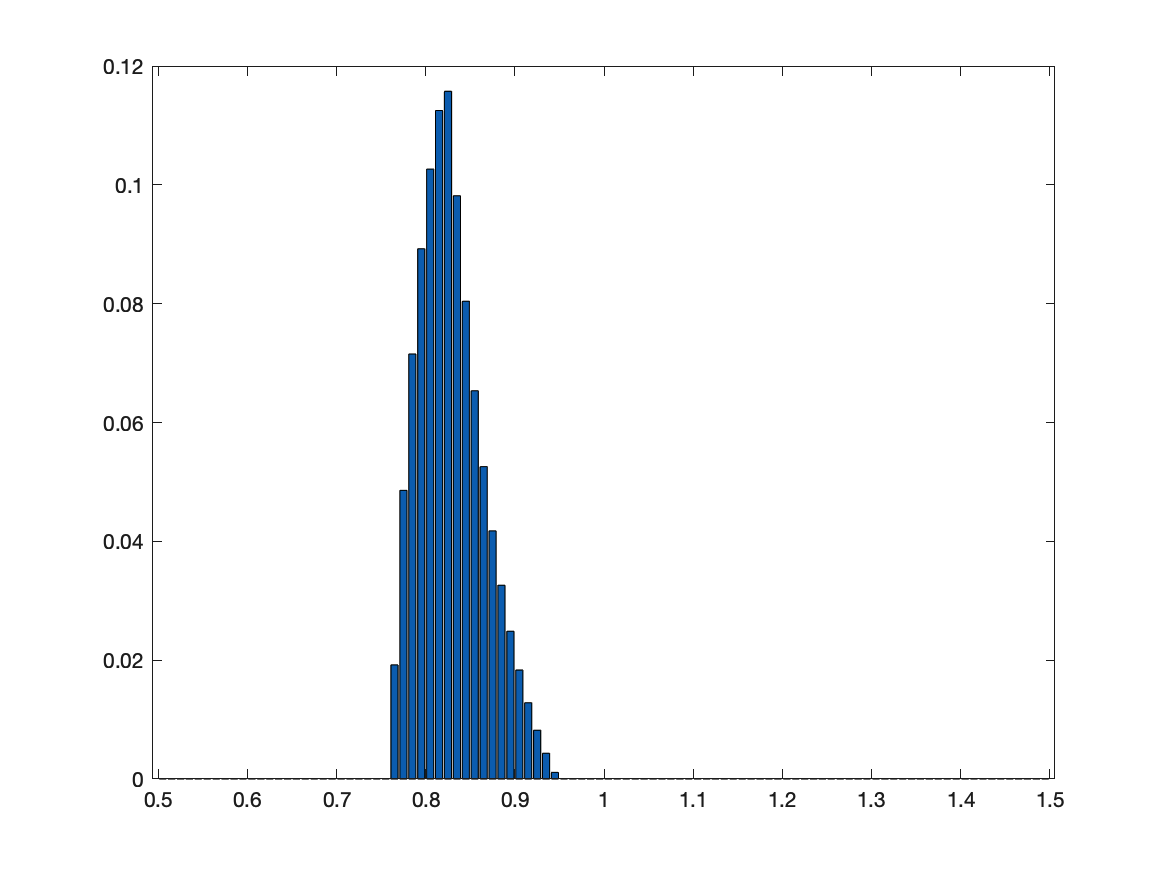}}\\
    \subfloat[$p=-1, q =  2, r = 0.25, \beta = 0.4$ \label{fg:tri_p100q200r025}]{%
  \includegraphics[width=0.3\textwidth]{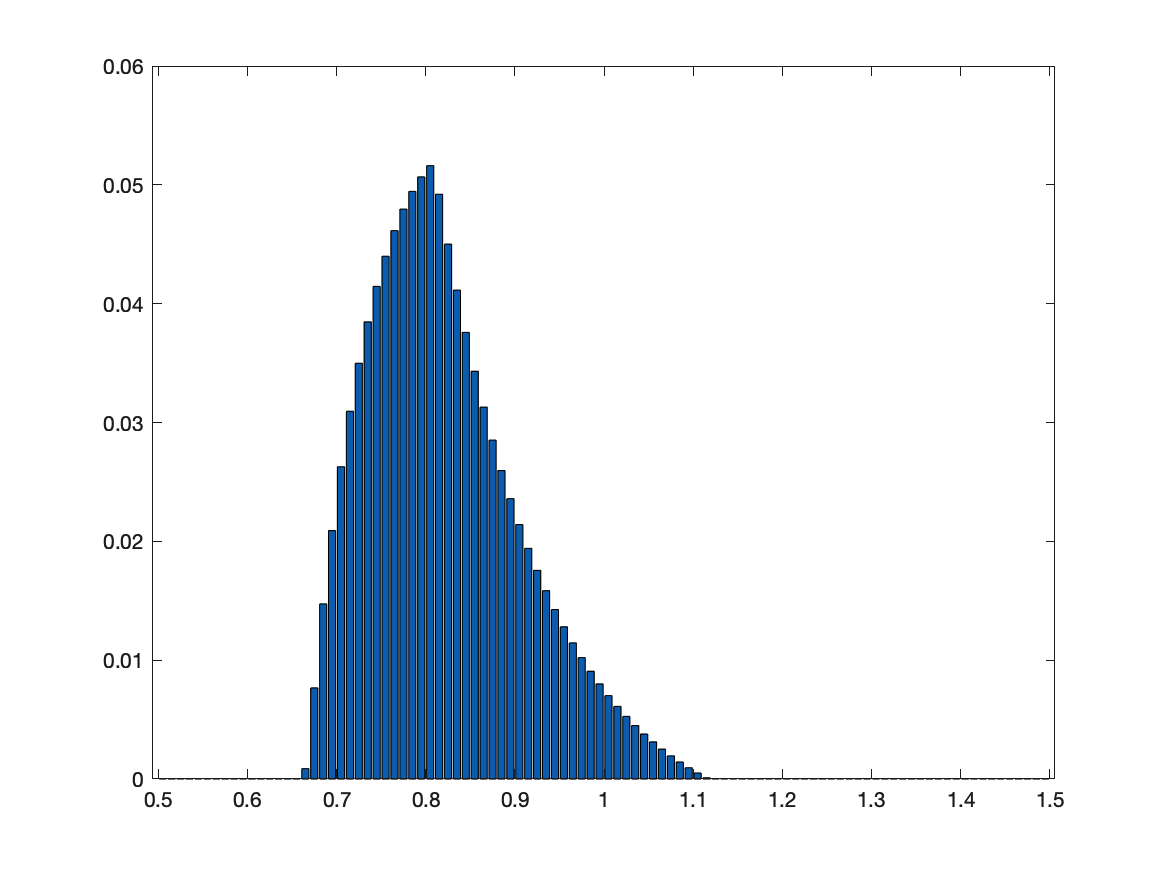}}&
  \subfloat[$p=-1, q =  2, r = 0.75, \beta = 0.4$ \label{fg:tri_p100q200r075}]{%
  \includegraphics[width=0.3\textwidth]{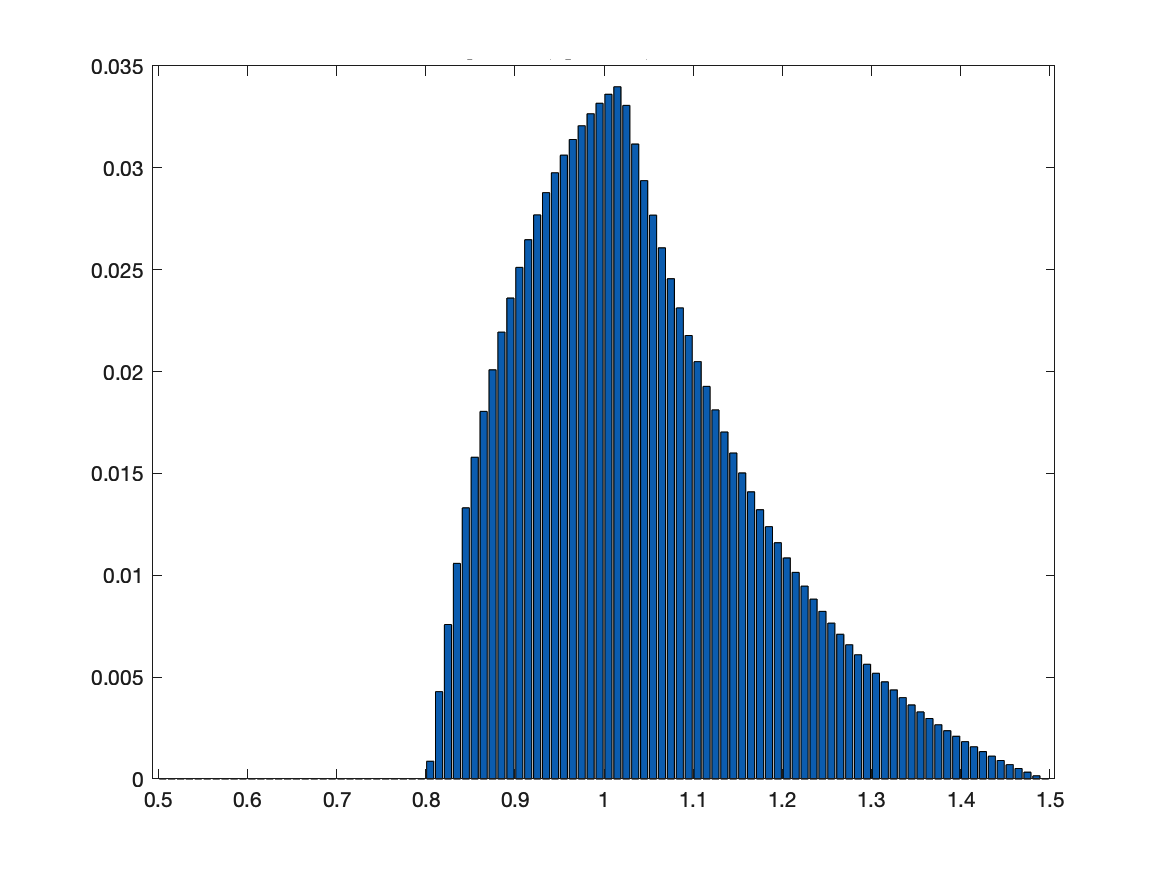}} &
   \subfloat[$p=-1, q =  2, r = 1, \beta = 0.4$ \label{fg:tri_p100q200r100}]{%
  \includegraphics[width=0.3\textwidth]{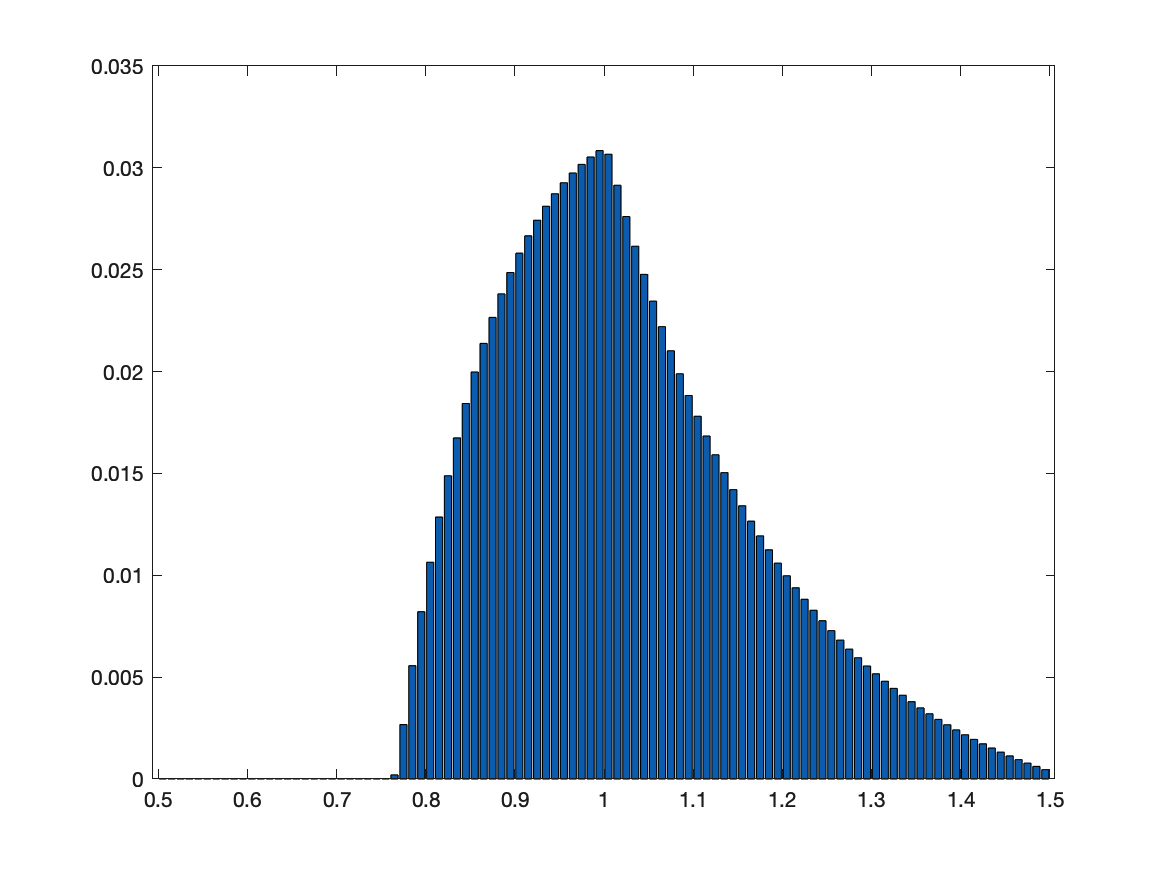}}\\
      \subfloat[$p=-1, q =  2, r = 0.5, \beta = 0.2$ \label{fg:tri_dist_beta02}]{%
  \includegraphics[width=0.3\textwidth]{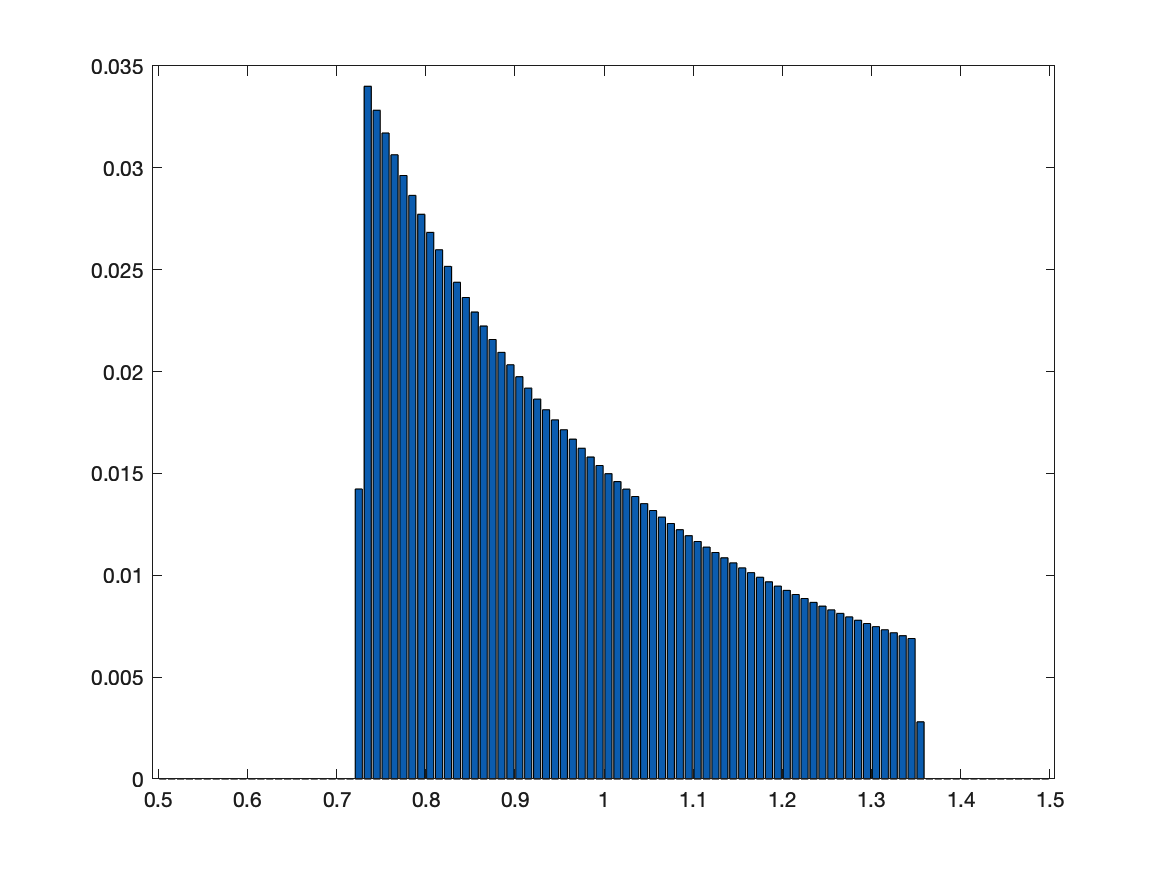}}&
  \subfloat[$p=-1, q =  2, r = 0.5, \beta = 0.6$ \label{fg:tri_dist_beta06}]{%
  \includegraphics[width=0.3\textwidth]{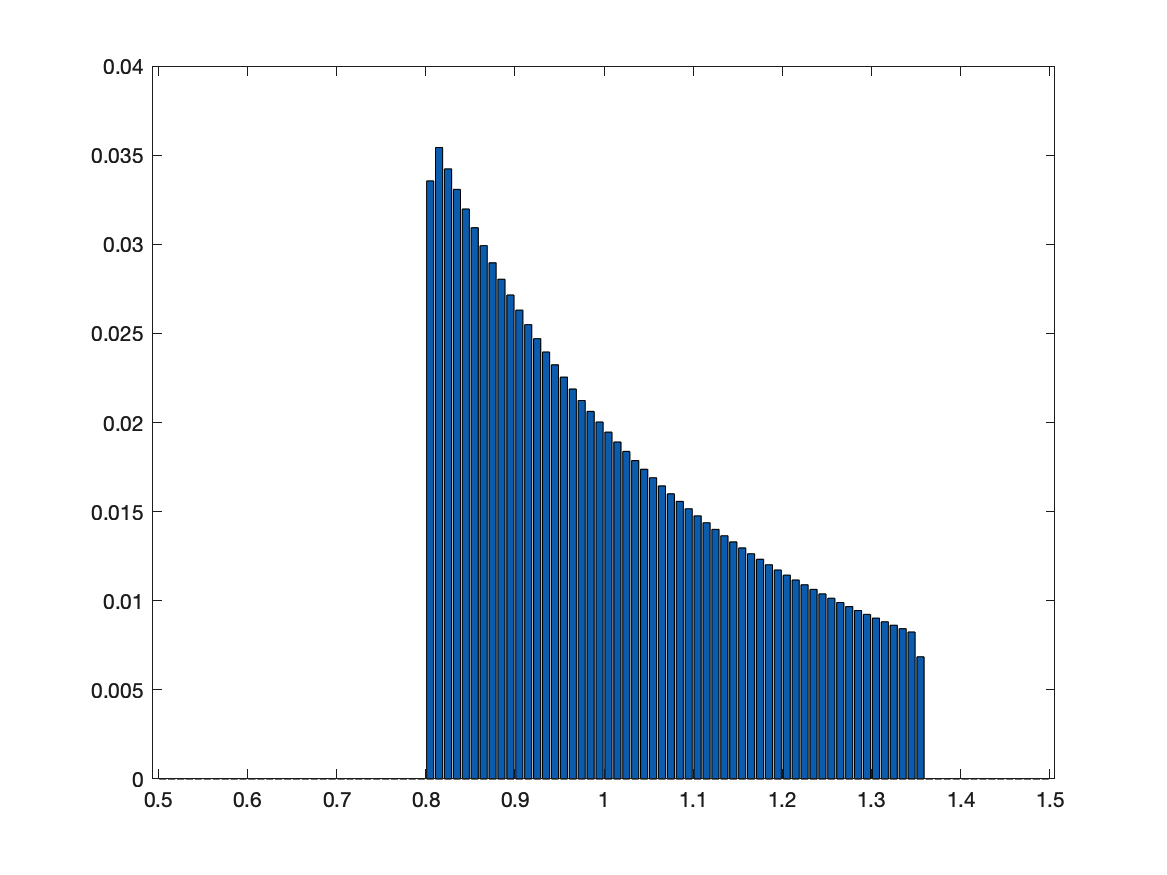}} &
   \subfloat[$p=-1, q =  2, r = 0.5, \beta = 0.8$ \label{fg:tri_dist_beta08}]{%
  \includegraphics[width=0.3\textwidth]{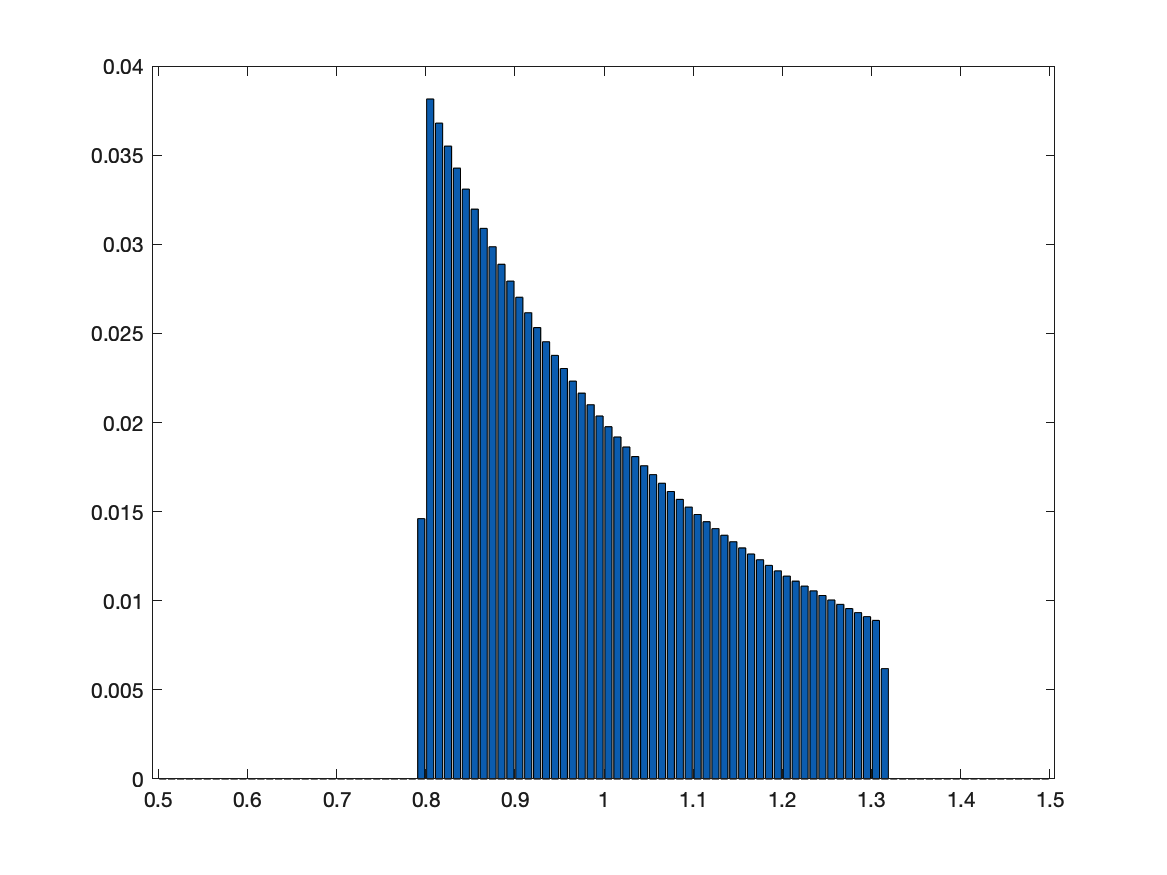}}
  \end{tabular}
  \label{fg:tri_mu_distributions}
 \end{figure}

 \begin{figure}
\caption{The equilibrium distributions for different parametric setups with a  triangular distribution with an increasing density}
    \centering
    \begin{tabular}{ccc}
       \subfloat[$p=-1, q =  2, r = 0.5, \beta = 0.4$ \label{fg:tri1_p100q200r050}]{%
  \includegraphics[width=0.3\textwidth]{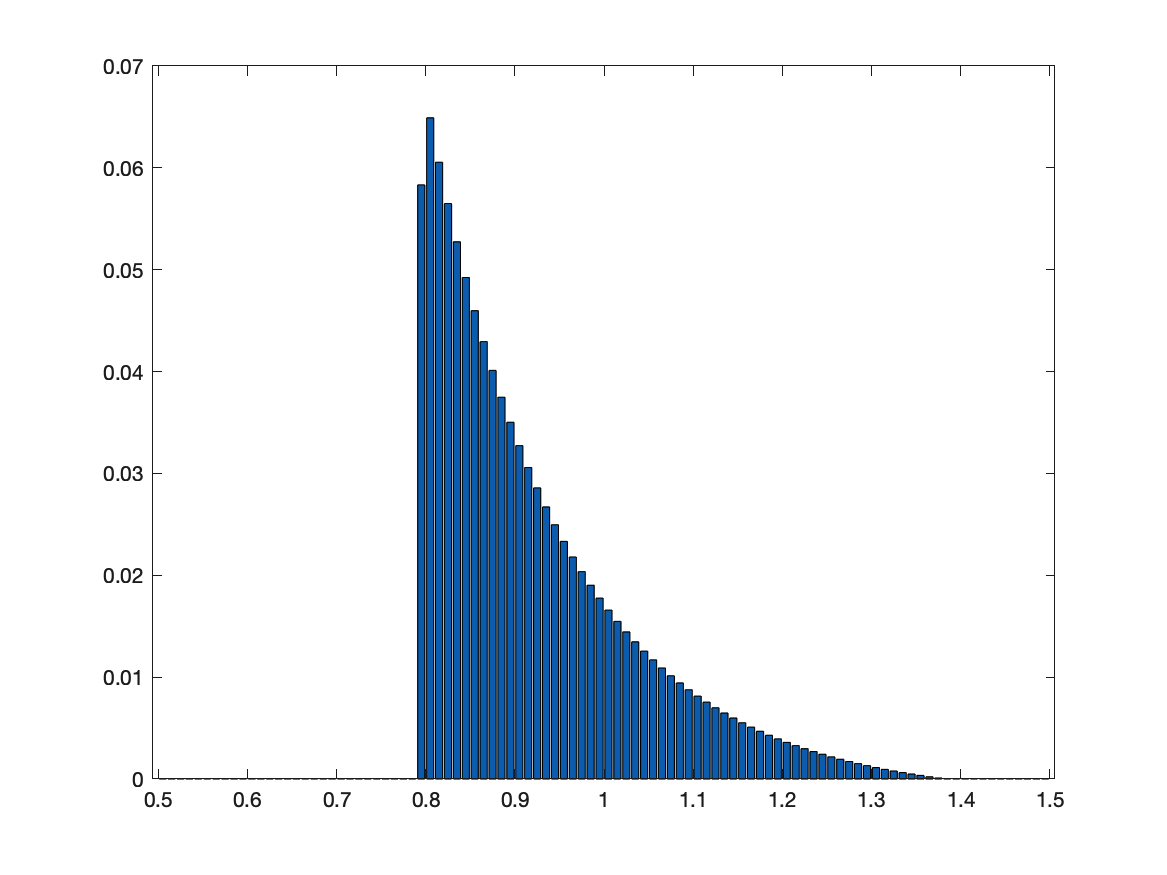}}&
  \subfloat[$p=-0.6, q =  2,r = 0.5, \beta = 0.4$ \label{fg:tri1_p060q200r050}]{%
  \includegraphics[width=0.3\textwidth]{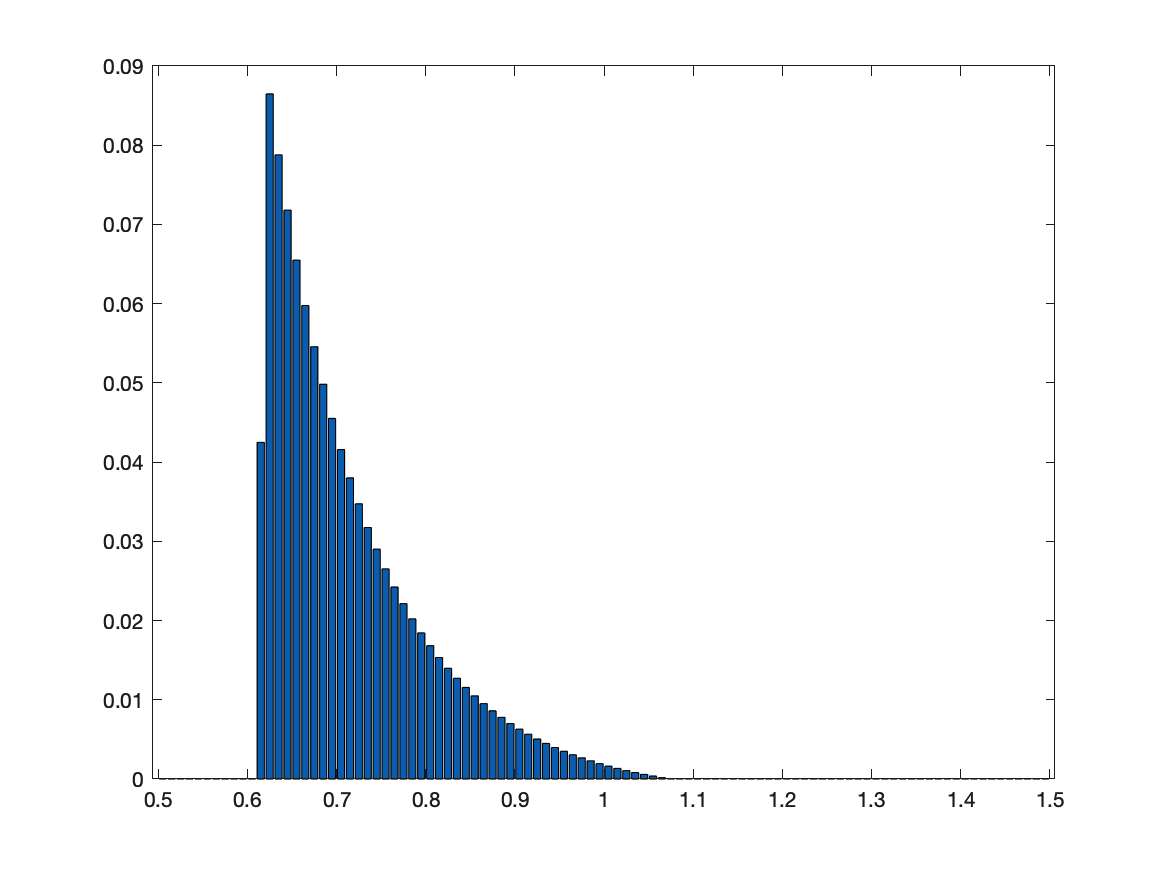}} &
  \subfloat[$p=-0.2, q =  2,r = 0.5, \beta = 0.4$ \label{fg:tri1_p020q200r050}]{%
  \includegraphics[width=0.3\textwidth]{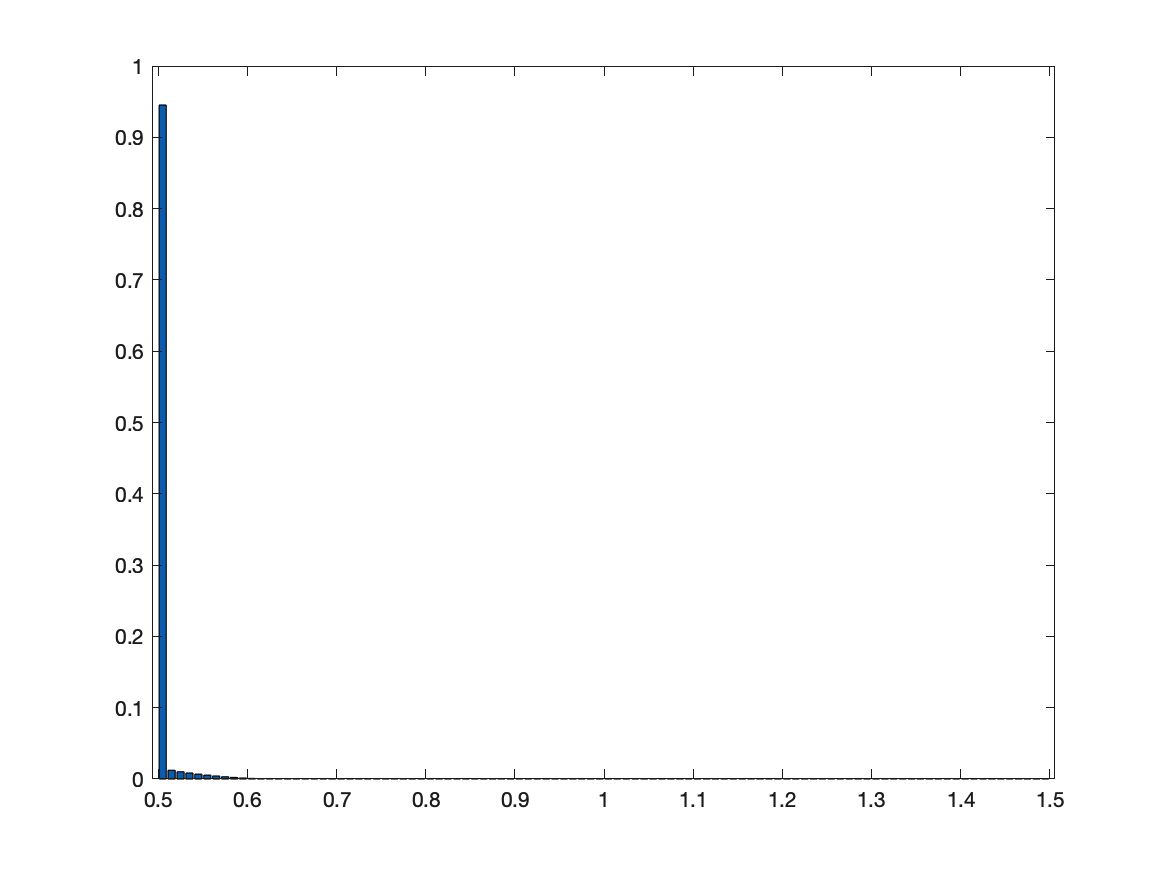}}\\
    \subfloat[$p=-1, q =  3, r = 0.5, \beta = 0.4$\label{fg:tri1_p100q300r050}]{%
  \includegraphics[width=0.3\textwidth]{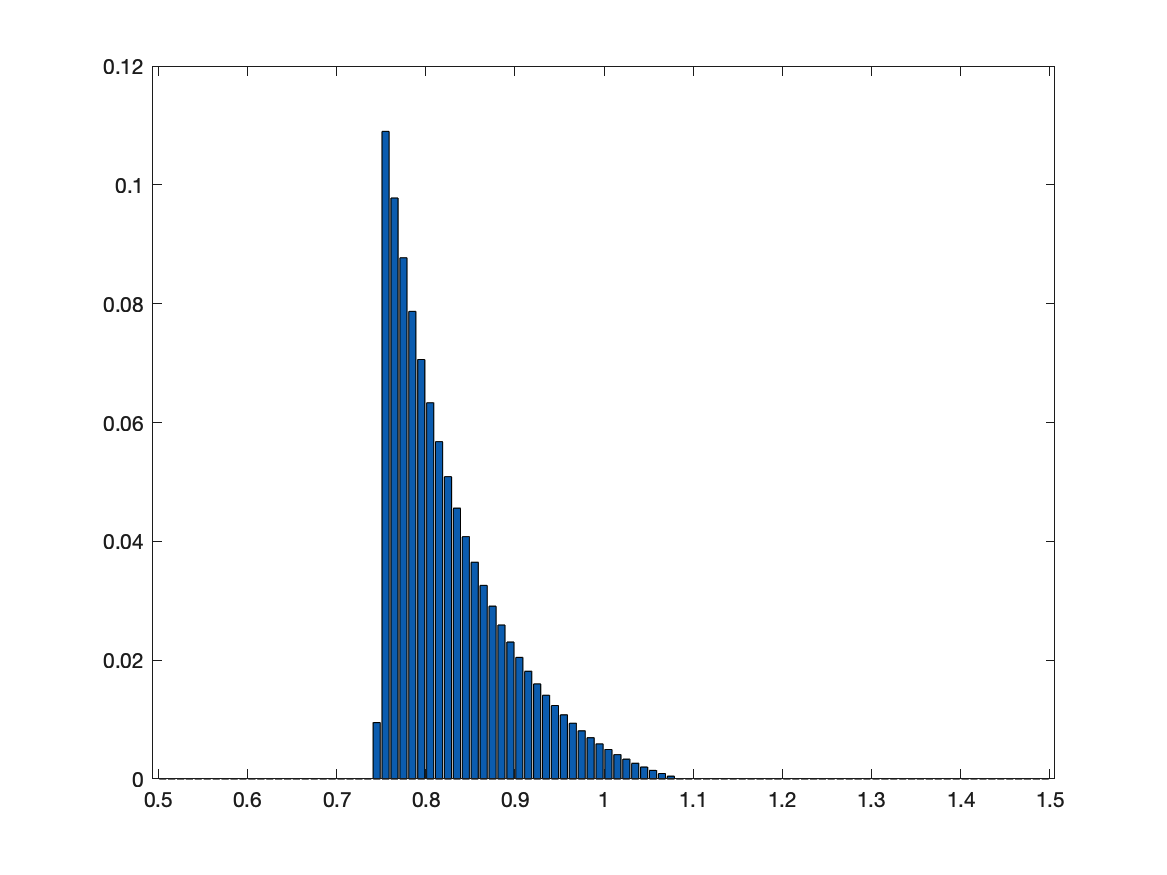}}&
  \subfloat[$p=-1, q =  4, r = 0.5, \beta = 0.4$ \label{fg:tri1_p100q400r050}]{%
  \includegraphics[width=0.3\textwidth]{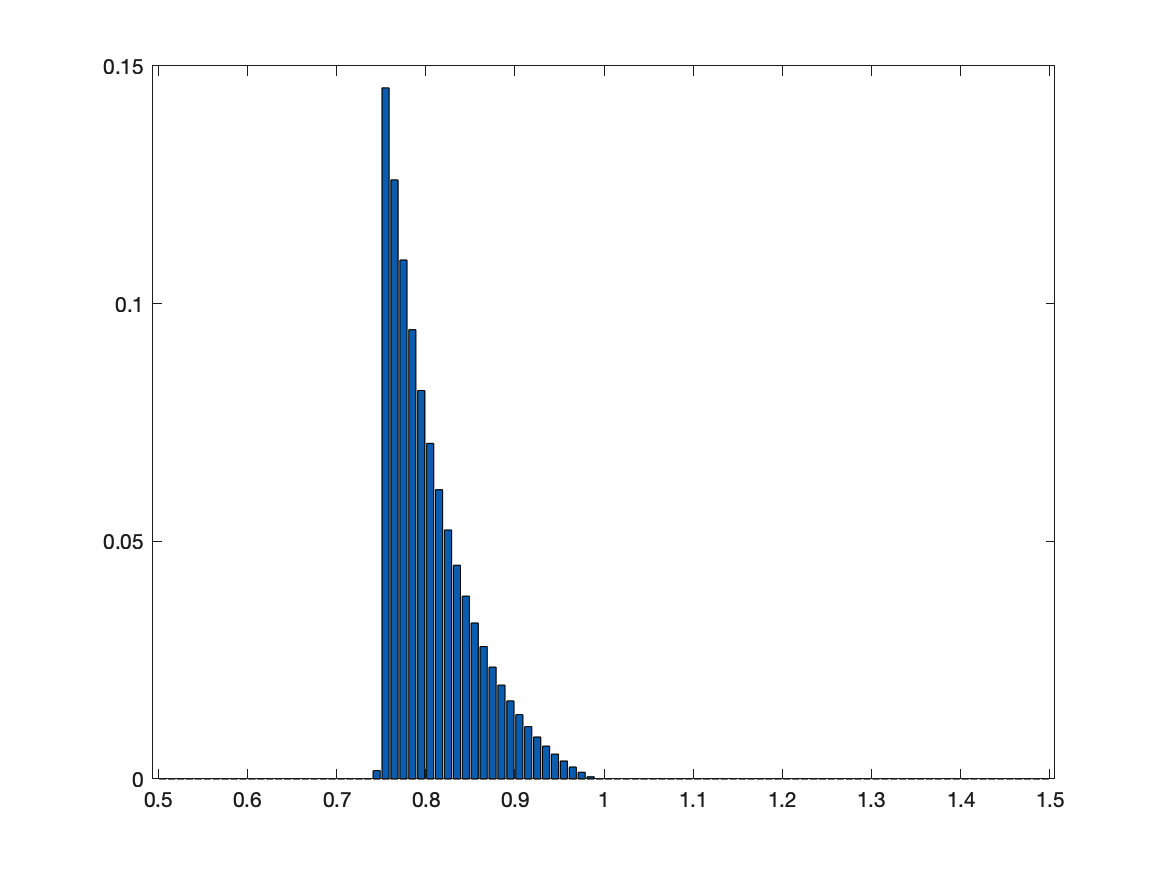}} &
   \subfloat[$p=-1, q =  5, r = 0.5, \beta = 0.4$ \label{fg:tri1_p100q500r050}]{%
  \includegraphics[width=0.3\textwidth]{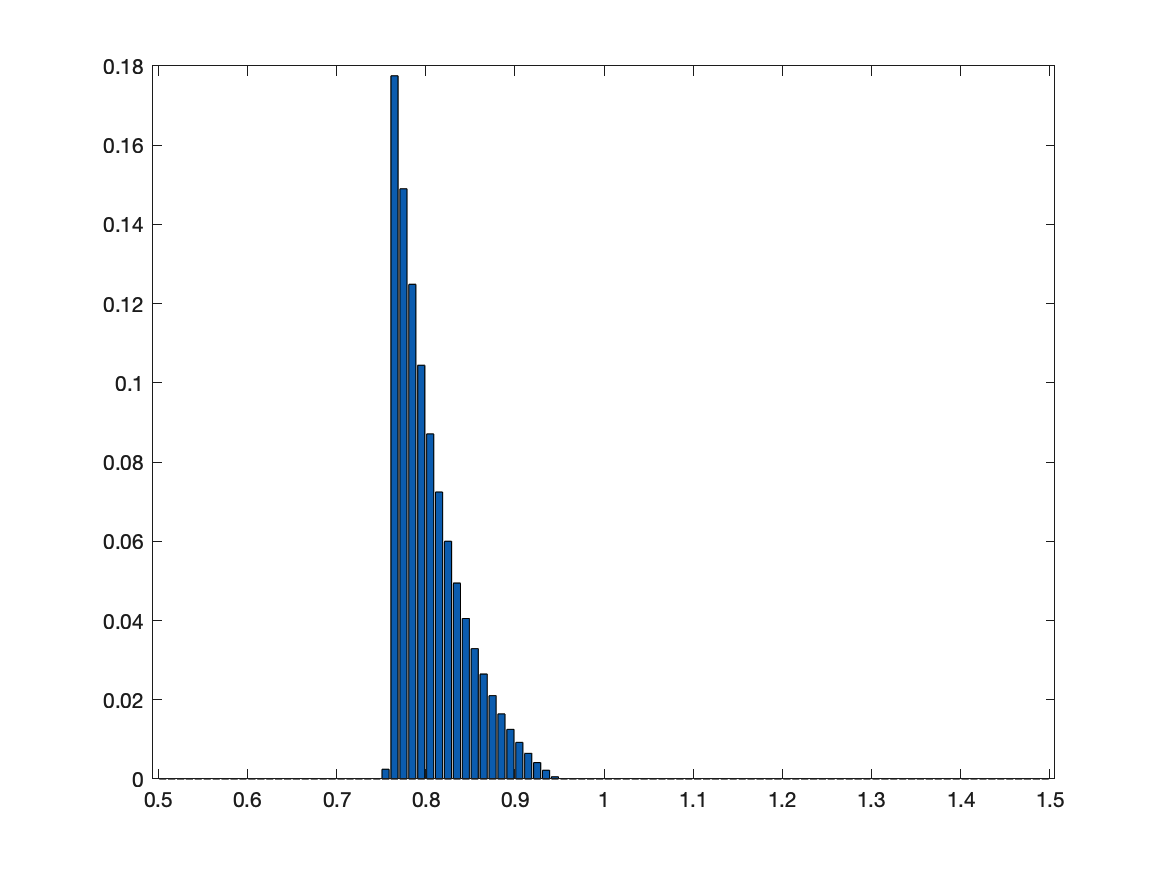}}\\
    \subfloat[$p=-1, q =  2, r = 0.25, \beta = 0.4$ \label{fg:tri1_p100q200r025}]{%
  \includegraphics[width=0.3\textwidth]{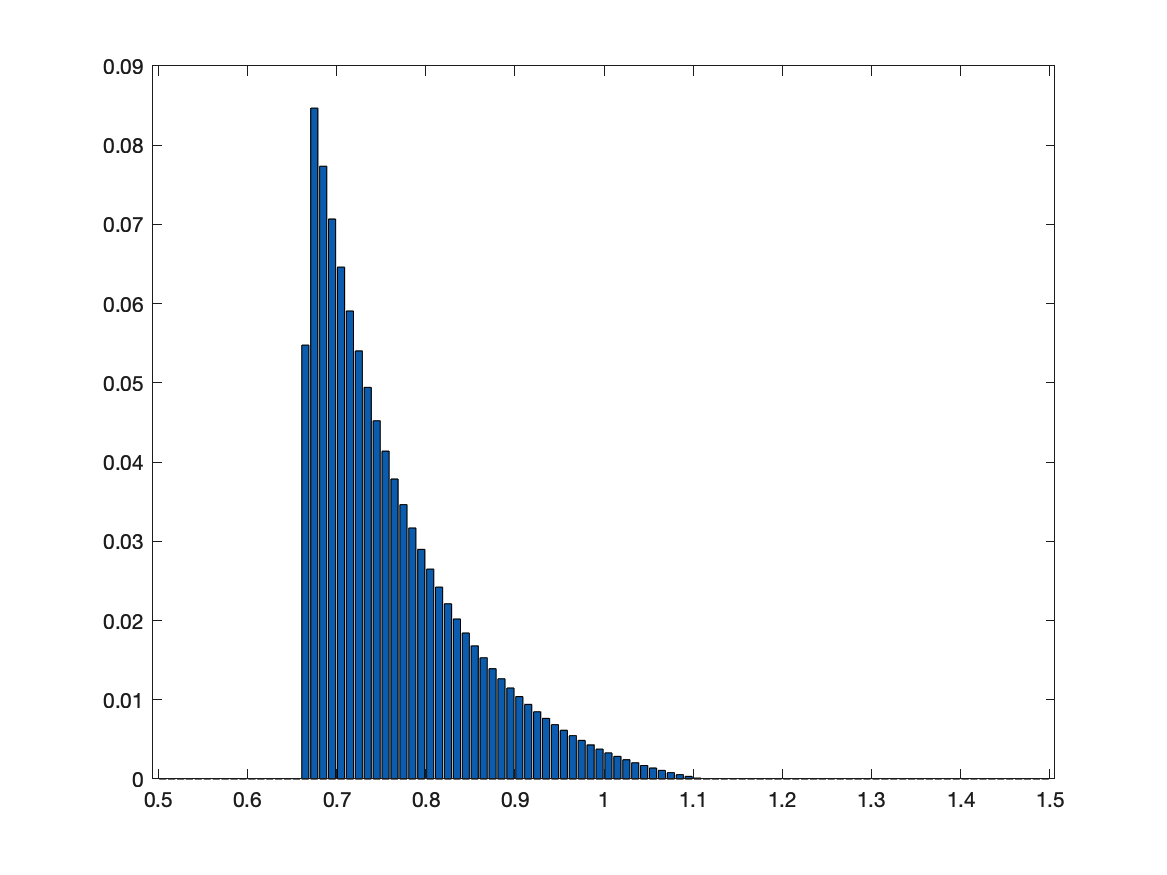}}&
  \subfloat[$p=-1, q =  2, r = 0.75, \beta = 0.4$ \label{fg:tri1_p100q200r075}]{%
  \includegraphics[width=0.3\textwidth]{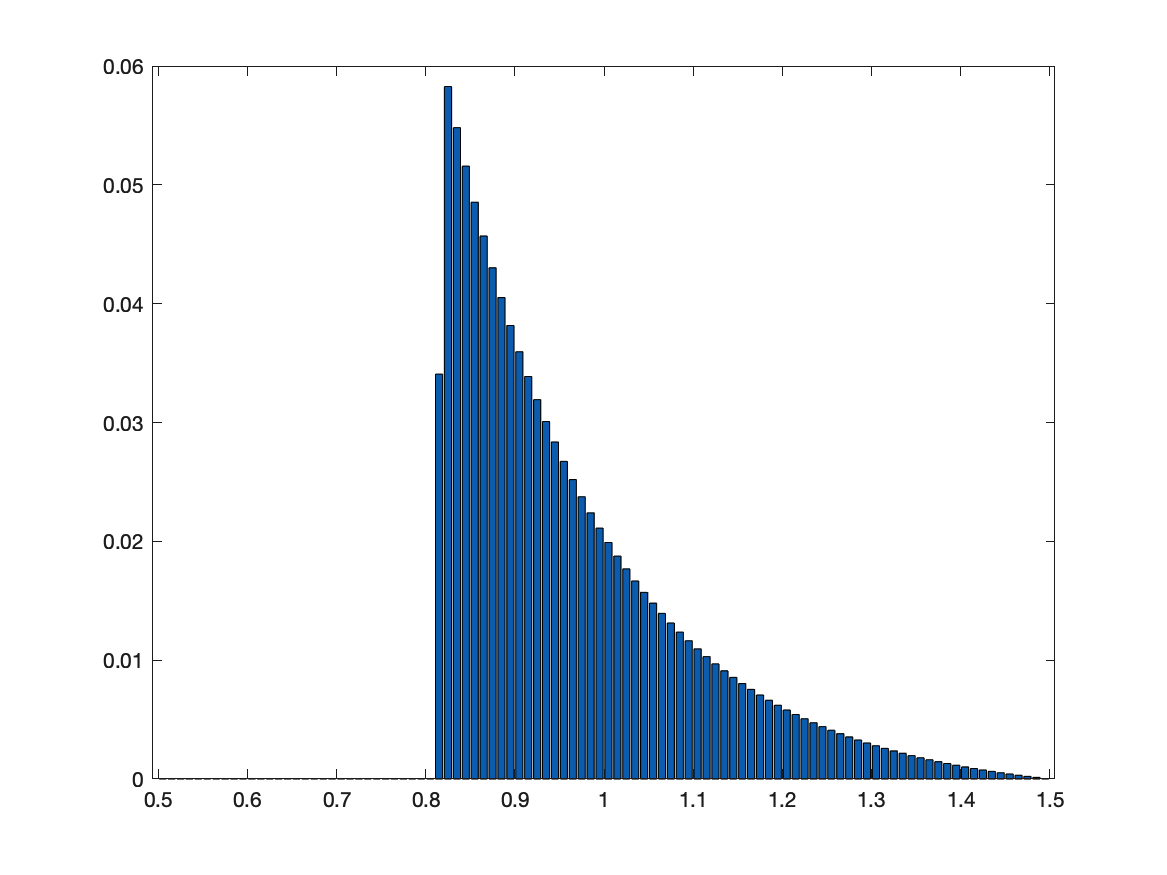}} &
   \subfloat[$p=-1, q =  2, r = 1, \beta = 0.4$ \label{fg:tri1_p100q200r100}]{%
  \includegraphics[width=0.3\textwidth]{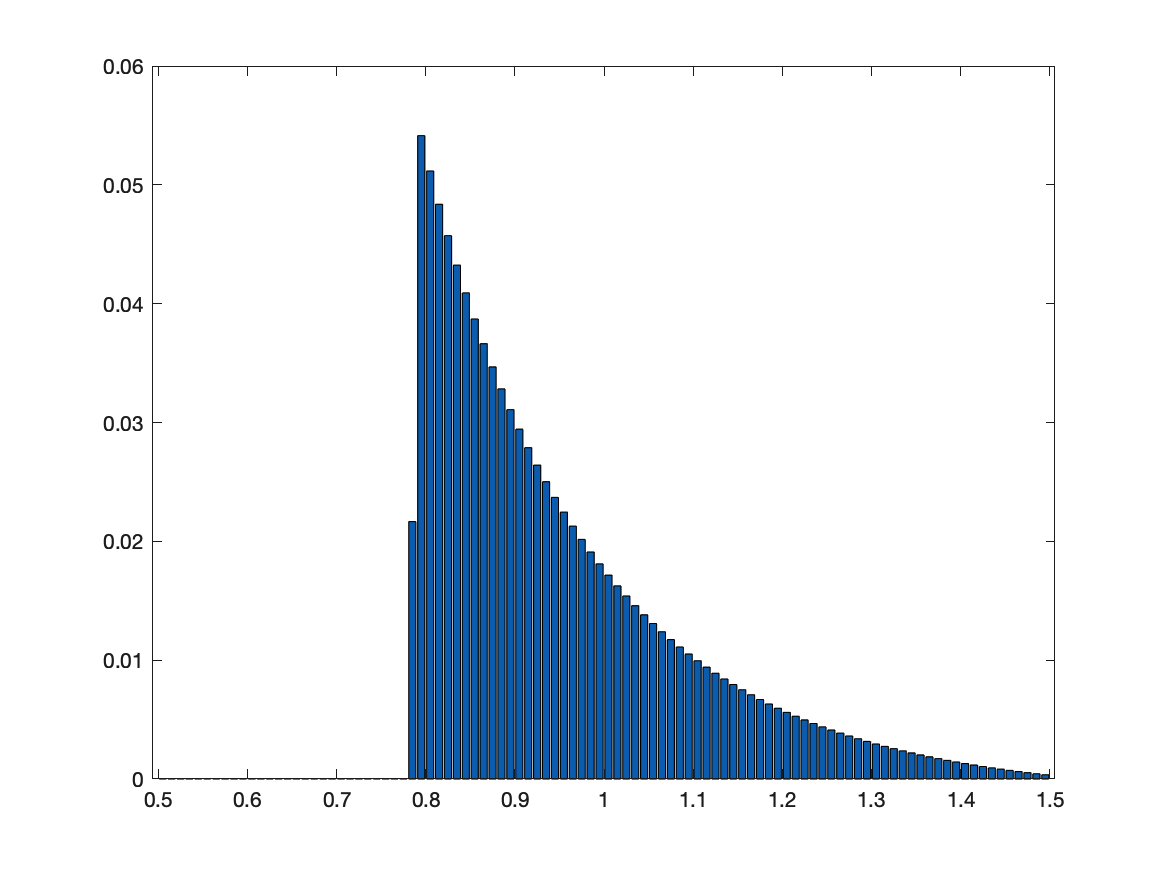}}\\
      \subfloat[$p=-1, q =  2, r = 0.5, \beta = 0.2$ \label{fg:tri1_dist_beta02}]{%
  \includegraphics[width=0.3\textwidth]{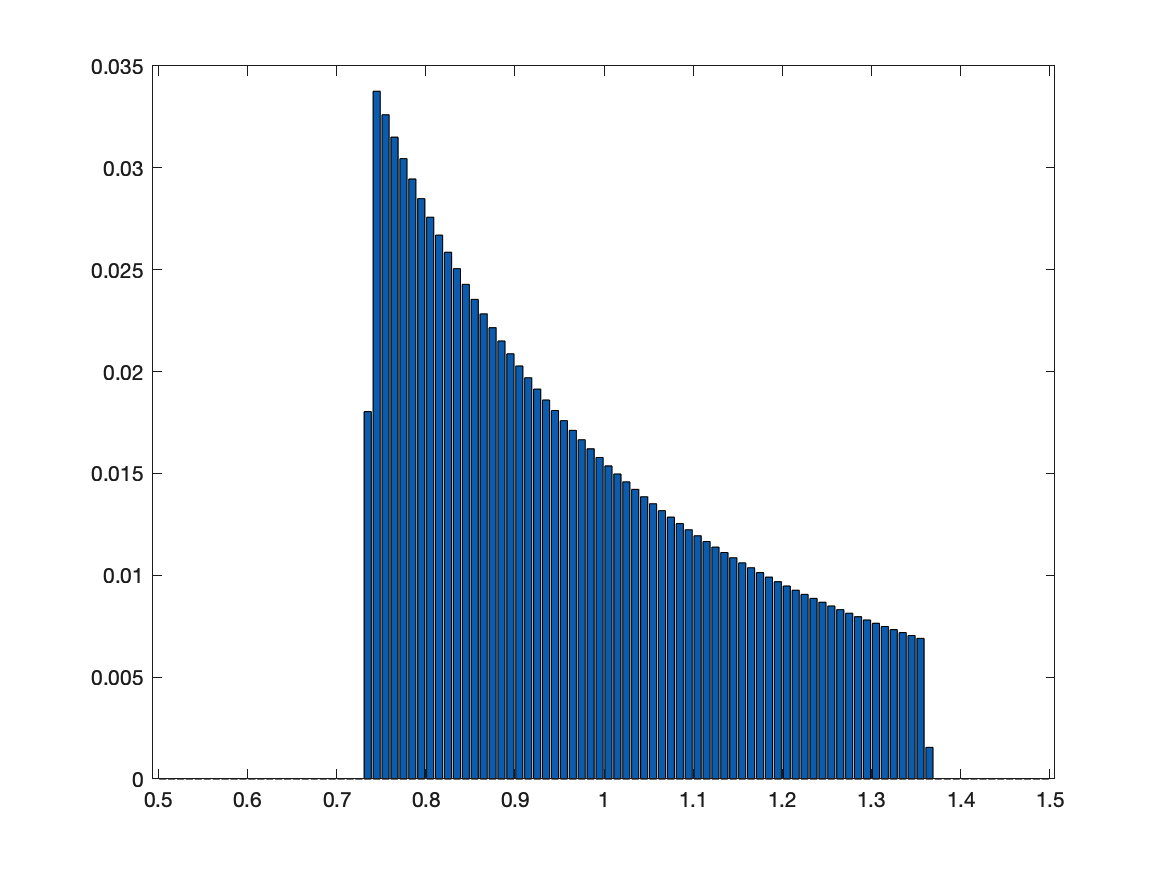}}&
  \subfloat[$p=-1, q =  2, r = 0.5, \beta = 0.6$ \label{fg:tri1_dist_beta06}]{%
  \includegraphics[width=0.3\textwidth]{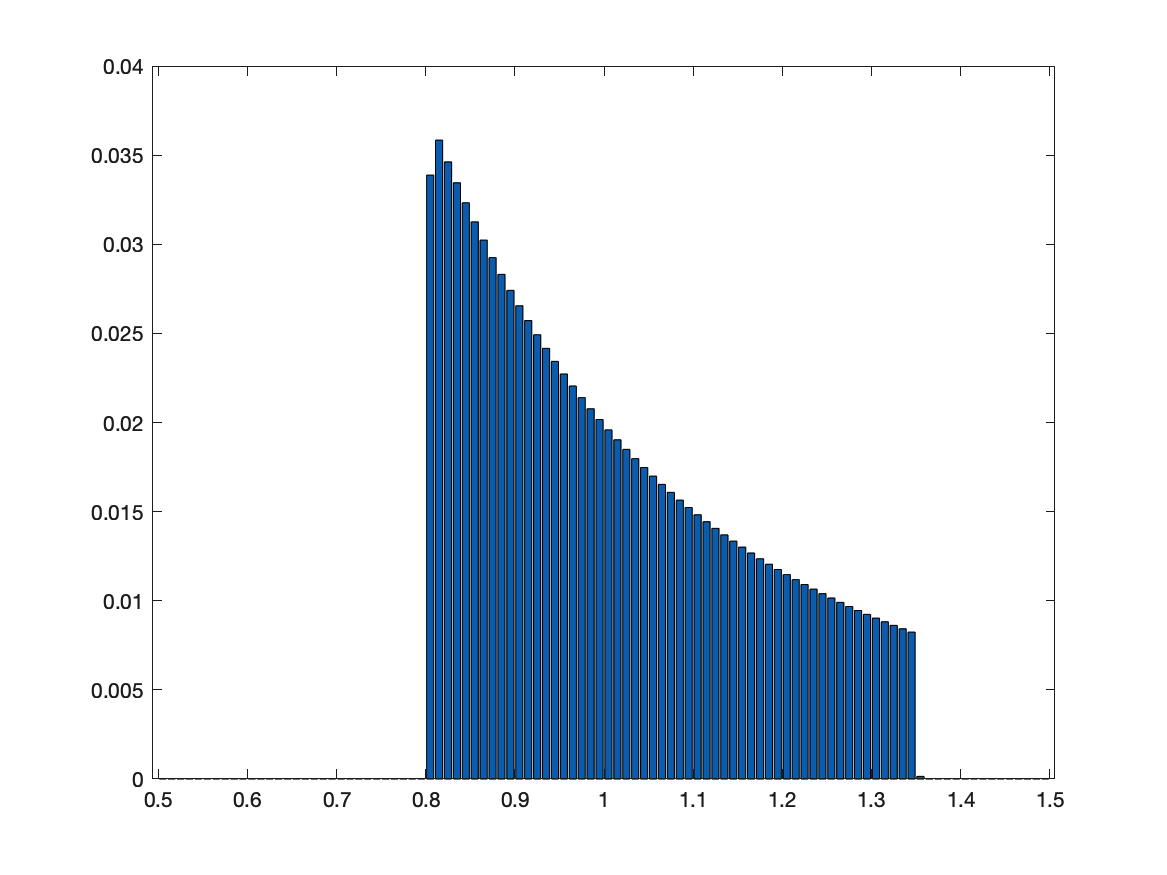}} &
   \subfloat[$p=-1, q =  2, r = 0.5, \beta = 0.8$ \label{fg:tri1_dist_beta08}]{%
  \includegraphics[width=0.3\textwidth]{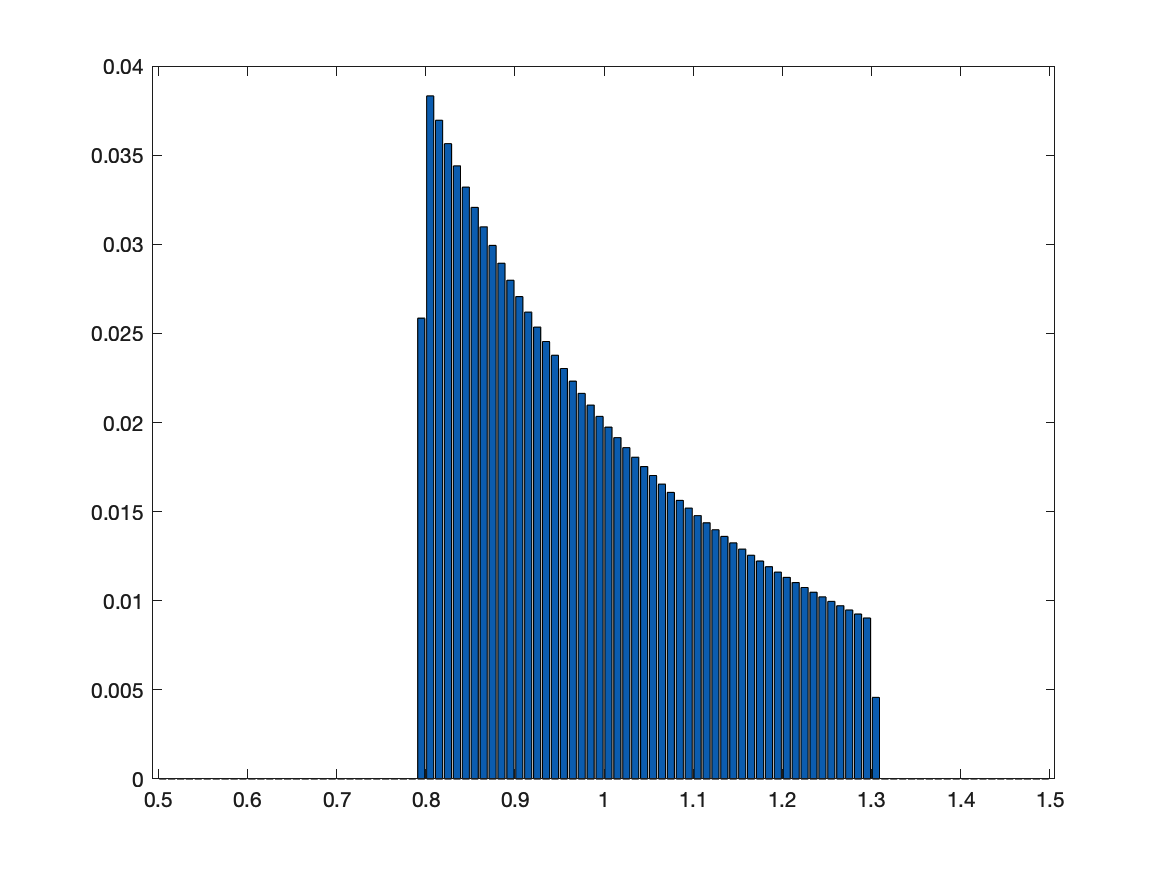}}
  \end{tabular}
  \label{fg:tri1_mu_distributions}
 \end{figure}

  \begin{figure}
\caption{The equilibrium distributions for different parametric setups with a  triangular distribution with an increasing density}
    \centering
    \begin{tabular}{ccc}
       \subfloat[$p=-1, q =  2, r = 0.5, \beta = 0.4$ \label{fg:tri2_p100q200r050}]{%
  \includegraphics[width=0.3\textwidth]{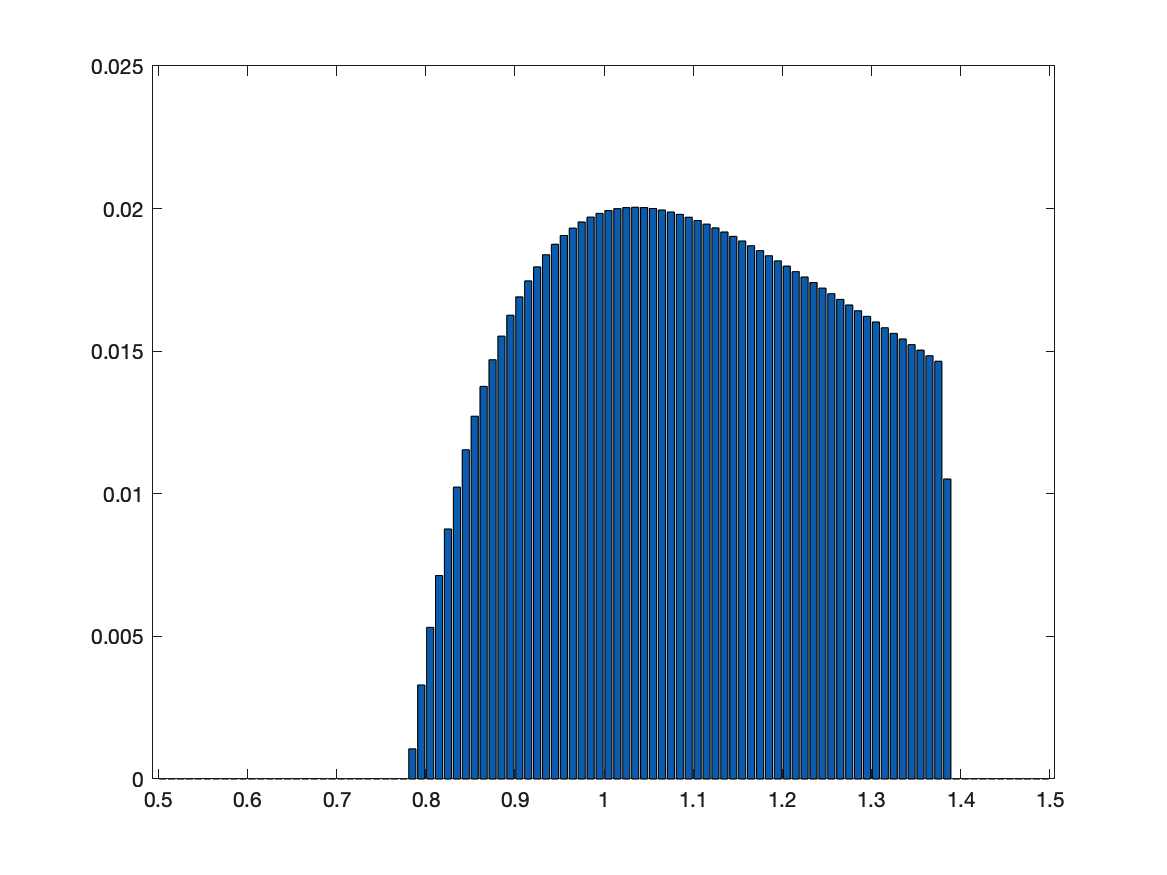}}&
  \subfloat[$p=-0.6, q =  2,r = 0.5, \beta = 0.4$ \label{fg:tri2_p060q200r050}]{%
  \includegraphics[width=0.3\textwidth]{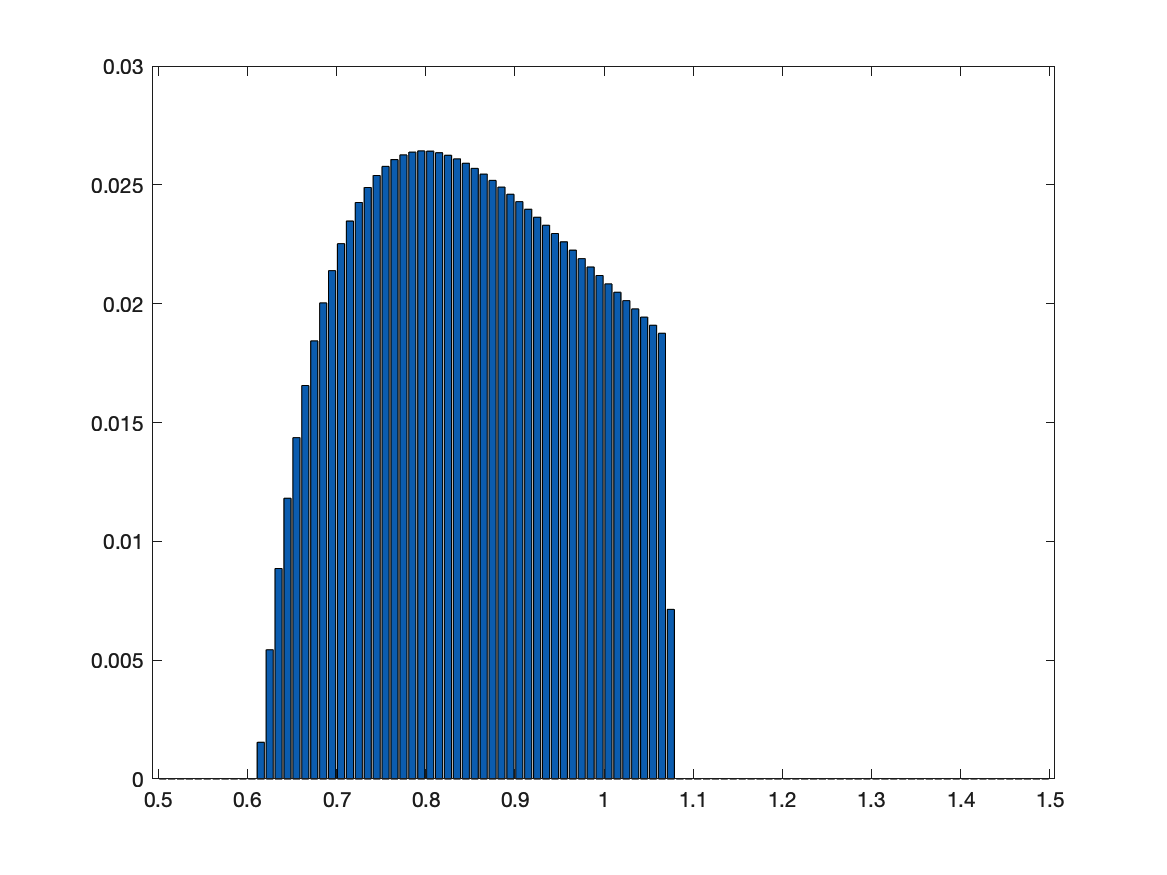}} &
  \subfloat[$p=-0.2, q =  2,r = 0.5, \beta = 0.4$ \label{fg:tri2_p020q200r050}]{%
  \includegraphics[width=0.3\textwidth]{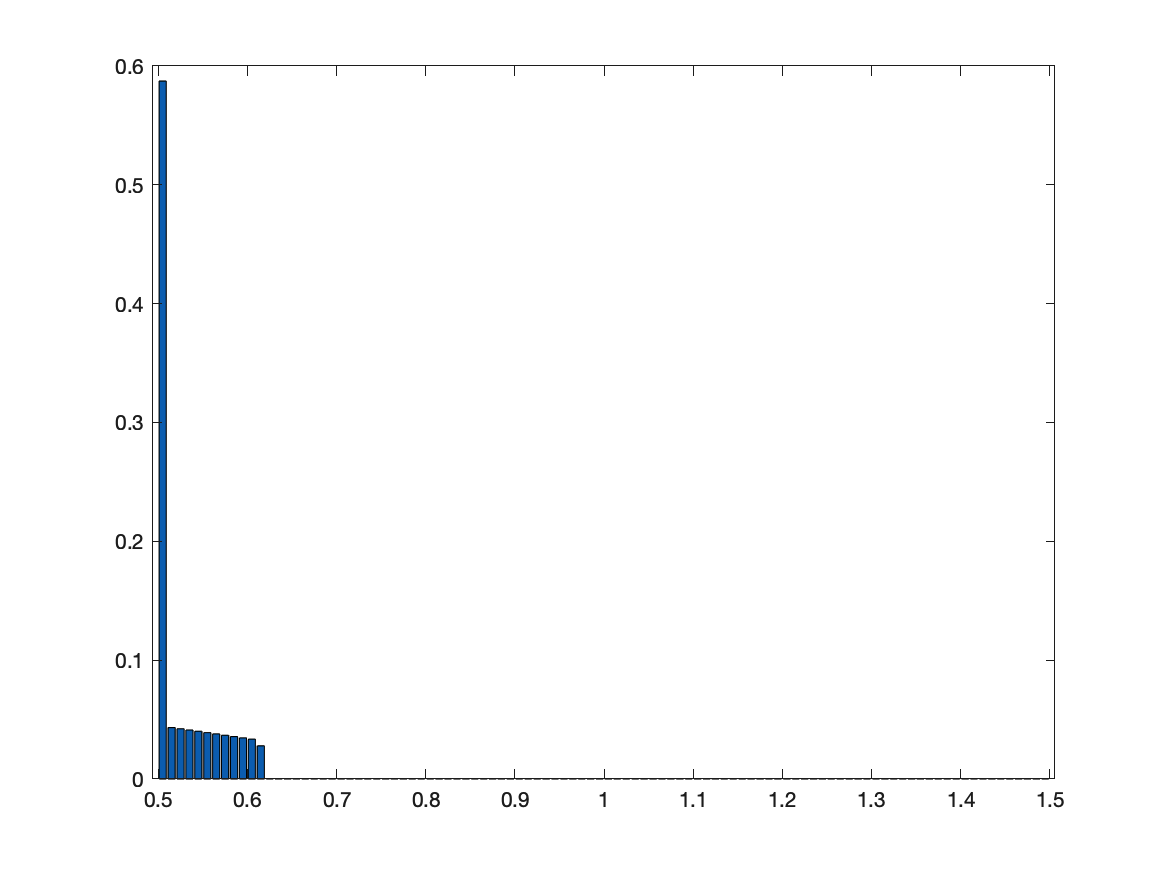}}\\
    \subfloat[$p=-1, q =  3, r = 0.5, \beta = 0.4$\label{fg:tri2_p100q300r050}]{%
  \includegraphics[width=0.3\textwidth]{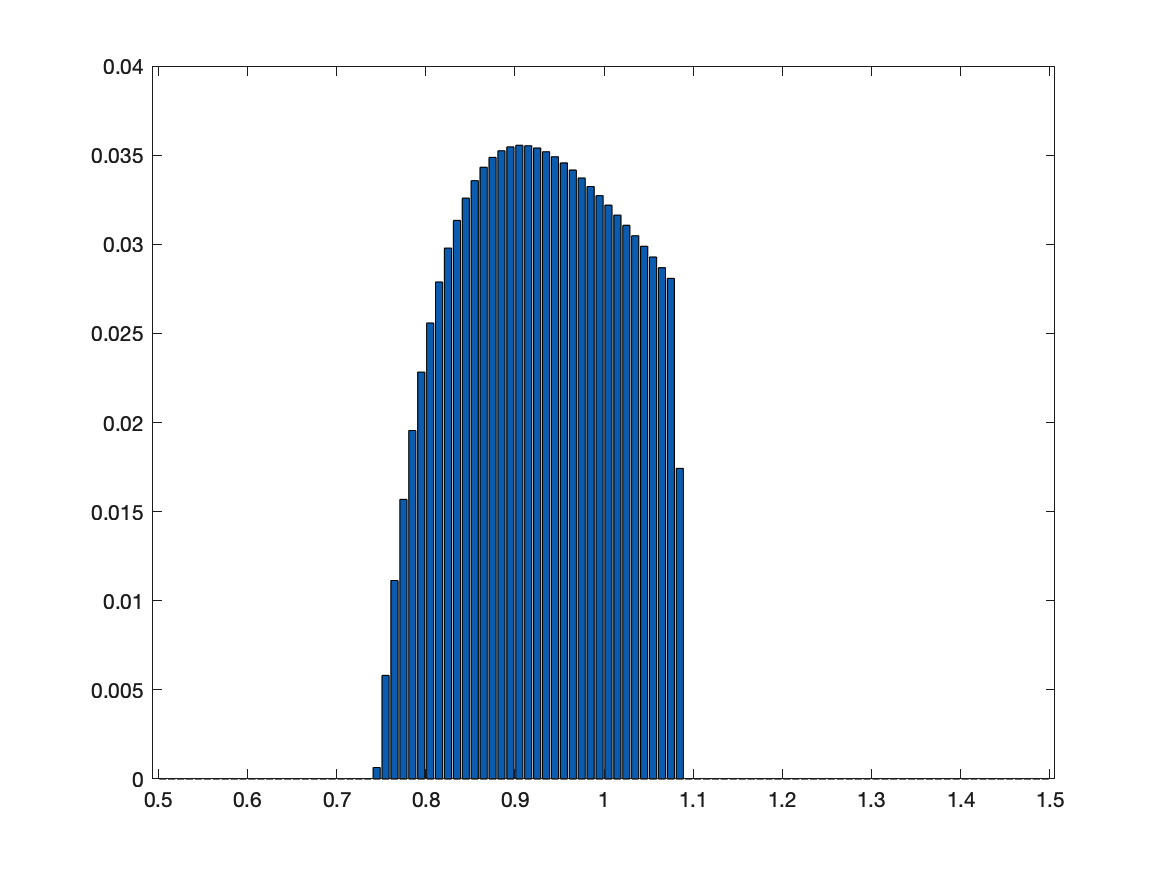}}&
  \subfloat[$p=-1, q =  4, r = 0.5, \beta = 0.4$ \label{fg:tri2_p100q400r050}]{%
  \includegraphics[width=0.3\textwidth]{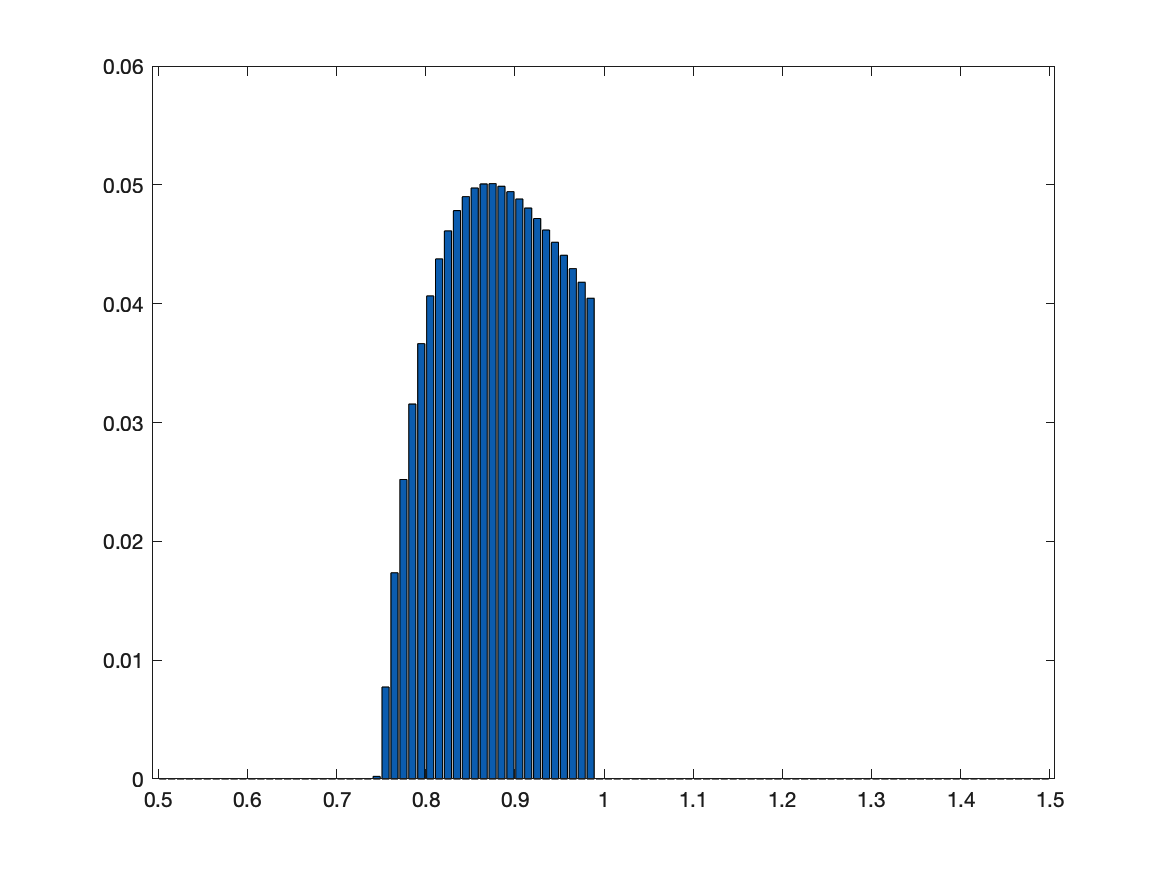}} &
   \subfloat[$p=-1, q =  5, r = 0.5, \beta = 0.4$ \label{fg:tri2_p100q500r050}]{%
  \includegraphics[width=0.3\textwidth]{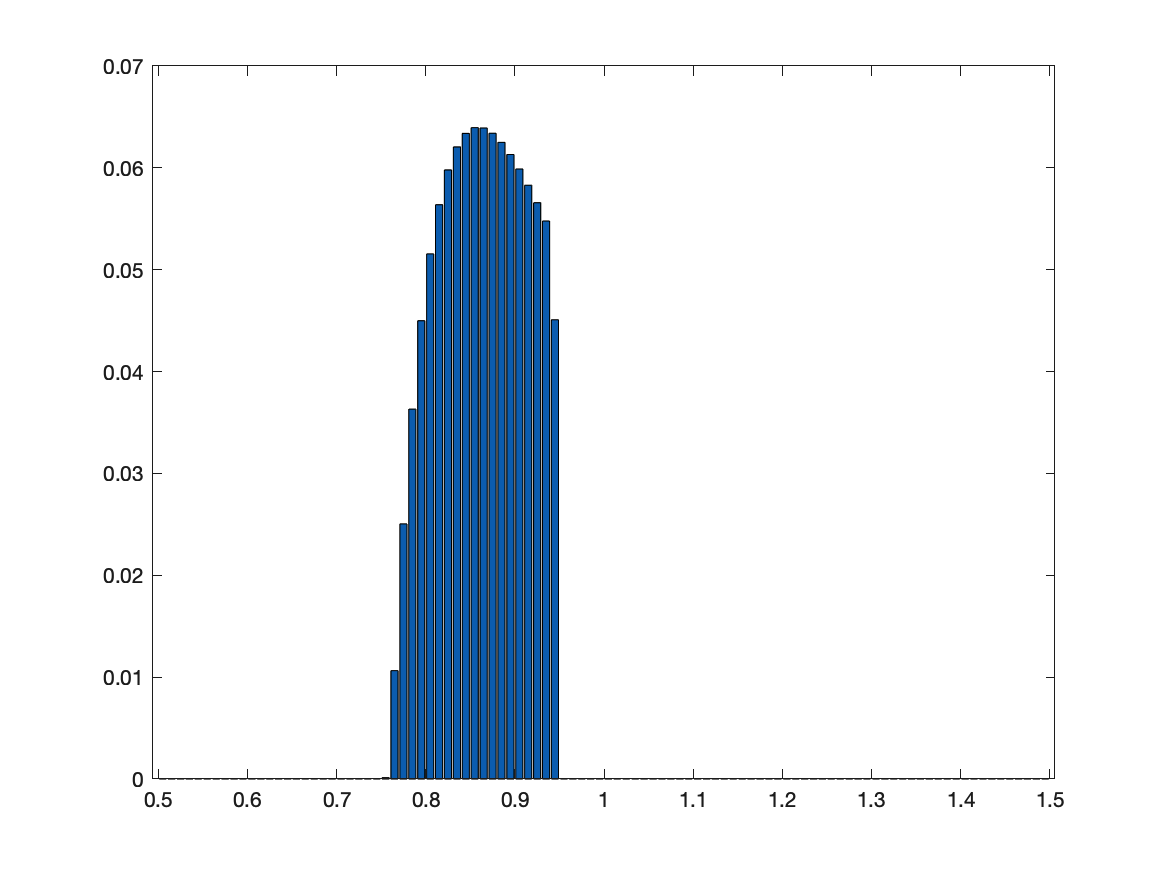}}\\
    \subfloat[$p=-1, q =  2, r = 0.25, \beta = 0.4$ \label{fg:tri2_p100q200r025}]{%
  \includegraphics[width=0.3\textwidth]{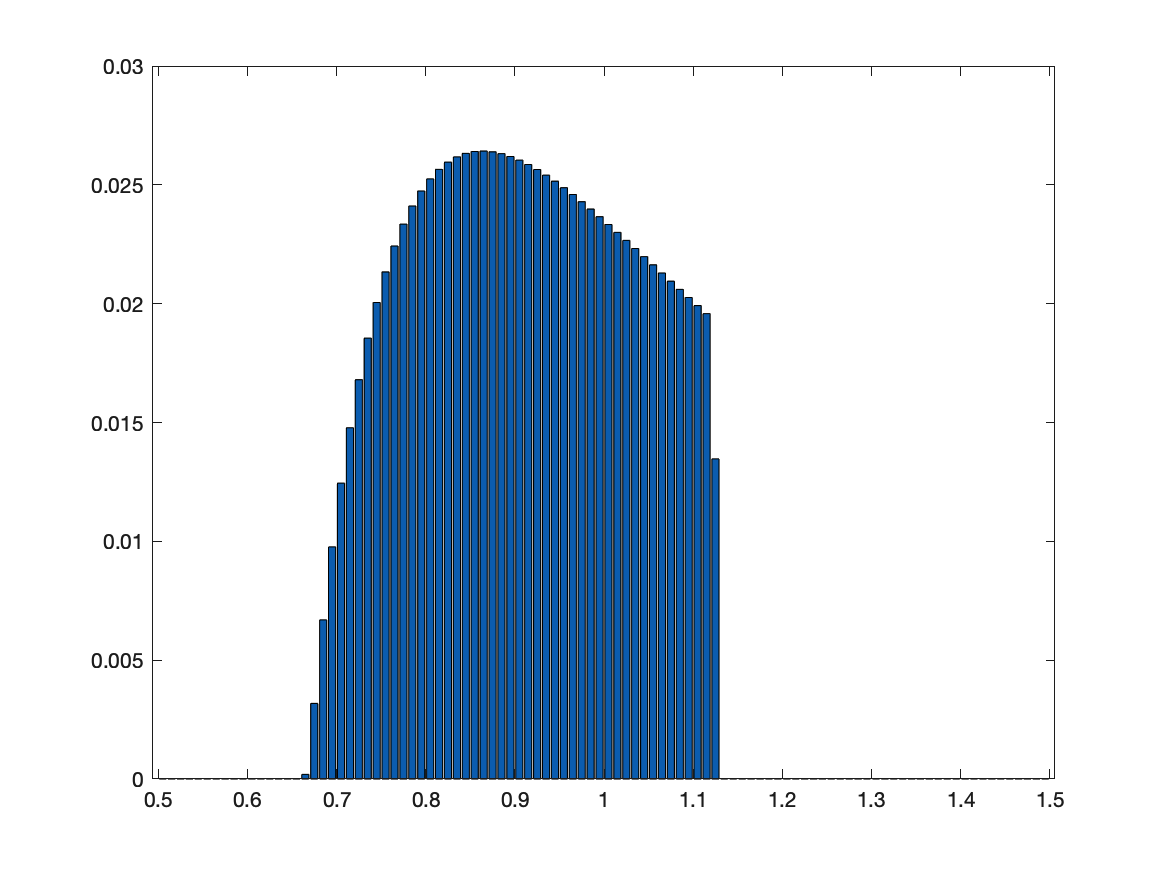}}&
  \subfloat[$p=-1, q =  2, r = 0.75, \beta = 0.4$ \label{fg:tri2_p100q200r075}]{%
  \includegraphics[width=0.3\textwidth]{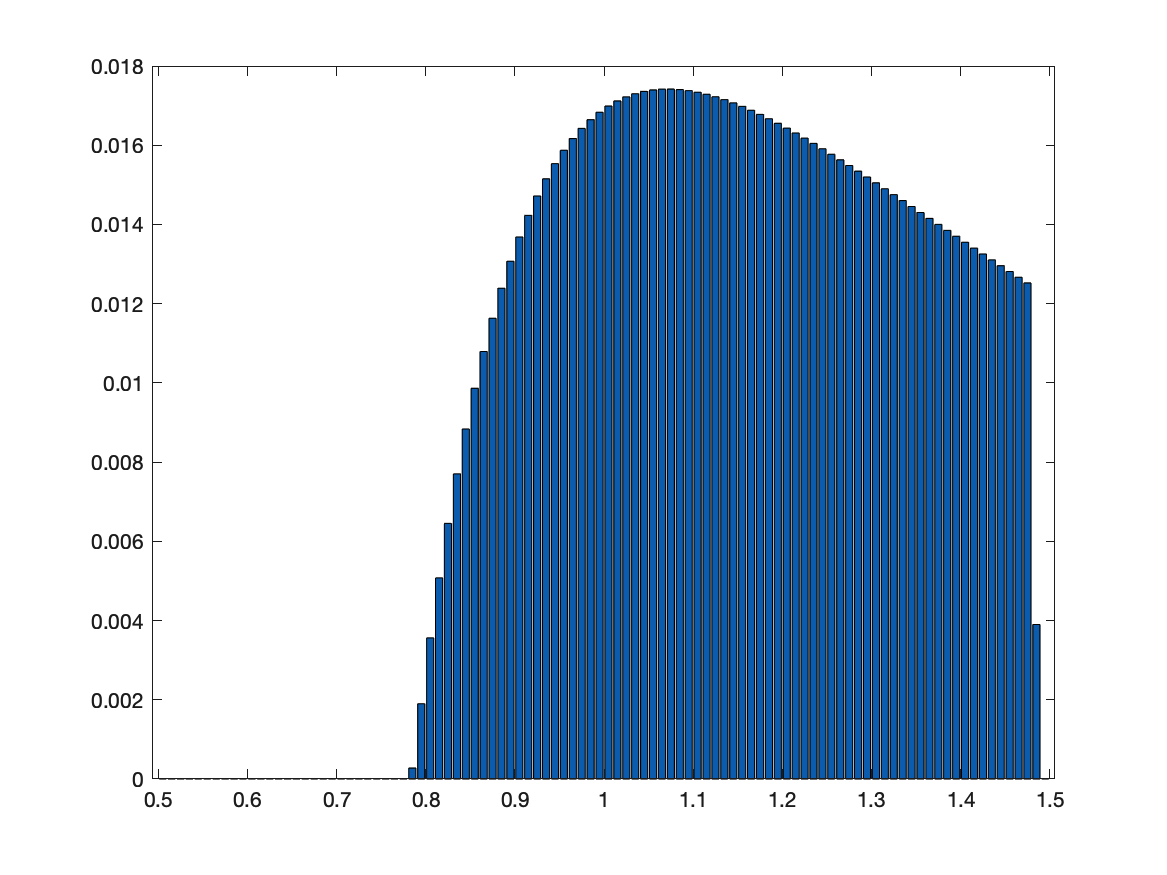}} &
   \subfloat[$p=-1, q =  2, r = 1, \beta = 0.4$ \label{fg:tri2_p100q200r100}]{%
  \includegraphics[width=0.3\textwidth]{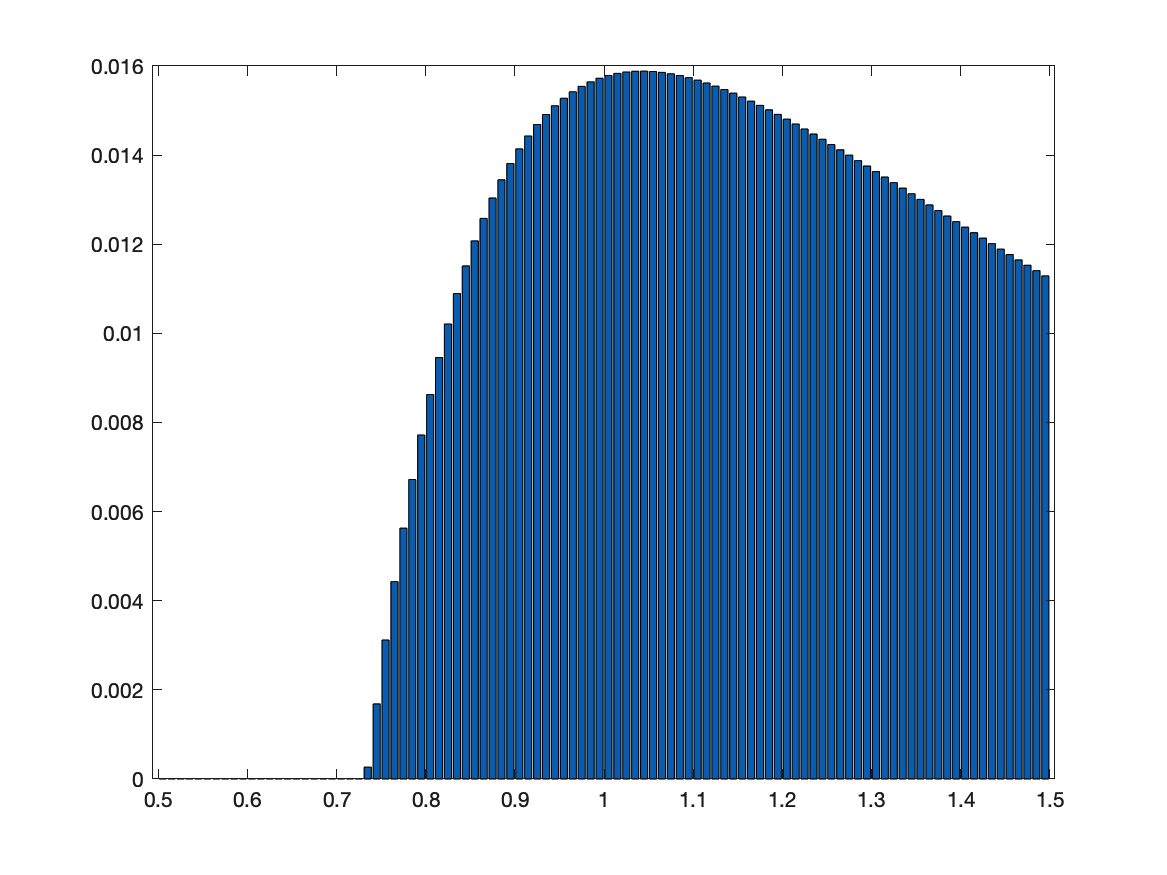}}\\
      \subfloat[$p=-1, q =  2, r = 0.5, \beta = 0.2$ \label{fg:tri2_dist_beta02}]{%
  \includegraphics[width=0.3\textwidth]{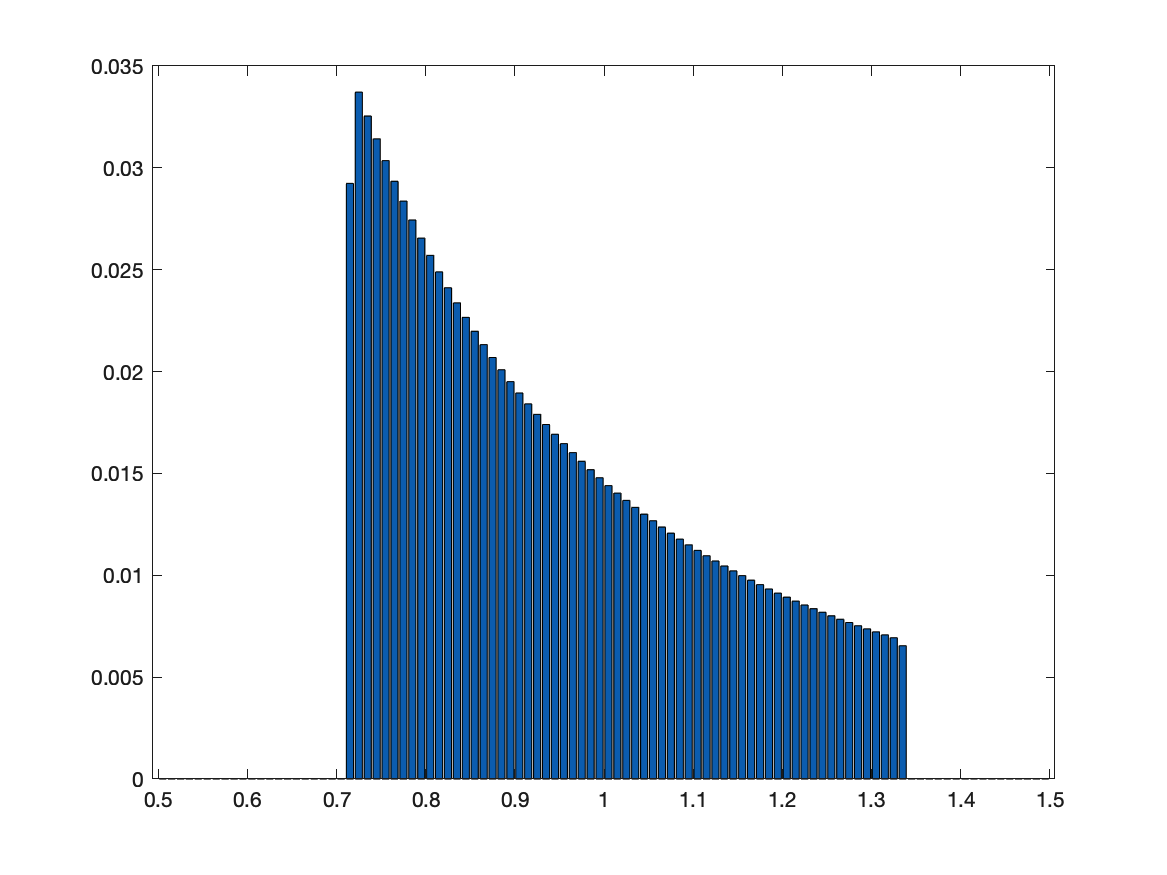}}&
  \subfloat[$p=-1, q =  2, r = 0.5, \beta = 0.6$ \label{fg:tri2_dist_beta06}]{%
  \includegraphics[width=0.3\textwidth]{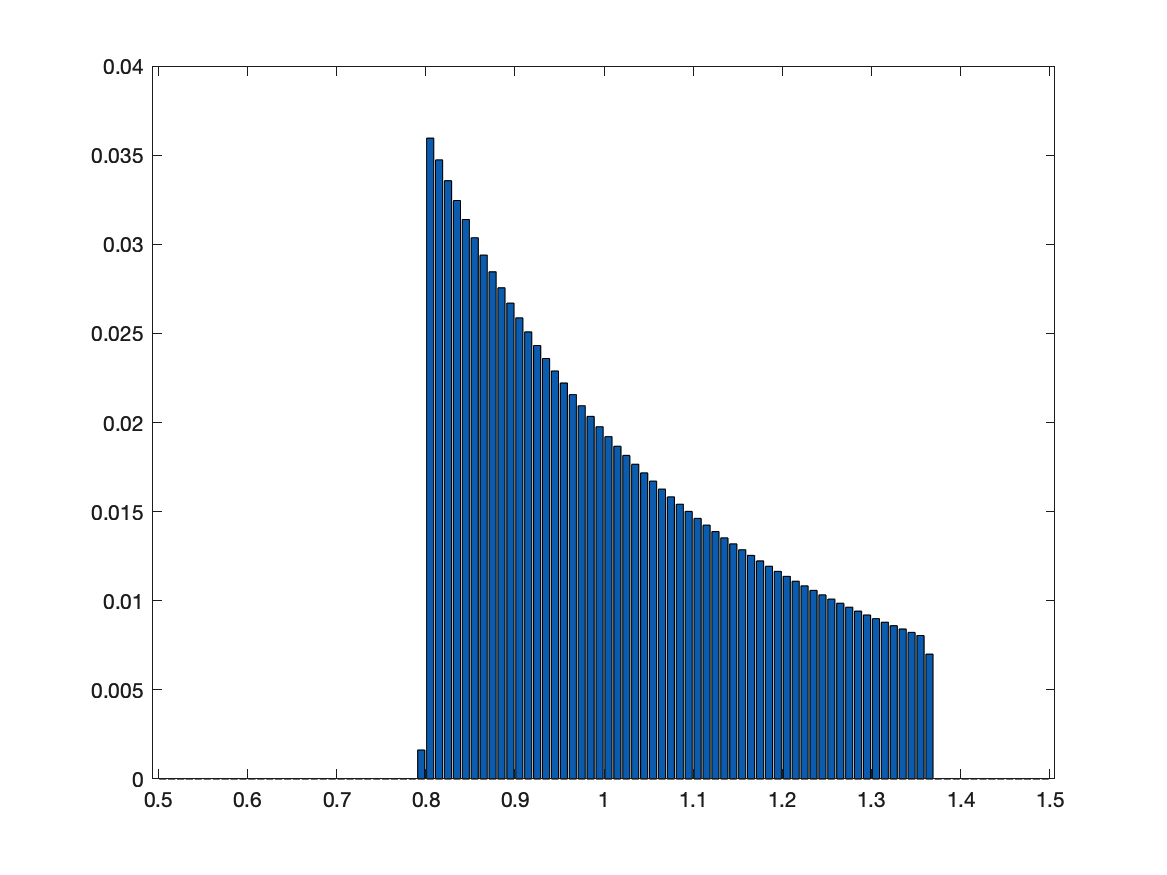}} &
   \subfloat[$p=-1, q =  2, r = 0.5, \beta = 0.8$ \label{fg:tri2_dist_beta08}]{%
  \includegraphics[width=0.3\textwidth]{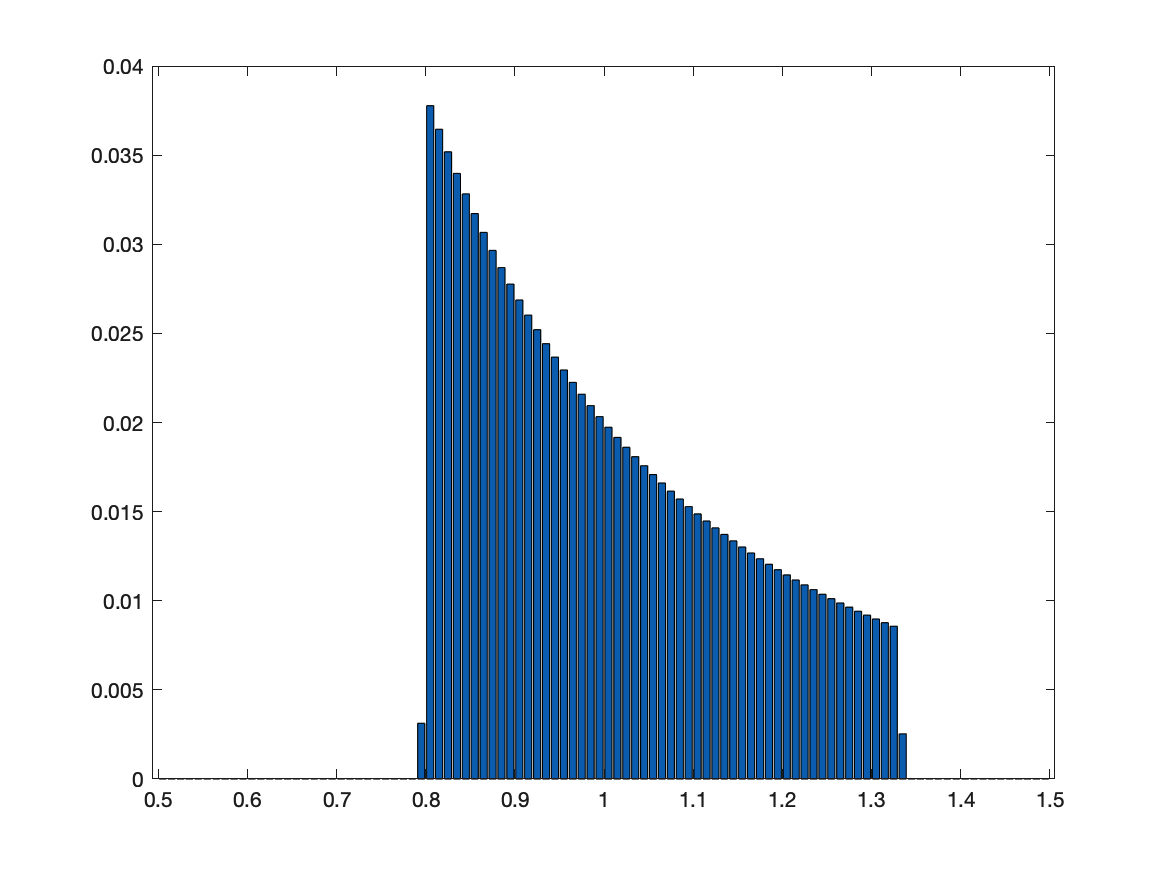}}
  \end{tabular}
  \label{fg:tri2_mu_distributions}
 \end{figure}

  \bibliographystyle{abbrv}           
  \bibliography{final_bibliography}

% \bib, bibdiv, biblist are defined by the amsrefs package.
\begin{bibdiv}
\begin{biblist}

\bib{Allon2010}{article}{
      author={Allon, G},
      author={Gurvich, I},
       title={{Pricing and dimensioning competing large-scale service
  providers}},
        date={2010},
     journal={Manufacturing and Service Operations Management},
      volume={12},
      number={3},
       pages={449\ndash 469},
}

\bib{Armony2005}{article}{
      author={Armony, M},
       title={{Dynamic routing in large-scale service systems with
  heterogeneous servers}},
        date={2005},
     journal={Queueing Systems},
      volume={51},
      number={3-4},
       pages={287\ndash 329},
}

\bib{armony2021capacity}{article}{
      author={Armony, Mor},
      author={Roels, Guillaume},
      author={Song, Hummy},
       title={Capacity choice game in a multi-server queue: Existence of a nash
  equilibrium},
        date={2021},
     journal={Naval Research Logistics},
      volume={68},
      number={5},
       pages={663\ndash 678},
}

\bib{atarschwartzshaki11}{article}{
      author={Atar, R.},
      author={Shaki, Y.Y.},
      author={Shwartz, A.},
       title={A blind policy for equalizing cumulative idleness},
        date={2011},
     journal={Queueing Systems},
      volume={67},
       pages={275\ndash 293},
}

\bib{Atar2008}{article}{
      author={Atar, R},
       title={{Central Limit Theorem for a Many-Server Queue with Random
  Service Rates}},
        date={2008},
     journal={The Annals of Applied Probability},
      volume={18},
      number={4},
       pages={1548\ndash 1568},
}

\bib{armward10}{article}{
      author={Armony, M.},
      author={Ward, A.R.},
       title={Fair dynamic routing in large-scale heterogeneous-service
  systems},
        date={2010},
     journal={Operations Research},
      volume={3},
      number={58},
       pages={624\ndash 637},
}

\bib{Bayraktar2019}{article}{
      author={Bayraktar, E},
      author={Budhiraja, A},
      author={Cohen, A},
       title={{Rate Control Under Heavy Traffic with Strategic Servers}},
        date={2019},
     journal={The Annals of Applied Probability},
      volume={29},
      number={1},
       pages={1\ndash 35},
}

\bib{bil99}{book}{
      author={Billingsley, P.},
       title={Convergence of probability measures},
     edition={Second},
      series={Wiley Series in Probability and Statistics: Probability and
  Statistics},
   publisher={John Wiley \& Sons Inc.},
     address={New York},
        date={1999},
}

\bib{bukeqin2019}{article}{
      author={B\"{u}ke, B.},
      author={Qin, W.},
       title={Many-server queues with random service rates in the
  {Halfin-Whitt} regime: A measure-valued process approach},
        date={2023},
     journal={Mathematics of Operations Research},
      volume={48},
      number={2},
       pages={748\ndash 783},
}

\bib{boydvandenberghe2004}{book}{
      author={Boyd, Stephen},
      author={Vandenberghe, Lieven},
       title={Convex optimization},
   publisher={Cambridge University Press},
        date={2004},
}

\bib{BrowneWhitt1995Piecewise}{incollection}{
      author={Browne, Sid},
      author={Whitt, Ward},
       title={Piecewise-linear diffusion processes},
        date={1995},
   booktitle={Advances in queueing},
      series={Probab. Stochastics Ser.},
   publisher={CRC, Boca Raton, FL},
       pages={463\ndash 480},
      review={\MR{1395170}},
}

\bib{CallCentreHelper17}{misc}{
      author={CallCentreHelper.com},
       title={{7 Tricks That Call Centre Employees Play}},
        date={2017},
  url={https://www.callcentrehelper.com/7-tricks-that-call-centre-employees-play-67004.htm},
        note={[Accessed 17-December-2024]},
}

\bib{cohencharash2001}{article}{
      author={Cohen-Charash, Y.},
      author={Spector, P.E.},
       title={The role of justice in organizations: A meta-analysis},
        date={2001},
     journal={Organizational Behavior and Human Decision Processes},
      volume={86},
      number={2},
       pages={278\ndash 321},
}

\bib{colquittetal2001}{article}{
      author={Colquitt, J.A.},
      author={Colon, D.E.},
      author={Wesson, M.J.},
      author={Porter, C.O.L.H.},
      author={Ng, K.~Y.},
       title={Justice at the millenium: A meta-analytic review of 25 years of
  organizational justice research},
        date={2001},
     journal={Journal of Applied Psychology},
      volume={86},
      number={3},
       pages={425\ndash 445},
}

\bib{Cachon2002}{article}{
      author={Cachon, GP},
      author={Harker, PT},
       title={{Competition and outsourcing with scale economies}},
        date={2002},
     journal={Management Science},
      volume={48},
      number={10},
       pages={1314\ndash 1333},
}

\bib{Cachon2007}{article}{
      author={Cachon, GP},
      author={Zhang, F},
       title={{Obtaining fast service in a queueing system via
  performance-based allocation of demand}},
        date={2007},
     journal={Management Science},
      volume={53},
      number={3},
       pages={408\ndash 420},
}

\bib{deloitte2017}{techreport}{
      author={{Deloitte}},
       title={{Global Contact Center Survey}},
 institution={Deloitte},
        date={2017},
}

\bib{Gopalakrishnan_etal2016}{article}{
      author={Gopalakrishnan, R.},
      author={Doroudi, S.},
      author={Ward, A.~R.},
      author={Wierman, A.},
       title={Routing and staffing when servers are strategic},
        date={2016},
     journal={Operations Research},
      volume={64},
      number={4},
       pages={1033\ndash 1050},
}

\bib{Gans2003}{article}{
      author={Gans, N},
      author={Koole, G},
      author={Mandelbaum, A},
       title={Telephone call centers: Tutorial, review, and research
  prospects},
        date={2003},
     journal={Manufacturing and Service Operations Management},
      volume={5},
      number={2},
       pages={79\ndash 141},
}

\bib{Gans2010}{incollection}{
      author={Gans, N},
      author={Liu, N},
      author={Mandelbaum, A},
      author={Shen, H},
      author={Ye, H},
       title={{Service times in call centers: Agent heterogeneity and learning
  with some operational consequences}},
        date={2010},
   booktitle={Borrowing strength: Theory powering applications},
      volume={6},
   publisher={Institute of Mathematical Statistics},
       pages={99\ndash 123},
}

\bib{garnettetal02}{article}{
      author={Garnett, O},
      author={Mandelbaum, A},
      author={Reiman, MI},
       title={{Designing a Call Center with Impatient Customers}},
        date={2002},
     journal={Manufacturing {\&} Service Operations Management},
      volume={4},
      number={3},
       pages={208\ndash 227},
}

\bib{go4cust21}{misc}{
      author={Go4Customer},
       title={{What is Call Avoidance?}},
        date={2021},
         url={https://go4customer.co.uk/glossary/c/what-is-call-avoidance},
        note={[Accessed 17-December-2024]},
}

\bib{gopalakrishnan2021}{misc}{
      author={Gopalakrishnan, R.},
       title={On the architecture of service systems when servers are
  strategic},
        date={2021},
        note={working paper, Smith School of Business, Queen's University},
}

\bib{Gumbel1960}{article}{
      author={Gumbel, H},
       title={Waiting lines with heterogeneous servers},
        date={1960},
     journal={Operations Research},
      volume={8},
      number={4},
       pages={504\ndash 511},
}

\bib{gurwhitt09}{article}{
      author={Gurvich, I.},
      author={Whitt, W.},
       title={Queue-and-idleness-ratio controls in many-server service
  systems},
        date={2009},
     journal={Mathematics of Operations Research},
      volume={34},
       pages={363\ndash 396},
}

\bib{Gilbert1998}{article}{
      author={Gilbert, SM},
      author={Weng, ZK},
       title={{Incentive Effects Favor Nonconsolidating Queues in a Service
  System : The Principal-Agent Perspective}},
        date={1998},
     journal={Management Science},
      volume={44},
      number={12},
       pages={1662\ndash 1669},
}

\bib{hassin_2020}{book}{
      author={Hassin, Refael},
       title={Rational queueing},
   publisher={CRC Press},
        date={2020},
}

\bib{hassin_haviv_2003}{book}{
      author={Hassin, Refael},
      author={Haviv, Moshe},
       title={To queue or not to queue: Equilibrium behavior in queueing
  systems},
   publisher={Kluwer},
        date={2003},
}

\bib{halfinwhitt1981}{article}{
      author={Halfin, S},
      author={Whitt, H},
       title={{Heavy-Traffic Limits for Queues with Many Exponential
  Serverss}},
        date={1981},
     journal={Operations Research},
      volume={29},
      number={3},
       pages={567\ndash 588},
}

\bib{jak86}{article}{
      author={Jakubowski, A.},
       title={On the {S}korokhod topology},
        date={1986},
     journal={Annales de l'institut Henri Poincare (B) Probability and
  Statistics},
      volume={22},
      number={3},
       pages={263\ndash 285},
}

\bib{Kalai1992}{article}{
      author={Kalai, E},
      author={Kamien, MI},
      author={Rubinovitch, M},
       title={{Optimal Service Speeds in a Competitive Environment}},
        date={1992},
     journal={Management Science},
      volume={38},
      number={8},
       pages={1154\ndash 1163},
}

\bib{ptw07}{article}{
      author={Pang, G.},
      author={Talreja, R.},
      author={Whitt, W.},
       title={Martingale proofs of many-server heavy-traffic limits for
  {M}arkovian queues},
        date={2007},
     journal={Probability Surveys},
      volume={4},
       pages={193\ndash 267},
}

\bib{pwc2018}{techreport}{
      author={{PWC}},
       title={{Experience is everything: Here’s how to get it right}},
 institution={PWC},
        date={2018},
        note={Available at \url{pwc.com/future-of-cx}},
}

\bib{reedshaki2015}{article}{
      author={Reed, J.},
      author={Shaki, Y.},
       title={A fair policy for the {$G/GI/N$ Queue with Multiple Pools}},
        date={2015},
     journal={Mathematics of Operations Research},
      volume={40},
      number={3},
       pages={558\ndash 595},
}

\bib{statista2020}{techreport}{
      author={{Statista Research Department}},
       title={{Global call center market size 2020-2027}},
 institution={Statista},
        date={2020},
}

\bib{wararm13}{article}{
      author={Ward, A.R.},
      author={Armony, M.},
       title={Blind fair routing in large-scale service systems with
  heterogeneous customers and servers},
        date={2013},
     journal={Operations Research},
      volume={1},
      number={61},
       pages={228\ndash 243},
}

\bib{zhong2021}{misc}{
      author={Zhong, Y.},
      author={Gopalakrishnan, R.},
      author={Ward, A.},
       title={Behavior-aware queueing : The finite-buffer setting with many
  strategic servers},
        date={2021},
        note={submitted to \emph{Operations Research}},
}

\bib{zhan2019staffing}{article}{
      author={Zhan, Dongyuan},
      author={Ward, Amy~R},
       title={Staffing, routing, and payment to trade off speed and quality in
  large service systems},
        date={2019},
     journal={Operations Research},
}

\end{biblist}
\end{bibdiv}

\end{document}